\documentclass[a4paper, 10pt, twoside]{article}

\usepackage{exscale, fullpage,url}
\usepackage[centertags]{amsmath}
\usepackage{color}
\usepackage{amssymb}
\usepackage{amsthm}
\usepackage{dsfont}
\usepackage[all]{xy}
\usepackage{mathrsfs}
\usepackage{upgreek}
\usepackage{tikz-cd}
\usepackage{hyperref}
\usepackage{graphicx,fp}

\newif\ifpdf
\pdffalse
\pdftrue

\newtheorem{definition}{Definition}[section]
\newtheorem*{theorem-non}{Theorem}
\newtheorem{theorem}[definition]{Theorem}
\newtheorem{proposition}[definition]{Proposition}
\newtheorem{lemma}[definition]{Lemma}
\newtheorem{corollary}[definition]{Corollary}

\newtheorem{remark}[definition]{Remark}

\newtheorem{example}[definition]{Example}
\newtheorem{thm}[definition]{Theorem}
\newtheorem{defn}[definition]{Definition}

\newcommand{\nd}{\noindent}

\newcommand{\dV}{{\mathds V}}
\newcommand{\dR}{{\mathds R}}
\newcommand{\dC}{{\mathds C}}
\newcommand{\dQ}{{\mathds Q}}
\newcommand{\dN}{{\mathds N}}

\newcommand{\dZ}{{\mathds Z}}
\newcommand{\dP}{{\mathds P}}

\newcommand{\dL}{{\mathbb L}}

\newcommand{\bD}{{\mathbb D}}

\newcommand{\mbh}{\mathds{H}}

\newcommand{\cD}{\mathcal{D}}
\newcommand{\cE}{\mathcal{E}}

\newcommand{\cH}{\mathcal{H}}
\newcommand{\cI}{\mathcal{I}}

\newcommand{\cK}{\mathcal{K}}
\newcommand{\cL}{\mathcal{L}}
\newcommand{\cM}{\mathcal{M}}
\newcommand{\cN}{\mathcal{N}}
\newcommand{\cO}{\mathcal{O}}

\newcommand{\cR}{\mathcal{R}}

\newcommand{\SSC}{\scriptscriptstyle}

\DeclareMathOperator{\Spec}{\textup{Spec}\,}

\DeclareMathOperator{\DR}{\mathit{DR}}

\DeclareMathOperator{\ord}{\textup{ord}}

\DeclareMathOperator{\FL}{\textup{FL}}
\DeclareMathOperator{\id}{\textup{id}}
\DeclareMathOperator{\Gr}{\textup{Gr}}
\DeclareMathOperator{\gr}{\textup{Gr}}

\newcommand{\QDM}{\textup{QDM}}

\newcommand{\MF}{\textup{MF}}
\newcommand{\qdeg}{\textup{qdeg}}
\newcommand{\ck}[1]{{\overset{\tiny \;\SSC(\vee)}{#1}}}
\newcommand{\MHM}{\textup{MHM}}

\newcommand{\XSig}{X_{\!\Sigma}}
\newcommand{\KM}{\cK\!\cM^\circ}

\newcommand{\mbc}{\mathds{C}}
\newcommand{\mbd}{\mathbb{D}}
\newcommand{\mbl}{\mathds{L}}
\newcommand{\mbn}{\mathds{N}}
\newcommand{\mbp}{\mathds{P}}
\newcommand{\mbq}{\mathds{Q}}
\newcommand{\mbr}{\mathds{R}}

\newcommand{\mbz}{\mathds{Z}}

\newcommand{\mcd}{\mathcal{D}}

\newcommand{\mcf}{\mathcal{F}}

\newcommand{\mch}{\mathcal{H}}
\newcommand{\mci}{\mathcal{I}}

\newcommand{\mck}{\mathcal{K}}
\newcommand{\mcl}{\mathcal{L}}
\newcommand{\mcm}{\mathcal{M}}
\newcommand{\mcn}{\mathcal{N}}
\newcommand{\mco}{\mathcal{O}}

\newcommand{\mcr}{\mathcal{R}}

\newcommand{\msc}{\mathscr{C}}
\newcommand{\msj}{\mathscr{J}}
\newcommand{\msk}{\mathscr{K}}
\newcommand{\msl}{\mathscr{L}}
\newcommand{\msm}{\mathscr{M}}
\newcommand{\msn}{\mathscr{N}}
\newcommand{\msp}{\mathscr{P}}
\newcommand{\msr}{\mathscr{R}}
\newcommand{\mss}{\mathscr{S}}
\newcommand{\mst}{\mathscr{T}}
\newcommand{\msv}{\mathscr{V}}
\newcommand{\msw}{\mathscr{W}}
\newcommand{\msx}{\mathscr{X}}
\newcommand{\msy}{\mathscr{Y}}

\newcommand{\tJ}{\textup{J}}
\newcommand{\tI}{\textup{I}}
\newcommand{\tK}{\textup{K}}

\newcommand{\tN}{\textup{N}}
\newcommand{\tR}{\textup{R}}
\newcommand{\tS}{\textup{S}}

\newcommand{\ra}{\rightarrow}
\newcommand{\lra}{\longrightarrow}
\newcommand{\p}{\partial}

\newcommand{\ch}{\mbc^{\beta,H}_T}
\newcommand{\pch}{{^p}\mbc^{\beta,H}_T}
\newcommand{\iasu}{{\ck{\cI}_{\!A^s_u}}{}^{\!\!\!\!\!\!\!\!\widetilde{\beta}}}
\newcommand{\siasu}{{\ck{I}_{\!A^s_u}}{}^{\!\!\!\!\!\!\!\!\widetilde{\beta}}}

\newsavebox\foobox
\newcommand\slbox[2]{  \FPdiv{\result}{#1}{57.296}  \FPtan{\result}{\result}  \slantbox[\result]{#2}}\newcommand{\slantbox}[2][30]{        \mbox{        \sbox{\foobox}{#2}        \hskip\wd\foobox
        \pdfsave
        \pdfsetmatrix{1 0 #1 1}        \llap{\usebox{\foobox}}        \pdfrestore
}}
\newcommand\rotslant[3]{\rotatebox{#1}{\slbox{#2}{#3}}}

\begin{document}
\title{Hypergeometric Hodge modules}
\author{Thomas Reichelt and Christian Sevenheck}

\maketitle

\begin{abstract}
We consider mixed Hodge module structures on GKZ-hypergeometric differential systems.
We show that the Hodge filtration on these $\cD$-modules is given by the order filtration,
up to a suitable shift. As an application, we prove a conjecture on
the existence of non-commutative Hodge structures on the reduced quantum $\cD$-module
of a nef complete intersection inside a toric variety.
\end{abstract}

\renewcommand{\thefootnote}{}
\footnote{
The authors are supported by the DFG grants RE 3567/1-1 and SE 1114/5-1.
}

\tableofcontents

\section{Introduction}
\label{sec:Introduction}

In a series of papers Gel'fand, Graev, Kapranov and Zelevinski{\u\i} \cite{GGZ87}, \cite{GKZ89} introduced a system of differential equations which generalize the classical differential
systems satisfied by the hypergeometric functions of
Gau\ss, Appell, Bessel and others. These generalized systems are nowadays called GKZ-systems. The initial data of a GKZ system consists of a $d \times n$ integer matrix and a parameter vector $\beta$. Although the definition of a GKZ system has a combinatorial flavor it was early realized that at least for non-resonant parameter vectors $\beta$ GKZ-systems come from geometry \cite{GKZ1}, i.e. they are isomorphic to a direct image of some twisted  structure sheaf on an algebraic variety.  In \cite{Reich2}, the first-named author has shown that certain GKZ-systems actually carry a much richer structure, namely, they underlie \emph{mixed Hodge modules} in the sense of M.~Saito (see \cite{SaitoMHM}). One of the main goals of this paper is the explicit calculation of the corresponding Hodge filtration on these modules.\\

An important application of GKZ systems is mirror symmetry for weak Fano complete intersections in toric varieties. We have shown
in our previous papers \cite{ReiSe,ReiSe2} how to express variants of the mirror correspondence as an equivalence of differential systems of ``GKZ-type''. However, an important point was left open in these articles: The mirror statements given there actually
involve differential systems (i.e., holonomic $\cD$-modules) with some additional data, sometimes called \emph{lattices}.
These are constructed by a variant of the Fourier-Laplace transformation from regular holonomic \emph{filtered} $\cD$-modules.
The filtration in question is the
Hodge filtration on these modules, but a concrete description of it is missing in \cite{ReiSe, ReiSe2}. As a consequence,
the most important Hodge theoretic property of the differential system entering in the mirror correspondence
was formulated only as a conjecture in \cite{ReiSe2} (conjecture 6.15): the so-called \emph{reduced quantum $\cD$-module},
which governs certain Gromov-Witten invariants of nef complete intersections in toric varieties conjecturally underlies
a \emph{variation of non-commutative Hodge structures}. We prove this conjecture here (see Theorem \ref{thm:ncHodge2}),
it appears as a consequence of the main result of the present paper, which determines the Hodge filtration
on the GKZ-systems. More precisely, as GKZ-systems are defined as cyclic quotients of the Weyl algebra, we obtain (Theorem \ref{thm:HodgeGKZ}) that this Hodge filtration is given by the filtration induced
from the order of differential operators up to a suitable shift.
In some sense, this finishes the Hodge theoretic study of mirror symmetry for this class of varieties since we can now express the mirror correspondence as an isomorphism of non-commutative Hodge structures, which are the correct generalization of ordinary Hodge structures in the case where the underlying differential equations acquire irregular singularities, as it is the case
for the quantum $\cD$-module of weak Fano  varieties (in contrast to the Calabi-Yau case).
\\

Another application of our main result, which can be found in the two papers \cite{CDS} and\cite{SevCastReich}, is the calculation of the so-called irregular Hodge filtration on certain $1$-dimensional classical hypergeometric modules. The irregular Hodge filtration has been introduced by C.~Sabbah (see \cite{Sa15}) in order to attach Hodge-type numerical invariants (namely dimensions of graded parts of a filtration) to differential systems acquiring irregular singularities. In geometric situations, like those where regular functions on quasi-projective manifolds are studied as Landau-Ginzburg models of certain quantum cohomology theories, the irregular Hodge filtration has a concrete description using certain logarithmic de Rham complexes, as has been shown by Esnault, Sabbah and Yu  (\cite{SaEsYu}), see also the discussion in \cite{KKP2}). Classical hypergeometric systems are also the most prominent example of \emph{rigid} $\cD$-modules (see \cite{KatzN1})), so the computation of these invariants for them is of particular interest. It turns out that confluent classical hypergeometric modules (these are precisely those with irregular singularities) are obtained from GKZ-systems by a dimensional reduction and a Fourier-Laplace transformation. Using our result (i.e., Theorem \ref{thm:HodgeGKZ}), one can explicitly describe the irregular Hodge filtration (and give  closed formulas for irregular Hodge numbers) of certain such systems
(see \cite[Theorem 4.7]{CDS} and \cite[Theorem 5.9]{SevCastReich} for more details).
\\

Let us give a short overview on the content of this article and the precise statements of the main results. Notice that section \ref{sec:Example} below provides a detailed description of these
results and parts of their proofs for rather simple example, which is related to the quantum $\cD$-module of $\dP^1$. We advise the reader to go through this example in order to understand the strategy of the proof in the general case in the main body of this article.

The main result of this paper is obtained in two major steps, which occupy section \ref{sec:HodgeOnTorusEmbed} resp. section \ref{sec:Radon}. First we study embeddings of tori into affine spaces given by a monomial map $h_A: T=(\dC^*)^d \hookrightarrow \dC^n; (t_1,\ldots,t_d)
\mapsto (\underline{t}^{\underline{a}_1},\ldots,\underline{t}^{\underline{a}_n})$, where $\underline{t}^{\underline{a}_i}=\prod_{k=1}^d t_j^{a_{ki}}$
and where the matrix of columns $A=(\underline{a}_i)_{i=1,\ldots,n}
\in \textup{M}(d\times n,\dZ)$ satisfies certain combinatorial properties related to the geometry
of the semi-group ring $\dC[\dN A]$. Consider the twisted structure sheaf $\mco_T^\beta := \mcd_T/\mcd_T(\p_{t_1}t_1+\beta_1, \ldots, \p_{t_d} t_d +\beta_d)$. It was shown in \cite{SchulWalth2} that the direct image $h_{A +} \mco_T^\beta$ has an explicit description as a Fourier-Laplace transformed GKZ-system $\check{\mcm}^\beta_A$ (cf. Definition \ref{def:SWGarbe}) in case the parameter vector $\beta$ is not strongly resonant (cf. Definition \ref{def:sRes}). We consider the corresponding direct image $\mch^0( h_{A*} {^p}\mbc^{H,\beta}_T)$ in the category of complex mixed Hodge modules,
and calculate its Hodge filtration (cf. Theorem \ref{thm:hpot}) in case that the parameter vector $\beta$ lies in the set of admissible parameters $\mathfrak{A}_A$ (cf. Formula \eqref{eq:defAdm}). More precisely, this first result can be stated as follows.
\begin{theorem-non}[Theorem \ref{thm:hpot} below]
For $\beta \in \mathfrak{A}_A$ the Hodge filtration on $\check{\mcm}^\beta_A$ is equal to the order filtration shifted by $n-d$, i.e.
\[
F_{p+(n-d)}^H \check{\mcm}^\beta_A = F_{p}^{ord} \check{\mcm}^\beta_A\, .
\]
\end{theorem-non}

 If the matrix $A$ we started with satisfies
a homogeneity property, then the underlying $\cD$-module of this mixed Hodge module is a (monodromic) Fourier-Laplace transformation
of the GKZ-system we are interested in. It should be noticed that Theorem \ref{thm:hpot} is of independent interest, its statement is related to
the description of the Hodge filtration on various cohomology groups associated with singular toric varieties. We plan to discuss this question
in a subsequent work.
The main point in Theorem \ref{thm:hpot} is to determine the canonical $V$-filtration on the direct image module
along the boundary divisor $\overline{im(h_A)}\backslash im(h_A)$, i.e., the calculation of some Bernstein polynomials.\\

The second step, carried out in section \ref{sec:Radon} consists in studying the behavior of a twisted structure sheaf on a torus under a certain integral transformation which generalizes the Radon transformation in \cite{Reich2}.
It is well-known (see \cite{Brylinski} and \cite{AE}) that there is a close relation
between the Fourier-Laplace transformation and the Radon transformation for holonomic $\cD$-modules, however, the former  does not a priori preserve
the category of mixed Hodge modules whereas the latter does. This fact is one of the main points in the proof of the existence of a mixed Hodge module
structure on GKZ-systems in \cite{Reich2}. We calculate the behaviour of the Hodge filtration under the various functors entering into
the integral transformation functor, an essential tool for these calculations is the so called \emph{Euler-Koszul-complex} (or some variants of it) as
introduced in \cite{MillerWaltherMat}. We finally get the following statement for the Hodge filtration on the GKZ system $\mcm^{\widetilde{\beta}}_{\widetilde{A}}$ (cf. Theorem \ref{thm:HodgeGKZ}). We call a matrix homogenous if all
of its columns lie in an affine hyperplane. Moreover, an integer matrix is called normal, if the semi-group generated by its columns is the intersection of the cone generated by these
columns with the lattice generated by them (see Formula \eqref{eq:Bnormal} below).
\begin{theorem-non}[Theorem \ref{thm:HodgeGKZ} below]
Let $\widetilde{A}$ be a homogeneous, normal $(d+1) \times (n+1)$ integer matrix, $\widetilde{\beta} \in \mathfrak{A}_{\widetilde{A}}$ and $\beta_0 \in (-1,0]$. Then the GKZ-system $\mcm^{\widetilde{\beta}}_{\widetilde{A}}$ carries the structure of a mixed Hodge module whose Hodge filtration is given by the shifted order filtration, i.e.
\[
(\mcm^{\widetilde{\beta}}_{\widetilde{A}}, F^H_\bullet) \simeq (\mcm^{\widetilde{\beta}}_{\widetilde{A}}, F^{ord}_{\bullet +d}).
\]
\end{theorem-non}

The second last part of section \ref{sec:Radon} deals with the Hodge module structure on the holonomic dual GKZ-system (which, under the assumptions on the inital data, is also a GKZ-system). The last subsection of section \ref{sec:Radon} explains how one can deduce from our main result
the computation of Batyrev (see \cite{Bat4}) of the Hodge filtration on the relative cohomology of
smooth affine hypersurfaces in algebraic tori.

Finally, in section \ref{sec:LGModel} we explain the above mentioned conjecture from \cite{ReiSe2} and show how its proof can be deduced from our main result (cf. Theorem  \ref{thm:ncHodge2}). To formulate this result, consider a smooth projective toric variety $\XSig$ and a split globally generated vector bundle $\cE$ satisfying some positivity assumptions (ensuring, in particular, that the subvariety $Y:=s^{-1}(0)$ for a generic section $s\in\Gamma(\XSig,\cE)$
is smooth weak Fano). Then we have the reduced quantum $\mcd$-module  $\overline{\QDM}(\XSig,\cE)$ (see \cite[Section 4.1]{ReiSe2})
and the mirror map $\textup{Mir}$ (see Theorem \ref{thm:MirMM17} below as well as \cite[Theorem 6.9]{ReiSe2} and the references therein). We get the following result.
\begin{theorem-non}[Theorem  \ref{thm:ncHodge2} below]
Let $\XSig$ be a $k$-dimensional smooth, toric variety, let $\cL_1,\ldots,\cL_l$  be globally generated line bundles
such that $-K_{\XSig}-\cE$ is nef, where $\cE=\oplus_{j=1}^l\cL_j$, $\cL_j$ being ample for $j=1,\ldots,l$. Then the smooth $\cR_{\dC_z\times\KM}$-module
$(\id_{\dC_z}\times \textup{Mir})^*\overline{\QDM}(\XSig,\cE)$
underlies a variation of pure polarized non-commutative Hodge structures.
\end{theorem-non}

While we were working on this paper, a preprint of T.~Mochizuki (\cite{Mo15}) appeared where \cite[Conjecture 6.15]{ReiSe2} is shown with rather different methods. The arguments in his paper work entirely in the category of mixed twistor modules, using the full strength of \cite{Mo12}. It seems possible from his approach to obtain our main result by considering equivariant twistor modules (see \cite[Section 6.4]{Mo15}). However, this treatment of the regular case in loc.cit. (i.e., the case of mixed Hodge modules) seems to be done in a more restrictive setup, in two ways: First of all, there is an extra assumption (\cite[Assumption 6.51]{Mo15}) excluding the calculation of the Hodge filtration on GKZ-systems for examples coming from local mirror symmetry (however, the mirror symmetry consequence, i.e., the proof of
\cite[Conjecture 6.15]{ReiSe2} is coverd by \cite{Mo15}). On the other hand, there is no discussion on GKZ-systems  (or its twistor versions) for non-zero parameter vectors $\beta$ in loc.cit. in contrast to our main result (i.e., in constrast to Theorem
\ref{thm:HodgeGKZ}, which holds for non-zero parameters satisfying a natural
combinatorial condition). Besides, we feel that the calculation of the Hodge filtration on GKZ-systems which only uses properties of the category of mixed Hodge modules (not passing via twistor modules) may be of independent interest.

To finish this introduction, we introduce some notation and conventions used throughout the paper.
Let $X$ be a smooth algebraic variety over $\mbc$ of dimension $d_X$.
We denote by $M(\mcd_X)$ the abelian category of algebraic left $\mcd_X$-modules on $X$ and denote the abelian subcategory of (regular) holonomic $\mcd_X$-modules by $M_h(\mcd_X)$ (resp. $(M_{rh}(\mcd_X))$.  The full triangulated subcategory in $D^b(\mcd_X)$, consisting of objects with (regular) holonomic cohomology, is denoted by $D^b_{h}(\mcd_X)$ (resp. $D^b_{rh}(\mcd_X)$).

\noindent Let $f: X \ra Y$ be a map between smooth algebraic varieties. Let $M \in D^b(\mcd_X)$ and $N \in D^b(\mcd_Y)$, then we denote by
\[
f_+ M := Rf_*(\mcd_{Y \leftarrow X} \overset{L}{\otimes} M) \quad  \text{resp.}\quad  f^+ M:= \mcd_{X \ra Y} \overset{L}{\otimes} f^{-1}M [d_X - d_Y]
\]
the direct resp. inverse image for $\mcd$-modules.
Recall that the functors $f_+,f^+$ preserve (regular) holonomicity (see e.g., \cite[Theorem 3.2.3]{Hotta}).
\noindent We denote by $\mbd: D^b_h(\mcd_X) \ra (D^b_h(\mcd_X))^{opp}$ the holonomic duality functor.
Recall that for a single holonomic $\mcd_X$-module $M$, the holonomic dual is also a single holonomic $\mcd_X$-module (\cite[Proposition 3.2.1]{Hotta}) and that holonomic duality preserves regular holonomicity ( \cite[Theorem 6.1.10]{Hotta}).

\noindent For a morphism $f: X \ra Y$ between smooth algebraic varieties we additionally define the functors $f_\dag := \mbd \circ f_+ \circ \mbd$ and $f^\dag := \mbd \circ f^+ \circ \mbd$.\\

Let $MF(\mcd_X)$ be the category of filtered $\mcd_X$-modules $(M,F)$ where the ascending filtration $F_{\bullet}$ satisfies
\begin{enumerate}
\item $F_p M = 0$ for $p \ll 0 $
\item $\bigcup_p F_p M = M$
\item $(F_p \mcd_X) F_q M \subset F_{p+q} M$ for $p \in \mbz_{\geq 0},\; q \in\mbz$
\end{enumerate}
where $F_{\bullet} \mcd_X$ is the filtration by the order of the differential operator.\\

We denote by $\MHM(X)$ the abelian category of algebraic mixed Hodge modules and by $D^b \MHM(X)$ the corresponding bounded derived category. The forgetful functor to the bounded derived category of regular holonomic $\mcd$-modules is denoted by
\[
Dmod: D^b \MHM(X) \lra D^b_{rh}(\mcd_X)\, .
\]
\noindent For each morphism $f:X \ra Y$ between complex algebraic varieties, there are induced functors
\[
f_*,f_!: D^b \MHM(X) \lra D^b \MHM(Y)
\]
and
\[
f^*, f^!: D^b \MHM(Y) \ra D^b \MHM(X)\, ,
\]
which are interchanged by $\mbd$. The functors $f_*, f_!,f^*,f^!$  lift the analogous functors $f_+, f_\dag, f^\dag,  f^+$ on $D^b_{rh}(\mcd_X)$. Let $\mbq^{H}_{pt}$ be the unique mixed Hodge structure with $\gr^W_i = \gr^F_i = 0$ for $i \neq 0$ and underlying vector space $\mbq$. Denote by $a_{X} : X \lra \{pt\}$ the map to the point and set
\[
\mbq^H_X := a_X^* \mbq^H_{pt}\, .
\]
The shifted object ${^p}\mbq^H_X:= \mbq^H_X[d_X]$ lies in $\MHM(X)$ and is equal to $( \mco_X, F, \mbq_X[d_X],W)$ with $\gr^F_p = 0$   for $p \neq 0$ and $\gr^W_i = 0$ for $i \neq d_X$.
We have $\mbd \mbq^H_X \simeq a_X^! \mbq^H_{pt}$ and, since $X$ is smooth, the isomorphism
\begin{equation}\label{eq:dualQ}
\mbd \mbq^H_X  \simeq \mbq^H_X (d_X)[2 d_X]\, .
\end{equation}
Here $(d_X)$ denotes the Tate twist (see e.g., \cite[page 257]{SaitoMHM}). \\

We also have to consider the category $\MHM(X,\mbc)$ of complex mixed Hodge modules which can be defined as follows (see \cite[Definition 3.2.1]{DettSa}): First note that one can naturally extend the notion of a $\dQ$-mixed Hodge module (i.e., an object of $\MHM(X)$) to $\dR$-mixed Hodge module, due to the work of Mochizuki on mixed twistor modules (see in particular \cite[Section 13.5]{Mo12}, where the notion of a $K$-mixed Hodge module is considered, $K$ being any subfield of $\dR$). Then we say that a filtered $\cD_X$-module
$(M,F_\bullet)$ underlies a complex mixed Hodge module if it is a direct summand of a filtered $\cD_X$-module that underlies an $\dR$-mixed Hodge module. Many properties of the category of $\dR$-mixed Hodge modules carry over to $\MHM(X,\dC)$ since they are stable by direct summands.

Let $T = (\mbc^*)^d$ be a torus with coordinates $t_1,\ldots, t_d$ and $\beta \in \mbr^d$. We denote by $\mco^{\beta}_T$ the $\mcd_T$-module
\[
\mco^\beta_T = \mcd_T/ ((\p_{t_i} t_i + \beta_i)_{i=1,\ldots,d})
\]
 and  by $\mbc^{H,\beta}_T$ the complex Hodge module $(\mco_T^\beta,F,W)$ with $\gr^F_p \mbc^{H,\beta}_T = 0$ for $p \neq 0$ and $\gr^W_i \mbc^{H,\beta}_T = 0$ for $i \neq d$. Finally we set ${^p}\mbc^{H,\beta}_T := \mbc^{H,\beta}_T[d_T]$.

\vspace*{0.5cm}

\textbf{Acknowledgments:} We thank Takuro Mochizuki, Claude Sabbah and Uli Walther for many stimulating discussions on the content of this paper.

\section{A guiding example}
\label{sec:Example}

In this section, we intend to discuss a particular example, related to the quantum differential equation
of $\dP^1$, where most of the techniques used in the main body of the paper can be written down quite explicitly.
We hope that this section will help the reader to find his way through the technical difficulties of the paper.

Let $\widetilde{A}$ be the following $2\times 3$-matrix with integer entries:
$$
\begin{pmatrix}
1 & 1 & 1 \\
0 & 1 & -1
\end{pmatrix}.
$$
As explained in Definition \ref{def:GKZ} below, any $d\times n$-integer matrix $A$ together
with a vector $\beta\in \dC^d$ defines
a cyclic $\cD_{\dC^n}$-module called a GKZ-system and
denoted by $\cM_A^\beta$. For the above matrix, this system for the vector $\beta=0$ is given by
$$
\cM_{\widetilde{A}}^0=\cD_{\dC^3}/\left(\partial_{\lambda_0}^2-\partial_{\lambda_1}\partial_{\lambda_2},
\lambda_0\partial_{\lambda_0}+\lambda_1\partial_{\lambda_1}+\lambda_2\partial_{\lambda_2},\lambda_1\partial_{\lambda_1}-\lambda_2\partial_{\lambda_2}\right)
$$
It is well known that $\cM_{\widetilde{A}}^0$ is holonomic (see \cite{Adolphson}) and regular (\cite{SchulWalth1}). Moreover, it follows from \cite[Theorem 3.5]{Reich2} that $\cM_{\widetilde{A}}^0$ underlies
a mixed Hodge module ${^H\!}\cM_{\widetilde{A}}^0\in \MHM(\dC^3)$. The purpose
of this introductory section is to explain and partly prove the following statement, which is a very special case of the main theorem of this paper (Theorem \ref{thm:HodgeGKZ}). \begin{thm}\label{thm:HodgeMirroP1}
We have an isomorphism of filtered $\mcd_{\mbc^3}$-modules
$$
(\mcm_{\widetilde{A}}^0, F_\bullet^H )\simeq (\mcm_{\widetilde{A}}^0, F_{\bullet+1}^{ord}),
$$
where $F_\bullet^H$ denotes the filtration such that the filtered module $(\cM_{\widetilde{A}}^0, F_\bullet^H)$ underlies the mixed Hodge module
${^H\!}\cM_{\widetilde{A}}^0$ and where $F_\bullet^{ord}$ is
the filtration induced on $\cM_{\widetilde{A}}^0$ by the filtration on $\cD_{\dC^3}$ by orders of differential operators.
\end{thm}

The proof of this theorem will be done in several steps that parallel the main steps of the proof of Theorem \ref{thm:HodgeGKZ}. The two major simplifications are that we are dealing with the very special matrix $\widetilde{A}$, whereas
in Theorem \ref{thm:HodgeGKZ} any $(d+1)\times (n+1)$-matrix $\widetilde{A}$ satisfying some combinatorial conditions is considered, and
moreover we restrict to the parameter value $\beta=0$ (compare with the general definition of GKZ-systems $\cM_{\widetilde{A}}^\beta$ in Definition
\ref{def:GKZ} resp. with the Definition \ref{def:SWGarbe} for the sheaf $\check{\cM}_{\widetilde{A}}^\beta$). This avoids considering some rather
involved combinatorial condition relating ${\widetilde{A}}$ and $\beta$ (see the definition of the set $\mathfrak{A}_{\widetilde{A}}$ in Equation
\eqref{eq:defAdm} below).

We first consider the morphism
$$
\begin{array}{rcl}
h_{\widetilde{A}}: (\dC^*)^2 & \longrightarrow & \dC^3 =:W \\
    (t_0,t_1) & \longmapsto & (t_0,t_0\cdot t_1, t_0\cdot t_1^{-1})
    =:(w_0,w_1,w_2)
\end{array}
$$
where the exponents of the monomials in the components of the map are exactly the columns of $\widetilde{A}$.

As explained in more detail in Theorem \ref{thm:SW} and Proposition \ref{prop:SWgeneral} below, it follows from a result of Schulze and Walther (\cite{SchulWalth2}) that we have the following isomorphism of $\cD_W$-modules
$$
\begin{array}{rcl}
h_{\widetilde{A},+} \cO_{(\dC^*)^2}  & \cong &
h_{\widetilde{A},+} \cD_{(\dC^*)^2}/(\partial_{t_0}t_0+2,\partial_{t_1}t_1-1) \\ \\
& \stackrel{!}{\cong} & \cD_W/\left(w_0^2-w_1w_2, \partial_{w_0}w_0+\partial_{w_1}w_1+\partial_{w_2}w_2,\partial_{w_1}w_1-\partial_{w_2}w_2\right) =: \check{\cM}_{\widetilde{A}}^0.
\end{array}
$$
The main point in this result is to show that (left) multiplication by $w_0$
is invertible on $\check{\cM}_{\widetilde{A}}^0$, then the isomorphism follows since the map $h_{\widetilde{A}}$ can be decomposed
as
$$
h_{\widetilde{A}}:(\dC^*)^2
\stackrel{h_{\widetilde{A},1}}{\longrightarrow}
\dC^*\times \dC^2 \stackrel{h_{\widetilde{A},2}}{\longrightarrow} W
$$
where both $h_{\widetilde{A},1}$ and $h_{\widetilde{A},2}$ are embeddings and where
$h_{\widetilde{A},1}$ sends
$(t_0,t_1)$ to $(t_0,t_0\cdot t_1, t_0\cdot t_1^{-1})$
and is closed (this follows since
the map $\dC^*\rightarrow \dC^2, t\mapsto (t,t^{-1})$ is a closed embedding), whereas $h_{\widetilde{A},2}$ is the canonical open embedding of $\dC^*\times \dC^2$ into $W$.

A second consequence of Proposition \ref{prop:SWgeneral} below is that $\check{\cM}_{\widetilde{A}}^0$ underlies a mixed Hodge module on $W$, namely the object $h_{\widetilde{A},*}
{^p}\mbq^H_T \in \MHM(W)$. A first step to prove Theorem \ref{thm:HodgeMirroP1} above is to compute the Hodge filtration on $\check{\cM}_{\widetilde{A}}^0$. This task can be divided into two steps:
First we have to compute the Hodge filtration on the module ${^*\!}\check{\cM}_{\widetilde{A}}^0:=h_{\widetilde{A},1,+} \cO_{(\dC^*)^2}$ which underlies the mixed Hodge module
$h_{\widetilde{A},1,*} {^p}\mbq^H_T \in \MHM(\dC^*\times \dC^2)$. This is done via a rather direct argument
since $h_{\widetilde{A},1}$ is closed. The second step, which is more delicate, is to obtain from this the Hodge filtration
on $h_{\widetilde{A},2,+}{^*\!}\check{\cM}^0_{\widetilde{A}}=\check{\cM}^0_{\widetilde{A}}$. \\

Here we are faced with the fundamental problem of extending a mixed Hodge module from the complement of a (smooth) divisor to the total space. While this operation is easily understood at the level of $\cD$-modules, one cannot simply use the direct image functors for $\cO$-modules to calculate the extension of the filtration steps of the Hodge filtration since those are by definition $\cO$-coherent, a property that is lost under direct images of open embeddings. As is explained in more detail at the beginning of section \ref{subsec:CalcHodgeFiltrationSWGarbe}, this problem is solved by intersecting the direct image with the canonical $V$-filtration along the divisor in question. In order to compute the Hodge filtration
on $h_{\widetilde{A},2,+}{^*\!}\check{\cM}^0_{\widetilde{A}}$, we thus have to calculate this $V$-filtration along $w_1=0$.

Notice that in the main body of the text, we follow a slightly different strategy, due to the fact that the factorization of $h_{\widetilde{A}}$ used above
into a closed an open embedding may look different depending on the shape of the matrix $\widetilde{A}$. In general, one can always consider the factorization into a map between tori, which is a closed embedding (which would be the map from $(\dC^*)^2$ to $(\dC^*)^3$ in the above example) followed by the canonical open embedding from the torus into affine space. The latter, however, is the extension over a normal crossing divisor, which is not smooth. In order to apply the techniques sketched above, one has to compose further with a graph embedding with respect to the equation of the normal crossing divisor (see diagram \eqref{eq:DiagGraphEmbeddingNCD} and the arguments following it).

The first (easy) step of the calculation of the Hodge filtration on $\check{\cM}_{\widetilde{A}}^0$ can be formulated as follows.
\begin{lemma}[Compare Lemma \ref{lem:directMapk} below for the general case]\label{lem:HodgeFiltOpDirectImageEx}
The direct image $\cH^0 h_{\widetilde{A},1,+} \cO_{(\dC^*)^2}$ is isomorphic to the
cyclic $\cD_{\dC^* \times \dC^2}$-module
$$
\cD_{\dC^* \times \dC^2}/
\left (
w_0^2-w_1w_2, \partial_{w_0}w_0+\partial_{w_1}w_1+\partial_{w_2}w_2,\partial_{w_1}w_1-\partial_{w_2}w_2
\right)
$$
and the Hodge filtration on this module is given, under this isomorphism, by the induced order filtration, shifted by one, i.e., we have
$$
F_p^H \cH^0 h_{\widetilde{A},1,+} \cO_{(\dC^*)^2} =
F^{ord}_{p-1} \left[\cD_{\dC^* \times \dC^2}/
\left (
w_0^2-w_1w_2, \partial_{w_0}w_0+\partial_{w_1}w_1+\partial_{w_2}w_2,\partial_{w_1}w_1-\partial_{w_2}w_2\right)\right]
$$
\end{lemma}
\begin{proof}
We can factor $h_{\widetilde{A},1}$ further as $h_{\widetilde{A},1}=\widetilde{\widetilde{h}}_{\widetilde{A},1}\circ\widetilde{h}_{\widetilde{A},1}$,
where
$$
\begin{array}{rcl}
\widetilde{h}_{\widetilde{A},1}: (\dC^*)^2 & \longmapsto & (\dC^*)^3 \\ \\
    (t_0,t_1) & \longmapsto & (t_0,t_0\cdot t_1, t_0\cdot t_1^{-1})
\end{array}
$$
and where $\widetilde{\widetilde{h}}_{\widetilde{A},1}:(\dC^*)^3\hookrightarrow\dC^*\times \dC^2$ is
the canonical open embedding. Then since $h_{\widetilde{A},1}$ is a closed embedding, we know that
the support of ${^*\!}\cM_{\widetilde{A}}^0$ is disjoint from the divisor $(\dC^*\times \dC^2)\backslash(\dC^*)^3$,
which implies that $F_p^H {^*\!}\cM_{\widetilde{A}}^0 = \widetilde{\widetilde{h}}_{\widetilde{A},1,*} F_p^H \widetilde{h}_{\widetilde{A},1,+} \cO_{(\dC^*)^2}$.
It therefore suffices to determine the filtration steps $F_p^H \widetilde{h}_{\widetilde{A},1,+} \cO_{(\dC^*)^2}$,
or, more precisely, to show that
$$
F_p^H \widetilde{h}_{\widetilde{A},1,+} \cO_{(\dC^*)^2} = F^{ord}_{p-1} \left[\cD_{(\dC^*)^3}/
\left (
w_0^2-w_1w_2, \partial_{w_0}w_0+\partial_{w_1}w_1+\partial_{w_2}w_2,\partial_{w_1}w_1-\partial_{w_2}w_2\right)\right]
$$
Consider the coordinate change
$$
\begin{array}{rcl}
\phi:(\dC^*)^3 & \longrightarrow & (\dC^*)^3 \\
    (w_0,w_1,w_2) & \longmapsto &
    (w_0,w_1/w_0,w_1w_2/w_0^2)=:(u_0,u_1,u_2)
\end{array}
$$
so that $(\phi\circ \widetilde{h}_{\widetilde{A},1})(t_0,t_1)=(t_0,t_1,1)$ and
$$
(\phi\circ \widetilde{h}_{\widetilde{A},1})_+\cO_{(\dC^*)^2} =
\cD_{(\dC^*)^3}/(u_2-1,\partial_{u_0} u_0+2, \partial_{u_1} u_1) = \left(\cD_{(\dC^*)^2}/(\partial_{u_0} u_0+2, \partial_{u_1} u_1)\right) [\partial_{u_2}].
$$
According to \cite[Formula (1.8.6)]{Saitobfunc}, we have
$$
F^H_{p+1} (\phi\circ \widetilde{h}_{\widetilde{A},1})_+\cO_{(\dC^*)^2} =
\sum_{p_1+p_2=p} F_{p_1}^H \left(\cD_{(\dC^*)^2}/(\partial_{u_0} u_0+2, \partial_{u_1} u_1\right) \partial^{p_2}_{u_2}
$$
Since $F^H_{p_1}\left(\cD_{(\dC^*)^3}/(\partial_{u_0} u_0, \partial_{u_1} u_1\right))=F^{ord}_{p_1}\left(\cD_{(\dC^*)^3}/(\partial_{u_0} u_0, \partial_{u_1} u_1\right)$, we obtain
from the above formula that
$$
F^H_p (\phi\circ \widetilde{h}_{\widetilde{A},1})_+\cO_{(\dC^*)^2} =
F^{ord}_{p-1}\left(\cD_{(\dC^*)^3}/(u_2-1,\partial_{u_0} u_0, \partial_{u_1} u_1)\right).
$$
Since $\phi$ is invertible, we obtain that the Hodge filtration
on $\phi^{-1}_+((\phi\circ \widetilde{h}_{\widetilde{A},1})_+ \cO_{(\dC^*)^2})  = \widetilde{h}_{\widetilde{A},1,+}\cO_{(\dC^*)^2}$ is the
order filtration on $\cD_{(\dC^*)^3}/
\left (
w_0^2-w_1w_2, \partial_{w_0}w_0+\partial_{w_1}w_1+\partial_{w_2}w_2,\partial_{w_1}w_1-\partial_{w_2}w_2\right)$,
shifted by one. As discussed above, the closure of the support of $\widetilde{h}_{\widetilde{A},1,+}\cO_{(\dC^*)^2}$ in $\dC^*\times \dC^2$ lies entirely in the torus $(\dC^*)^3$, therefore, we obtain
$$
F_p^H h_{\widetilde{A},1,+} \cO_{(\dC^*)^2} =
F^{ord}_{p-1} \left[\cD_{\dC^* \times \dC^2}/
\left (
w_0^2-w_1w_2, \partial_{w_0}w_0+\partial_{w_1}w_1+\partial_{w_2}w_2,\partial_{w_1}w_1-\partial_{w_2}w_2\right)\right],
$$
as required.
\end{proof}

The next step is to compute the Hodge filtration of the open direct image $h_{\widetilde{A},2,+}{^*\!}\check{\cM}^0_{\widetilde{A}}=\check{\cM}^0_{\widetilde{A}}$. As mentioned
above, this needs information on the canonical $V$-filtration of the module $\check{\cM}^0_{\widetilde{A}}$ with respect to
the smooth divisor $\{w_0=0\}$. More precisely, we have the following important formula (see Formula
\eqref{eq:defHodgefilt} below, as well as \cite[Proposition 4.2.]{Saitobfunc}):
\begin{equation}\label{eq:SaitoFormulaFirst}
F_p^H \check{\cM}^0_{\widetilde{A}} = \sum_{i\geq 0} \partial_{w_0}^i \left(V^0 \check{\cM}^0_{\widetilde{A}} \cap h_{\widetilde{A},2,*} F^H_{p-i} {^*\!}\check{\cM}_{\widetilde{A}}^0\right)
\end{equation}
Hence we need to determine the object $V^0 \check{\cM}^0_{\widetilde{A}}$. For what follows, it is more convenient to work out everything at the level of global sections. Since all modules we are considering here are defined on affine spaces, this is obviously sufficient.

We refer to \cite{MebkhMais} for details
on the $V$-filtration. For what follows in this introduction, we only need that $V^0 D_{\dC^3}=\dC[w_0,w_1,w_2]\langle w_0\partial_{w_0}, \partial_{w_1},\partial_{w_2}\rangle$ and the following characterization of the canonical $V$-filtration of the holonomic module $\check{M}_{\widetilde{A}}^0=\Gamma(\dC^3, \check{\cM}^0_{\widetilde{A}})$, copied from \cite[Definition 4.3-3, Proposition 4.3-9]{MebkhMais} (notice again that we work at the level of global sections):
For any $m\in \check{M}_{\widetilde{A}}^0$ we consider its Bernstein-Sato polynomial $b_m(x)\in \dC[x]$,
which is the unique monic polynomial of smallest degree satisfying the functional equation $b_m(\partial_{w_0}w_0)m
\in w_0\cdot V^0(D_{\dC^3})m$. The set of roots of $b_m(x)$ is denoted by $\ord(m)$. Then we have
$$
V^\alpha \check{M}^0_{\widetilde{A}} := \left\{m\in \check{M}^0_{\widetilde{A}}\;|\, \ord(m)\subset [\alpha,\infty)\right\}.
$$
Our first step is to compute the Bernstein-Sato polynomial for the class $[1]\in \check{M}_{\widetilde{A}}^0$. Here we have the following result.
\begin{proposition}[Compare Lemma \ref{lem:compV1} below for the general case]\label{prop:BPolExample}
Consider the class of $1\in D_{\dC^3}$ in the quotient
$\check{M}_{\widetilde{A}}^0$, denoted by $[1]$. Then we have
$b_{[1]}(s)=s^2$.
\end{proposition}
\begin{proof}
It is sufficient to find a functional equation in $D_{\dC^3}$ of the form $(\partial_{w_0}w_0)^2 = w_0\cdot P+\check{I}_{\widetilde{A}}^0
$, where
$$
\check{I}_{\widetilde{A}}^0:=\left(w_0^2-w_1w_2, \partial_{w_0}w_0+\partial_{w_1}w_1+\partial_{w_2}w_2,\partial_{w_1}w_1-\partial_{w_2}w_2\right)
$$
and where $P\in\dC[w_0,w_1,w_2]\langle w_0\partial_{w_0},\partial_{w_1},\partial_{w_2}\rangle$.
We will show that
\begin{equation}\label{eq:FuncEqBernstein}
(-\partial_{w_1}w_1-\partial_{w_2}w_2)^2 \in  w_0\cdot \dC[w_1,w_2]\langle \partial_{w_1},\partial_{w_2} \rangle +\check{I}_{\widetilde{A}}^0,
\end{equation}
which suffices to conclude.
We have
$$
\begin{array}{rcl}
(-\partial_{w_1}w_1-\partial_{w_2}w_2)^2 & = & 2\partial_{w_1}\partial_{w_2}w_1w_2+(\partial_{w_1}w_1)^2+(\partial_{w_2}w_2)^2 \\ \\
& \equiv & 2 w_0^2 \cdot \partial_{w_1}\partial_{w_2} +(\partial_{w_1}w_1)^2+(\partial_{w_2}w_2)^2\quad\text{mod}\quad \check{I}_{\widetilde{A}}^0 \\ \\
&=&  2 w_0^2 \cdot \partial_{w_1}\partial_{w_2} +(\partial_{w_1}w_1 - \partial_{w_2}w_2)^2 +2\partial_{w_1}\partial_{w_2}w_1w_2\\ \\
&\equiv &  w_0^2 \cdot 4\cdot\partial_{w_1}\partial_{w_2} \quad\text{mod}\quad \check{I}_{\widetilde{A}}^0,
\end{array}
$$
which shows Formula \eqref{eq:FuncEqBernstein}.
\end{proof}

Notice that the above calculation can be extended to any matrix $\widetilde{A}$ of the form
$$
\widetilde{A}=
\left(
\begin{matrix}
1 & 1 & \ldots & 1 \\ 0 & a_{11} & \dots & a_{1n} \\ \vdots & \vdots & &\vdots \\ 0 & a_{d1} & \dots &a_{dn}
\end{matrix} \right)
$$
where the columns of the matrix
$$
A=
\left(
\begin{array}{ccc}
 a_{11} & \ldots & a_{1n}  \\
\vdots  &  & \vdots \\
a_{n1} & \ldots & a_{nn}
\end{array}
\right)
$$
are
the primitive integral generators of the fan of a smooth projective toric Fano manifold.
Then one has to consider the classical cohomology algebra of this manifold, which admits a toric description (see,
e.g. \cite[Section 5.2]{Fulton}), and the functional equation (i.e. the analogue of Formula \ref{eq:FuncEqBernstein})
can be deduced from the relations in this algebra.
Notice also that this is in fact an argument which is a much simplified version of the one used to prove
Lemma \ref{lem:compV1} below (actually, the proof of this lemma relies on the main result of the separate paper \cite{ReichSevWalth}, which is based on general arguments from toric algebra, such as Euler-Koszul complexes, toric modules etc.). Lemma \ref{lem:compV1} is also more general in the sense that the matrix considered there
is not necessarily defined by the rays of a smooth toric variety.

We have the following consequence of the above calculation which gives complete control
on the integer part of the canonical $V$-filtration on $\check{M}$.
\begin{corollary}[Compare Proposition \ref{prop:Vfilt} below for the general case]\label{cor:VFiltEx}
Denote by $V^\bullet_{ind} \check{M}_{\widetilde{A}}^0$ the filtration induced on $\check{M}_{\widetilde{A}}^0$ by the $V$-filtration
(with respect to $w_0$) on $D_{\dC^3}$. Then for all $k\in \dZ$, we have $V^k_{ind} \check{M}_{\widetilde{A}}^0 = V^k \check{M}_{\widetilde{A}}^0$.
\end{corollary}
\begin{proof}
The proof is more or less similar to the general case in Proposition \ref{prop:Vfilt} below. In the case $k \geq 0$ any element $[P]\in V^k_{ind} \check{M}_{\widetilde{A}}^0$ has an expression
$$
[P]=[\sum_{i=0}^{l} w_0^k (w_0\partial_{w_0})^i\cdot P_i] + [R]
$$
where $P_i\in \dC[w_1,w_2]\langle \partial_{w_1},\partial_{w_2} \rangle$ and where
$[R]\in V^{k+1}_{ind} \check{M}_{\widetilde{A}}^0$. On the other hand, if $k>0$, we can always write an element
$[P]\in V^{-k}_{ind} \check{M}_{\widetilde{A}}^0$ as
$$
[P]=[\sum_{i=0}^{l} \partial_{w_1}^k (w_1\partial_{w_1})^i\cdot P_i] + [R].
$$
where $P_i$ and $[R]$ are as above. One  easily deduces (see the calculations in the proof of
Proposition \ref{prop:Vfilt} below) from the functional equation
$(\partial_{w_0}w_0)^2[1]\in V^1_{ind}\check{M}_{\widetilde{A}}^0$ proved in Proposition \ref{prop:BPolExample} above that we have
$$
\begin{array}{rclll}
  (\partial_{w_0}w_0-k)^2[P] & \in & V^{k+1}_{ind} \check{M}_{\widetilde{A}}^0&&\textup{for }k\geq 0 \textup{ and }[P]\in V^k_{ind} \check{M}_{\widetilde{A}}^0\\ \\
  (\partial_{w_0}w_0+k)^2[P] & \in & V^{-k+1}_{ind} \check{M}_{\widetilde{A}}^0&&\textup{for }k > 0 \textup{ and }[P]\in V^{-k}_{ind} \check{M}_{\widetilde{A}}^0\\
\end{array}
$$
General considerations on the canonical $V$-filtration (see, e.g., \cite[section 4.2 and 4.3]{MebkhMais} and the argument in the proof of Proposition \ref{prop:Vfilt}) then imply
that $V^k_{ind} \check{M}_{\widetilde{A}}^0 = V^k \check{M}_{\widetilde{A}}^0$.
\end{proof}

Finally, we arrive at the following first main step toward the proof of theorem \ref{thm:HodgeMirroP1}.
\begin{proposition}[Compare Theorem \ref{thm:hpot} below for the general case]
\label{prop:HodgeTorusEmbExamples}
Let $\widetilde{A}$ and $h_{\widetilde{A}}:(\dC^*)^2\rightarrow \dC^3=W$ be as above, then for any $k\in \dZ$ we have the following isomorphism of $\cO_W$-modules
$$
F_p^H h_{\widetilde{A},+} \cO_{(\dC^*)^2} \cong F^{ord}_{p-1} \left[\cD_W/\left (
w_0^2-w_1w_2, \partial_{w_0}w_0+\partial_{w_1}w_1+\partial_{w_2}w_2,\partial_{w_1}w_1-\partial_{w_2}w_2\right)\right]
$$
\end{proposition}
\begin{proof}
The main tool to obtain a description of the Hodge filtration is Formula \eqref{eq:SaitoFormulaFirst} from above (see
\cite[Proposition 4.2.]{Saitobfunc}) which at the level of global sections reads
\begin{equation}\label{eq:SaitoFormulaEx}
F^H_p \check{M}_{\widetilde{A}}^0 = \sum_{i \geq 0} \partial_{w_0}^i \left(V^0 \check{M}_{\widetilde{A}}^0 \cap  F^H_{p-i} {^*\!}\check{M}_{\widetilde{A}}^0\right).
\end{equation}
Recall from Lemma \ref{lem:HodgeFiltOpDirectImageEx} that for all $l\in \dZ$ we have
$$
F^H_l {^*\!}\check{M}_{\widetilde{A}}^0 = F^{ord}_{l-1} {^*\!}\check{M}_{\widetilde{A}}^0.
$$
In particular, since $F^{ord}_p {^*\!}\check{M}_{\widetilde{A}}^0=0$ for all $p<0$, we have
$F^H_p \check{M}_{\widetilde{A}}^0=0$ for all $p<1$. On the other hand, we have seen in Proposition
\ref{prop:BPolExample} that $[1]\in V^0 \check{M}_{\widetilde{A}}^0$. Obviously, we have
$[1]\in F_0^{ord} {^*\!}\check{M}_{\widetilde{A}}^0$, which implies that
$[1]\in F_1^H \check{M}_{\widetilde{A}}^0$. Since $\check{M}_{\widetilde{A}}^0$ is a cyclic $D_W$-module and since
both filtrations $F^H$ and $F^{ord}$ on it are good filtrations, we obtain
the inclusion
$$
F_{p-1}^{ord} \check{M}_{\widetilde{A}}^0 \subset F_p^H \check{M}_{\widetilde{A}}^0
$$
for all $p\in \dZ$. It remains to show the reverse inclusion
$F_p^H \check{M}_{\widetilde{A}}^0 \subset F_{p-1}^{ord} \check{M}_{\widetilde{A}}^0$. Using formula
\eqref{eq:SaitoFormulaEx} as well as Corollary \ref{cor:VFiltEx} and Lemma \ref{lem:HodgeFiltOpDirectImageEx} this amounts to
$$
\sum_{i\geq 0} \partial_{w_0}^i\left(V_{ind}^0 \check{M}_{\widetilde{A}}^0 \cap F_{p-i}^{ord} {^*\!}\check{M}_{\widetilde{A}}^0\right)
\subset F^{ord}_p \check{M}_{\widetilde{A}}^0.
$$
We obviously have $\partial_{w_0}^i F^{ord}_{p-i} \check{M}_{\widetilde{A}}^0  \subset F^{ord}_p \check{M}_{\widetilde{A}}^0$ for all $i$ so
it only remains to show that
$$
V_{ind}^0 \check{M}_{\widetilde{A}}^0 \cap F_l^{ord} {^*\!}\check{M}_{\widetilde{A}}^0 \subset F^{ord}_l \check{M}_{\widetilde{A}}^0
$$
for all $l\in \dZ$. Consider any class $[P] \in V_{ind}^0 \check{M}_{\widetilde{A}}^0 \cap F_l^{ord} {^*\!}\check{M}_{\widetilde{A}}^0$.
Since we have ${^*\!}\check{M}_{\widetilde{A}}^0= \check{M}_{\widetilde{A}}^0[w_0^{-1}]$, we can write
$$
P=w_0^{-k} P_k + w_0^{-k+1} P_{-k+1}+\ldots,
$$
where $P_i\in\dC[w_1,w_2]\langle \partial_{w_0},\partial_{w_1},\partial_{w_2}\rangle$. It follows that
$w_0^k\cdot [P]\in V_{ind}^k \check{M}_{\widetilde{A}}^0 \cap F_l^{ord} \check{M}_{\widetilde{A}}^0$. We thus have to prove
$$
V_{ind}^k \check{M}_{\widetilde{A}}^0 \cap F_l^{ord} \check{M}_{\widetilde{A}}^0 \subset w_0^k F_l^{ord}\check{M}_A^0.
$$
for all $k,l\in \dZ$. Take any class $[Q]\in V_{ind}^k \check{M}_{\widetilde{A}}^0 \cap F_l^{ord} \check{M}_{\widetilde{A}}^0$,
then suppose that we can find a representative $Q\in V^k D_W \cap F_l D_W$ of $[Q]$. This means that
$Q=w_0^k\cdot \widetilde{Q}$, with $\widetilde{Q}\in F_lD_W$, as required. Hence we obtain $[Q]\in w_0^k F_l^{ord}\check{M}_{\widetilde{A}}^0$.
It thus remains  to show the existence of such a representative $Q\in V^k D_W \cap F_l D_W$, and this is exactly the content of the next lemma.
\end{proof}
\begin{lemma}[Compare Proposition \ref{prop:CompFiltOnSWGarbe} below for the general case]
Let $A$ be as above, then for all $k,l\in \dZ$, the morphism
$$
V^kD_W\cap F_l D_W \longrightarrow V_{ind}^k \check{M}_{\widetilde{A}}^0 \cap F^{ord}_l \check{M}_{\widetilde{A}}^0
$$
is surjective.
\end{lemma}
\begin{proof}
The proof relies on the theory of Gr\"obner bases in the Weyl algebra. We will not give any definition here, but we refer to subsection \ref{subsec:Vfilt} for details about monomial orders and Gr\"obner bases in the non-commutative setup.

Consider any class $m\in V^k_{ind} \check{M}_{\widetilde{A}}^0 \cap F_l^{ord} \check{M}_{\widetilde{A}}^0$. Then we can find $P\in F_l D_W$ and $Q\in V^k D_W$ such that $[P]=[Q]=m$, that is, $P=Q+i$ for some
$$
i\in \check{I}^0_{\widetilde{A}}=\left(w_0^2-w_1w_2, \partial_{w_0}w_0+\partial_{w_1}w_1+\partial_{w_2}w_2,\partial_{w_1}w_1-\partial_{w_2}w_2\right).
$$
We chose a minimal $r\in\dN$ with $Q\in F_r D_W$. If $r\leq l$, we are done since then $Q\in V^kD_W\cap F_l D_W$ is the preimage of $m$ we are looking for. Hence suppose $r>l$. It is easy to see that then
$i\in F_r D_W$ and the class of $i$ in $\gr_r^F D_W$ is non-zero. For any operator $R\in D_W$, write
$$
\sigma(R)\in \gr^F_\bullet D_W=\dC[w_0,w_1,w_2,\xi_0,\xi_1,\xi_2]=:\dC[\underline{w},\underline{\xi}]
$$
for its symbol. The three generators of $\check{I}_{\widetilde{A}}^0$ from above
form a Gr\"obner basis of this ideal with respect to the partial ordering given by the weight vector (\underline{0},\underline{1}) (i.e. where
$w_i$ has weight $0$ and $\partial_{w_i}$ has weight $1$), notice that this weight vector induces the filtration $F_\bullet$ on $D_W$.
This can be directly shown by a Macaulay2 calculation, the corresponding general
result is Corollary \ref{cor:constrGroebbasis}
below, where we treat a slightly different situation however (we add a column to our matrix which is the sum of all other columns, this corresponds to composing with the graph embedding of the equation of a normal crossing divisor, see
also the remarks before Lemma \ref{lem:directMapk} above).

We can therefore  conclude that there is an expression
$$
\sigma(i) = \widetilde{i}_1 \cdot (w_0^2-w_1w_2)+
\widetilde{i}_2 \cdot (\xi_0w_0+\xi_1w_1+\xi_2w_2)+
\widetilde{i}_3 \cdot (\xi_1w_1-\xi_2w_2).
$$
with $\widetilde{i}_1,\widetilde{i}_2,\widetilde{i}_3 \in\dC[\underline{w},\underline{\xi}]$. Let $i_1,i_2,i_3\in D_W$ be the normally ordered operators
obtained from $\widetilde{i}_1,\widetilde{i}_2,\widetilde{i}_3$ by replacing $\xi_k$ by $\partial_{w_k}$. Then we
define
$$
i':=
i_1(w_0^2-w_1w_2)+i_2(\partial_{w_0}w_0+\partial_{w_1}w_1+\partial_{w_2}w_2)+i_3(\partial_{w_1}w_1-\partial_{w_2}w_2)\in \check{I}_{\widetilde{A}}^0
$$
and clearly $i'\in F_l D_W$.
However, we also have $i'\in V^k D_W$,
since it can again be shown by a direct computation that the three polynomials $w_0^2-w_1w_2, \xi_0w_0+\xi_1w_1+\xi_2w_2, \xi_1w_1-\xi_2w_2$ form a Gr\"obner basis of
the ideal they generate with respect to a partial ordering
given by the weight vector $(-1,0,0,1,0,0)$
(i.e. the weight of $w_0$ is $-1$, the weight of $\xi_0$
is $1$ and all other weights are zero). Notice again
that this weight vector yields the filtration induced
from $V^\bullet D_W$ on $\dC[\underline{w},\underline{\xi}]$.
The corresponding general result (for the case
of the extended matrix with one added column) is found in Corollary \ref{cor:constrGroebbasis}, 2. below.

Summarizing, we obtain that
the operator $Q-i'$ satisfies
\begin{enumerate}
    \item $[Q-i']=[Q]=m$.
    \item $Q-i'\in V^k D_W$
    \item $Q-i'\in F_{l-1} D_W$ (this follows from $\sigma(Q)=\sigma(i)=\sigma(i')$).
\end{enumerate}
Hence we see by descending induction on $l$
that we can construct an operator $Q'\in F_l D_W \cap V^k D_W$ such that $[Q']=m$. This shows the statement.
\end{proof}

We have finished the proof of Proposition \ref{prop:HodgeTorusEmbExamples} above, which roughly summarizes the content of section \ref{sec:HodgeOnTorusEmbed} in the main body of the paper for our particular example. We now turn to the statements corresponding to section \ref{sec:Radon} below (for our example), that is, we are going to complete the proof of
Theorem \ref{thm:HodgeMirroP1}.

Recall that the GKZ system $\mcm^0_{\widetilde{A}}$ can be described as a Fourier-Laplace transform of a torus embedding
\[
\mcm^0_{\widetilde{A}} \simeq \FL(h_{\widetilde{A},+} \mco_{(\mbc^*)^2})
\]
Since the matrix $\widetilde{A}$ is homogeneous (i.e. $(1,\ldots,1)$ is in its row span) the $\mcd$-module $\mcm^0_{\widetilde{A}}$ has a different presentation involving only (proper) direct image functors and inverse image functors but excluding
the use of the Fourier-Laplace transformation (see \cite[Proposition 2.7(iii)]{Reich2}). Let us recall some ingredients of this construction in the present situation. Consider the torus embedding
\begin{align*}
g : \mbc^* &\lra \mbp^2  \\
t &\mapsto (1:t:t^{-1}),
\end{align*}
then $\mcm^0_{\widetilde{A}}$ can be described by a Radon type transform of the $\mcd$-module $g_{+} \mco_{\mbc^*}$. More precisely, we have a commutative diagram
\[
\xymatrix{ & U \ar[dl]_{\pi_1^U}  \ar[dr]^{\pi_2^U } \ar[d]_{j_U} & \\\mbp^2 & \mbp^2 \times \mbc^3 \ar[l]_{\pi_1} \ar[r]^{\pi_2} & \mbc^3 \\ }
\]
where $U$ is the complement of the universal hyperplane in $\mbp^2 \times \mbc^3$, i.e. $U:=\{\lambda_0 w_0+\lambda_1 w_1+\lambda_2 w_2 \neq 0\}$.
The GKZ-system $\mcm^0_{\widetilde{A}}$ is now given by
\[
\mcm^0_{\widetilde{A}} \simeq \mcr^\circ_c(g_{+} \mco_{\mbc^*}) \simeq  \pi_{2\dag}^U \pi_1^{U,\dag} g_{+} \mco_{\mbc^*} \simeq \pi_{2,+} \left(\pi_1^\dag  g_{+} \mco_T \otimes j_{U,\dag} \mco_U \right)
\]
where the last isomorphism follows from the projection formula.

Since $\mco_{\mbc^*}$ carries a trivial Hodge module structure and since the category of algebraic mixed Hodge modules is stable under the (proper) direct image functor and the (exceptional) inverse image functor this induces a mixed Hodge module structure on $\mcm^0_{\widetilde{A}}$.

For technical reasons that mainly occur when  dealing with the case $\beta \neq 0$ as we do in the main body of this paper, we will pursue a slightly different approach.

Consider the map $F: \mbc^* \times \mbc^3 \ra \mbc$  given by the Laurent polynomial $\lambda_0 +\lambda_1 t + \lambda_2 t^{-1}$ and let $j: \mbc^* \ra \mbc$ the canonical embedding. Denote by $p$ resp. $q$ the projection from $\mbc^* \times \mbc^3$ to the second resp. first factor. We consider the integral transform of $\mco_{\mbc^*}$ from $\mbc^*$ to $\mbc^3$ with kernel $F^\dag j_\dag \mco_{\mbc^*}$ and prove in Proposition \ref{prop:GKzeqtwRad} that this integral transfrom is isomorphic to the GKZ-system $\mcm_{\widetilde{A}}^0$ (In the more general case of a non-zero $\widetilde{\beta}=(\beta_0,\beta)$ we would start with $\mco_{\mbc^*}^\beta$ and use the kernel $F^\dag j_\dag \mco_{\mbc^*}^{\beta_0}$).

Since this integral transformation preserves the category of mixed Hodge modules we define a Hodge module structure on the GKZ system by
\begin{equation}\label{eq:exintegraltrans}
{^H}\mcm^0_{\widetilde{A}} := \mch^{6}(p_*(q^* {^p} \mbc_{\mbc^*} \otimes F^* j_! {^p} \mbc_{\mbc^*}))
\end{equation}
In proposition \ref{prop:compRadtwRad} we prove that this approach coincides with the Radon transform for integer $\beta_0$.

In order to compute the Hodge filtration on \eqref{eq:exintegraltrans} explicitly we have to consider a partial compactification of $\mbc^* \times \mbc^3$, since the projection $p: \mbc^* \times \mbc^2$ is not proper.  For this we use the locally closed embedding $g: \mbc^* \ra \mbp^2$. As an intermediate step we compute the Hodge filtration on the mixed Hodge module
\[
{^H}\mcn:= ( g \times id)_* (q^* {^p} \mbc_{\mbc^*} \otimes F^* j_! {^p} \mbc_{\mbc^*})
\]
The space $\mbp^2 \times \mbc^3$ is covered by the three charts $W_u = \{w_u \neq 0\}$ for $u=0,1,2$. The map $g$ factors over each chart and is given by
\begin{align*}
g_0: \mbc^*  & \lra W_0  &  g_1: \mbc^* &\lra W_1  & g_2: \mbc^* &\lra W_2 \\
t &\mapsto (t, t^{-1}) & t &\mapsto (t^{-1}, t^{-2}) &  t &\mapsto (t,t^2)
\end{align*}

We obtain in Formula \eqref{eq:presWu} below that the restriction of $^{H}\mcn$ to $W_u$ can be written as a direct product
\[
{^H}\mcn_u = \mch^0({^p}g_{u*} \mbc_{\mbc^*}^{H}) \boxtimes \mch^2(pr_u j_! {^p} \mbc_{\mbc*}^{H})
\]
where $pr_u: \mbc^3 \ra \mbc $ is the projection to the $(u+1)$-th factor. The Hodge filtration on the first factor can be computed by using Theorem \ref{thm:hpot}, the computation of the second factor is straightforward (cf. Remark \ref{rem:pdi1}) (we check the assumption of Theorem \ref{thm:hpot} in Lemma \ref{lem:MatrixAu}).

Define the matrices \setlength{\arraycolsep}{4pt}
\[
A^s_0 := \left(\begin{matrix}
0 & 0 & 1 & 1 & 1 \\
1 & -1 & 0 & 1 &-1
\end{matrix} \right)\quad
A^s_1 := \left(\begin{matrix}
0 & 0 & 1 & 1 & 1 \\
-1 & -2 & 0 & 1 &-1
\end{matrix} \right)\quad
A^s_2 := \left(\begin{matrix}
0 & 0 & 1 & 1 & 1 \\
1 & 2 & 0 & 1 &-1
\end{matrix} \right)
\]
\setlength{\arraycolsep}{6pt}
We show in Lemma \ref{lem:descrNu} that the $\mcd$-module underlying ${^H}\mcn_u$ is isomorphic to a partial Laplace transformation of $\mcm^0_{A^s_u}$ in the $w$-variables. More precisely we have
\begin{align*}
\mcn_0 = \mcd_{\mbc^2 \times \mbc^3} / \ck{\cI}_{\!A^s_0}, \qquad \mcn_1 = \mcd_{\mbc^2 \times \mbc^3} / \ck{\cI}_{\!A^s_1}, \qquad \mcn_2 = \mcd_{\mbc^2 \times \mbc^3} / \ck{\cI}_{\!A^s_2}
\end{align*}
where $\ck{\cI}_{\!A^s_0}$ is generated by Euler operators
\[
 \ck{E^0_0}:=\lambda_0\p_{\lambda_0} + \lambda_1 \p_{\lambda_1} + \lambda_2 \p_{\lambda_2},\quad \ck{E^0_1}:= -w_{10}\p_{w_{10}} + w_{20}\p_{w_{20}} + \lambda_1 \p_{\lambda_1} - \lambda_2 \p_{\lambda_2}
\]
and the box operators
\[
\ck{\Box}_{(1,1,0,0,0)} := w_{10}w_{20}-1, \quad \ck{\Box}_{(0,0,2,-1,-1)}:= \p_{\lambda_0}^2 - \p_{\lambda_1} \p_{\lambda_2}, \qquad  \ck{\Box}_{(0,1,-1,1,0)} := w_{20} \p_{\lambda_1} - \p_ {\lambda_0}
\]
The ideal $\ck{\cI}_{\!A^s_1}$ is generated by Euler operators
\[
\ck{E^1_0}:= \lambda_0 \p_{\lambda_0} + \lambda_1 \p_{\lambda_1} + \lambda_2 \p_{\lambda_2},\quad \ck{E^1_1}:= w_{01}\p_{w_{01}} + 2 w_{21}\p_{w_{21}} + \lambda_1 \p_{\lambda_1} - \lambda_2 \p_{\lambda_2}
\]
and box operators
\[
\ck{\Box}_{(2,-1,0,0,0)} := w_{01}^2-w_{21}, \quad \ck{\Box}_{(0,0,2,-1,-1)}:= \p_{\lambda_0}^2 - \p_{\lambda_1} \p_{\lambda_2}, \qquad \ck{\Box}_{(1,0,-1,1,0)} := w_{01} \p_{\lambda_1} - \p_{\lambda_0}
\]
The ideal $\ck{\cI}_{\!A^s_2}$ is generated by Euler operators
\[
\ck{E^2_0}:=\lambda_0 \p_{\lambda_0} + \lambda_1 \p_{\lambda_1} + \lambda_2 \p_{\lambda_2},\quad \ck{E^2_1}:=-w_{02}\p_{w_{01}} - 2 w_{12}\p_{w_{21}} + \lambda_1 \p_{\lambda_1} - \lambda_2 \p_{\lambda_2}.
\]
and box operators
\[
\ck{\Box}_{(2,-1,0,0,0)}:= w_{02}^2-w_{12}, \quad \ck{\Box}_{(0,0,2,-1,-1)}:=\p_{\lambda_0}^2 - \p_{\lambda_1} \p_{\lambda_2}, \qquad \ck{\Box}_{(1,0,1,-1,0)}:=   w_{02}\p_{\lambda_0} - \p_{\lambda_1}
\]
The Hodge filtration on these systems is given by $F^H_{p+1} \mcn_u = F_p^{ord} \mcd_{W_u \times \mbc^3} / \ck{\cI}_{\!A^s_u}$.\\

Using the fact that $\mcn_u$ is a partial Fourier-Laplace transform of $\mcm_{A^s_u}$ we use the results of section \ref{subsec:GKZ} to construct a strict resolution of $(\mcn_u,F^H_\bullet)$ at the level of global sections which is given by the Euler-Koszul complex
\[
\xymatrix{\ck{K}{}^\bullet_u \;:= \;D_{W_u \times \mbc^3} / \ck{J_{\!A^s_u}} \ar[rr]^{\cdot (-\ck{E^u_1},\,\ck{E^u_0})^t} && ( D_{W_u \times \mbc^3} / \ck{J_{\!A^s_u}})^2 \ar[rr]^{\cdot (\ck{E^u_0},\, \ck{E^u_1})} && D_{W_u \times \mbc^3} / \ck{J_{\!A^s_u}} }
\]
where the left ideal $J_{\!A^s_u}$ is generated by the "box-type" generators from above.

It remains to compute the projection of $(\mcn, F^H_\bullet)$ under the map $\mbp^2 \times \mbc^3\ra \mbc^3$.
In order to do this we lift the filtered $D$-modules $(N_u, F^H_\bullet)$ as well as their strictly filtered resolution $(\ck{K}{}^\bullet_u ,F_\bullet)$ to the category of $\textup{R}_{\msw_u \times \msc^3}$-modules where $\msw_u \times \msc^3 := \mbc_z \times W_u \times \mbc^3$ and
\[
\textup{R}_{\msw_u \times \msc^3} = \mbc[z,(w_{iu})_{i \neq u},\lambda_0,\lambda_1,\lambda_2]\langle (z\p_{w_{iu}})_{i \neq u}, z\p_{\lambda_0},z\p_{\lambda_1},z\p_{\lambda_2} \rangle
\]

This is done by the Rees construction, i.e. we associate to the filtered $D_{W_U \times \mbc^3}$-modules $(N_u,F^H_\bullet)$  resp. $(\ck{K}{}^\bullet_u,F_\bullet)$ the $\textup{R}_{\msw_u \times \msc^3}$-modules
\[
\textup{N}_u  := R_F N_u = \bigoplus_{p \in \mbz} F_p N_u z^p , \qquad \textup{K}^\bullet_u := R_F \ck{K}{}^\bullet_u
\]
and similarly for the filtered $\mcd_{\mbp^2 \times \mbc^3}$-module $(\mcn,F^H_\bullet)$ to which we associate $\msn := \mco_{\msp \times \msc^3} \otimes_{\mco_{\mbp^2 \times \mbc^3}[z]} R_{F^H} \mcn$, where $\msp\times \msc^3:=\dC_z\times \dP^2\times \dC^3$.

Instead of computing the projection of the filtered $\mcd$-module $(\mcn, F^H_\bullet)$ we compute the projection of the $\msr$-module $\msn$. This is given by
\[
\pi_{2+} \msn \simeq R\pi_{2*} \DR_{\msp \times \msc^3 / \msc^3}(\msn),
\]
where this time $\pi_2$ denotes the map $\msp \times \msc^3 \ra \msc^3$.

Since this is hard to compute directly we construct a resolution $\msk^\bullet$ from the local resolutions $\textup{K}^\bullet_u$ and get the double complex $\Omega^{\bullet+2}_{\msp \times \msc^3 / \msc^3} \otimes \msk^\bullet$:
\[
\xymatrix{\msk^0  \ar[rr]^-{_{II}d^{1,0}} & & \Omega^1_{\msp \times \msc^3/\msc^3} \otimes \msk^0 \ar[rr]^-{_{II}d^{2,0}} & & \Omega^2_{\msp \times \msc^3/\msc^3} \otimes \msk^0 \\
\msk^{-1}  \ar[u]^-{_{I}d^{0,0}} \ar[rr]^-{_{II}d^{1,-1}} & & \Omega^1_{\msp \times \msc^3/\msc^3} \otimes \msk^{-1} \ar[u]^-{_{I}d^{1,0}} \ar[rr]^-{_{II}d^{2,-1}} & & \Omega^2_{\msp \times \msc^3/\msc^3} \otimes \msk^{-1} \ar[u]^-{_{I}d^{2,0}} \\
\msk^{-2} \ar[u]^-{_{I}d^{0,-1}}  \ar[rr]^-{_{II}d^{1,-2}} & & \Omega^1_{\msp \times \msc^3/\msc^3} \otimes \msk^{-2} \ar[u]^-{_{I}d^{1,-1}} \ar[rr]^-{_{II}d^{2,-2}} & & \Omega^2_{\msp \times \msc^3/\msc^3} \otimes \msk^{-2} \ar[u]^-{_{I}d^{2,-1}} }
\]
This double complex gives rise to two spectral sequences: The first one is given by first taking cohomology in the vertical direction. This gives the ${_I}E_1$-page where only the ${_I}E_1^{0,q}$-terms are non-zero and are isomorphic to $\Omega^{q+2}_{\msp \times \msc^3 /\msc^3} \otimes \mcn$. If we consider the second spectral sequence and take cohomology in the horizontal direction we get the $_{II}E_1$-page. Here $_{II}E_1^{p,q}= 0$ for $ q \neq 0$ and we set $\mcl^\bullet:= _{II}E_1^{p,0}$. Since both spectral sequences degenerate at the second page we get a quasi-isomorphism $\Omega^{\bullet+2}_{\msp \times \msc^3 /\msc^3} \otimes \mcn \simeq \mcl^\bullet$ (cf. Propsoition \ref{prop:quasiIsos}).

In order to get an explicit representation of $\mcl^\bullet$ we introduce a sheaf of rings $\mss$ on $\msp \times \msc^3$ and an ideal $\msj \subset \mss$ which are locally given by
\[
\Gamma(\msw_u \times \msc^3,\mss) := S_{\msw_u \times \msc^3} :=   \mbc[z,\lambda_0,\lambda_1,\lambda_2,(w_{iu})_{i \neq u}]\langle z \p_{\lambda_0}, z \p_{\lambda_1}, z \p_{\lambda_2}\rangle
\]
resp.
\[
\Gamma(\msw_u \times \msc^3,\msj) = J_{A^s_u}
\]
where $J_{A^s_u}$ is the left ideal in $S_{\msw_u \times \msc^3}$ generated by the corresponding box operators.

Define the following Euler operators
\[
\tilde{E}_0 = \lambda_0 z \p_{\lambda_0} + \lambda_1 z \p_{\lambda_1} +\lambda_2 z \p_{\lambda_2},\qquad  \tilde{E}_1 = \lambda_1 z \p_{\lambda_1} - \lambda_2 z \p_{\lambda_2}
\]
we get the following quasi-isomorphism
\[
\mcl^\bullet \simeq Kos^\bullet(z^{-1} \mss / \msj, (\tilde{E}_k)_{k=0,1})
\]
hence we get $\pi_{2+} \msn \simeq R \pi_{2*} \left( Kos^\bullet(z^{-1} \mss / \msj, (\tilde{E}_k)_{k=0,1}) \right)$.

Since $\msc^3=\dC_z\times \dC^3$ is affine it is enough to compute the global sections of $\pi_{2+} \msn$ which are given by
\[
R \Gamma \left( Kos^\bullet(z^{-1} \mss / \msj, (\tilde{E}_k)_{k=0,1}) \right).
\]
 We will show that each term in the Koszul complex is $\Gamma$-acyclic which boils down to the fact that $\mss/ \msj$ is $\Gamma$-acyclic.

Define the matrix
\[
A^s := \left(\begin{matrix}
1 & 1 & 1 & 0 & 0 & 0\\
0 & 0 & 0 & 1 & 1 & 1\\
0 & -1 & -1 & 0 & -1 & -1
\end{matrix} \right)
\]
and the ring
\[
\textup{S}:= \mbc[z,w_0,w_1,w_2,\lambda_0,\lambda_1,\lambda_2]\langle z \p_{\lambda_0}, z \p_{\lambda_1} , z \p_{\lambda_2} \rangle
\]
Let $J_{A^s}$ be the left ideal in $S$ generated by
\[
w_0^2-w_1w_2, \qquad \p_{\lambda_0}^2 - \p_{\lambda_1} \p_{\lambda_2},\qquad w_1\p_{\lambda_1}-w_0 \p_{\lambda_0},\qquad w_2\p_{\lambda_2} - w_0 \p_{\lambda_0}
\]
(these are box operators with respect to the matrix $A^s$). The associated sheaf $\widetilde{S /J_{A^s}}$ is isomorphic to $\mss / \msj$. The associated graded module is defined by
\[
\Gamma_*(\mss /\msj) := \bigoplus_{a \in \mbz} \Gamma(\msp \times \msc^3, (\mss/\msj )(a))
\]
The difference between $S/J_{A^s}$ and the associated graded module is measured by local cohomology modules of $S/J_{A^s}$, more precisely we have (cf. Proposition \ref{prop:GrLocCoh})
\begin{align*}
0 \lra &H^0_{(\underline{w})}(S/J_{A^s}) \lra S/J_{A^s} \lra \Gamma_*(\mss /\msj) \lra H^1_{(\underline{w})}(S/J_{A^s}) \lra 0 \\
\end{align*}
and
\begin{equation}
\label{eq:locCohomExample}
\bigoplus_{a\in \mbz} H^i(\msp \times \msc^3, (\mss / \msj)(a)) \simeq H^{i+1}_{(\underline{w})}(S/J_{A^s})
\end{equation}
where $(\underline{w})$ is the ideal in $\mbc[z,w_0,w_1,w_2,\lambda_0,\lambda_1,\lambda_2]$ generated $w_0,w_1,w_2$.  Notice that all terms involved carry a natural $\mbz$-grading  by setting $\deg(w_i) =1$ and $\deg(\lambda_i) = deg(\p_{\lambda_i}) = 0$ for $i=0,1,2$. The generators of $J_{A^s}$ lie in the commutative subring $T := \mbc[w_0,w_1,w_2,\p_{\lambda_0},\p_{\lambda_1},\p_{\lambda_2}] \subset S$.
We denote by $K_{A^s}$ the corresponding ideal in $T$. It is easily seen that the ring $T/K_{A^s}$ is isomorphic to the semi-group ring $\mbc[\mbn A^s]$.

We prove that the local cohomology modules turning up in Formula \eqref{eq:locCohomExample} above can be rewritten as follows (cf. Lemma \ref{lem:tensorprod}):
\[
H^k_{(\underline{w})}(S/ J_{A^s}) \simeq S \otimes_T H^k_I(\mbc[\mbn A^s])
\]
where the ideal $I \subset \mbc[\mbn A^s]$ is generated by $w_0,w_1,w_2$. Hence, we have reduced the problem to a well-known subject in commutative algebra, since the local cohomology groups $H^k_I(\mbc[\mbn A^s])$  can be explicitly computed by the so-called Ishida complex. Let $\sigma$ be th face which is generated by the first three columns of $A^s$ (the columns which correspond to variables $w_0,w_1,w_2$). For a face $\tau \subseteq \sigma$ we define the localization $\mbc[\mbn A^s]_\tau := \mbc[\mbn A^s +  \mbz(A^s \cap \tau)]$. Put
\[
L^k_\tau := \bigoplus_{\tau \subset \sigma \atop \dim \tau =k} \mbc[\mbn A^s]_\tau
\]
The Ishida complex therefore  takes the form
\[
L^\bullet_\sigma : 0 \lra L^0_\sigma \lra L^1_\sigma \lra L^2_\sigma \lra 0
\]
We prove in Proposition \ref{prop:LocCohIshida} that $H^k_I(\mbc [\mbn A^s]) \simeq H^k(L^\bullet_\sigma)$. Finally we show in Corollary \ref{cor:IshidaNegative} that $H^k(L^\bullet_\sigma) = 0$ for $k \neq 2$ (so we have local cohomology only in the top degree) and that the $\mbz$-degrees in $H^2(L^\bullet_\sigma)$ are purely negative. We refer the reader to Example \ref{ex:locCohP1s} for more details in this particular case.\\

We can therefore conclude that
\[
S/J_{A^s} \simeq \Gamma_*(\mss/\msj)\qquad \text{and} \qquad H^i(\msp \times \msc^3, \mss/\msj)) = 0 \quad \text{for all}\; i \geq 1
\]

Putting things together we conclude that the global sections of $\pi_{2+} \mcn$ are given by $\Gamma(Kos^\bullet(z^{-1} \mss / \msj, (\tilde{E}_k)_{k=0,1}))$. The latter one can be easily computed and gives
\[
\Gamma \mch^0 \pi_{2+} \mcn = z^{-1} R_{\msc^3} / I^\lambda_{\widetilde{A}} \qquad \text{and} \qquad \Gamma \mch^i \pi_{2+} \mcn = 0 \quad \text{for}\; i \geq 1
\]
where $R_{\msc^3} := \mbc[z,\lambda_0,\lambda_1,\lambda_2]\langle z \p_{\lambda_0}, z \p_{\lambda_1}, z \p_{\lambda_2}\rangle$ and the left ideal $I^\lambda_{\widetilde{A}}$ is generated by the box operator $(z\p_{\lambda_0})^2 - (z \p_{\lambda_1})(z \p_{\lambda_2})$ and the Euler operators $\tilde{E}_0, \tilde{E}_1$. But this shows that
\[
(\mcm_{\widetilde{A}}^0,F_\bullet^H) \simeq (\mcm^0_{\widetilde{A}}, F^{ord}_{\bullet +1})
\]
which is the statement of Theorem \ref{thm:HodgeMirroP1} resp. that of Theorem \ref{thm:HodgeGKZ} below in
the general case.

\section{GKZ-systems and the Fourier-Laplace transformation}

We start by introducing GKZ-systems as well as their Fourier-Laplace transformed versions. Throughout the whole paper, we let $W$ be a finite-dimensional vector space over $\mbc$ and denote by $V$ its dual vector space. We will fix coordinates $w_1,\ldots, w_n$ on $W$ and
dual coordinates $\lambda_1,\ldots,\lambda_n$ on $V$.

\subsection{GKZ-systems and strict resolutions}\label{subsec:GKZ}

Given a $d \times n$ integer matrix $A=(a_{ki})$ we denote by $\underline{a}_1,\ldots, \underline{a}_n$ its columns. We define
\[
\mbn A := \sum_{i=1}^n \mbn \underline{a}_i
\]
and similarly for $\mbz A$ and $\mbr_{\geq 0} A$. Throughout the paper we assume that the matrix $A$ satisfies
\[
\mbz A = \mbz^d\, .
\]
\begin{definition}\label{def:GKZ}
Let $A =(a_{ki})$ be a $d\times n$ integer matrix with $\dZ A=\dZ^d$ and  $\beta = (\beta_1, \ldots , \beta_d) \in \mbc^d$. Write $\mbl_A$ for the $\mbz$-module of integer relations among the columns of $A$ and write $\mcd_V$ for the sheaf of rings of differential operators on $V$. Define
\[
\mcm^\beta_A := \mcd_V/ \cI_A\, ,
\]
where $\cI_A$ is the sheaf of left ideals generated by \[
\Box_{\underline{l}} := \prod_{i : l_i <0} \p_{\lambda_i}^{-l_i} - \prod_{i: l_i > 0} \p_{\lambda_i}^{l_i}
\]
for all $\underline{l}\in\dL_A$
and
\[
E_k -\beta_k := \sum_{i=1}^n a_{ki} \lambda_i \p_{\lambda_i} - \beta_k
\]
for $i=1,\ldots,d$.
\end{definition}

Since GKZ-systems are defined on the affine space $V\cong\mbc^{n}$, we will often work with the $D$-modules of global sections $M^\beta_A := \Gamma(\mbc^{n} , \mcm^\beta_A)$ rather than with the sheaves themselves.\\

We will now discuss filtrations on GKZ-systems given by a weight vector $(u,v) \in \mbz^{2n}$.  This weight vector induces an increasing  filtration on $D_{V}$  given by
\[
F^{(u,v)}_p D_{V} =  \left\lbrace \sum_{\sum_i u_i \gamma_i + v_i \delta_i \leq p \atop \text{finite}} c_{\gamma \delta} \lambda^\gamma \p_{\lambda}^\delta  \mid \gamma, \delta \in \mbz_{\geq 0}^n \right\rbrace
\]
where we set $\lambda^\gamma := \prod_{i=1}^n \lambda_i^{\gamma_i}$ etc. . For an element $P = \sum c_{\gamma\delta} \lambda^\gamma \p_{\lambda}^\delta$ we define $ord_{(u,v)}(P) := \max\{\sum_i u_i \gamma_i + v_i \delta_i \mid c_{\gamma \delta} \neq 0 \}$.
The associated graded ring $\gr^{(u,v)}_\bullet D_{V}$ is given by $\bigoplus_p F_p^{(u,v)}D_{V} / F_{p-1}^{(u,v)} D_{V}$.\\

In order to construct a strictly filtered resolution of $\mcm^\beta_A$, we use the theory of Euler-Koszul complexes as introduced in
\cite{MillerWaltherMat}. We will work at the level of global sections.
We briefly recall the definition of the Euler-Koszul complex $(K^\bullet, E-\beta)$
from \cite[Definition 4.2]{MillerWaltherMat} (where it is called  $\mck_\bullet(E-\beta;\mbc[\mbn A])$ and placed in positive homological degrees).
Its terms are given by
\[
K^{-l} = \bigoplus_{0 \leq i_1 < \ldots < i_l \leq l} (D_{V} / J_A) e_{i_1\ldots i_{i_l}}\, ,
\]
where the left ideal $J_A \subset \mbc[\underline{\p}]:= \mbc[\p_{\lambda_1},\ldots, \p_{\lambda_n}]$ is generated by
\[
\Box_{\underline{l}} := \prod_{i : l_i <0} \p_{\lambda_i}^{-l_i} - \prod_{i: l_i > 0} \p_{\lambda_i}^{l_i},\quad\forall \underline{l}\in \mbl_A
\]
A simple computation using the fact that $\sum_{i=1}^n l_i a_{ki} = 0$ shows that the maps
\begin{align}
D_{V}/D_V J_{A} &\lra D_{V}/ D_V J_{A} \notag \\
P &\mapsto P \cdot (E_k-\beta_k) \qquad \text{for} \quad k= 1, \ldots ,d\label{eq:EulerFieldsWellDefined}
\end{align}
are well defined. Moreover, we have $[E_{k_1}-\beta_{k_1}, E_{k_2}-\beta_{k_2}] = 0$ for $k_1, k_2 \in \{1, \ldots ,d \}$, and hence
we can build the Koszul complex
\[
(K^\bullet,E-\beta) = (\ldots \overset{d_{-2}}\lra K^{-1} \overset{d_{-1}}\lra K^0 \ra 0):= \textup{Kos}( D_{V}/D_V J_{A}, (E_k-\beta_k)_{1 = 0 , \ldots ,d})\, .
\]
with $D_V$-linear differential
\[
d_{-l}(e_{i_1\ldots i_l}) := \sum_{k=1}^l (-1)^{l-1}(E_{i_k}-\beta_{i_k})e_{i_1\ldots \hat{i}_k\ldots i_l}
\]
If we assume that the semigroup $\mbn A$ satisfies
\[
\mbn A  =\mbz^d \cap \mbr_{\geq 0} A
\]
then by a classical
result due to Hochster (\cite[theorem 1]{Hoch}) it follows that the semigroup ring $\mbc[\mbn A]$ is Cohen-Macaulay. It was shown in \cite[Remark 6.4]{MillerWaltherMat} that in this case $(K^\bullet,E-\beta)$ is a resolution of $M^\beta_A$ for all $\beta \in \mbc^d$.\\

Notice that the filtration $F^{(u,v)}_\bullet$ on $D_V$ induces a filtration on $D_V/D_V J_A$ which we denote by the same symbol. We define the following filtration on each term of the Koszul complex $(K^\bullet,E-\beta)$:
\[
F^{(u,v)}_{p}K^{-l} := \bigoplus_{0 \leq i_1 < \ldots < i_l \leq l} F^{(u,v)}_{p -\sum_{k=1}^l c_{i_k} }(D_{V} / D_V J_A) e_{i_1\ldots i_{i_l}}\, ,
\]
where $c_{i} = ord_{(u,v)}(E_i-\beta_i)$. This shows that the complex $((K^\bullet,E-\beta),F^{(u,v)}_\bullet)$ is filtered, i.e. that the differential $d$ respects the filtration
\[
d_{-l}(F^{(u,v)}_p K^{-l}) \subset F^{(u,v)}_p\, d_{-l}(K^{-l}) := \textup{im} ( d_{-l}) \cap F^{(u,v)}_p\, K^{-l+1} \, .
\]
We recall the following well-known criterion for a complex to be strictly filtered, which means
\[
d_{-l}(F^{(u,v)}_p K^{-l}) = F^{(u,v)}_p d_{-l}(K^{-l}) = im(d_{-l}) \cap F^{(u,v)}_p K^{-l+1}
\]

\begin{lemma}\label{lem:equivChar}
Let
\[
0 \lra (M_1,F) \overset{d_1}{\lra} \ldots \overset{d_{n-1}}{\lra} (M_n,F) \ra 0
\]
be a sequence of filtered $D$-modules with bounded below filtration. The following properties are equivalent.
\begin{enumerate}
\item The map $d_k$ is strict.
\item $H^{k} (F_p M_\bullet) \simeq F_p H^{k}(M_\bullet)$ for all $p$.
\item $H^{k} (gr^F_p M_\bullet) \simeq\gr^F_p H^k (M_\bullet)$ for all $p$.
\end{enumerate}
\end{lemma}

\begin{remark}\label{rem:strictresprop} Suppose that we have $\dN A=\dZ^d \cap \dR_{\geq 0} A$, then in order to prove that the filtered complex $((K^\bullet,E-\beta),F^{(u,v)}_\bullet)$ is strict it is enough to show that $H^{-l}(\gr^{(u,v)}_\bullet K^{\bullet}) = 0$ for $l > 1$ and $H^0(\gr^{(u,v)}_\bullet K^{\bullet}) = \gr^{(u,v)}_\bullet M^\beta_A$, since we already know that $H^{-l}(K^\bullet) = 0$ for $l > 1$ and $H^0(K^\bullet) = M_A^\beta$.
\end{remark}

\subsection{Fourier-Laplace transformed GKZ-systems}

Let as before $W$ be a $n$-dimensional vector space over $\mbc$ and denote by $V$ its dual vector space. Let $X$ be a smooth algebraic variety and $E = X \times W$ be a trivial vector bundle and $E':= X \times V$ its dual. We write $\langle , \rangle: W \times V \ra \mbc$ for the canonical pairing which extends to a function $\langle,\rangle: E \times E' \ra \mbc$.
\begin{definition}
Define $\mcl := \mco_{E \times_X E'} e^{-\langle, \rangle}$ which is by definition the free rank one module with differential given by the product rule. Denote by $p_1: E \times_X E' \ra E$, $p_2: E \times_X E \ra E'$  the canonical projections. For $\mcm \in D^b_h(\mcd_E)$ the Fourier-Laplace transformation is then defined by
\[
\FL_X(\mcm) := p_{2+}(p_1^+ \mcm \overset{L}\otimes \mcl)[-n]
\]
\end{definition}

\begin{definition}\label{def:SWGarbe}
Let $A=(a_{ki})$ be a $d \times n$ integer matrix. Let $\beta \in \mbc^d$. Write $\mbl_A$ for the $\mbz$-module of relations among the columns of $A$ and write $\mcd_{W}$ for the sheaf of rings of algebraic differential operators on $W$. Define
\[
\check{\mcm}^\beta_A := \mcd_{W}/\left( (\check{\Box}_{\underline{m}})_{\underline{m} \in \mbl_A} , (\check{E}_k +\beta_k)_{k=1, \ldots ,d}\right)\, ,
\]
where
\begin{align}
\check{E}_k &:= \sum_{i=1}^{n} a_{ki} \p_{w_i} w_i  \quad  \text{for}\;\; k= 1, \ldots , d \notag \\
\check{\Box}_{\underline{m} \in \mbl_A} &:= \prod_{m_i > 0} w_i^{m_i} - \prod_{m_i < 0} w_i^{-m_i}\, . \label{eq:boxopSW}
\end{align}
\end{definition}
Again we will often work with the $D_W$-module of global sections
\[
\check{M}^\beta_A := \Gamma(W, \check{\mcm}^\beta_A)
\]
of the $\mcd_W$-module $\check{\mcm}^\beta_A$. Sometimes we will be interested  in the case $\beta = 0$ and will write
\[
\check{\mcm}_A := \check{\mcm}_A^0 \qquad \text{and} \qquad M_A := \Gamma(W, \check{\mcm}_A)\, .
\]

\begin{remark}
Notice that $\check{\mcm}^\beta_A$ is just a Fourier-Laplace transformation (in all variables) of the GKZ-system $\mcm^\beta_A$ (cf. Definition \ref{def:GKZ}).
\end{remark}

\noindent The semigroup ring associated with the  matrix $A$ is
\[
\mbc[\mbn A] \simeq \mbc[\underline{w}]/\left((\check{\Box}_{\underline{m}})_{\underline{m} \in \mbl_A}\right)\, ,
\]
where $\mbc[\underline{w}]$ is the commutative ring $\mbc[w_1, \ldots, w_n]$ and the isomorphism follows from \cite[Theorem 7.3]{MillSturm}.
The rings $\mbc[\underline{w}]$ and $\mbc[\mbn A]$ are naturally $\mbz^{d}$-graded if we define $\deg(w_j) = \underline{a}_j$ for $j=1, \ldots , n$. This is compatible with the $\mbz^{d}$-grading of the Weyl algebra $D_W$ given by $\deg(\p_{w_j}) = - \underline{a}_j$ and $\deg(w_j) = \underline{a}_j$.

\begin{defn}[{\cite[Definition 5.2]{MillerWaltherMat}}]
Let $N$ be a finitely generated $\mbz^{d}$-graded $\mbc[\underline{w}]$-module. An element
$\alpha \in \mbz^{d}$ is called a true degree of $N$ if $N_\alpha$ is non-zero. A vector $\alpha \in \mbc^{d}$ is called a quasi-degree of $N$, written $\alpha \in qdeg(N)$, if
$\alpha$ lies in the complex Zariski closure $qdeg(N)$ of the true degrees of $N$ via the natural embedding $\mbz^{d} \hookrightarrow \mbc^{d}$.
\end{defn}

\noindent Schulze and Walther now define the following set of parameters:
\begin{defn}[\cite{SchulWalth2}]\label{def:sRes}
The set
\[
sRes(A) := \bigcup_{j=1}^n sRes_j(A)\, ,
\]
where
\[
sRes_j(A) := \{ \beta \in \mbc^{d} \mid \beta \in -(\mbn +1)\underline{a}_j + qdeg(\mbc[\mbn A] / ( w_j ))\}
\]
is called the set of strongly resonant parameters of $A$.
\end{defn}

The matrix $A$ is called pointed if  $0$ is the only unit in $\mbn A$. The matrix $A$ gives rise to a map from a torus $T = (\mbc^*)^d$ with coordinates $(t_1, \ldots ,t_d)$ into the affine space $W = \mbc^n$ with coordinates $w_1, \ldots , w_n$:
\begin{align}
h_A: T &\lra W \notag \\
(t_1, \ldots ,t_d) &\mapsto (\underline{t}^{\underline{a}_1},\ldots, \underline{t}^{\underline{a}_n})\, , \notag
\end{align}
where $\underline{t}^{\underline{a}_i} := \prod_{k=1}^d t_k^{a_{ki}}$. Notice that the map $h_A$ is affine and a locally closed embedding, hence the direct image functor for $\mcd_T$-modules $(h_A)_+$ is exact.\\

For a pointed matrix $A$ Schulze and Walther computed the direct image of the twisted structure sheaf
\[
\mco_T^\beta := \mcd_T / \mcd_T \cdot \left(  \p_{t_1} t_1 + \beta_1, \ldots \p_{t_d} t_d  + \beta_d\right)
\]
under the morphism $h_A$.

\begin{thm}[\cite{SchulWalth2}\label{thm:SW} Theorem 3.6, Corollary 3.7]\label{thm:GKZisoSW}
Let $A$ a pointed $(d \times n)$ integer matrix satisfying $\mbz A = \mbz^d$, then the following statements are equivalent
\begin{enumerate}
\item $\beta \notin sRes(A)$.
\item $\check{\mcm}^\beta_A \simeq (h_A)_+ \, \mco_{T}^\beta$.
\item Left multiplication with $w_i$ is invertible on $\check{M}^{\beta}_{A}$ for $i= 1, \ldots ,n$.
\end{enumerate}
\end{thm}

Notice that Schulze and Walther \cite{SchulWalth2} use the GKZ-system $\mcm^\beta_A$ and the convention $deg(\p_{\lambda_j}) = \underline{a}_j$. We will use $\check{\mcm}^\beta_A$ and  $deg(w_j) = \underline{a}_j $ instead.\\

The aim of section is to generalize the implication $1. \Rightarrow 2.$ to the case of a non-pointed matrix $A$. For this we set $\underline{a}_0 := 0$.
We will associate to the matrix $A$ the homogenized $(d+1 \times n+1)$ matrix $\widetilde{A}$ with columns $\widetilde{\underline{a}}_i := (1, \underline{a}_i)$ for $i = 0, \ldots ,n$. Notice that $\mbz \widetilde{A} = \mbz^{d+1}$ holds and that the matrix $\widetilde{A}$ is pointed in any case. Consider now the augmented map
\begin{align}
h_{\widetilde{A}}: \widetilde{T} \lra \widetilde{W} \notag \\
(t_0, \ldots , t_d) &\mapsto (t_0 \underline{t}^{\underline{a}_0}, t_0 \underline{t}^{\underline{a}_1}, \ldots, t_0 \underline{t}^{\underline{a}_n})\, ,
\end{align}
where $\widetilde{T} = (\mbc^*)^{d+1}$ and $\widetilde{W} = \mbc^{n+1}$ with coordinates $w_0, \ldots ,w_n$. Let $\widetilde{W}_0$ be the subvariety of $\widetilde{W}$ given by $w_0 \neq 0$ and denote by $k_0: \widetilde{W}_0 \ra \widetilde{W}$ the canonical embedding. The map $h_{\widetilde{A}}$ factors through $\widetilde{W}_0$ which gives rise to a map $h_0$ with $ h_{\widetilde{A}} = k_0 \circ h_0$. We get the following commutative diagram
\begin{equation}\label{eq:facdiagu}
\begin{tikzcd}
\widetilde{T} \ar{r}{h_0} \ar{d}{\pi} \ar[bend left=35]{rr}{h_{\widetilde{A}}} & \widetilde{W}_0 \ar{d}{\pi_0} \ar{r}{k_0}& \widetilde{W} \\ T \ar{r}{h_A} & W &
\end{tikzcd}
\end{equation}
where $\pi$ is the projection which forgets the first coordinate and $\pi_0$ is given by
\begin{align}
\pi_0: \widetilde{W}_0 &\lra W \notag \\
(w_0,w_1,\ldots ,w_n) &\mapsto (w_1 /w_0, \ldots , w_n /w_0)\, . \notag
\end{align}

\begin{lemma}\label{lem:SWchart1}
For each $\beta_0 \in \mbz$ we have an isomorphism:
\[
\mch^0 \left(h_{A +} \mco_T^\beta\right) \simeq
\mch^0\left(\pi_{0 +} k_0^+\left(h_{\widetilde{A} +} \, \mco_{\widetilde{T}}^{(\beta_0,\beta)}\right)\right)\, .
\]
\end{lemma}
\begin{proof}
We show the claim by using the following isomorphisms
\begin{align}
\mch^0 h_{A+} \mco_T^\beta &\simeq
\mch^0 h_{A+} \mch^0 \pi_+ \mco_{\widetilde{T}}^{(\beta_0,\beta)} \simeq \mch^0 h_{A +} \pi_+ \mco_{\widetilde{T}}^{(\beta_0,\beta)} \simeq
\mch^0 \pi_{0 +}  h_{0 +} \mco_{\widetilde{T}}^{(\beta_0,\beta)} \notag \\ &\simeq
\mch^0 \pi_{0 +} k_0^+ k_{0 +} h_{0 +} \mco_{\widetilde{T}}^{(\beta_0,\beta)} \simeq
\mch^0 \pi_{0 +}  k_0^+ h_{\widetilde{A} +} \mco_{\widetilde{T}}^{(\beta_0,\beta)}\, . \notag
\end{align}
The first isomorphism follows  from the fact that $\pi$ is a projection with fiber $\mbc^*$, the second isomorphism follows from the exactness of $(h_A)_+$ and the fourth from the fact that $k_0^+ (k_0)_+ \simeq id_{\widetilde{W}_0}$.

\end{proof}

The following proposition is the generalization of Theorem \ref{thm:SW} to the non-pointed case.

\begin{proposition}\label{prop:SWgeneral}
Let $A=(a_{ki})$ be a $d \times n$ integer matrix satisfying $\mbz A = \mbz^d$ and let $\beta\in \mbc^d$ with $\beta \notin sRes(A)$, then $\mch^0 \left((h_A)_+ \mco_T^\beta\right)$ is isomorphic to $\check{\mcm}_{A}^\beta$.
\end{proposition}
\begin{proof}
The proof relies on Lemma  \ref{lem:SWchart1} and the theorem of Schulze and Walther in the pointed case. Notice that we can find a $\beta_0 \in \mbz$ with $\beta_0 \gg 0$ such that $(\beta_0,\beta) \notin sRes(\widetilde{A})$ by \cite[Lemma 1.16]{Reich2} (in loc. cit. the statement is formulated for $\beta \in \mbq^d$ but the proof carries over almost word for word in this more general case).

Consider the following isomorphism on $\widetilde{W}_0$:
\begin{align}
f: \widetilde{W}_0 &\lra  \mbc^*_{w_0} \times W \notag \\
(w_0, \ldots w_n) &\mapsto ((w_0, w_1/w_0, \ldots ,w_{n}/w_0) \notag
\end{align}
together with the canonical projection $p: W \times \mbc^*_{w_0} \ra W$ which forgets the first coordinate. This factors $\pi_0 = p \circ f$, which gives (using Lemma \ref{lem:SWchart1} above)
\begin{align}
\mch^0 \left((h_A)_+ \mco_T^\beta\right) \simeq &\mch^0\left( (\pi_0)_+ \left((h_{\widetilde{A}})_+ \mco_{\widetilde{T}}^{(\beta_0,\beta)}\right)_{\mid \widetilde{W}_0}\right) \simeq  \mch^0\left(p_{+} f_{+} \left((h_{\widetilde{A}})_+ \mco_{\widetilde{T}}^{(\beta_0,\beta)}\right)_{\mid \widetilde{W}_0}\right) \notag \\
\simeq  &\mch^0 \left( p_{+} f_{+} (\check{\mcm}^{(\beta_0,\beta)}_{\widetilde{A}})_{\mid \widetilde{W}_0}\right). \notag
\end{align}
The $\mcd$-module $\mch^0 f_{+} (\check{\mcm}^{(\beta_0,\beta)}_{\widetilde{A}})_{\mid \widetilde{W}_0}$ is isomorphic to $\mcd_{W\times \mbc^*_{w_0}} / \mci'_0$ where $\mci'_0$ is generated by
\[
\check{\Box}_{\underline{m} \in \mbl_{A}} = \prod _{i : m_i >0,} w_{i}^{m_i} - \prod_{i: m_i <0 , } w_{i}^{-m_i}
\]
and
\[
Z_0 = \p_{w_0} w_0 + \beta_0 \qquad \text{and} \qquad E_k = \sum_{i=1}^n a_{ki }\p_{w_i} w_i + \beta_k \]
Hence $\mch^0 f_{+} (\check{\mcm}_{\widetilde{A}}^{(\beta_0,\beta)})_{\mid \widetilde{W}_0}$ is isomorphic to $\check{\mcm}^\beta_A \boxtimes \mcd_{\mbc^*_{w_0}}/ ( \p_{w_0} w_0+ \beta_0)$ as a $\mcd$-module. We therefore have
\[
\mch^0 \left( p_{+} f_{+} (\check{\mcm}_{\widetilde{A}}^{(\beta_0,\beta)})_{\mid \widetilde{W}_0}\right)\simeq \mch^0 \left( p_{+} \mch^0 f_{+} (\check{\mcm}_{\widetilde{A}}^{(\beta_0,\beta)})_{\mid \widetilde{W}_0}\right)  \simeq \mch^0 p_+ \left(\check{\mcm}^\beta_A \boxtimes \mcd_{\mbc^*_{w_0}}/ ( \p_{w_0} w_0 +\beta_0)\right) \simeq \check{\mcm}^\beta_A \, .
\] \\
\end{proof}

\section{Hodge filtration on torus embeddings}

\label{sec:HodgeOnTorusEmbed}

The aim of this section is to compute explicitly the Hodge filtration
of $(h_A)_+ \mco^\beta_T$ as a mixed Hodge module  for certain values of $\beta$ (cf. Theorem \ref{thm:hpot}). We will use this result in section \ref{sec:Radon} where the behavior of mixed Hodge modules obtained by such torus embeddings under the twisted Radon transformation is studied.

\subsection{V-filtration}\label{subsec:Vfilt}

As above let $A$ be a $d \times n$ integer matrix s.t. $\mbz A = \mbz^d$.  In this section we additionally assume that the matrix $A$ satisfies the following conditions:\\
\begin{equation}\label{eq:Bnormal}
\mbn A = \mbz^d \cap \mbr_{\geq 0}A \qquad \text{and} \qquad \mbn A \neq \mbz^d
\end{equation}
where $\mbr_{\geq 0}A$ is the cone generated by the columns of $A$.
As already noticed above, the first  condition is equivalent to the fact that the semigroup ring $\mbc[\mbn A]$ is normal
(see, e.g., \cite[Section 6.1]{HerzBr}).
We will again consider the locally closed embedding
\begin{align}
h_A: T &\lra W \notag \\
(t_1, \ldots t_d) &\mapsto (\underline{t}^{\underline{a}_1}, \ldots , \underline{t}^{\underline{a}_n})\, . \notag
\end{align}
Put $D:=\{w_1\cdot\ldots\cdot w_n =0\}\subset W$, $W^*:=W\backslash D$, and consider the decomposition $h_A=l_A\circ k_A$, where
\begin{align}
k_A: T &\lra W^* \notag \\
(t_1, \ldots t_d) &\mapsto (\underline{t}^{\underline{a}_1}, \ldots , \underline{t}^{\underline{a}_n})\, . \notag
\end{align}
and where $l_A: W^* \rightarrow W$ is the canonical open embedding.
\begin{lemma}\label{lem:SWclosedembed}
The morphism $k_A:T\rightarrow W^*$ is a closed embedding.
\end{lemma}
\begin{proof}
This is clear, as the image of $k_A$ is precisely the vanishing locus
of $(\check{\Box}_{\underline{m}})_{\underline{m}\in \mbl_A} \subset \Gamma(W^*,\mco_{W^*})$.
\end{proof}

The aim of this subsection is to compute parts of the canonical (descending) $V$-filtration of $\check{\mcm}^\beta_A  \simeq h_{A+} \mco^\beta_T$ (or Kashiwara-Malgrange filtration) along the normal crossing divisor $D$ for certain values of $\beta$.

We review very briefly some facts about the $V$-filtration for differential modules. Let $X=\Spec(R)$ be a smooth affine variety and $Y = div(t)$ be
a smooth reduced principal divisor. Denote by $I = (t)$ the corresponding ideal. Let as before $D_X=\Gamma(X,\cD_X)$ be the
ring of algebraic differential operators on $X$, then the $V$-filtration on $D_X$ is defined by
\[
V^k D_X = \{ P \in D_X \mid P I^j \subset I^{j+k} \quad \text{for any}\quad j \in \mbz\}\, ,
\]
where $I^j = R$ for $j \leq 0$. One has
\begin{align}
V^k D_X &= t^k V^0 D_X\, , \notag \\
V^{-k}D_X &= \sum_{0 \leq j \leq k} \p_t^j V^0 D_X\, . \notag
\end{align}
Choose a total ordering $<$ on $\mbc$ such that, for any $\alpha, \beta \in \mbc$, the following conditions hold:
\begin{enumerate}
\item $\alpha < \alpha+1$,
\item $\alpha < \beta \quad \text{if and only if} \quad \alpha +1 < \beta+1$,
\item $\alpha < \beta +m \quad \text{for some}\; m \in \mbz$.
\end{enumerate}
We recall the definition of the canonical $V$-filtration (see, e.g., \cite[Section 1]{Saitobfunc}).
\begin{definition}\label{def:CanonicalVFiltration}
Let $N$ be a coherent $D_X$-module. The canonical $V$-filtration (or Kashiwara-Malgrange filtration) is an exhaustive filtration on $N$ indexed discretely by $\mbc$ with total order as above and is uniquely determined by the following conditions
\begin{enumerate}
\item $(V^k D_X)(V^\alpha N) \subset V^{\alpha+k}N$ for all $k,\alpha$
\item $V^{\alpha}N$ is coherent over $V^0D_X$ for any $\alpha$
\item $t (V^\alpha N) = V^{\alpha+1}N$ for $\alpha \gg 0$
\item the action of $\p_t t - \alpha$ on $\gr^\alpha_V N = V^\alpha N / V^{> \alpha} N$ is nilpotent
\end{enumerate}
where $V^{> \alpha} N := \bigcup_{\beta > \alpha} V^\beta$.\\

The canonical $V$-filtration is unique if it exists. Its existence is guaranteed if $N$ is $D_X$-holonomic.
\end{definition}

We reduce the computation of the $V$-filtration on $\check{M}_A^\beta$ along the possibly singular divisor $D$ to the computation of a $V$-filtration along a smooth divisor by considering the following graph embedding:
\begin{align}
i_{g} : W &\lra W \times \mbc_t \notag \\
(w_1, \ldots , w_{n}) &\mapsto (w_1, \ldots , w_{n} , w_1 \cdot \ldots \cdot w_n)\, . \notag
\end{align}
Instead of computing the $V$-filtration on $\check{M}^\beta_A$, we will compute it on $\Gamma(W \times \mbc_t, \mch^0(i_{g +} \check{\mcm}_A^\beta))$ along $t=0$ (notice that $i_g$ is an affine embedding hence $i_{g +}$ is exact). In order to compute the direct image we consider the composed map
\begin{align}
i_g \circ h_A: T &\lra W \times \mbc_t \notag \\
(t_1, \ldots , t_{d}) &\mapsto (\underline{t}^{\underline{a}_1},  \ldots , \underline{t}^{\underline{a}_{n}}, \underline{t}^{\underline{a}_1 + \ldots + \underline{a}_n})\, . \label{eq:graphembed}
\end{align}
Notice that the matrix $A'$, which is built from the columns $\underline{a}_{1}, \ldots , \underline{a}_{n}, \underline{a}_1 + \ldots  + \underline{a}_n$, gives a saturated semigroup $\mbn A'=\mbn A$. Hence we can apply again Proposition \ref{prop:SWgeneral}  to compute
\[
\check{\mcm}_{A'}^\beta \simeq \mch^0 i_{g +} \check{\mcm}_A^\beta \simeq \mch^0(i_g \circ h_A)_+ \mco_T^\beta\, .
\]
This means that $\mch^0 i_{g +} \check{\mcm}_A^\beta$ is a
cyclic $\mcd_{W \times \mbc_t}$-module $\mcd_{W\times \mbc_t} / \mci'$, where $\mci'$ is generated by
\begin{equation}\label{eq:Eprime}
\check{E}'_k := \sum_{i=1}^{n} a_{ki} \p_{w_i} w_i + c_k \p_t t + \beta_k \quad \text{for}\;\; k= 1, \ldots , d\, ,
\end{equation}
where $c_k = a_{k1}+ \ldots + a_{k n}$ is the $k$ -th component of $c \in \mbz^d$
and
\begin{equation}\label{eq:Boxprime}
\check{\Box}_{\underline{m} \in \mbl_{A'}} := \begin{cases}\prod_{m_i > 0} w_i^{m_i}t^{m_{n+1}} - \prod_{m_i < 0} w_i^{-m_i} & \text{for}\; m_{n+1} \geq 0 \\
\prod_{m_i > 0} w_i^{m_i} - \prod_{m_i < 0} w_i^{-m_i} t^{-m_{n+1}} & \text{for}\; m_{n+1} < 0 \end{cases}
\end{equation}
where $\mbl_{A'}$ is the $\mbz$-module of relations among the columns of $A'$.\\

We are going to use the following characterization of the canonical $V$-filtration along $t = 0$.
\begin{proposition}\cite[Definition 4.3-3, Proposition 4.3-9]{MebkhMais}\label{prop:KMfilteq}
Let $n \in N$ and set $E:= \p_{t} t$.  The Bernstein-Sato polynomial of $n$ is the unitary polynomial of smallest degree, satisfying
\[
b(E)n \in V^{1}(D_X)n\, .
\]
We denote it by $b_n(x) \in \mbc[x]$ and denote the set of roots of $b_n(x)$ by $ord(n)$. The canonical $V$-filtration on $N$ is then given by
\[
V^{\alpha} N = \{n \in N \mid ord(n) \subset [\alpha, \infty)\}\, .
\]
\end{proposition}

We will use this characterization to compute the canonical $V$-filtration on $\check{M}_{A'}^\beta$ along $t=0$ for certain $\beta \in \mbr^d$.\\

Let $\underline{c}:= \underline{a}_1 + \ldots +\underline{a}_n$. For all facets $F$ of $\mbr_{\geq 0} A' =\mbr_{\geq 0} A$ let $0 \neq \underline{n}_F \in \mbz^d$ be the uniquely determined primitive, inward-pointing, normal vector of $F$, i.e. $\underline{n}_F$ satisfies $\langle \underline{n}_F, F \rangle = 0$, $\langle \underline{n}_F, \mbn A \rangle  \subset \mbz_{\geq 0}$ and $\lambda \cdot \underline{n}_F \not \in \mbz^d$ for $\lambda \in [0,1)$ (where $\langle \cdot, \cdot \rangle$ is the Euclidean pairing). Set
\[
e_F := \langle \underline{n}_F, \underline{c} \rangle \in \mbz_{\geq 0}.
\]
We show that $e_F$  is always positive. We have $\underline{c} \neq 0$ since otherwise $0 = - \underline{a}_1 - \ldots -\underline{a}_n \in \mbn A$ and therefore $-\underline{a}_i \in \mbn A$ for all $i \in \{1,\ldots,n\}$ which contradicts the assumption $\mbn A \neq \mbz^d$. Furthermore $\underline{c}$ lies in the interior of $\mbr_{\geq 0}A'$. In order to see this assume to the contrary that $\underline{c}$ lies on some facet $F$ of $\mbr_{\geq 0} A'$.  Then $\langle \underline{n}_F, \underline{c} \rangle = 0 $ holds. For $\underline{a}_i \not \in F$ we have on the one hand $\underline{c} - \underline{a}_i \in \mbn A$ and on the other hand $\langle \underline{c}- \underline{a}_i, \underline{n}_F\rangle < 0$ which is a contradiction. Hence $\underline{c}$ is in the interior of $\mbr_{\geq 0} A$, which shows $e_F \in \mbz_{> 0}$.

We define the following set of admissible parameters $\beta$:
\begin{equation}\label{eq:defAdm}
\mathfrak{A}_A := \bigcap_{F : F \; \text{facet}} \lbrace \mbr \cdot F - [0, \frac{1}{e_F}) \cdot \underline{c} \rbrace
\end{equation}

\begin{lemma}\label{lem:compV1}
Suppose as above that $\mbn A = \mbz^d \cap \mbr_{\geq 0} A$. Consider the cyclic $D_{W\times \dC_t}$-module $\check{M}_{A'}^\beta$, and its generator $[1]\in\check{M}_{A'}$.
Then we have
$ord([1]) \subset [0,1)$ if $\beta \in \mathfrak{A}_A$.
\end{lemma}
\begin{proof}
It was shown in \cite[Theorem 3.5]{ReichSevWalth} that the roots of $b_{[1]}(x)$ for $[1] \in \check{M}^\beta_A$ are contained in the set $\{\epsilon \in \mbc \mid \epsilon \cdot \underline{c} \in \qdeg(\mbc[\mbn A'] / (t)) -\beta  \}$ which is discrete since $\underline{c}$ lies in the interior of $\mbr_{\geq 0} A' = \mbr_{\geq 0} A$ and $\qdeg(\mbc[\mbn A']/(t))$ is a finite union of parallel translates of the complex span of faces of $\mbr_{\geq 0}A'$ (cf. \cite{MillerWaltherMat}).  We will now compute an estimate of the quasi-degrees $\qdeg(\mbc[\mbn A'] / (t))$. For this we remark that $0 = [P] \in \mbc[\mbn A'] /(t)$ for $P \in \mbc[\mbn A']$ iff $\exists P' \in \mbc[\mbn A']$ with $ P = P' \cdot t$. In this case we have $\deg(P)  \in \mbn A +\underline{c}$. \\

Set $ L_F := \{\frac{k}{e_F} \cdot \underline{c} + \mbc \cdot F \mid k = 0,\ldots e_F-1\}$. Then $L = \bigcup_{F : \text{facet}} L_F$ is Zariski closed  and we will show that the set $\deg(\mbc[\mbn A'] / (t))$ is contained in L. Let $P \in \mbc[\mbn A']$ with $0 \neq [P] \in \mbc[\mbn A'] / (t)$ and set $\underline{p}:= deg(P) \in \mbn A$.  Since $- \underline{c} \notin \mbr_{\geq 0}A$ there exist  a facet $F$ and some $\lambda \in [0,1)$ such that $\underline{p} - \lambda \underline{c} \in F$, i.e. $\underline{p} = \lambda \underline{c} + \underline{f}$ for some $\underline{f} \in F$.  We have $\lambda \cdot e_F = \langle \lambda \underline{c}+ \underline{f}, \underline{n}_F\rangle = \langle \underline{p}, \underline{n}_F \rangle \in \mbz_{\geq 0}$. Hence $\underline{p} \in L_F \subset L$.\\

Since  $\qdeg(\mbc[\mbn A'] / (t))$ is by definition the Zariski closure of $\deg(\mbc[\mbn A'] / (t)) $ the former set is contained in $L$. In particular this shows that the roots of $b_{[1]}(x)$ are contained in the set $\{\epsilon \in \mbc \mid  \epsilon \cdot \underline{c} \in L - \beta\}$. Since $L$ is a union of hypersurfaces which are defined over $\mbr$, $\underline{c} \in \mbz^d$  and $\beta \in \mbr^d$, this set is equal to $\{\epsilon \in \mbr \mid  \epsilon \cdot \underline{c} \in L - \beta\}$. Hence for $\beta \in \bigcap_{F : \text{facet}} \lbrace \mbr \cdot F - [0,\frac{1}{e_F}) \cdot \underline{c}\rbrace $ we can guarantee that the roots of $b_{[1]}(x)$ are contained in $[0,1)$.
\end{proof}

We will prove a basic lemma on the set $\mathfrak{A}_A$ which will be of importance later.
\begin{lemma}\label{lem:AvalsRes}
Suppose $\mbn A = \mbz^d \cap \mbr_{\geq 0}$. Then $\mathfrak{A}_A \cap sRes(A) = \emptyset$.
\end{lemma}
\begin{proof}
Recall that $sRes(A) = \bigcup_{j=1}^n sRes_j(A) = \bigcup_{j=1}^n -(\mbn +1)\underline{a}_j + \textup{qdeg}(\mbc[\mbn A] /((w_j))$. Therefore it is enough to show that
\begin{equation}\label{eq:Aqdegemptyset}
\mathfrak{A}_A \cap \left\lbrace -(\mbn +1) \underline{a}_j + \textup{qdeg}(\mbc[\mbn A] /(w_j)) \right\rbrace = \emptyset
\end{equation}
holds.  The following estimate of the quasi-degrees of $\mbc[\mbn A]/(w_j))$ can be shown similarly as in the proof of the lemma above
\[
\textup{qdeg}(\mbc[\mbn A] /((w_j)) \subset L_j := \bigcup_{F : \underline{a}_j \not \in F} \{ \frac{k}{e_{F,j}} \cdot \underline{a}_j + \mbc \cdot F \mid k= 0, \ldots , e_{F,j} -1\}
\]
where $e_{F,j} := \langle \underline{n}_F, \underline{a}_j \rangle$. Hence it is enough to show that for each $j \in \{1,\ldots, n\}$ and each facet $F$ with $\underline{a}_j \not \in F$ the following holds
\begin{equation}\label{eq:sRescutaP}
\left\lbrace \mbr \cdot F- [0, \frac{1}{e_F}) \cdot \underline{c} \right\rbrace \; \cap \; \left\lbrace-(\mbn +1) \underline{a}_j + \bigcup_{k=0}^{e_{F,j} -1}  \frac{k}{e_{F,j}} \cdot  \underline{a}_j + \mbr \cdot F \right\rbrace = \emptyset
\end{equation}
Since $F$ has codimension one in $\mbr^d$ and $\underline{a}_j, \underline{c} \not \in \mbr \cdot F$ we can write $ \underline{c} = \lambda \underline{a}_j + f$ for some $f \in \mbr \cdot F$. We get $e_F = \lambda e_{F,j}$. We conclude that \eqref{eq:sRescutaP} is equivalent to
\[
\left\lbrace \mbr \cdot F- [0, \frac{1}{e_{F,j}}) \cdot \underline{a}_j \right\rbrace \; \cap \; \left\lbrace-(\mbn +1) \underline{a}_j + \bigcup_{k=0}^{e_{F,j} -1}  \frac{k}{e_{F,j}} \cdot  \underline{a}_j + \mbr \cdot F \right\rbrace = \emptyset \, .
\]
But this holds since $(-\frac{1}{e_{F,j}},0] \cap \lbrace -(\mbn +1) +  \{0, \frac{1}{e_{F,j}}, \ldots, \frac{e_{F,j} -1}{e_{F,j}}\} \rbrace = \emptyset$.
\end{proof}

\begin{example}
The sets $sRes(A)$ and $\mathfrak{A}_A$  for the matrix
\[
A = \left(\begin{matrix}-1 & 0 & 1 & 2 \\ \phantom{-}1 & 1 & 1 & 1 \end{matrix}\right)
\]
are sketched below.
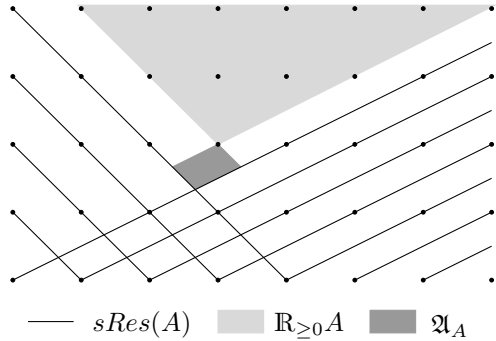
\begin{center}
\newdimen\scale
\scale=0.9cm
\begin{tikzpicture}
 \filldraw[black,opacity=.15] (0,0) -- (\scale*-2.05,\scale*2.05) -- (\scale*4.05,\scale*2.05)  -- (0,0);
 \filldraw[black,opacity=.4] (0,0) -- (\scale*-.66,\scale*-.33) -- (\scale*-.33,\scale*-.66) -- (\scale*.33,\scale*-.33)  -- (0,0);

 \foreach \x in {-3,...,4}{
   \foreach \y in {-2,...,2}{
     \node[draw,circle,inner sep=0.5pt,fill] at (\scale*\x,\scale*\y) {};
   }
 }
\draw (\scale*-3,\scale* -2) -- (\scale*4, \scale *1.5);
\draw (\scale*-2,\scale* -2) -- (\scale*4, \scale *1);
\draw (\scale*-1,\scale* -2) -- (\scale*4, \scale *0.5);
\draw (\scale*0,\scale* -2) -- (\scale*4, \scale *0);
\draw (\scale*1,\scale* -2) -- (\scale*4, \scale *-0.5);
\draw (\scale*2,\scale* -2) -- (\scale*4, \scale *-1);
\draw (\scale*3,\scale* -2) -- (\scale*4, \scale *-1.5);
\draw (\scale*-3,\scale* 2) -- (\scale*1, \scale *-2);
\draw (\scale*-3,\scale* 1) -- (\scale*0, \scale *-2);
\draw (\scale*-3,\scale* 0) -- (\scale*-1, \scale *-2);
\draw (\scale*-3,\scale* -1) -- (\scale*-2, \scale *-2);

\draw(-2.5,-2.35)--(-1.9,-2.35);
\node[] at (-1.0,-2.39) {$sRes(A)$};
\filldraw[black,opacity=.15] (0,-2.2) -- (0.6,-2.2) -- (0.6,-2.5) -- (0,-2.5) -- (0,-2.2);
\node[] at (1.2,-2.41) {$\mbr_{\geq 0} A$};
\filldraw[black,opacity=.4] (2,-2.2) -- (2.6,-2.2) -- (2.6,-2.5) -- (2,-2.5) -- (2,-2.2);
\node[] at (3.05,-2.41) {$\mathfrak{A}_A$};
\end{tikzpicture}
\end{center}
\end{example}

We give another estimate of the set $\mathfrak{A}_A$:
\begin{lemma}
Suppose $\mbn A= \mbz^d \cap \mbr_{\geq 0} A$. Let $\underline{b} \in \mbz^d \cap \textup{int}(\mbr_{\geq 0}A)$, then $\mathfrak{A}_{A} \subset-\underline{b} + \textup{int}(\mbr_{\geq 0} A)$.
\end{lemma}
\begin{proof}
Let $F$ be a face of $\mbr_{\geq 0}A$. Since $F$ has codimension and $\underline{b}, \underline{c} \not \in \mbr F$ we can write $\underline{c} = \lambda \underline{b} +f$ for some $f \in \mbr \cdot F$. If we set $e_{F,\underline{b}} := \langle\underline{n}_F, \underline{b}\rangle \in \mbz_{> 0}$ we get $e_F = \lambda e_{F,\underline{b}}$. hence
\[
\lbrace \mbr \cdot F - [0, \frac{1}{e_F}) \cdot \underline{c}) \rbrace = \lbrace \mbr \cdot F - [0, \frac{1}{e_{F,\underline{b}}}) \cdot \underline{b}  \rbrace \subset \lbrace \mbr \cdot F - [0,1) \underline{b}\rbrace \subset \lbrace \mbr \cdot F - \underline{b} + (0,\infty)\underline{b}  \rbrace \subset  - \underline{b} + \lbrace \mbr \cdot F + (0,\infty)n_{\underline{F}}  \rbrace
\]
hence
\[
\mathfrak{A} = \bigcap_F \lbrace \mbr \cdot F - [0, \frac{1}{e_F}) \cdot \underline{c}) \rbrace \subset - \underline{b} + \text{int}(\mbr_{\geq 0}A)
\]
\end{proof}

Next we draw a consequence for the canonical $V$-filtration with respect to $t=0$ on $\mch^0 i_{g +} \check{\mcm}_A^\beta$. We will not compute all of its filtration steps, but
those corresponding to integer indices, which is sufficient for our purpose.
For this consider the induced $V$-filtration on $\check{M}_{A'}^\beta = \Gamma(W \times \mbc_t, \mch^0 i_{g+} \check{\mcm}_A^\beta)$
\[
V_{ind}^k \check{M}_{A'}^\beta := \{ [P] \in  \check{M}_{A'}^\beta \mid P \in V^k D_{W \times \mbc_t} \}\, .
\]
It is readily checked that $V_{ind}^k \check{M}_{A'}^\beta$ is a good $V$-filtration on
$\check{M}_{A'}^\beta$. As $\check{M}_{A'}$ is holonomic, hence specializable along any smooth hypersurface, it admits
a Bernstein polynomial $b_{V_{ind}}(x)$ in the sense of \cite[D\'efinition 4.2-3]{MebkhMais}. On the other hand, for any
section $\sigma:\dC/\dZ \rightarrow \dC$ of the canonical projection $\dC\rightarrow \dC/\dZ$, there is a \emph{unique}
good ($\dZ$-indexed) $V$-filtration $V^\bullet_\sigma \check{M}_{A'}^\beta$ on $\check{M}_{A'}^\beta$ such that
the roots of $b_{V^\bullet_\sigma}(x)$ lie in $\textit{Im}(\sigma)$ (see loc.cit., Proposition 4.2-6). From this we deduce the following result, which describes the integral part of the canonical $V$-filtration on $\check{M}_{A'}^\beta$.

\begin{proposition}\label{prop:Vfilt}
If $\dN A = \dZ^d\cap \dR_{\geq 0} A$ and $\beta \in \mathfrak{A}_A$, then for any $k\in \dZ$, we have the following equality
\[
V^k \check{M}_{A'}^\beta = V^k_{ind} \check{M}_{A'}^\beta\, .
\]
\end{proposition}

\begin{proof}
Recall (see \cite[Proposition 4.3-5]{MebkhMais}) that we have $V^{\alpha+k} \check{M}_{A'}^\beta = V^k_{\sigma_\alpha} \check{M}_{A'}^\beta$
for any $\alpha \in \dC, k\in \dZ$, where $\sigma_\alpha:\dC/\dZ \rightarrow \dC$ is the section of $\dC\rightarrow \dC/\dZ$ with image equal to $[\alpha,\alpha+1)$.  Hence, in order to prove the proposition it is enough to show that $V^k_{\sigma_0} \check{M}_{A'}^\beta = V^k_{ind} \check{M}_{A'}^\beta$. Using loc. cit., Proposition 4.2-6 it remains to show that the roots of the Bernstein polynomial $b_{V_{ind}}(x)$ are contained in $[0,1)$.\\

An element $[P]$ of $V^k_{ind}\check{M}_{A'}^\beta$ for $k \geq 0$ can be written as
\[
[P] = [\sum_{i=0}^l t^k (\p_{t} t)^i P_i] + [R]\, ,
\]
where $[R] \in V^{k+1}_{ind} \check{M}_{A'}^\beta$ and $P_i \in \mbc[w_1,\ldots,w_{n}]\langle \p_{w_1}, \ldots , \p_{w_{n}} \rangle$.
We have
\begin{align}
b_{[1]}(\p_t t-k) \cdot [P] &= [\sum_{i=0}^l t^k (\p_{t} t)^i P_i  \cdot b_{[1]}(\p_t t)]  + b_{[1]}(\p_t t-k) \cdot [R]  \notag \\
&= \sum_{i=0}^l t^k (\p_{t} t)^i P_i  \cdot b_{[1]}(\p_t t) \cdot [1] + b_{[1]}(\p_t t-k) \cdot [R]\, . \notag
\end{align}
But $\sum_{i=0}^l t^k (\p_{t} t)^i P_i  \cdot b_{[1]}(\p_t t) \cdot [1] \in V^{k+1}_{ind} \check{M}_{A'}^\beta$ because $\sum_{i=0}^l t^k (\p_{t} t)^i P_i \in V^k D$ and $b_{[1]}(\p_t t) \cdot [1] \in V^1_{ind} \check{M}_{A'}^\beta$. Therefore
\[
b_{[1]}(\p_{t}t - k) \cdot [P]  \in  V^{k+1}_{ind} \check{M}_{A'}^\beta.
\]
Now let $[P] \in V^{-k}_{ind}\check{M}_{A'}^\beta$ with $k >0$. It can be written as
\[
[P] = [\sum_{i=0}^l \p_{t}^k(\p_{t} t)^i P_i] +[R]\, ,
\]
where $[R] \in V^{k+1}_{ind} \check{M}_{A'}^\beta$. By a similar argument we have
\[
b_{[1]}(\p_{t} t+k)\cdot [P] \in V^{-k+1}_{ind} \check{M}_{A'}^\beta\, .
\]
This shows $b_{V_{ind}}(x) \mid b_{[1]}(x)$. Because of Lemma \ref{lem:compV1} the roots of $b_{V_{ind}}(x)$ are contained in $[0,1)$, the claim follows.

\end{proof}

\subsection{Compatibility of filtrations}

In this subsection we are going to show a compatibility result between two filtrations on the
$D_{W'}$-module $\check{M}_{A'}^\beta$ (recall that the $d\times (n+1)$-matrix $A'$ has
columns $\underline{a}_1,\ldots,\underline{a}_n,
\underline{a}_1+\ldots+\underline{a}_n$). Let, as before, $F^{ord}$ be the
filtration induced on $M_{A'}^\beta$ by the filtration $F_\bullet D_{W'}$ by orders
of differential operators. Moreover, let $V^\bullet D_{W'}$ be the $V$-filtration on $D_{W'}$ with respect to the coordinate $w_{n+1}$,
and denote as before by $V^\bullet_{ind} M_{A'}^\beta$ the induced
filtration on $M_{A'}^\beta$.
Then the main result of this subsection can be stated as follows.

\begin{proposition}\label{prop:CompFiltOnSWGarbe}
Let $A$ be a $d\times n$-integer matrix and suppose that $\mbn A = \mbz^d \cap \mbr_{\geq 0} A$, $\mbn A \neq \mbz^d$. Let
$A':=(\underline{a}_1,\ldots,\underline{a}_n,
\underline{a}_1+\ldots+\underline{a}_n)$ and consider the left $D_{W'}$-module
$$
\check{M}_{A'}^\beta = D_{W'} /\left( (\check{\Box}_{\underline{m}})_{\underline{m} \in \mbl_{A'}} + (\check{E}'_k+\beta_k)_{k=1, \ldots ,d}\right).
$$
Then the map
$$
V^k D_{W'} \cap F_p D_{W'} \lra V^k_{ind} \check{M}_{A'}^\beta \cap F_p^{ord} \check{M}^\beta_{A'}\, .
$$
is surjective.
\end{proposition}

The proof of this result will occupy this entire section. Before going into it, let us comment on how this result will enter in the calculation of the Hodge filtration on $\check{M}_{A'}^\beta$. As will be explained in more detail at the beginning of section  \ref{subsec:CalcHodgeFiltrationSWGarbe}, we consider the mixed Hodge module
$h_{A*} \pch$ with underlying $\cD_W$-module $h_{A+}\cO_T^\beta$.
If $A$ and $\beta$ satisfy the assumptions of Proposition \ref{prop:SWgeneral}, then this $\cD_W$-module is $\check{\cM}_A^\beta$. In order to compute the Hodge filtration on its module of global
sections $\check{M}_A^\beta$, we will first consider the module $\check{M}_{A'}^\beta=\Gamma(W', h_{A'+}^\beta \cO_T^\beta)$ and compute
the Hodge filtration on it. We will use the fact that the embedding $h_{A'}:T \hookrightarrow W'$ can be factored as
$$
\begin{tikzcd}
T \ar{r} & W\times \dC^*_t \ar{r}{j} & W'=W \times \dC_t,
\end{tikzcd}
$$
where the first morphism is a \emph{closed} embedding, and the second one is
the canonical open embedding of $W\times \dC^*_t$ into $W'$.
Then the main tool to compute the Hodge filtration on $\check{M}_{A'}^\beta$ is the following formula of Saito
(see formula \eqref{eq:defHodgefilt} below). Let $(\cM, F^H_\bullet)$ be any filtered $\cD_{W \times \dC^*_t}$-module underlying a complex mixed Hodge module in $\MHM(W\times \dC^*_t,\dC)$. Then the direct image $j_+ \cM$ underlies
a complex mixed Hodge module on $W'$, and its Hodge filtration is given as
$$
F^H_p j_+ \mcm = \sum_{i \geq 0}\p_t^i \left(V^0 j_+ \mcm \cap j_*( F^H_q \mcm)\right),
$$
where $V^\bullet j_+ \cM$ denotes the canonical $V$-filtration on
$j_+\cM$ with respect to the divisor $\{t=0\}$.
We are going to apply this formula for the case where $\cM$ is the direct image of $\cO_T^\beta$ under the map
$T\rightarrow W\times \dC^*_t$ (so that $j_+\cM=\check{\cM}_{A'}^\beta$). Since this map is a closed immersion, we can explicitly calculate
the Hodge filtration on this direct image, i.e., it is given as the shifted order filtration for
a cyclic presentation. Moreover, if $\beta$ satisfies the assumptions of Proposition
\ref{prop:Vfilt}, then we have
$V^k \check{M}_{A'}^\beta = V^k_{ind} \check{M}_{A'}^\beta$ for all $k\in \dZ$ (in particular, for $k=0$), so that we have to compute the intersection of the order filtration on $M_{A'}^\beta$ with the \emph{induced} $V$-filtration on that module.
As we will see below in section \ref{subsec:CalcHodgeFiltrationSWGarbe}, this is possible since these two filtrations
satisfy the compatibility statement of the above proposition.
Its proof relies
on the very specific structure of the hypergeometric ideal $\check{I}_{A'} = \left((\check{\Box}_{\underline{m}})_{\underline{m}\in\mbl_{A'}}+(\check{E}_k+\beta_k)_{k=1,\ldots,d}\right)\subset D_{W'}$ and uses
non-commutative Gr\"obner basis techniques, a good reference for results needed is \cite{SST}. We will recall the
main definitions for the reader's convenience.\\

To simplify the notation, we rename the coordinate $t$ on $W'$ to be $w_{n+1}$,
that is $\Gamma(W',\mco_{W'})=\mbc[w_1,\ldots,w_n,w_{n+1}]$. We work in the Weyl algebra
$D_{W'}=\mbc[w_1,\ldots,w_{n+1}]\langle \p_{w_1},\ldots,\p_{w_{n+1}}\rangle$.
Any operator $P\in D_{W'}$ has the so-called normally ordered expression $P=\sum_{(\gamma,\delta)} c_{\gamma\delta}  w^{\gamma} \p_{w}^\delta  \in D_{W'}$,
where the sum runs over all pairs $(\gamma,\delta)$ in some finite subset of $\dN^{2(n+1)}$.\\

First we define partial orders on the set of monomials in $D_{W'}$ resp. $\mbc[w]:=\mbc[{w_1},\ldots, {w_{n+1}}]$ resp. $\mbc[w,\xi]:= \mbc[w_1,\ldots,w_{n+1},\xi_1,\ldots,\xi_{n+1}]$ by choosing the weight vectors $(u,v)\in \mbz^{2(n+1)}$ with $u_i + v_i \geq 0$ resp. $u \in \mbz^{n+1}$. This means that the variables $w_i$ have weight $u_i$ and the partial differentials $\p_{w_i}$ resp. $\xi_i$ have weight $v_i$ . The associated partial order in $D_{W'}$ is defined as follows: If for two monomials
$ w^\gamma \p_{w}^\delta ,  w^c \p_{w}^d $ we have $\sum_i u_i c_i + v_i d_i < \sum_i u_i \gamma_i + v_i \delta_i $, then by definition $ w^\gamma \p_{w}^\delta$ is larger then $ w^c \p_{w}^d $, we write $  w^c \p_{w}^d  \prec_{(u,v)}  w^\gamma \p_{w}^\delta$ and similarly for $\mbc[w]$ and $\mbc[w,\xi]$. The weight vector $(u,v)$ induces an increasing resp. decreasing filtrations on $D_{W'}$  given by
\[
F^{(u,v)}_p D_{W'} =  \left\lbrace \sum_{\sum_i u_i \gamma_i + v_i \delta_i \leq p} c_{\gamma \delta} w^\gamma \p_w^\delta \right\rbrace \qquad \text{resp.} \qquad F_{(u,v)}^p D_{W'} =  \left\lbrace \sum_{\sum_i u_i \gamma_i + v_i \delta_i \geq p} c_{\gamma \delta} w^\gamma \p_w^\delta \right\rbrace
\]
We define the graded ring $\gr^{(u,v)}_\bullet D_{W'} := \bigoplus_p F_p^{(u,v)}D_{W'} / F_{p-1}^{(u,v)} D_{W'}$   associated with the weight $(u,v)$. Notice that for $(u,v) = (0,\ldots,0,1\ldots,1)$ (i.e. the $w_i$ have weight $0$ and the $\p_{w_i}$ have weight $1$) the ascending filtration $F^{(u,v)}_\bullet D_{W'}$ is the order filtration $F_\bullet D_{W}'$ and for $(u,v) = (0,\ldots,0, -1,0,\ldots,0,1)$ the descending filtration $F^\bullet_{(u,v)}D_{W'}$ is the $V$-filtration with respect to $w_{n+1}$.\\

We get well-defined maps

\begin{align*}
in_{(u,v)}: D_{W'} & \longrightarrow  \Gr_\bullet^{(u,v)} D_{W'} = \mbc[w,\xi]  \\
P= \sum_{\gamma,\delta} c_{\gamma\delta}  w^{\gamma} \p_{w}^\delta   &\longmapsto  in_{(u,v)}(P) :=\!\!\!\! \sum_{\sum_i u_i\gamma_i+v_i\delta_i = m}\!\!\!\!\!\!\!\!\!\! c_{\gamma\delta}  w^{\gamma} \xi^\delta  \\
\end{align*}
where $m:= ord_{(u,v)}(P) := \max\{\sum_i u_i \gamma_i + v_i \delta_i \mid c_{\gamma \delta} \neq 0 \}$ and \\
\begin{align*}
in_u : \mbc[w] & \lra \gr_\bullet^u \mbc[w] = \mbc[w] \\
  Q= \sum_{\gamma} c_{\gamma} w^\gamma  &\longmapsto  in_{u}(Q) :=\!\!\!\! \sum_{\sum_i u_i \gamma_i =m}\!\!\!\! c_{\delta} w^\delta\\
\end{align*}
where $m = \max\{\sum_u u_i \gamma_i \mid c_{\gamma} \neq 0 \}$ and \\
\begin{align*}
in_{(u,v)} : \mbc[w,\xi] & \lra \gr_\bullet^{(u,v)} \mbc[w,\xi] = \mbc[w,\xi] \\
  R= \sum_{\gamma,\delta} c_{\gamma} w^\gamma \xi^\delta  &\longmapsto  in_{(u,v)}(Q) :=\!\!\!\! \sum_{\sum_i u_i \gamma_i + v_i \delta_i =m}\!\!\!\! c_{\delta} w^\delta\\
\end{align*}
where $m:= ord_{(u,v)}(R) := \max\{\sum_i u_i \gamma_i + v_i \delta_i \mid c_{\gamma \delta} \neq 0 \}$. Notice that, in constrast to the case of a total ordering, the initial terms $in_{(u,v)}$ resp. $in_u$ are not monomials.\\

Let $I' \subset D_{W'}$ be a left ideal.  The set $in_{(u,v)}(I')$ is an ideal in $\gr_\bullet^{(u,v)}D_{W'}$  and is called initial ideal of $I'$ with respect to the weight vector $(u,v)$. A finite subset $G$ of $D_{W'}$ is a Gr\"obner basis of $I'$ with respect to $(u,v)$ if $I'$ is generated by $G$ and $in_{(u,v)}(I')$ is generated by $in_{(u,v)}(G)$. Similarly, let $J' \subset \mbc[w]$  resp. $K' \subset \mbc[w,\xi]$ be an ideal. The set $in_{u}(J')$  resp. $in_{(u,v)}(K')$ is an ideal in $\gr_\bullet^u \mbc[w]$ resp. $\gr_\bullet^{(u,v)}\mbc[w,\xi]$ and is called initial ideal of $J'$ resp. $K'$ with respect to the weight vector $u$ resp. $(u,v)$. The definition of a Gr\"obner basis is parallel to the definition above.\\

Let $\check{I}_{A'} =\left((\check{\Box}_{\underline{m}})_{\underline{m}\in\mbl_{A'}}+(\check{E'_k}+\beta_k)_{k=1,\ldots,d}\right)$ be the hypergeometric ideal.  The fake initial ideal $fin_{(u,v)}(\check{I}_{A'})$ is the following ideal in $\gr_{(u,v)}D_{W'}$:
\[
fin_{(u,v)}(\check{I}_{A'}) := \gr_{(u,v)}\mcd_{W'} \cdot in_u(\check{J}_{A'}) + \sum_{k=1}^d \gr_{(u,v)}\mcd_{W'} \cdot in_{(u,v)}( \check{E}_k + \beta_k)
\]
where $\check{J}_{A'} \subset \mbc[w]$ is the ideal generated by $(\check{\Box}_{\underline{m}})_{\underline{m}\in\mbl_{A'}}$.\\

Consider the Koszul complex
\[
\ldots \overset{d_{-2}}{\lra} K^{-1}(\gr_\bullet^{(u,v)}(D_{W'}/D_{W'}\check{J}_{A'})) \overset{d_{-1}}{\lra} K^0(\gr_\bullet^{(u,v)}(D_{W'}/D_{W'}\check{J}_{A'})) \lra 0
\]
where
\[
K^{-p}(\gr_\bullet^{(u,v)}(D_{W'}/D_{W'}\check{J}_{A'}) = \bigoplus_{1\leq i_1 < \ldots < i_p \leq s+1} \gr_\bullet^{(u,v)}(D_{W'}/D_{W'}\check{J}_{A'})e_{i_1\ldots i_p}
\]
and
\[
d_{-p}(e_{i_1 \ldots i_p}) = \sum_{k=1}^{p}(-1)^{k-1}in_{(u,v)}(\check{E}_k + \beta_k)e_{i_1\ldots \widehat{i}_k \ldots i_p}
\]

The following statement is an easy adaption of \cite[Theorem 4.3.5]{SST}
\begin{proposition}\label{prop:infinH1}
If the cohomologgy $H^{-1}(K^ \bullet (\gr_\bullet^{(u,v)}(D_{W'}/D_{W'}\check{J}_{A'})))$ vanishes, then the initial ideal satisfies $in_{(u,v)}(\check{I}_{A'}) =  fin_{(u,v)}(\check{I}_{A'})$.
\end{proposition}
\begin{proof}
After a Fourier-Laplace transform $w_i \ra \p_{x_i}$ and $ \p_{w_i} \ra -x_i$ the proof carries over word for word from loc. cit. (Notice that in Chapter 4 of loc. cit. the homogenity of $A'$ assumed, however the proof of this statement does not need this requirement).
\end{proof}

Recall that $A'$ is a matrix built from the matrix $A$ by adding a column which is the sum over all columns of $A$.  Let $\check{J}_A \subset \mbc[w_1,\ldots,w_n]$ be the ideal generated by $(\check{\Box}_l)_{l \in \mbl_A}$. We choose generators $g_1,\ldots, g_{\ell-1}$ of $\check{J}_{A}$. Notice that $g_1,\ldots, g_{\ell-1}, g_\ell := w_{n+1}- w_1\cdot \ldots \cdot w_n$ is a basis of $\check{J}_{A'} \subset \mbc[w_1,\ldots, w_{n+1}]$.
\begin{lemma}\label{lem:defg}
The elements $g_1,\ldots, g_\ell$ form a Gr\"obner basis of $\check{J}_{A'}$ with respect to the weight vector $(0,\ldots,0,-e)$ with $e > 0$.
\end{lemma}
\begin{proof}
We have already seen that $g_1,\ldots, g_\ell$ is a basis of $\check{J}_{A'}$. It remains to prove that $in_{(0,\ldots,0,-e)}(g_1)= g_1,\ldots,in_{(0,\ldots,0,-e)}(g_{\ell-1})=g_{\ell-1},in_{(0,\ldots,0,-e)}(g_{\ell})=w_1\cdot \ldots \cdot w_n$ is a basis of $in_{(0,\ldots,0,-e)}(\check{J}_{A'})$. Let
\begin{equation}\label{eq:smallGroeb}
x= \sum_{i=1}^{\ell} x_i g_i
\end{equation}
 and $-e\cdot N:= \max \{ \ord_{(0,\ldots,0,-e)}(x_i g_i) \mid i= 1,\ldots, \ell\} $. Assume that $\ord_{(0,\ldots,0,-e)}(x) < -e \cdot N$, then the maximal $w_{n+1}$-degree component of the equation \eqref{eq:smallGroeb} is given by
\[
0 = \sum_{i=1}^{\ell-1} w_{n+1}^N p_i g_i + w^N_{n+1}p_\ell \cdot (w_1\cdot \ldots \cdot w_n)
\]

for polynomials $p_i \in \mbc[w_1,\ldots,w_{n}]$. Since $\check{J}_A = (g_1,\ldots,g_{\ell-1})$ is a prime ideal and $w_1\cdot \ldots \cdot w_n \not \in \check{J}_A$ we conclude that $p_\ell \in \check{J}_{A}$. Hence there exist polynomial $q_i \in \mbc[w_1,\ldots,w_n]$ such that $ p_\ell = \sum_{i=1}^{\ell-1} q_i g_i$. We get
\[
x = \sum_{i=1}^\ell x_i g_i - \sum_{i=1}^{\ell-1} w_{n+1}^N p_i g_i - w^N_{n+1}p_\ell \cdot g_\ell + w^{N+1}_{n+1} \left( \sum_{i=1}^{\ell-1} q_i g_i \right) = \sum_{i=1}^{\ell} x'_i g_i
\]
for $x'_i \in \mbc[w_1,\ldots,w_{n+1}]$ with $\max\{ ord_{(0,\ldots,0,-e)}(x'_i g_i) \mid i =1,\ldots ,\ell \}< -e \cdot N $. By induction we can reduce to the case   $\ord_{(0,\ldots,0,-e)}(x) = -e \cdot N$. In this case we get for the maximal  $w_{n+1}$-degree component
\[
in_{(0,\ldots,0,-e)}(x) = \sum w^N_{n+1} p'_i g_i + w^N_{n+1} p'_\ell \cdot (w_1 \cdot \ldots \cdot w_n) =  \sum w^N_{n+1} p'_i in_{(0,\ldots,0,-e)}(g_i) + w^N_{n+1} p'_\ell in_{(0,\ldots,0,-e)}(g_\ell)
\]
for polynomials $p_i' \in \mbc[w_1,\ldots,w_n]$. This shows the claim.

\end{proof}

\begin{proposition}\label{prop:fineqin}
Let $A$ be a $d \times n$ integer matrix such that $\mbn A = \mbr_{\geq 0}A \cap \mbz^{d}$ and $\mbn A \neq \mbz^d$. Let $A'$ be the matrix built from $A$ by adding a column which is the sum over all columns of $A$.  Then
\[
fin_{(u,v)}(\check{I}_{A'}) = in_{(u,v)}(\check{I}_{A'})
\]
if
\begin{enumerate}
\item $(u,v) = (0,0,\ldots, 0,1,1,\ldots ,1)$
\item $(u,v) = (0,\ldots,0,-e,1,\ldots,1,1+e)$ for $ 0 < e < 1$.
\end{enumerate}
\end{proposition}
\begin{proof}

The first case was proven in \cite[Corollary 4.36]{SST} for homogeneous $A$.
In order to prove the statement for $(u,v) = (0,0,\ldots,0,1,1,\ldots,1)$ in the general case we first observe that $\gr_{(v,u)}(D_{W'} / \check{J}_{A'})$  is isomorphic to
\[
\mbc[\xi_1,\ldots, \xi_{n+1}]\otimes_\mbc \mbc[\mbn A']
\]
which is Cohen-Macaulay by the assumption $\mbn A = \mbr_{\geq 0}A \cap \mbz^{d}$ and the fact that $\mbn A = \mbn A'$ as well as $\mbr_{\geq 0} A = \mbr_{\geq 0} A'$. It follows from \cite[Theorem 1.2]{Berk} that the elements $in_{(u,v)}(\check{E}_k+ \beta_k)$ are part of a system of parameters in $\mbc[\xi_1,\ldots, \xi_{n+1}]\otimes_\mbc \mbc[\mbn A']$ and since this ring is Cohen-Macaulay they also form a regular sequence. Therefore $H^{-1}(K^\bullet(\gr_{(u,v)}(D_{W'}/D_{W'}\check{J}_{A'}))) = 0$ and the claim follows from Proposition \ref{prop:infinH1}.

We prove the second claim. Since $\mbn A \neq \mbz^d$ holds the last column of $A'$, which is the sum of the columns of $A$, is non-zero (this was shown above Lemma \ref{lem:compV1}). Hence we can assume (by elementary row manipulations of $A'$, which do not change the ideal $\check{I}_{A'}$) that the last column of $A'$ is zero except for the entry in the first row. Set
\[
\check{e}_k := \sum_{i=1}^n a_{ki} w_i \xi_i \qquad \text{for} \quad k=1,\ldots ,d.
\]
We will use the generators $g_1,\ldots, g_{\ell-1}$ of $\check{J}_A$ from Lemma \ref{lem:defg}. It follows from \cite[Theorem 1.2]{Berk} that $\check{e}_1,\ldots , \check{e}_d$ is part of a system of parameters for
\begin{align*}
\mbc[\xi_1,\ldots,\xi_n] \otimes_\mbc \mbc[\mbn A] &\simeq \mbc[\xi_1,\ldots,\xi_n,w_1,\ldots,w_n]/ \mbc[\xi_1,\ldots,\xi_n,w_1,\ldots,w_n] \check{J}_A \\
&\simeq \mbc[\xi_1,\ldots,\xi_n,w_1,\ldots,w_n]/(g_1,\ldots, g_{\ell-1})
\end{align*}
where $\mbc[\mbn A]$ has Krull dimension $d$.
Therefore
\[
\mbc[\xi_1,\ldots,\xi_n,w_1,\ldots,w_n]/(g_1,\ldots, g_{\ell-1},\check{e}_1,\ldots, \check{e}_r)
\]
has Krull dimension $n$.\\

We will show that the Krull dimension of
\begin{equation}\label{eq:diffKrull}
\mbc[\xi_1,\ldots,\xi_n,w_1,\ldots,w_n]/(g_1,\ldots, g_{\ell-1},w_1\cdot\ldots\cdot w_n, \check{e}_2,\ldots,\check{e}_r)
\end{equation}
is also $n$ (notice that we omitted $\check{e}_1$). The variety corresponding  to $\mbc[\xi_1,\ldots, \xi_n] \otimes_\mbc \mbc[\mbn A]$ is
\[
\mbc^n \times X_A \subset \mbc^n \times \mbc^n
\]
where $X_A := \textup{Spec}\, \mbc[\mbn A]$.  The toric variety $X_A$ is a finite disjoint union of torus orbits where the big dense torus lies in $\{ w_1 \cdot \ldots \cdot w_n \neq 0\}$ and the smaller dimensional tori lie in $\{w_1 \cdot \ldots \cdot w_n = 0\}$. Hence
\begin{equation}\label{eq:Compring1}
\mbc[\xi_1,\ldots, \xi_n,w_1,\ldots,w_n]/(g_1,\ldots,g_{\ell-1},w_1\cdot \ldots \cdot w_n)
\end{equation}
has Krull dimension $n+d-1$. The torus orbits of $X_A$ correspond to the faces of the cone $\mbr_{\geq 0} A$ where the big dense torus corresponds to $\mbr_{\geq 0}A$ itself. For a face $\tau \subsetneq \mbr_{\geq 0}A$ the torus orbit $Orb(\tau)$ is given by $Orb(\tau) = X_A \cap (\mbc^*_w)^\tau$ where $(\mbc^*)_w^\tau = \{w \in \mbc^n \mid w_i = 0 \; \text{for} \; a_i \not \in \tau\, , w_j \neq 0 \; \text{for} \; a_j \in \tau\}$. Hence it suffices to prove that $(\mbc^n \times Orb(\tau)) \cap V((\check{e}_2,\ldots,\check{e}_d)$ has dimension $n$, where $V((\check{e}_2,\ldots,\check{e}_d)$ is the vanishing locus of the ideal generated by $\check{e}_2,\ldots,\check{e}_d$. Set $\mbc^\tau_\xi = \{\xi \in \mbc^n \mid \xi_i = 0 \; \text{for}\; a_i \not \in \tau\}$. It is enough to show that $\mbc^\tau_\xi \times Orb(\tau) \cap V((\check{e}_2^\tau,\ldots, \check{e}_d^\tau))$ has dimension at most $\sharp\{i \mid a_i \in \tau \}$, where $\check{e}^\tau_k := \sum_{i : a_i \in \tau}a_{ki}w_i \xi_i$. The codimension of $V((\check{e}_2^\tau,\ldots, \check{e}_d^\tau))$ is $\dim(\tau)$ since $(1,0,\ldots,0) = \frac{1}{c}(a_1+ \ldots + a_n) $ ( for a suitable $c \in \mbz\setminus \{0\}$) lies in the interior of $\mbr_{\geq 0}A$, hence not in $\tau$ and therefore the matrix $(a_{ki})_{k \geq 2, i : a_i \in \tau}$ has rank $\dim(\tau)$. By \cite[Lemma 1.1]{Berk} the intersection of $\mbc^\tau_\xi \times Orb(\tau)$ with $V((\check{e}^\tau_2,\ldots,\check{e}^\tau_d))$ is transverse.
Since the codimension of $V((\check{e}^\tau_2,\ldots,\check{e}^\tau_d))$ is $\dim Orb(\tau) = \dim(\tau)$ the intersection has dimension $\sharp\{i \mid a_i \in \tau \}$. This shows that the Krull dimension of \eqref{eq:diffKrull} is $n$.\\

Let
\begin{align*}
\check{e}'_1 &:=  \check{e}_1 + (\sum_{i=1}^n a_{1i})x_{n+1} \xi_{n+1} = in_{(0,\ldots,0,-e,1,\ldots,1,1+e)}(\check{E}'_1 +\beta_1) \\
\check{e}'_k &:= \check{e}_k = in_{(0,\ldots,0,-e,1,\ldots,1,1+e)}(\check{E}'_k + \beta_k) \qquad \text{for} \quad k= 2, \ldots ,d.
\end{align*}
and
\begin{align*}
\overline{g}_i &:= in_{(0,\ldots,0,-e,1,\ldots,1,1+e)}(g_i) = g_i \qquad \qquad \quad \text{for} \quad i=1,\ldots , \ell-1 \\
\overline{g}_\ell &:= in_{(0,\ldots,0,-e,1,\ldots,1,1+e)}(g_\ell) = w_1\cdot \ldots \cdot w_n
\end{align*}
\end{proof}
Since $\overline{g}_1,\ldots, \overline{g}_\ell$ and $\check{e}'_2,\ldots, \check{e}'_d$ are independent of $w_{n+1},\xi_{n+1}$ and $\check{e}'_1 = \check{e}_1 + (\sum_{i=1}^n a_{1i})x_{n+1} \xi_{n+1}$ is (for degree reasons) a non-zerodivisor on
\[
\mbc[w_{n+1},\xi_{n+1}] \otimes_\mbc \mbc[\xi_1,\ldots,\xi_n,w_1,\ldots,w_n]/(\overline{g}_1,\ldots, \overline{g}_{\ell}, \check{e}_2,\ldots,\check{e}_d)\, ,
\]
one easily sees that
\[
\mbc[w_1,\ldots, w_{n+1},\xi_1,\ldots, \xi_{n+1}] / (\overline{g}_1,\ldots, \overline{g}_{\ell},\check{e}'_1,\ldots, \check{e}'_d)
\]
has Krull dimension $n+1$. It follows from \eqref{eq:Compring1} that $\mbc[w_1,\ldots, w_{n+1}, \xi_1,\ldots, \xi_{n+1}] /(\overline{g}_1,\ldots, \overline{g}_\ell)$ has Krull dimension $(n+d+1)$, hence $\check{e}'_1, \ldots, \check{e}'_d$ is part of a system of parameters.
By the assumption on $A$ the ring
\[
\mbc[w_1,\ldots,w_{n+1},\xi_1,\ldots,\xi_{n+1}] / (\overline{g}_1,\ldots, \overline{g}_{\ell-1}) \simeq \mbc[w_{n+1}, \xi_1,\ldots, \xi_{n+1}] \otimes_\mbc \mbc[\mbn A]
\]
is Cohen-Macaulay. Since $\overline{g}_\ell = w_1 \cdot \ldots \cdot w_n$ is not a zero-divisor in the ring above (because $\mbc[\mbn A]$ has no non-zero zero-divisors), we see that the ring
\[
\mbc[w_1,\ldots, w_{n+1},\xi_1,\ldots, \xi_{n+1}] / (\overline{g}_1,\ldots, \overline{g}_{\ell})
\]
is also Cohen-Macaulay and therefore $\check{e}'_1,\ldots , \check{e}'_d$ is a regular sequence in $\mbc[w_1,\ldots, w_{n+1},\xi_1,\ldots, \xi_{n+1}] / (\overline{g}_1,\ldots, \overline{g}_{\ell})$.\\

Since
\[
\gr_{(0,\ldots,0,-e,1,\ldots,1,1+e)} (D_{W'}/ D_{W'} \check{J}_{A'} ) \simeq \mbc[w_1,\ldots, w_{n+1},\xi_1,\ldots, \xi_{n+1}] / (\overline{g}_1,\ldots, \overline{g}_{\ell})
\]
and $\check{e}'_k = in_{(0,\ldots,0,-e,1,\ldots,1,1+e)}(\check{E}'_k+\beta_k)$ for $k=1,\ldots ,d$, we have
\[
H^{-1}(K^\bullet(\gr_{(0,\ldots,0,-e,1,\ldots,1,1+e)}(D_{W'}/D_{W'}\check{J}_{A'}))) = 0.
\]
Using again Proposition \ref{prop:infinH1}, this shows the second claim.
\begin{corollary}\label{cor:constrGroebbasis}
Let $g_1,\ldots, g_\ell \in \mbc[w_1,\ldots,w_{n+1}]$ be the generators  of $\check{J}_{A'}$ defined above Lemma \ref{lem:defg}.
\begin{enumerate}
\item The $(g_i)_{i=1,\ldots, \ell}$ together with  $(\check{E}'_k+ \beta_k)_{k=1,\ldots,d}$ form a Gr\"obner basis  of $\check{I}_{A'}$ with respect to the weight vector $(u,v) =(0,\ldots,0,1\ldots,1)$.
\item  Let $(u,v) = (0,\ldots,0,1,\ldots,1)$ and set $\tilde{g}_i := in_{(u,v)}(g_i)$ and $\tilde{E}'_k = in_{(u,v)}(\check{E}'_k+ \beta_k)$. The elements $(\tilde{g}_i)_{i=1,\ldots,\ell}$ and $(\tilde{E}'_k)_{k=1,\ldots,d}$ form a Gr\"obner-basis  of
\begin{align*}
in_{(u,v)}(\check{I}_{A'})  &= in_{(u,v)}((g_1,\ldots, g_\ell, \check{E}'_1+\beta_1,\ldots, \check{E}'_d +\beta_d))\\
&= (\tilde{g}_1,\ldots, \tilde{g}_\ell, \tilde{E}'_1,\ldots,\tilde{E}'_d) \subset \mbc[w_1,\ldots, w_{n+1},\xi_1,\ldots, \xi_{n+1}]
\end{align*}
with respect to the weight vector $(0, \ldots ,0 ,-1,0,\ldots, 0, 1)$ (i.e. $w_{n+1}$ has weight $-1$ and $\xi_{n+1}$ has weight $+1$).
\end{enumerate}
\end{corollary}
\begin{proof}
1.) The set $(g_i)_{i = 1,\ldots, \ell}$ is a Gr\"obner basis for $\check{J}_{A'}$. Therefore the elements $in_u(g_i)=g_i$ generate $in_u(\check{J}_{A'})=\check{J}_{A'}$. The elements $(g_i)_{i =1,\ldots,\ell}$ and $(\check{E}_k+\beta_k)_{k=1,\ldots , d}$ generate $\check{I}_{A'}$ and the elements $(in_{u,v}(g_i))_{i=1,\ldots,\ell}$ and $(in_{(u,v)}(\check{E}_k+\beta_k))_{k=1,\ldots, d}$ generate $fin_{(u,v)}$ by definition. The claim follows now from Proposition \ref{prop:fineqin} 1. .\\

2.) It follows from the first point that the $\tilde{g}_i = in_{(u,v)}(g_i)$ and the $\tilde{E}'_k = in_{(u,v)}(\check{E}'_k+\beta_k)$ generate $in_{(u,v)}(\check{I}_{A'})$. We have to show that the $in_{(0,\ldots,0,-1,0,\ldots,0,1)}(\tilde{g}_i)$ for $i=1,\ldots,d$ and the $in_{(0,\ldots,0,-1,0,\ldots,0,1)}(\tilde{E}'_k)$ generate $in_{(0,\ldots,0,-1,0,\ldots,0,1)}(in_{(u,v)}(\check{I}_{A'}))$. But this follows from (cf. \cite[Lemma 2.1.6 (2)]{SST})
\begin{align*}
in_{(0,\ldots,0,-1,0,\ldots,0,1)}(\tilde{g}_i) &= in_{(0,\ldots,0,-e,1,\ldots,1,1+e)}(g_i)  & &\text{for} \; i=1,\ldots, \ell \\
in_{(0,\ldots,0,-1,0,\ldots,0,1)}(\tilde{E}'_k) &= in_{(0,\ldots,0,-e,1,\ldots,1,1+e)}(\check{E}'_k +\beta_k) & &\text{for} \; k=1,\ldots, d \\
in_{(0,\ldots,0,-1,0,\ldots,0,1)}(in_{(u,v)}(\check{I}_{A'})) &= in_{(0,\ldots,0,-e,1,\ldots,1,1+e)}(\check{I}_{A'})
\end{align*}
for $0 < e \ll 1$ and Proposition \ref{prop:fineqin} 2.  .
\end{proof}

The second notion we are going to introduce relates the order filtration $F_\bullet$ on $D_{W'}$ with the $V$-filtration
that already occurred in the last subsection. Here we consider the
descending $V$-filtration on $D_{W'}$ with respect to $w_{n+1} = 0$, which we denote again by by $V^\bullet D_{W'}$. We have
\[
V^0D_{W'} = \sum_{i,k \geq 0} (\p_{w_{n+1}} w_{n+1})^i (w_{n+1})^k P_i
\]
for $P_i \in \mbc[ w_1, \ldots , w_{n}] \langle \p_{w_1}, \ldots , \p_{w_{n}} \rangle$ and
\begin{equation}\label{eq:Vfiltcomp}
V^k D_{W'} = w_{n+1}^k V^0 D_{W'} \quad \text{and} \quad V^{-k} D_{W'} = \sum_{j \geq 0} \p_{w_{n+1}}^j V^0 D_{W'}
\end{equation}
for $k >0$.\\

Recall the left ideal
$\check{I}_{A'} \subset D_{W'}$  and the left $D_{W'}$-modules $\check{M}^\beta_{A'} := D_{W'} / \check{I}_{A'}$ from above.
We define filtrations $V^\bullet_{ind}$ and $F^{ord}_{\bullet}$ on $\check{M}^\beta_{A'}$ by:
\[
V^k_{ind} \check{M}_{A'}^\beta := \frac{V^k D_{W'} + \check{I}_{A'}}{\check{I}_{A'}} \quad \text{and} \quad F_p^{ord} \check{M}_{A'}^\beta := \frac{F_p D_{W'} +\check{I}_{A'}}{\check{I}_{A'}}\, .
\]

We are now ready to prove the main result of this subsection.

\begin{proof}[Proof of Proposition \ref{prop:CompFiltOnSWGarbe}]
 Let $m \in  V^k_{ind} \check{M}^\beta_{A'} \cap F_p^{ord} \check{M}^\beta_{A'}$. We can find $P,Q \in D_{W'}$ such that $P \in F_p D_{W'}$, $Q \in V^k D_{W'}$ and $[P] = m = [Q]$, i.e. $P = Q - i$ for some $i \in \check{I}_{A'}$. We have to find a $Q'$ with $Q' \in V^k D_{W'} \cap F_p D_{W'}$ with $P= Q'-i'$ for $i' \in I$. We will construct this element $Q'$ by decreasing induction on the order of $Q$ by killing its leading term in each step. For this we will use the special Gr\"obner basis of $\check{I}_{A'}$ which we constructed in Corollary \ref{cor:constrGroebbasis} above.\\

Recall that the  weight vector $(u,v):= (0,\ldots,0,1,\ldots,1)$ induces the order filtration $F_\bullet^{(u,v)} = F_\bullet^{ord}$ on $D_{W'}$. If $ R \in D_{W'}$ and $k:=ord_{(u,v)}(R)$ we define the symbol of $R$ by $\sigma_k(R) = in_{(u,v)}(R)$ and set $\sigma_q(R) = 0$ for $q \neq k$. We define a second weight vector $(u',v'):= (0,\ldots,0,-1,0,\ldots,0,1)$ which induces the descending $V$-filtration from \eqref{eq:Vfiltcomp} on $D_{W'}$. The $V$-filtration and $F$-filtration also induce filtrations $\tilde{V}$  and $\tilde{F}$ on $\gr_\bullet^{(u,v)} D_{W'} = \gr_\bullet^{F}D_{W'} =\mbc[w_1,\ldots, w_{n+1},\xi_1,\ldots,\xi_{n+1}]$.\\

Let $t_Q := \ord_{(u,v)} Q$, $t_i := \ord_{(u,v)} i$ and set $t := \max(t_Q,t_i)$. Obviously we have $t \geq p$. If $t = p$ we are done. Hence, we assume $t >p$, thus we have
\[
0 =  \sigma_t(P) = \sigma_t(Q -i)
\]
and therefore $t= t_Q = t_i$ which implies $\sigma_t (Q) = \sigma_t(i) \neq 0$. Set $k_Q:= ord_{(u',v')}(\sigma_t(Q))$, then we have $\sigma_t(i) =\sigma_t(Q) \in \tilde{V}^{k_Q}$.\\

Recall from Corollary \ref{cor:constrGroebbasis} that $\check{I}_{A'}$ is generated by $\{G_1,\ldots,G_m\} := \{g_1,\ldots, g_\ell,\check{E}'_1+\beta_1,\ldots, \check{E}'_d+\beta_d\}$ and these elements form a Gr\"obner basis with respect to weight vector $(u,v)$  and  their initial forms
\[
\{\tilde{G}_1,\ldots, \tilde{G}_m\} := \{in_{(u,v)}(G_1),\ldots, in_{(u,v)}(G_m)\}
\]
are a Gr\"obner basis of $in_{(u,v)}(\check{I}_{A'})$ with respect to the weight vector $(u',v')$. Therefore we can write
\[
\sigma_t(i) = \sum_{l=1}^m \tilde{i}_l \tilde{G}_l
\]
with $\tilde{i}_l \in \mbc[w_1,\ldots,w_{n+1},\xi_1,\ldots, \xi_{n+1}]$. Using a commutative version of \cite[Theorem 1.2.10]{SST} we can assume that $\tilde{i}_l \in \tilde{V}^{k_Q - k_{l}}$ where $k_{l} = ord_{(u',v')}(\tilde{G}_l)$. Since the elements $\tilde{G}_l$ are homogeneous with respect to the variables $\xi_1,\ldots, \xi_{n+1}$ we can also assume that $\tilde{i}_l \in \tilde{F}_{t-t_l} \gr^F_\bullet D_{W'}$ where $t_l = ord_{(u,v)}(\tilde{G}_l)$.    Let $i_l \in D_{W'}$ be the normally ordered element  which we obtain from $\tilde{i}_l$ by replacing $\xi_i$ with $\p_{w_i}$. One sees easily that $i_l G_l \in F_t D_{W'} \cap V^{k_Q} D_{W'}$. Therefore the element $i' := \sum_{l =1}^m i_l G_l$ has the following two properties
\[
\sigma_t(i') = \sigma_t(i) =\sigma_t(Q) \qquad \text{and} \qquad i' \in V^{k_{Q}}D_{W'}
\]

where the second property follows from $ord_{(u',v')}(G_l) = ord_{(u',v')}(\tilde{G}_l)$.
We therefore have
 \[
 P = Q - i' - (i - i')
 \]
 with $Q - i' \in F_{t-1} D_{W'} \cap V^{k_Q} D_{W'}$.
 Since obviously we have $k \leq k_Q$, we conclude that $Q - i' \in F_{t-1} D_{W'} \cap V^k D_{W'}$. The claim now follows by descending induction on the order $t$.
\end{proof}

\subsection{Calculation of the Hodge filtration}
\label{subsec:CalcHodgeFiltrationSWGarbe}

In this subsection we want to compute the Hodge filtration on the mixed Hodge module
\[
\mch^0( h_{A*} \pch)\, ,
\]
recall from section \ref{sec:Introduction} that $\pch=\ch[d] \in \MHM(T)$.
Also recall that $\pch$ has the underlying filtered $\mcd$-module $(\mco^\beta_T,F^H_\bullet \mco^\beta_T)$, where the Hodge filtration is given by
\[
F_p^H \mco^\beta_T = \begin{cases}\mco^\beta_T & \text{for}\; p \geq 0 \\ 0 & \text{else}\, . \end{cases}
\]

We will use several different presentations of $\mco^\beta_T$ as a $\mcd_T$-module, namely, for each $\alpha =(\alpha_k)_{k=1,\ldots ,d} \in \mbz^d$ we have a $D_T$-linear isomorphism
\[
\Gamma(T,\mco^\beta_T) \simeq D_T/(\p_{t_k} t_k  + \beta_k +  \alpha_k)_{k=1,\ldots,d}
\]
such that the Hodge filtration is simply the order filtration on the right hand side.\\

As we have seen in Lemma \ref{lem:SWclosedembed}, the morphism
$h_A$ can be decomposed into the closed embedding $k_A:T \rightarrow W^*=W\backslash D$ and the canonical
open embedding $l_A:W^*\rightarrow W$. We have to determine the Hodge filtration on the direct image modules
for both mappings. The former is (after some coordinate change) a rather direct calculation, and will be carried
out in Lemma \ref{lem:directMapk} below. However, understanding the behaviour of the Hodge filtration under the direct
image of an open embedding of the complement of a divisor (like the map $l_A$) is more subtle and at the heart of the theory of mixed Hodge modules (see, e.g., \cite[Section (2.b)]{SaitoMHM}).
More precisely, since the steps of the Hodge filtration of a mixed Hodge module are coherent modules
over the structure sheaf of the underlying variety, the usual direct image functors are not suitable
for the case of an open embedding as they do not preserve coherence. In order to circumvent this
difficulty, one uses the canonical $V$-filtration along the boundary divisor, as computed in subsection \ref{subsec:Vfilt} above. Let us give an overview
of the strategy to be used below. The actual calculation will be finished only in Theorem \ref{thm:hpot},
the main step being Proposition \ref{Prop:ExtensionGraphSmoothDivisor}.

We will need the following formula (copied from \cite[Proposition 4.2.]{Saitobfunc} ) which describes the extension of a mixed Hodge module over a smooth hypersurface. Let $X$ be a smooth variety, let $t, x_1, \ldots x_n$ be local coordinates on $X$ and $j:Y \hookrightarrow X$ be a smooth hypersurface given by $t= 0$. Let ${^H\!\!}\mcm$ be a mixed Hodge module on $X \setminus Y$ with underlying filtered $\mcd$-module $(\mcm,F^H_\bullet \mcm)$, then
\begin{equation}\label{eq:defHodgefilt}
F^H_p \mch^0 j_+ \mcm = \sum_{i \geq 0}\p_{t}^iF^H_{p-i}V^0 \mch^0 j_+ \mcm, \quad \text{where} \quad F^H_q V^0 \mch^0 j_+ \mcm := V^0 \mch^0 j_+ \mcm \cap j_*( F^H_q \mcm)\, \, ,
\end{equation}
here $V^0 \mch^0 j_+ \mcm$ is the canonical $V$-filtration on the $\mcd$-module $\mch^0 j_+ \mcm \simeq j_* \mcm$, as introduced in Definition \ref{def:CanonicalVFiltration}.\\

If $Y$ is a non-smooth hypersurface locally given by $g = 0$, we consider (locally) the graph embedding
\begin{align}
i_g: X &\lra X \times \mbc_t \notag \\
x &\mapsto (x, g(x)) \notag
\end{align}
together with its restriction $i_g^\circ: X\setminus Y \ra X \times \mbc^*_t$. Notice that $i^\circ_g$ is a closed embedding. Given a mixed Hodge module $\mcm$ on $X \setminus Y$ we proceed as follows. We first extend the Hodge filtration of $(i^\circ_g)_+ \mcm$ over the smooth divisor given by $\{t = 0\}$ as explained above. Afterward, we restrict the mixed Hodge module which we obtained to the smooth divisor given by $\{t = g\}$.\\

After these general remarks, we come back to the situation of the torus embedding $h_A: T \rightarrow W$ described at the beginning of
this section. Consider the following commutative diagram
\begin{equation}\label{eq:DiagGraphEmbeddingNCD}
\begin{tikzcd}
T \ar{r}{k_A} \ar[bend left=35]{rr}{h_A} & W^* \ar{r}{l_A} \ar{drr}[swap]{i^\circ_g} & W \ar{r}{i_g} & W \times \mbc_t \\ & & & W \times \mbc^*_t \ar{u}{j_t}
\end{tikzcd}
\end{equation}
where $W^*:= W \setminus D = W \setminus \{w_1 \ldots w_n =0\} \simeq  (\mbc^*)^n$ and where  $i_g$ is the graph embedding
\begin{align}
i_g: W &\lra W \times \mbc_t \label{eq:graphem} \\
w &\mapsto (w, w_1 \cdot \ldots \cdot w_n) \notag
\end{align}
associated with the function $g:W\rightarrow \dC_t,w\mapsto w_1\cdot\ldots\cdot w_n$.
Notice that $i_g \circ l_A$ factors over $W \times \mbc^*_t$. We have the following isomorphisms
\[
i_{g + } h_{A+} \mco_T^\beta \simeq i_{g + } l_{A+} k_{A+} \mco_T^\beta \simeq j_{t+} i^\circ_{g +} k_{A+} \mco_T^\beta\, .
\]
\begin{lemma}\label{lem:directMapk}
The direct image $\mch^0 k_{A+} \mco_T^\beta$ is isomorphic to the cyclic $\mcd_{W^*}$-module
$$
{^*\!}\check{\mcm}^\beta_{A'}:=\mcd_{W^*} / \check{\mci}^*_\beta
$$ where $\check{\mci}^*_\beta$ is the left ideal generated by $(\check{E}_k+ \beta_k)_{k=1,\ldots,
d}$ for $\beta = (\beta_k)_{k=1, \ldots , d} \in \mbr^{d}$ and $(\check{\Box}_{\underline{m}})_{\underline{m} \in \mbl_A}$.
Furthermore, the Hodge filtration on $\check{\mcm}_{A'}$ is equal to the induced order filtration, shifted by $n-d$, i.e.
\[
F^H_{p}\,{^*\!}\check{\mcm}_{A'}\,  = F^{ord}_{p-(n-d)}\, \mcd_{W^*}/\check{\mci}^*_\beta\, .
\]
\end{lemma}
\begin{proof}
We factor the map $k_A$ from above in the following way. Let $A = C \cdot E \cdot F$ be the Smith normal form of $A$, i.e. $C =(c_{pq}) \in GL(d,\mbz)$, $F=(f_{uv}) \in GL(n,\mbz)$ and $E = ( I_d, 0_{d,n-d})$. This gives rise to the maps
\begin{align}
k_C : T &\lra T \notag \\
(t_1, \ldots, t_{d}) &\mapsto (\tilde{t}_1, \ldots , \tilde{t}_{d}) =(\underline{t}^{\underline{c}_1}, \ldots , \underline{t}^{\underline{c}_{d}}) \notag \\
k_E : T &\lra (\mbc^*)^{n} \notag \\
(\tilde{t}_1,\ldots , \tilde{t}_{d}) &\mapsto (\tilde{w}_1, \ldots, \tilde{w}_{n})= (\tilde{t}_1,\ldots, \tilde{t}_{d},1,\ldots ,1) \notag \\
k_F : (\mbc^*)^{n} &\lra W^* \notag \\
(\tilde{w}_1, \ldots , \tilde{w}_{n}) &\mapsto (w_1, \ldots , w_{n})= (\underline{\tilde{w}}^{\underline{f}_1}, \ldots , \underline{\tilde{w}}^{\underline{f}_{n}})\, . \notag
\end{align}

For $\gamma \in \mbz^d$ we have
\[
k_{A+} \mco_T^\gamma \simeq ( k_F \circ k_E \circ k_C)_+ \mco_T^\gamma \simeq  k_{F +} k_{E +} k_{C +} \mco_T^\gamma\, .
\]
Since all maps and spaces involved are affine, we will work at the level of global sections.
We have $\Gamma(T, \mco_T^{\gamma}) = D_T/(\p_{t_k}t_k+\gamma_k )_{k=1,\ldots,d}=D_T/(t_k \p_{t_k} +\gamma_k+ 1 )_{k=1,\ldots,d}$. Notice that the Hodge filtration in this presentation is simply the order filtration.
Since $k_C$ is a change of coordinates we have $\Gamma(T, \mch^0 (k_C)_+ \mco_T^\gamma) \simeq D_T/ (\sum_{i=1}^d c_{ki} \tilde{t}_i \p_{\tilde{t}_i} + \gamma_k + 1)_{k=1, \ldots ,d}$ and again the Hodge filtration is equal to the order filtration.
We now calculate $\Gamma((\mbc^*)^{n}, k_{E+}k_{C+}\mco_T^\gamma)$. We have
\begin{align}
&\Gamma((\mbc^*)^{n}, \mch^0 k_{E+}k_{C+}\mco_T^\gamma) \simeq \Gamma(T, \mch^0 k_{C+}\mco_T^\gamma)[\p_{\tilde{w}_{d+1}}, \ldots, \p_{\tilde{w}_{n}}] \notag \\ \simeq &D_{(\mbc^*)^{n}} /(\sum_{i=1}^{d} c_{ki} \tilde{w}_i \p_{\tilde{w}_i} + \gamma_k + 1)_{k=1, \ldots ,d}, (\tilde{w}_i -1)_{i=d+1, \ldots , n}\, . \label{eq:dimagek1}
\end{align}
The Hodge filtration is (cf. \cite[Formula (1.8.6)]{Saitobfunc}) \begin{align}
&F_{p+(n-d)}^H \left(\Gamma((\mbc^*)^{n}, \mch^0 k_{E+}k_{C+}\mco_T^\gamma)\right) = \sum_{p_1+p_2 =p} F_{p_1}^H\Gamma(T, \mch^0 k_{C+}\mco_T^\gamma) \otimes \underline{\p}^{p_2} \notag \\
=  &\sum_{p_1+p_2 =p} F_{p_1}^{ord}\Gamma(T, \mch^0 k_{C+}\mco_T^\gamma) \otimes \underline{\p}^{p_2} =F_p^{ord} \left( \Gamma((\mbc^*)^{n}, \mch^0 k_{E+}k_{C+}\mco_T^\gamma) \right)\, .
\end{align}
Hence we see that the Hodge filtration on the presentation \eqref{eq:dimagek1} shifted by $(n-d)$ is equal to the order filtration, i.e. $F_{p+(n-d)}^H = F^{ord}_p$.

The map $k_F$ is again a change of coordinates, so we have
\begin{align}
\Gamma((\mbc^*)^{n},\mch^0 k_{F +} k_{E_+} k_{C +} \mco_T^\gamma) &\simeq D_{(\mbc^*)^{n}}/ \left( (\sum_{j=1}^{n} a_{kj} w_j \p_{w_j}+ \gamma_k +1)_{k=1, \ldots , d}, (\underline{w}^{\underline{m}_i} -1 )_{i=d+1,\ldots ,n}\right) \notag \\
&\simeq D_{(\mbc^*)^{n}}/ \left( (\sum_{j=1}^{n} a_{kj} w_j  \p_{w_j}+\gamma_k +1)_{k=1, \ldots , d},(\check{\Box}_{\underline{m}})_{\underline{m} \in \mbl_A}\right)\, , \label{eq:dimagek2}
\end{align}
where $\underline{m}_i$ are the columns of the inverse matrix $M = F^{-1}$. The first isomorphism follows from the equality $A = C \cdot E \cdot F$. The second isomorphism follows from the fact that an element $\underline{m} \in \mbz^{n}$ is a relation between the columns of $A$ if and only if it is a relation between the columns of $E \cdot F$.
So the Hodge filtration on the presentation \eqref{eq:dimagek2} shifted by $(n-d)$ is again the order filtration.
We have
\[
\sum_{j=1}^{n} a_{kj} w_j \p_{w_j} +\gamma_k + 1 = \sum_{j=1}^{n} a_{kj}\p_{w_j} w_j - \sum_{j=1}^n a_{kj} +\gamma_k + 1\, .
\]
Setting $\gamma_k := \sum_{j=1}^{n} a_{kj} +\beta_k -1$, shows that
\begin{equation}\label{eq:tTransform}
 {^*\!}\check{\mcm}^\beta_{A'} = \mch^0 k_{A+}\mco_T^{\sum_{j=1}^{n} \underline{a}_{j} +\beta -\mathbf{1}} \simeq \mch^0 k_{A+}\mco_T^\beta
\end{equation}
where the last isomorphism is given by right multiplication with $\underline{t}^{-\sum_{j=1}^{n} \underline{a}_{j} +\mathbf{1}}$ (here $\mathbf{1}\! :=\! (1,\ldots,1)\! \in\! \mbz^d$).
\end{proof}

The next step is to compute the Hodge filtration of $h_{A +} \mco_T^\beta \simeq l_{A+} k_{A+} \mco_T^\beta$
from that of $k_{A+}\mco_T^\beta$. As the map $l_A$ is an open embedding of the complement of a normal crossing divisor,
we need to consider the graph embedding embedding $i^\circ_g$ with respect to the function $g=w_1\cdots w_n$.
We proceed as described at the beginning of this section, i.e., we first extend the module
$i^\circ_{g +} k_{A +}\mco_T^\beta $ over the smooth divisor $\{t=0\}$.

\begin{lemma}\label{lem:resgraphembed}
The direct image $i^\circ_{g+} k_{A+} \mco_T^\beta$ is isomorphic to the cyclic $\mcd_{W \times \mbc_t^*}$-module $\mcd_{W \times \mbc_t^*} / \mci^\circ$ where $\mci^\circ$ is the left ideal generated by $(\check{E}'_k+\beta_k)_{k=1,\ldots d}$ for $\beta = (\beta_k)_{k=1,\ldots ,d} \in \mbr^{d}$ and $(\check{\Box}_{\underline{m}})_{\underline{m} \in \mbl_{A'}}$.
Recall that the vector fields $\check{E}'_k$ have been defined in formula \eqref{eq:Eprime} as
$\check{E}'_k := \sum_{i=1}^{n} b_{ki} \p_{w_i} w_i + c_k \p_t t \quad \text{for}\;\; k= 1, \ldots , d$.
Furthermore, the Hodge filtration shifted by $n-d+1$ is equal to the induced order filtration, i.e., we have
$$
F^H_{p}i^\circ_{g+} k_{A+} \mco_T^\beta \simeq F^{ord}_{p-(n-d+1)}\mcd_{W\times \dC_t^*}/\mci^\circ\, .
$$
\end{lemma}
\begin{proof}
We define
\[
\widetilde{W} := (W^* \times \mbc_{\tilde{t}}) \setminus \{ \tilde{t} + g(\underline{w}) = 0\}
\]
and factor the map $i^\circ_g$ in the following way.
Set
\begin{align}
l_1: W^* &\lra \widetilde{W} \notag \\
\underline{w} &\mapsto (\underline{w},0) \notag\\
l_2: \widetilde{W} &\lra W^* \times \mbc^*_t \notag \\
(\underline{w},\tilde{t}) &\mapsto (\underline{w}, \tilde{t}+g(\underline{w})) \notag
\end{align}
and let $l_3: W^*\times \mbc_t^* \ra W \times \mbc_t^*$ be the canonical inclusion. We have $i^\circ_g = l_3 \circ l_2 \circ l_1$.
For the convenience of the reader, let us summarize these maps in the following diagram
$$
\begin{tikzcd}
W^* \ar{rr}[swap]{l_A}{  \renewcommand{\arraystretch}{0.5}\begin{array}{l}\textup{\tiny open embedding }\\ \textup{\tiny complement of NCD}\end{array}}  \ar{dd}{l_1}[swap]{\textup{\tiny closed}} \ar{rrrrdd}{i_g^\circ} & &  W \ar{rr}{i_g} && W\times \dC_t \\ \\
\widetilde{W} \ar{rr}{l_2}[swap]{\simeq} & & W^*\times \dC_t^* \ar{rr}{l_3}[swap]{\renewcommand{\arraystretch}{0.5}\begin{array}{l}\textup{\tiny closed on}\\ \textup{\tiny support}\end{array}} & & W\times \dC_t^* \ar[swap]{uu}[swap]{j_t}{\renewcommand{\arraystretch}{0.5}\begin{array}{l}\textup{\tiny open embedding }\\ \textup{\tiny complement of smooth divisor}\end{array}}
\end{tikzcd}
$$
Notice again that all spaces involved are affine, hence we will work with the modules of global sections.
 Since $l_1$ is just the inclusion of a coordinate hyperplane we have
 \[
 \Gamma(\widetilde{W}, \mch^0 l_{1+} k_{A+} \mco_T^\beta) \simeq \Gamma(W^*, \mch^0 k_{A+} \mco_T^\beta)[\p_{\tilde{t}}]\, .
 \]
The Hodge filtration is given by
 \begin{equation}\label{eq:graphHshift}
 \Gamma(\widetilde{W}, F^H_{p+1}(\mch^0 l_{1+} k_{A+} \mco_T^\beta)) \simeq \sum_{p_1+p_2 = p} \Gamma(W^*, F^H_{p_1} \mch^0 k_{A+} \mco_T^\beta)\otimes \p_{\tilde{t}}^{p_2}\, .
 \end{equation}
Notice that $\Gamma(\widetilde{W}, \mch^0 l_{1+} k_{A+} \mco_T^\beta) \simeq D_{\widetilde{W}} / I_1^\circ$ where $I_1^\circ$ is the left ideal generated by $(\check{E}_k+ \beta_k)_{k=1,\ldots,d}$, $(\check{\Box}_{\underline{m}})_{\underline{m} \in \mbl_{A}}$ and $\tilde{t}$.\\

Under this isomorphism the Hodge filtration on $\Gamma(\widetilde{W}, \mch^0 l_{1+} k_{A+} \mco_T^\beta)$ shifted by $(n-d)+1$ is equal to the order filtration by Lemma \ref{lem:directMapk} and \eqref{eq:graphHshift}. The map $l_2$ is just a change of coordinates, hence under the substitutions $\tilde{t} \mapsto t= \tilde{t} +g(\underline{w})$, $w_i \p_{w_i} \mapsto w_i \p_{w_i} + g(\underline{w}) \p_t \equiv w_i \p_{w_i} + \p_t t$ for $i=1,\ldots, n$
and by using the presentation of $\mch^0(k_{A+}\mco_T^\beta)$ as acyclic $\mcd$-module, we get that
\begin{align}
\Gamma(W^* \times \mbc^*_t, \mch^0 l_{2+}l_{1+}k_{A+}\mco_T^\beta)&\simeq D_{W^* \times \mbc_t^*} / \left((\check{E}'_k+\beta_k)_{k=1,\ldots ,d} + (\check{\Box}_{\underline{m}})_{\underline{m} \in \mbl_{A}} + (t - w_1 \cdot\ldots\cdot w_n)\right) \notag \\
&\simeq D_{W^* \times \mbc_t^*} / \left((\check{E}'_k+\beta_k)_{k=1,\ldots ,d} + (\check{\Box}_{\underline{m}})_{\underline{m} \in \mbl_{A'}}\right)\, , \label{eq:basechanget}
\end{align}
where $\check{E}'_k$ was defined in formula \eqref{eq:Eprime} and $A'$ is the matrix defined just before that formula.
Notice that the Hodge filtration shifted by $(n-d)+1$ is again equal to the order filtration.\\

Since the support of $\mch^0 l_{2+}l_{1+}k_{A+} \mco_T^\beta$ lies in the subvariety $\{t = g(\underline{w})\}$, the closure of the support in $W \times \mbc^*_t$ does not meet $D\times \mbc^*_t$. We conclude that
\begin{align}
\Gamma(W \times \mbc^*_t, \mch^0 i^\circ_{g + } k_{A+} \mco_T^\beta)  &\simeq \Gamma(W \times \mbc^*_t, \mch^0 l_{3 +} l_{2 +} l_{1 +} k_{A+} \mco_T^\beta) \notag \\
&\simeq \Gamma(W^* \times \mbc^*_t, \mch^0 l_{2 + }l_{1+} k_{A+} \mco_T^\beta) \notag \\
&\simeq D_{W^* \times \mbc_t^*} / \left((\check{E}'_k+\beta_k)_{k=1,\ldots ,d} + (\check{\Box}_{\underline{m}})_{\underline{m} \in \mbl_{A'}}\right) \notag\\
&\simeq D_{W \times \mbc_t^*} / \left((E'_k +\beta_k)_{k=1,\ldots ,d} + (\check{\Box}_{\underline{m}})_{\underline{m} \in \mbl_{A'}}\right)\, . \notag
\end{align}
The Hodge filtration is then simply extended by using the following formula
\[
F_p^H \mch^0 i^\circ_{g+}  k_{A+} \mco_T^\beta \simeq F_p^H \mch^0 l_{3 +} l_{2 +} l_{1 +} k_{A+} \mco_T^\beta \simeq l_{3 *} F^H_p \mch^0 l_{2 +} l_{1 +} k_{A+} \mco_T^\beta\, .
\]
\end{proof}

\begin{proposition}\label{Prop:ExtensionGraphSmoothDivisor}
Let $\beta \in \mbr^d \setminus sRes(A)$. The direct image $\mch^0 j_{t+} i^\circ_{g+} k_{A+} \mco_T^\beta$ is isomorphic to the quotient $\check{\mcm}_{A'}^\beta = \mcd_{W \times \mbc_t} /\mci'$, where $\mci'$ is the left ideal generated by $(\check{E}'_k+\beta_k)_{k=1,\ldots,d}$ and $(\check{\Box}_{\underline{m}})_{\underline{m} \in \mbl_{A'}}$. Furthermore, if $\beta\in \mathfrak{A}_A$, then the Hodge filtration on $\mcd_{W \times \mbc_t}/\mci'$ shifted by $(n-d)+1$ is equal to the induced order filtration, that is,
$$
F^H_{p} \mch^0 j_{t+} i^\circ_{g+} k_{A+} \mco_T^\beta = F^{ord}_{p-(n-d+1)}\cD_{W\times \dC_t}/\cI'\, .
$$
\end{proposition}
\begin{proof}
First recall that we  have an isomorphism $\mch^0 j_{t+} i^\circ_{g+} k_{A+} \mco_T^\beta \simeq \mch^0 i_{g+} h_{A+} \mco^\beta_T$. The composed map $i_g \circ h_A$ is a torus embedding given by the matrix $A'$. Hence, we have an isomorphism $\mch^0 j_{t+} i^\circ_{g+} k_{A+} \mco_T^\beta \simeq \check{\mcm}_{A'}^{\beta}$ for $\beta \in \mbr^{d}$ and $\beta \notin sRes(A')=sRes(A)$. This shows the first claim.\\

For the second statement, suppose that $\beta \in \mathfrak{A}_A$. The formula for extending the Hodge filtration over the smooth divisor $\{t=0\}$ is
\begin{equation}\label{eq:Hodgeext}
F_p^H \check{\mcm}^\beta_{A'} = \sum_{i \geq 0} \p_t^i(V^0 \check{\mcm}^\beta_{A'} \cap j_{t *}j^{-1}_t\, F_{p-i}^H\,\check{\mcm}^\beta_{A'} )\, .
\end{equation}
At the level of global sections the adjunction morphism $\check{\mcm}^\beta_{A'} \lra j_{t+}j_t^+\check{\mcm}^\beta_{A'}$ is given by the inclusion
$\check{M}_{A'}^\beta \ra {^*\!}\check{M}_{A'}^\beta$, where ${^*\!}\check{M}_{A'}$ is the $D_{W \times \mbc^*_t}$-module from Lemma \ref{lem:resgraphembed} seen as a $D_{W \times \mbc_t}$-module. Hence, at the level of global sections, formula \eqref{eq:Hodgeext} becomes
\[
F_p^H \check{M}^\beta_{A'} = \sum_{i \geq 0} \p_t^i(V^0 \check{M}^\beta_{A'} \cap F^H_{p-i} {^*\!}\check{M}^\beta_{A'})\, .
\]
Since we have  $F_{p+(n-d+1)}^H {^*}\check{M}^\beta_{A'} = F_p^{ord} {^*\!}\check{M}^\beta_{A'}$ by the same lemma, we conclude that $F_{n-d}^H \check{M}^\beta_{A'} = 0$. The element $1 \in \check{M}^\beta_{A'}$ is in $V^0 \check{M}^\beta_{A'}$ by Proposition \ref{prop:Vfilt} and $1 \in F_{(n-d)+1}^H  {^*\!}\check{M}^\beta_{A'} = F_0^{ord} {^*\!}\check{M}^\beta_{A'}$ and therefore $1\in F^H_{(n-d)+1} \check{M}^\beta_{A'}$. Notice that both $(\check{M}^\beta_{A'}, F^H_\bullet)$ and $(\check{M}^\beta_{A'}, F^{ord}_{\bullet})$ are cyclic, well-filtered $D_{W \times \mbc_t}$-modules (see e.g. \cite[Section 2.1.1]{Saito1}, therefore we can conclude
\[
F^{ord}_p \check{M}^\beta_{A'} \subset F^H_{p+(n-d+1)} \check{M}^\beta_{A'}\, .
\]
In order to show the reverse inclusion, we have to show
\[
F_p^{ord} \check{M}^\beta_{A'} \supset F_{p+(n-d+1)}^H \check{M}^\beta_{A'} = \sum_{i \geq 0} \p_t^i(V^0 \check{M}^\beta_{A'} \cap F^H_{p+(n-d+1)-i} {^*\!}\check{M}^\beta_{A'}) = \sum_{i \geq 0} \p_t^i(V^0_{ind} \check{M}^\beta_{A'} \cap F^{ord}_{p-i} {^*\!}\check{M}^\beta_{A'})
\]
for all $p \geq 0$, where the last equality follows from Proposition \ref{prop:Vfilt} and Lemma \ref{lem:resgraphembed}. Since we have
\[
F^{ord}_{p} \check{M}^\beta_{A'}  \supset \p_t^i F^{ord}_{p-i} \check{M}^\beta_{A'} \; \text{for}\;\; p \geq 0\;\; \text{and} \;\; 0 \leq i \leq p
\]
it remains to show
\[
F_{p-i}^{ord} \check{M}^\beta_{A'} \supset V^{0}_{ind} \check{M}^\beta_{A'} \cap F^{ord}_{p-i} {^*\!}\check{M}^\beta_{A'} \; \text{for}\;\; p \geq 0 \;\; \text{and} \;\; 0 \leq i \leq p
\]
resp.
\[
F_{p}^{ord} \check{M}^\beta_{A'} \supset V^{0}_{ind} \check{M}^\beta_{A'} \cap F^{ord}_{p} {^*\!}\check{M}^\beta_{A'} \;\; \text{for} \;\; p \geq 0\, .
\]
Now let $[P] \in V^{0}_{ind} \check{M}^\beta_{A'} \cap F^{ord}_{p} {^*\!}\check{M}^\beta_{A'}$, then $P \in D_{W \times \mbc_t}$ can be written as
\[
P = t^{-k}P_k + t^{-k+1} P_{k-1} + \ldots
\]
with $P_i \in \mbc[w_1,\ldots,w_n]\langle \p_t, \p_{w_1}, \ldots , \p_{w_{n}} \rangle$ and $P_k \neq 0$. Since $t^k \cdot [P] \in V^k_{ind} \check{M}^\beta_{A'} \cap F_p^{ord} \check{M}^\beta_{A'}$ it is enough to prove
\[
t^k F_{p}^{ord} \check{M}^\beta_{A'} \supset V^{k}_{ind} \check{M}^\beta_{A'} \cap F^{ord}_{p} \check{M}^\beta_{A'} \;\; \text{for} \;\; p \geq 0\, .
\]
Given an element $[Q] \in  V^{k}_{ind} \check{M}^\beta_{A'} \cap F^{ord}_{p} \check{M}^\beta_{A'}$ we can find, using Proposition \ref{prop:CompFiltOnSWGarbe},  a $Q' \in V^k D_{W \times \mbc_t} \cap  F_p D_{W \times \mbc_t}$ with $[Q] = [Q']$. But this element $Q'$ can be written as a linear combination of monomials $t^{l_0}w_1^{l_1}\ldots w_{n}^{l_{n}} \p_t^{p_0} \p_{w_1}^{p_1}\ldots \p_{w_{n}}^{p_{n}}$ with $p_0 + \ldots p_{n} \leq p$ and $l_0 - p_0 \geq k$, hence $[Q'] \in t^k F^{ord}_p \check{M}_{A'}^\beta$. This shows the statement in the case $\mbn A \neq \mbz^d$ (recall that $\dN A \neq \dZ^d$ was a condition used in Proposition \ref{prop:CompFiltOnSWGarbe}).

In the case $\mbn A = \mbz^d$ the support of $\check{\mcm}^\beta_{A'}$ is disjoint from the divisor $\{t =0\}$, hence the extension of the Hodge filtration is simply given by $F_p^H \check{\mcm}^\beta_{A'} = j_{t *} F_p^H {^*\!}\check{\mcm}^\beta_{A'}$. Since $j_t$ is an open embedding we have $\check{M}^\beta_{A'} = {^*\!}\check{M}^\beta_{A'}$ and therefore also $F_p^H \check{M}^\beta_{A'} = F_p^H{^*\!}\check{M}^\beta_{A'}$. This shows the claim.

\end{proof}

Now we want to deduce the Hodge filtration on $h_{A+} \mco_T^\beta$ from the proposition above.

\begin{theorem}\label{thm:hpot}
Let $\mbn A = \mbz^d \cap \mbr_{\geq 0} A$ and $\beta \in \mbr^d \setminus sRes(A)$. The direct image $h_{A+} \mco_T^\beta$ is isomorphic to the cyclic $D_W$-module $\check{\mcm}^\beta_A:= \mcd_W/\check{\mci}$, where $\check{\mci}$ is the left ideal generated by $(\check{E}_k+\beta_k)_{k=1,\ldots, d}$ and $(\check{\Box}_{\underline{m}})_{\underline{m} \in \mbl_A}$. For $\beta \in \mathfrak{A}_A$ the Hodge filtration on $\check{\mcm}^\beta_A$ is equal to the order filtration shifted by $n-d$, i.e.
\[
F_{p+(n-d)}^H \check{\mcm}^\beta_A = F_{p}^{ord} \check{\mcm}^\beta_A\, .
\]
\end{theorem}
\begin{proof}
Recall that we have $j_{t} \circ i^\circ_g \circ k_A = i_g \circ h_A$ where $i_f$ is the graph embedding from $\eqref{eq:graphem}$. The map $i_f$ can be factored by
\begin{align}
i_0 : W &\lra W \times \mbc_{\widetilde{t}} \notag \\
\underline{w} &\mapsto (\underline{w},0) \notag \\
l_g: W \times \mbc_{\widetilde{t}} &\lra W \times \mbc_{t} \notag \\
(\underline{w},\widetilde{t}) &\mapsto (\underline{w}, \widetilde{t} + g(\underline{w}))\, . \notag
\end{align}
Once again, we summarize the relevant maps in the following diagram.
$$
\begin{tikzcd}
{ } && W\times \dC_{\widetilde{t}} \ar{ddrr}[swap]{l_g}{\simeq}\\ \\
W^* \ar{rr}{l} \ar{dd}{l_1}[swap]{\textup{\tiny closed}} \ar{rrrrdd}{i^\circ_g} & &  W \ar{rr}{i_g}[swap]{\textup{\tiny closed}} \ar{uu}[swap]{i_0}{\textup{\tiny closed}} && W\times \dC_t \\ \\
\widetilde{W} \ar{rr}{l_2}[swap]{\simeq} & & W^*\times \dC_t^* \ar{rr}{l_3}[swap]{\renewcommand{\arraystretch}{0.5}\begin{array}{c}\textup{\tiny closed on}\\ \textup{\tiny support}\end{array}} & & W\times \dC_t^* \ar[swap]{uu}[swap]{j_t}{\renewcommand{\arraystretch}{0.5}\begin{array}{l}\textup{\tiny open embedding }\\ \textup{\tiny complement of smooth divisor}\end{array}}
\end{tikzcd}
$$
We first compute $\mch^0(l_g^{-1})_+ \check{\mcm}^\beta_{A'}$ with its corresponding Hodge filtration. Since  $(l_g)^{-1}$ is just a coordinate change we get similarly to formula \eqref{eq:basechanget}
\begin{equation}\label{eq:lback}
\Gamma(W \times \mbc_{\widetilde{t}},\mch^0(l_g^{-1})_+ \check{\mcm}_{A'}^\beta) \simeq D_{W \times \mbc_{\widetilde{t}}} / \left( (\check{E}_k+\beta_k)_{k=1, \ldots, d}, (\check{\Box}_{\underline{m}})_{\underline{m} \in \mbl_A}, (\widetilde{t}) \right)\, ,
\end{equation}
where the Hodge filtration on the right hand side is the induced order filtration shifted by $(n-d+1)$. Notice that the right hand side of \eqref{eq:lback} is simply $\check{M}_A^\beta[\p_{\widetilde{t}}]$, hence the Hodge filtration on
\[
\check{M}^\beta_A = \Gamma(W, \check{\mcm}^\beta_A) = \Gamma(W \times \mbc_{\widetilde{t}}, \mch^0 i_0^+ (l_g^{-1})_+ \check{\mcm}^\beta_{A'})
\]
is simply the order filtration shifted by $(n-d)$ by \cite[Proposition 3.2.2 (iii)]{Saito1}.
\end{proof}
\begin{remark}\label{rem:pdi1}
Let $\mco^\beta_{\mbc^*} = \mcd_{\mbc^*}/(\p_t t + \beta)$ and $j_0: \mbc^* \lra \mbc$ be the inclusion. If $\beta \not \in \{-1,-2, -3 ,\ldots\}$ then
\[
j_{0+} \mco_{\mbc^*}^\beta \simeq \mcd_{\mbc}/(\p_t t + \beta)
\]
as well as
\[
j_{0 \dag} \mco^{-\beta-1}_{\mbc^*} \simeq \mbd j_{0 +} \mbd \mco^{\beta}_{\mbc^*} \simeq \mbd j_{0+} \mcd_{\mbc^*} / ( \p_t t +\beta) \simeq \mbd \, \mcd_{\mbc}/( \p_t t + \beta) \simeq \mcd_{\mbc}/(t \p_t - \beta)
\]
If, additionally, $\beta \in (-1,0]$ holds then by Theorem \ref{thm:hpot}
\[
 (j_{0+} \mco_{\mbc^*}^\beta, F^H_{\bullet}) \simeq  (\mcd_{\mbc}/ (\p_t t + \beta), F^{ord}_\bullet)
\]
In order to compute the Hodge filtration on $j_{0 \dag} \mco^{-\beta-1}_{\mbc^*}$ we use that $(\mbd \mco^{-\beta-1}_{\mbc^*}, F^H_p) = (\mcd_{\mbc^*}/(\p_t t + \beta), F^{ord}_{\bullet-1})$ and therefore
$(j_{0+}\mbd \mco^{-\beta-1}_{\mbc^*}, F^H_p) = (\mcd_{\mbc}/(\p_t t + \beta), F^{ord}_{\bullet-1})$ holds, since we assumed $\beta \in (-1,0]$.
We use the filtered resolution
\[
\xymatrix{(\mcd_{\mbc^*}, F^{ord}_{\bullet-2}) \ar[rr]^{\p_t t + \beta} && (\mcd_{\mbc^*}, F^{ord}_{\bullet-1}})
\]
to resolve $(\mcd_{\mbc}/(\p_t t + \beta), F^{ord}_{\bullet-1})$ and apply $\mch om(-,(\mcd_{\mbc}, F^{ord}_{\bullet-2}) \otimes \omega^\vee_X)$ to compute the dual of  $(\mcd_{\mbc}/(\p_t t + \beta), F^{ord}_{\bullet-1})$ (cf. \cite[page 55]{SaitoonMHM} for the choice of filtration on $\mcd_\mbc \otimes \omega_X^\vee$). This gives
\[
(j_{0 \dag} \mco^{-\beta-1}_{\mbc^*}, F^H_p) \simeq (\mcd_{\mbc}/(t \p_t - \beta),F^{ord}_p)
\]
for $\beta \in (-1,0]$.
\end{remark}

\section{Integral transforms of torus embeddings}
\label{sec:Radon}

In this section we investigate how the Hodge filtration of $\mbc_T^{H,\beta}$ behaves under a certain integral transformation, which depends on a matrix $\widetilde{A}$ and a parameter $\beta_0$. We show that the outcome of this transform is isomorphic to the GKZ-system $\mcm^{(\beta_0,\beta)}_{\widetilde{A}}$ (cf. Proposition \ref{prop:GKzeqtwRad}). In the special case $\beta_0 \in \mbz$ we will show that the integral transformation mentioned above is isomorphic to the  Radon transform of a torus embedding.

The Radon transformation has been extensively used in the
previous papers \cite{Reich2} and \cite{ReiSe2} in order to study hyperplane sections of toric varieties,
more precisely, fibres of Laurent polynomials and their compactifications. \\

In the first subsection below, we give a brief reminder of how certain GKZ-systems  can be constructed using the Radon transformation. The second subsection introduces an integral transform that  is able to produce all GKZ-systems $\mcm^{\widetilde{\beta}}_{\widetilde{A}}$ with $\widetilde{\beta} \not \in sRes(\widetilde{A})$ and homogeneous $\widetilde{A}$.

The next subsections (until \ref{subsec:MainThm}) constitute the main part of this
section, where we study in detail how the various functors entering in the definition of this integral transformation
act on the twisted structure sheaf. One can roughly divide the construction in two parts:
In subsection \ref{subsec:CalculationInCharts} we calculate the push-forward  of a tensor product between the twisted structure sheaf and a kernel to a partial compactification.
Then one has to study the projection to the parameter space (i.e., the space on which the GKZ-system is defined).
The calculation of
the behavior of the filtration steps is non-trivial, as the higher direct images of these filtration steps, being coherent $\cO$-modules,
do not, a priori, vanish. However, we can show that this is actually the case in the current situation. We formulate this result
in the language of $\msr$-modules (i.e., using the Rees construction for filtered $\cD$-modules), and make extensive use of (variants of) the Euler-Koszul complex of hypergeometric
modules. All these intermediate steps are contained in the subsections \ref{subsec:Koszul} until \ref{subsec:MainThm}.
A very important technical result is the calculation of some local cohomology groups of a  certain semi-group ring,
contained in subsection \ref{subsec:LocalCohom}. The culminating point is then Theorem \ref{thm:HodgeGKZ}
which gives a precise description of the Hodge filtration on certain GKZ-systems.

\subsection{Hypergeometric modules, Gau\ss-Manin systems and the Radon transformation}
\label{subsec:Reminder}

Here we give a brief reminder on the relationship between GKZ-hypergeometric systems, Gau\ss-Manin systems of families of Laurent polynomials as developed in \cite{Reich2}.

As in \cite{ReiSe,ReiSe2}, we will consider a homogenization of the above systems. Namely,
given the matrix $A=(a_{ki})$, we consider the system $\cM_{\widetilde{A}}^{\widetilde{\beta}}$, where
$\widetilde{A}$ is the $(d+1) \times (n+1)$ integer matrix
\begin{equation}\label{def:Atilde}
\widetilde{A} := (\widetilde{\underline{a}}_0, \ldots, \widetilde{\underline{a}}_n) :=\left(\begin{matrix}
1 & 1 & \ldots & 1 \\ 0 & a_{11} & \dots & a_{1n} \\ \vdots & \vdots & &\vdots \\ 0 & a_{d1} & \dots &a_{dn}
\end{matrix} \right)
\end{equation}
and $\widetilde{\beta}\in \dC^{d+1}$.\\

In order to show that such a homogenized GKZ-system comes from geometry we have to review briefly the so-called Radon transformation for $\mcd$-modules which was introduced by Brylinski \cite{Brylinski} and variants were later added by d'Agnolo and Eastwood \cite{AE}.\\

Let $W$ be the dual vector space of $V$ with coordinates $w_0, \ldots, w_n$, and let $\lambda_0, \ldots , \lambda_n$ be coordinates for $V$. We will denote by $Z \subset \mbp(W) \times V$ the universal hyperplane given by $Z := \{\sum_{i=0}^{n} \lambda_i w_i = 0 \}$ and denote its complement by $U := (\mbp(W) \times V) \setminus Z$. Consider the following diagram
\begin{equation}\label{eq:universaldiag}
\begin{tikzcd}
{}&& U \ar{drr}{\pi_2^U} \ar{dll}[swap]{\pi_1^U} \ar[hook]{d}{j_U} \\
\mbp(W) && \mbp(W) \times V \ar{ll}[swap]{\pi_1} \ar{rr}{\pi_2} && V  \\
&& Z \ar{ull}{\pi_1^Z} \ar[hook]{u}{i_Z} \ar{rru}[swap]{\pi_2^Z}
\end{tikzcd}
 \end{equation}
We will use in the sequel several variants of the so-called Radon transformation in the derived category of mixed Hodge modules. These are functors from $D^b \MHM(\mbp(W))$ to $D^b\MHM(\mcd_V)$ given by
\begin{align}
{^*}\mcr(M) &:= \pi_{2*}^Z (\pi_1^Z)^*M \simeq \pi_{2 *} i_{Z *} i_Z^* \pi_1^*M\, , \notag \\
{^!}\mcr(M) &:= \mbd \circ {^*}\mcr \circ \mbd \, (M) \simeq \pi_{2*}^Z (\pi_1^Z)^! M \simeq \pi_{2 *} i_{Z *} i_Z^! \pi_1^! M \notag \\
{^*}\mcr_{cst}(M) &:= \pi_{2*} \pi_1^*M\, . \notag \\
{^!}\mcr_{cst}(M) &:= \mbd \circ {^*}\mcr_{cst} \circ \mbd\,(M)\simeq  \pi_{2*} \pi_1^! M\, . \notag \\
{^*}\mcr^\circ_c(M) &:= \pi_{2 !}^U (\pi_1^U)^*(M) \simeq \pi_{2 *} j_{U !} j_U^* \pi_1^* (M) \notag \\
{^!}\mcr^\circ(M)&:= \mbd \circ {^*}\mcr^\circ_c \circ \mbd\, (M) \simeq  \pi_{2 *}^U (\pi_1^U)^!(M) \simeq \pi_{2 *} j_{U *} j_U^! \pi_1^! (M)\, . \notag
\end{align}

The adjunction triangle corresponding to the open embedding $j_U$ and the closed embedding $i_Z$ gives rise to the following triangles of Radon transformations
\begin{align}
{^!}\mcr(M) \lra {^!}\mcr_{cst}(M) \lra {^!}\mcr^\circ(M) \overset{+1}{\lra}\, , \label{eq:Radontri1} \\
{^*}\mcr^\circ_{c}(M) \lra {^*}\mcr_{cst}(M) \lra {^*}\mcr(M) \overset{+1}{\lra}\, , \label{eq:Radontri2}
\end{align}

where the second triangle is dual to the first.\\

We now introduce a family of Laurent polynomials defined on $T \times \Lambda := (\mbc^*)^d \times \mbc^n$ using the columns of the matrix $A$, more precisely, we put
\begin{align}\label{eq:FamLaurent}
\varphi_A: T \times \Lambda &\lra V=\mbc_{\lambda_0} \times \Lambda\, , \\
(t_1, \ldots , t_d, \lambda_1, \ldots , \lambda_n) &\mapsto (- \sum_{i=1}^n \lambda_i \underline{t}^{\underline{a}_i}, \lambda_1, \ldots, \lambda_n)\, .  \notag
\end{align}

The following theorem of \cite{Reich2} constructs a morphism between the Gau\ss-Manin system $\cH^0(\varphi_{A,+}\cO_{T\times \Lambda})$ resp.
its proper version $\cH^0(\varphi_{A,\dag}\cO_{T\times \Lambda})$ and certain GKZ-hypergeometric systems and identify both with a corresponding Radon transform.

For this we apply the triangle \eqref{eq:Radontri1}  to $M = g_!\, \mbd\, {^p}\mbq^H_T$ and the triangle \eqref{eq:Radontri2} to $M = g_* {^p}\mbq^H_T$, where the map $g$ was defined by
\begin{align}
g: T &\lra \mbp(W) \notag \\
(t_1, \ldots ,t_d) &\mapsto (1:\underline{t}^{\underline{a}_1}: \ldots : \underline{t}^{\underline{a}_n})\, . \notag
\end{align}

\begin{theorem}\cite[Lemma 1.11, Proposition 3.4]{Reich2}\label{thm:4termseq}
Let $A=(\underline{a}_1,\ldots, \underline{a}_n)\in M(d\times n, \dZ)$
and $\widetilde{A}=(\widetilde{\underline{a}}_0,\widetilde{\underline{a}}_1,\ldots,\widetilde{\underline{a}}_n)\in M((d+1)\times (n+1),\dZ)$ be as above and
assume that $\widetilde{A}$ satisfies
\begin{enumerate}
\item $\mbz \widetilde{A} = \mbz^{d+1}$
\item $\mbn \widetilde{A} = \mbr_{\geq 0}\widetilde{A} \cap \mbz^{d+1}$
\end{enumerate}

Then for every $\widetilde{\beta} \in \mbn\widetilde{A}$ and every $\widetilde{\beta}' \in int(\mbn \widetilde{A})$, we have that
\[
\mcm^{\widetilde{\beta}}_{\widetilde{A}} \simeq \textup{DMod}\left(\mch^{n+1}({^*}\mcr^\circ_c(g_* {^p}\mbq^H_T))\right)  \quad \text{and} \quad \mcm^{-\widetilde{\beta}'}_{\widetilde{A}} \simeq \textup{DMod}\left(\mch^{-n-1}({^!}\mcr^\circ(g_! \mbd\, {^p}\mbq^H_T))\right)
\]
If we define
\[
{^H\!\!}\mcm^{\widetilde{\beta}}_{\widetilde{A}} := \mch^{n+1}({^*}\mcr^\circ_c(g_* {^p}\mbq^H_T))  \quad \text{and} \quad {^H\!\!}\mcm^{-\widetilde{\beta}'}_{\widetilde{A}} := \mch^{-n-1}({^!}\mcr^\circ(g_! \mbd\, {^p}\mbq^H_T))
\]

the following sequences
of mixed Hodge modules are exact and dual to each other:
$$
\xymatrix@C=8pt{ & H^{d-1}(T,\mbc)\otimes {^p}\mbq^H_V & \mch^0(\varphi_{A*} {^p}\mbq^H_{T \times \Lambda}) & {^H\!\!}\mcm_{\widetilde{A}}^{\widetilde{\beta}} & H^{d}(T, \mbc)\otimes {^p}\mbq^H_V & \\
0 \ar[r] & \cH^{n}({^*}\mcr_{cst}(g_* {^p}\mbq^H_T)) \ar[u]_\simeq \ar[r] & \mch^n({^*}\mcr(g_* {^p}\mbq^H_T)) \ar[u]_\simeq \ar[r] & \mch^{n+1}({^*}\mcr^\circ_c(g_* {^p}\mbq^H_T)) \ar[u]_\simeq \ar[r] &
\mch^{n+1}({^*}\mcr_{cst}(g_* {^p}\mbq^H_T)) \ar[u]_\simeq \ar[r] & 0\\
0 & \mch^{-n}({^!}\mcr_{cst}(g_! \mbd\,{^p}\mbq^H_T)) \ar[l] \ar[d]^\simeq  & \mch^{-n}({^!}\mcr(g_! \mbd\,{^p}\mbq^H_T)) \ar[l] \ar[d]^\simeq & \mch^{-n-1}({^!}\mcr^\circ(g_! \mbd\, {^p}\mbq^H_T)) \ar[l]\ar[d]^\simeq & \mch^{-n-1}({^!}\mcr_{cst}(g_! {^p}\mbq^H_T)) \ar[l] \ar[d]^\simeq & 0 \, .\ar[l]\\
& H_{d+1}(T,\mbc)\otimes \mbd\,{^p}\mbq^H_V & \mch^0(\varphi_{A!} \mbd\,{^p}\mbq^H_{T \times \Lambda}) & {^H\!\!}\mcm_{\widetilde{A}}^{-\widetilde{\beta}'} & H_{d}(T, \mbc)\otimes \mbd\,{^p}\mbq^H_V\\ &}
$$
\end{theorem}

\begin{proposition}\label{prop:morphdualGKZ}
 Let $\widetilde{\beta} \in \mbn \widetilde{A}$ and $\widetilde{\beta}' \in int(\mbn \widetilde{A})$. There exists a natural morphism of mixed Hodge modules between
 ${^H\!\!}\mcm_{\widetilde{A}}^{-\widetilde{\beta}'}(-d-n)$ and ${^H\!\!}\mcm^{\widetilde{\beta}}_{\widetilde{A}}$, which is (up to multiplication with a non-zero constant) given on the underlying $\mcd_V$-modules by
\begin{align}
\mcm_{\widetilde{A}}^{-\widetilde{\beta}'} &\lra \mcm^{\widetilde{\beta}}_{\widetilde{A}} \notag \\
P &\mapsto P \cdot \p^{\widetilde{\beta} + \widetilde{\beta}'}\, , \notag
\end{align}
where $\p^{\widetilde{\beta} + \widetilde{\beta}'} := \prod_{i=0}^n \p_{\lambda_i}^{k_i}$ for any $\underline{k} = (k_0, \ldots ,k_n)$ with $\widetilde{A} \cdot \underline{k} = \widetilde{\beta} + \widetilde{\beta}'$.
\end{proposition}
\begin{proof}
First notice that there is a natural morphism of mixed Hodge modules
\[
\mch^0(\varphi_{A!} \mbd\,{^p}\mbq^H_{T \times \Lambda})(-d-n) \lra \mch^0(\varphi_{A*} {^p}\mbq^H_{T \times \Lambda})
\]
which is induced by the morphism $\mbd\,{^p}\mbq^H_{T \times \Lambda}(-d-n) \ra {^p}\mbq^H_{T \times \Lambda}$. Using the isomorphisms in the second column, this gives a morphism
\[
\mch^{-n}({^!}\mcr(g_! \mbd\,{^p}\mbq^H_T)) (-n-d) \lra \mch^n({^*}\mcr(g_* {^p}\mbq^H_T))\, .
\]
Now we can concatenate this with the following morphisms
\[
\begin{tikzcd}
 \mch^n({^*}\mcr(g_* {^p}\mbq^H_T))  \ar{r} & \mch^{n+1}({^*}\mcr^\circ_c(g_* {^p}\mbq^H_T)) \\
 \mch^{-n}({^!}\mcr(g_! \mbd\,{^p}\mbq^H_T))(-n-d)\ar{u} & \mch^{-n-1}({^!}\mcr^\circ(g_! \mbd\, {^p}\mbq^H_T))(-n-d)\, . \ar{l} \ar[dotted]{u}\end{tikzcd}
\]
This gives the desired morphism of mixed Hodge modules between ${^H\!\!}\mcm_{\widetilde{A}}^{-\widetilde{\beta}'}(-d-n)$ and ${^H\!\!}\mcm^{\widetilde{\beta}}_{\widetilde{A}}$.
Then it follows from \cite[Lemma 2.12]{ReiSe2} that the corresponding morphism of the underlying $\cD_V$-modules is (up to multiplication by a non-zero constant) right multiplication with $\partial^{\widetilde{\beta}+\widetilde{\beta}'}$.
\end{proof}
We will now prove a partial generalization of Theorem \ref{thm:4termseq} for non-integer $\beta$.
\begin{proposition}\label{prop:Radonhalfgeneral}
With the notation as above, let $\widetilde{\beta} = (\beta_0, \beta) \in (\mbz \times \mbr^{d}) \setminus \textup{sRes}(\widetilde{A})$, then we have the following isomorphism
\[
\textup{DMod}( \mch^{n+1}({^*}\mcr^\circ_c(g_* {^p}\mbc^{\beta,H}_T)) \simeq \mcm^{\widetilde{\beta}}_{\widetilde{A}}
\]
This induces the structure of a complex mixed Hodge module on $\mcm^{\widetilde{\beta}}_{\widetilde{A}}$ which we call ${^H\!\!}\mcm^{\widetilde{\beta}}_{\widetilde{A}}$.
\end{proposition}
\begin{proof}
Consider the following commutative diagram with cartesian square
\[
\xymatrix{ & W \\ \mbc^* \times T \ar[ur]^{h_{\widetilde{A}}} \ar[r]^{\tilde{h}} \ar[d]_{\pi_T} & W \setminus \{0\} \ar[u]_{\bar{j}} \ar[d]^\pi \\ T \ar[r]^g & \mbp(W)}
\]
where $\pi_T$ is the projection to the first factor.
We have
\[
\bar{j}_+ \pi^+ g_+ \mco_T^{\beta} \simeq  \bar{j}_+ \widetilde{h}_+ \pi_T^+ \mco_T^{\beta} \simeq h_{\widetilde{A}+} \mco_{\mbc^* \times T}^{(0,\beta)}[1]\simeq h_+ \mco_{\mbc^* \times T}^{(\beta_0,\beta)}[1]
\]
for every $\beta_0 \in \mbz$. Let ${^*}\mcr^\circ_c: D^b_{rh}(\mcd_X) \ra D^b_{rh}(\mcd_X)$ be the corresponding functor for $\mcd$-modules which is given by $M \mapsto \pi_{2+} j_{U \dag} j_U^\dag \pi_1^\dag M$. We have the following isomorphism
\[
{^*}\mcr^\circ_c(g_+ \mco^\beta_T)[-n-1] \simeq \FL(\bar{j}_+ \pi^+ g_+ \mco^\beta_T[-1]) \simeq \FL(h_{\widetilde{A}+} \mco_{\mbc^* \times T}^{(\beta_0,\beta)}) \simeq \mcm^{\widetilde{\beta}}_{\widetilde{A}}
\]
where the first isomorphism follows from \cite[Proposition 1]{AE}. Notice that the various shifts, occurring in the formulas above, stem from a different (shifted) definition of the (exceptional) inverse image for $\mcd$-modules in loc. cit. .
\end{proof}

\subsection{Integral transforms of twisted structure sheaves}\label{subsec:TwistedIntTrafo}
Unfortunately, the Radon transformation produces only GKZ-systems with $\beta_0 \in \mbz$, as we can see in Proposition \ref{prop:Radonhalfgeneral}. To remedy this fact, we introduce a integral transformation which takes care of that by twisting with a kernel which depends on $\beta_0$.\\

Let $T=(\mbc^*)^d$ resp.  $\widetilde{T}:= (\mbc^*)^{d+1}$ be tori with coordinates $t_1,\ldots, t_d$ resp. $t_0,\ldots, t_d$, $W= V= \mbc^{n+1}$ with coordinates $w_0,\ldots, w_n$ resp. $\lambda_0, \ldots, \lambda_n$ and consider the torus embedding with respect to the matrix  $\widetilde{A}$
\begin{align*}
h:=h_{\widetilde{A}}: \widetilde{T} &\lra W \\
(t_0,\ldots, t_d) &\mapsto (t_0, t_0 \underline{t}^{\underline{a}_1},\ldots, t_0 \underline{t}^{\underline{a}_n})
\end{align*}
If $\widetilde{\beta} \not \in sRes(\widetilde{A})$ the GKZ-system $\mcm^{\widetilde{\beta}}_{\widetilde{A}}$ is given by $\FL(h_+ \mco_{\widetilde{T}}^{\widetilde{\beta}})$. Consider the maps
\[
\xymatrix{\mbc^* \ar[r]^{j} &\mbc  & \ar[l]_-F T \times V \ar[r]^{p}  \ar[d]^{q} & V \\ & & T }
\]
where $F$ is given by $(t_1,\ldots,t_d, \lambda_0,\ldots, \lambda_n) \mapsto \lambda_0 + \sum_{i=1}^n \lambda_i \underline{t}^{\underline{a}_i}$, where $p_1$ resp. $p_2$  is the projection to the first resp. second factor and where $j$ is the inclusion.\\

Proposition \ref{prop:Radonhalfgeneral} showed that a GKZ-system with integer $\beta_0$ can be expressed by a Radon transformation, generalizing a result in \cite{Reich2}. The Radon transformation ${^*}\mcr^\circ_c$ can be seen as an integral transform from $\mbp^n$ to $\mbc^{n+1}$ with kernel $j_{U!}{^p}\mbc^{H}_T$. Now, in order to construct GKZ-systems with general $\beta_0$ we could twist the kernel $j_{U!}{^p}\mbc^{H}_U$ which means instead of using the constant module ${^p}\mbc^{H}_U$ on $U$ we could use a rank one local system on $U$ with monodormy $e^{2 \pi i \beta_0}$  (notice that $U$ has fundamental group isomorphic to $\mbz$). However, due to computational reason we use a slightly different approach. Instead of embedding the torus $T$ in $\mbp^n$ and considering an integral transform from $\mbp^n$ to $\mbc^{n+1}$ we directly define an integral transformation from $T$ to $V$ with a kernel depending on $\beta_0$ and $A$ (the matrix $A$ is encoded in the map $F$). We prove in Proposition \ref{prop:GKzeqtwRad} that the outcome of this integral transformation applied to the $\mcd_T$-module $\mco^\beta_T$ is indeed the GKZ-system $\mcm^{\widetilde{\beta}}_{\widetilde{A}}$. Finally, we prove in Proposition \ref{prop:compRadtwRad} that this approach is compatible with the original approach using the Radon transform from Proposition \ref{prop:Radonhalfgeneral}.

The following proposition is a variant of a theorem of d'Agnolo and Eastwood \cite{AE}.  It compares the Fourier-Laplace transform of the twisted structure sheaf under a torus embedding with an integral transform of the twisted structure on a smaller torus. The latter description is favorable since it naturally equips the GKZ-system with the structure of a mixed Hodge module.
\begin{proposition}\label{prop:GKzeqtwRad}
Let $\widetilde{\beta} = (\beta_0,\beta) \not \in sRes(\widetilde{A})$ then
\[
\mcm^{\widetilde{\beta}}_{\widetilde{A}}\simeq \FL(h_+ \mco_{\widetilde{T}}^{\widetilde{\beta}}) \simeq \mch^{2n+d+1}(p_{+} (q^\dag \mco_T^\beta \otimes_\mco   F^\dag (j_{\dag} \mco_{\mbc^*}^{-\beta_0-1})))
\]
\end{proposition}
\begin{proof}
Notice that the morphism $h: \widetilde{T} \lra \mbc^{n+1}$ factors as
\begin{equation}\label{eq:factorh}
\widetilde{T} \overset{\widetilde{j}}\lra  \mbc \times T \overset{k}\lra W
\end{equation}
where  $\widetilde{j}$ is the canonical embedding and $k$ is given by $(t_0,\ldots, t_d) \mapsto (t_0, t_0 \underline{t}^{\underline{a}_1},\ldots, t_0 \underline{t}^{\underline{a}_n})$.
Consider the diagram
\[
\xymatrix@C=4pc@R=3pc{ \mbc & \mbc \times T  \ar[l]_{l}
  \ar[r]^{k} & W & \mbc \times T  \ar[dr]^{f} & \\
  \mbc \times \mbc \ar[u]_{p_1} \ar[d]_{p_2} & \mbc \times T \times
  V  \ar[u]_{p_{12}} \ar[l]_{id_{\mbc} \times F}
  \ar[d]^{p_{13}}  \ar[r]^{k\times id_{\mbc^{n+1}}} & W
  \times V  \ar[u]_{q_{1}} \ar[d]^{q_{2}} & \mbc \times T
  \times V \ar[u]_{p_{12}} \ar[d]^{p_{13}} \ar[r]^{g} & T\\
  \mbc & T \times V \ar[l]_F  \ar[r]^p& V & T
   \times V \ar[ur]^{q} &}\qquad
\]
where $p_{ij}$ are the projections to the factors $i$ and $j$, the maps $l,q,p_1,q_1$ are the projections to the first and the maps $f,g,p,p_2,q_2,$ are the projection to the second factor.\\

We have that $\widetilde{j}_+ \mco_{\widetilde{T}}^{\widetilde{\beta}} \simeq j_{+} \mco_{\mbc^*}^{\beta_0} \boxtimes  \mco_{T}^{\beta}  \simeq l^+ j_{+} \mco^{\beta_0}_{\mbc^*}[-d] \otimes f^+ \mco_T^{\beta}[-1]$ hence we get the following isomorphisms
\begin{align*}
&\FL(h_+ \mco^{\widetilde{\beta}}_{\widetilde{T}})\\
&\simeq \FL(k_+ \widetilde{j}_+ \mco^{\widetilde{\beta}}_{\widetilde{T}}),  &&\text{factorization of }h \; \eqref{eq:factorh} \\
&\simeq q_{2,+} ( q_1^+ k_+ \widetilde{j}_{+} \mco_{\widetilde{T}}^{\widetilde{\beta}}  \otimes \mcl)[-n-1], \\
&\simeq q_{2,+} ( q_1^+ k_+ (l^+ j_{+} \mco_{\mbc^*}^{\beta_0} \otimes f^+ \mco_T^{\beta})  \otimes \mcl)[-n-d-2], &&\text{use } \mco_{\widetilde{T}}^{\widetilde{\beta}} \simeq \mco_{\mbc^*}^{\beta_0} \boxtimes \mco_T^\beta \\
&\simeq q_{2,+} ((k \times id)_+ p_{12}^+ (l^+ j_{+} \mco_{\mbc^*}^{\beta_0} \otimes f^+ \mco_T^{\beta})  \otimes \mcl)[-n-d-2], &&\text{base change}\\
&\simeq q_{2,+}(k \times id)_+ ( p_{12}^+ (l^+ j_{+} \mco_{\mbc^*}^{\beta_0} \otimes f^+ \mco_T^{\beta})  \otimes (k \times id)^+\mcl)[-2d-2],  &&\text{projection formula}\\
&\simeq q_{2,+}(k \times id)_+ ( (id \times F)^+ p_1^+ j_{+} \mco_{\mbc^*}^{\beta_0} \otimes g^+ \mco_T^{\beta}  \otimes (k \times id)^+\mcl)[-n - 2d-3], &&g= f \circ p_{12},\\ & &&l \circ p_{12} = p_1 \circ (id \times F)\\
&\simeq p_{+}p_{13,+} ( (id \times F)^+ p_1^+ j_{+} \mco_{\mbc^*}^{\beta_0} \otimes g^+ \mco_T^{\beta}  \otimes (k \times id)^+\mcl)[-n-2d-3], &&p \circ p_{13} = q_2 \circ(k \times id)\\
&\simeq p_{+}p_{13,+} ( (id \times F)^+ p_1^+ j_{+} \mco_{\mbc^*}^{\beta_0} \otimes g^+ \mco_T^{\beta}  \otimes (id \times F)^+\mcl_1)[-3n-2d-3], \\
&\simeq p_{+}p_{13,+} ( g^+ \mco_T^{\beta} \otimes (id \times F)^+ (p_1^+ j_{+} \mco_{\mbc^*}^{\beta_0} \otimes \mcl_1) )[-2n-d-3], \\
&\simeq p_{+}p_{13,+} ( p_{13}^+q^+ \mco_T^{\beta} \otimes (id \times F)^+ (p_1^+ j_{+} \mco_{\mbc^*}^{\beta_0} \otimes \mcl_1) )[-2n-d-3], &&g = q \circ p_{13} \\
&\simeq p_{+} (q^+ \mco_T^{\beta} \otimes  p_{13,+}(id \times F)^+ (p_1^+ j_{+} \mco_{\mbc^*}^{\beta_0} \otimes \mcl_1) )[-2n-d-2],  &&\text{projection formula}\\
&\simeq p_{+} ( q^+ \mco_T^{\beta} \otimes F^+ p_{2,+} (p_1^+ j_{+} \mco_{\mbc^*}^{\beta_0} \otimes \mcl_1) )[-2n-d-2], &&\text{base change}\\
&\simeq p_{+} ( q^+ \mco_T^{\beta} \otimes  F^+  j_{\dag} \mco_{\mbc^*}^{-\beta_0-1}  )[-2n-d-1], &&\FL(j_+ \mco_{\mbc^*}^{\beta_0}) \simeq j_\dag \mco_{\mbc^*}^{-\beta_0-1}\\
&\simeq p_{+} (  q^\dag \mco_T^{\beta} \otimes F^\dag  j_{\dag} \mco_{\mbc^*}^{-\beta_0-1})[2n+d+1], &&q^+\simeq q^\dag[2n+2],\\
&&& F^+ \simeq F^\dag[2n+2d]
\end{align*}
\end{proof}

The isomorphism $\mcm^{\widetilde{\beta}}_{\widetilde{A}} \simeq \mch^{2n+d+1} (p_+(q^\dag \mco^\beta_T \otimes F^\dag j_{\dag} \mco^{-\beta_0-1}_{\mbc^*} ))$ which holds for $\widetilde{\beta} \in \mbr^{d+1} \setminus sRes(\widetilde{A})$ endows the GKZ-system with the structure of a complex mixed Hodge module. We define the mixed Hodge module structure  by
\begin{equation}\label{eq:defHodgeGKZgen}
{^H\!\!}\mcm^{\widetilde{\beta}}_{\widetilde{A}}:= \mch^{2n+d+1} (p_*( q^* {^p}\mbc_T^\beta \otimes F^* j_{!} {^p}\mbc^{-\beta_0-1}_{\mbc^*} ))
\end{equation}
We now check that the Hodge module structure on $\mcm_{\widetilde{A}}^{\widetilde{\beta}}$ induced by the definition \eqref{eq:defHodgeGKZgen} coincides with the one of Proposition \ref{prop:Radonhalfgeneral} in the case $\beta_0 \in \mbz$.

\begin{proposition}\label{prop:compRadtwRad}
If $\beta_0 \in \mbz$ and $(\beta_0,\beta) \not \in sRes(\widetilde{A)}$ then there is an isomorphism
\[
{\!^*}\mcr^\circ_c(g_*{^p}\mbc^{H,\beta}_T)[n+1] \simeq p_*(q^* {^p}\mbc_T^\beta \otimes F^* j_{!}{^p} \mbc^{H,-\beta_0-1}_{\mbc^*})[2n+d+1]
\]
\end{proposition}
\begin{proof}
Consider the following commutative diagram whose squares are cartesian
\[
\xymatrix{ &\mbc^* \ar[d]^j & U_0  \ar[l]_{\dot{F}} \ar[r] \ar[d]^{k_U} & U \ar[d]^{j_U} \\
& \mbc   & W_0 \times V \ar[d]_{\widetilde{q}} \ar[l]_{\widetilde{F}} \ar[r]^{j_0 \times id} & \mbp(W) \times V \ar[d]^{\pi_1} \ar[r]^-{\pi_2}& V \\
& T\times V \ar[u]^F  \ar[d]_{q}\ar[ur]^{g_0 \times id} & W_0 \ar[r]^{j_0} & \mbp(W) \\
& T \ar[ur]^{g_0} \ar@/_1em/[urr]_g}
\]
where $W_0:= \mbp(W) \setminus\{w_0 \neq 0\} =\mbc^n$ with coordinates $w_1,\ldots,w_n$ and $U_0 = U \cap W_0$. We denote by $p,p_0,\pi_1$  the projections to the first factor and by $\widetilde{F}$ the map $ (w_1,\ldots,w_n,\lambda_0,\lambda_n) \mapsto \lambda_0 + \sum_{i=1}^n\lambda_i w_i$. We consider the coordinate change $\phi_0$ defined by
\[
\widetilde{t}_k = t_k\; \text{for}\; k=1,\ldots,d, \qquad \widetilde{\lambda}_0 = \lambda_0 + \sum_{i=1}^n \lambda_i \underline{t}^{\underline{a}_i} = F  \quad \text{and} \quad  \widetilde{\lambda}_i = \lambda_i \; \text{for}\, i=1,\ldots, n
\]
on $T \times V$ and the coordinate change $\psi_0$ defined by
\[
\widetilde{w}_j = w_j\; \text{for}\; j=1,\ldots,n, \qquad \widetilde{\lambda}_0 = \lambda_0 + \sum_{i=1}^n \lambda_i w_j =\widetilde{F}  \quad \text{and} \quad  \widetilde{\lambda}_i = \lambda_i \; \text{for}\, i=1,\ldots, n
\]
on $W_0 \times V$.
Notice that with respect to these coordinates the maps $F$ and $\widetilde{F}$ are given by the coordinate function $\widetilde{\lambda}_0$. Let $pr: V \ra \mbc$ be the projection $(\widetilde{\lambda}_0,\ldots , \widetilde{\lambda}_n) \mapsto \widetilde{\lambda}_0$. We also have $\psi_0\circ (g_0 \times id)\circ \varphi^{-1}_0 = g_0 \times id$ and the map $q$ factors as $\pi_2 \circ j_0 \circ (g\times id)$.  Hence we get
\begin{align*}
&\quad p_*(q^* {^p}\mbc_T^{H,\beta} \otimes  F^* j_{!} {^p}\mbc^{H,-\beta_0-1}_{\mbc^*} )[2n+d+1] \\
& \simeq p_*( q^* {^p}\mbc_T^{H,\beta} \otimes F^* j_{!} {^p}\mbc^H_{\mbc^*})[2n+d+1] && \text{use}\, \beta_0 \in \mbz\\
&\simeq p_*( {^p}\mbc_T^{H,\beta} \boxtimes pr^*j_!{^p}\mbc^H_{\mbc^*})[n+1] && F \text{ is a projection after coordinate change}\\
&\simeq (\pi_2 \circ (j_0\times id) \circ (g_0\times id))_*( {^p}\mbc_T^{H,\beta} \boxtimes pr^*j_!{^p}\mbc^H_{\mbc^*})[n+1] && p= \pi_2 \circ (j_0 \times id) \circ (g_0 \times id)\\
&\simeq (\pi_2 \circ (j_0\times id) )_*( g_{0*}{^p}\mbc_T^{H,\beta} \boxtimes pr^*j_!{^p}\mbc^H_{\mbc^*})[n+1] \\
&\simeq (\pi_2 \circ (j_0\times id) )_*( \widetilde{q}^* g_{0*}{^p}\mbc_T^\beta \otimes \widetilde{F}^*j_!{^p}\mbc^H_{\mbc^*})[n+1] && \widetilde{F} \text{ is a projection after coordinate change}\\
&\simeq (\pi_2 \circ (j_0 \times id))_*(\widetilde{q}^* g_{0*}{^p}\mbc_T^{H,\beta}  \otimes k_{U!}\dot{F}^* {^p}\mbc^H_{\mbc^*} )[n+1] && \text{base change,  } \dot{F} \text{ and } \widetilde{F} \text{ smooth}\\
&\simeq (\pi_2 \circ (j_0 \times id))_*(\widetilde{q}^* g_{0*}{^p}\mbc_T^{H,\beta}  \otimes  k_{U!} {^p}\mbc^H_{U_0})[n+1] \\
&\simeq (\pi_2 \circ (j_0 \times id))_*((j_0 \times id)^!\pi_1^* g_*\, {^p}\mbc^{H,\beta}_T \otimes k_{U!} {^p}\mbc^H_{U_0} )[n+1] && g= j_0 \circ g_0\\
&\simeq \pi_{2*} (\pi_1^* g_*\, {^p}\mbc^{H,\beta}_T \otimes (j_0 \times id)_* k_{U!} {^p}\mbc^H_{U_0} )[n+1] && \text{projection formula}\\
&\overset{(*)}\simeq \pi_{2*} (\pi_1^* g_*\, {^p}\mbc^{H,\beta}_T \otimes j_{U!} {^p}\mbc^H_{U} )[n+1] \\
&\simeq \pi_{2*} j_{U!}j_U^* \pi_1^*g_* {^p}\mbc^{H,\beta}_T[n+1] && \text{projection formula}\\
&\simeq {\!^*}\mcr^\circ_c(g_*{^p}\mbc^{H,\beta}_T)[n+1]
\end{align*}
where the isomorphism $(*)$ follows from the fact that $\pi_1^*g_*  {^p}\mbc^\beta_T$ is localized along the divisor $\mbp(W)\setminus W_0$.
\end{proof}

\subsection{Calculation in charts}\label{subsec:CalculationInCharts}

We saw Subsection \ref{subsec:TwistedIntTrafo} that the GKZ-system $\mcm^{\widetilde{\beta}}_{\widetilde{A}}$ can be expressed by an integral transformation from $T$ to $V$ with kernel $F^* j_! ^{p} \mbc_{\mbc^*}^{-\beta_0-1}$. As a first step we compute the Hodge filtration of an intermediate step in this integral transformation, namely $q^*\, {^p}\mbc^{H,\beta^u}_T \otimes F_u^* j_{!} {^p}\mbc^{H,-\beta_0-1}_{\mbc^*}[2n+d+1]$ (cf. Lemma \ref{lem:compFuCT}). We do this by giving different presentations of the kernel (cf. Lemma \ref{lem:kerneldiffpres}) and by using several adapted coordinate systems on $T \times V$ indexed by a variable $u$ which goes from $0$ to $n$. The reason for using these adapted coordinate systems (and not just one) is the fact that we can rewrite the intermediate step as a direct product whose factors are easy to compute. Since the projection $p:T \times V \ra V$ is not proper,  we have to extend the intermediate step to a partial compactification. Concretely we are using the factorization
\[
\xymatrix{T \times V \ar[r]^-{g \times id} & \mbp(W) \times V \ar[r]^-{\pi_2} & V}
\]
of the projection $q$.
Our goal in this section is to compute the underlying $\mcd$-module $\mcn$ of
\begin{equation}\label{eq:SheafN}
{^H\!}\mcn := \mch^{2n+d+1}( g \times id)_* (q^* {^p}\mbc_T^{H,\beta} \otimes F^* j_{!} {^p}\mbc^{H,-\beta_0-1}_{\mbc^*})
\end{equation}

 together with its Hodge filtration on affine charts $W_u$ of $\mbp(W) \times V$. The different adapted coordinate systems are now used to compute the direct image of the intermediate step under the embedding $T \times V \ra W_u \times V$ which is simply the restriction of ${^H}\mcn$ to the affine chart $W_u$. It turns out that the underlying $\mcd$-module of this direct image is a direct product of a torus embedding with respect to a matrix $A_u$ and another rather simple module (cf. equation \eqref{eq:presWu}). We use Theorem \ref{thm:hpot} to compute the Hodge filtration on the first factor in Proposition \ref{prop:Mbetau}, the Hodge filtration on the second factor was computed in Remark \ref{rem:pdi1} .\\

We define the map
\begin{align*}
F_u : T \times V &\lra \mbc \\
(t_1,\ldots, t_d, \lambda_0, \ldots, \lambda_n) &\mapsto \lambda_u + \sum_{i=0 \atop i \neq u}\lambda_i \underline{t}^{\underline{a}_i - \underline{a}_u} = \left(\lambda_0 + \sum_{i=1}\lambda_i \underline{t}^{\underline{a}_i}\right) \cdot \underline{t}^{-\underline{a}_u}
\end{align*}
(notice that $F_0 = F$). We need the following result
\begin{lemma}\label{lem:kerneldiffpres} There is an isomorphism
\[
F^* j_{!} {^p}\mbc_{\mbc^*}^{H,-\beta_0-1} \simeq F_u^* j_{!} {^p}\mbc_{\mbc^*}^{H,-\beta_0-1}
\]
for $u = 0, \ldots,n$.
\end{lemma}
\begin{proof}
For $u \in \{0,\ldots,n\}$ and $G := (\mbc^*)^d$ consider the action
\begin{align*}
\mu_u: G \times T \times V &\lra T \times V \\
(g_1,\ldots,g_d,t_1,\ldots,t_d,\lambda_0,\ldots,\lambda_n) &\mapsto (t_1,\ldots,t_d, \lambda_0 \underline{g}^{-\underline{a}_u},\ldots,\lambda_n \underline{g}^{-\underline{a}_u})
\end{align*}
and the action $G \times \mbc \ra \mbc$ given by $(g_1,\ldots,g_d,t) \mapsto \underline{g}^{-\underline{a}_u} \cdot t$. It is easy to see that the map $F: T \times V \ra \mbc$ is equivariant with respect to this action. Let $i: T  \ra G \times T $ the embedding $(t_1,\ldots,t_d) \mapsto (t_1,\ldots,t_d,t_1,\ldots,t_d)$. Since $j_{!}{^p}\mbc^{H,-\beta_0-1}_{\mbc^*}$ is a $G$-equivariant mixed Hodge module, the module $F^! j_{!}{^p}\mbc_{\mbc^*}^{H;-\beta_0-1}$ is also $G$-equivariant. Let $p: G\times T \times V \ra T \times V$ the projection. We have isomorphisms
\[
F^* j_{!}{^p}\mbc_{\mbc^*}^{H,-\beta_0-1} \simeq i^* p_2^* F^* j_{!}{^p}\mbc_{\mbc^*}^{H,-\beta_0-1} \simeq i^* \mu_u^* F^* j_{!}{^p}\mbc_{\mbc^*}^{H,-\beta_0-1} \simeq F_u^* j_{!}{^p}\mbc_{\mbc^*}^{H,-\beta_0-1}
\]
where the second isomorphism follows from the $G$-equivariance of  $F^* j_{!}{^p}\mbc_T^{H,-\beta_0-1}$.
\end{proof}

We define a coordinate change $\phi_u$ by
\begin{align*}
\widetilde{t}_k = t_k,\quad (\widetilde{\lambda}_i)_{i \neq u} = (\lambda_i)_{i \neq u}\quad \text{and} \quad \widetilde{\lambda}_u = \lambda_u + \sum_{i \neq u} \lambda_i \underline{t}^{\underline{a}_i - \underline{a}_u}
\end{align*}

Denote by $C_{u} \in GL(d+1,\mbz)$ the matrix
\[
C_u := \left(\begin{matrix} 1 & \\ -a_{1u} & 1 \\ \vdots & & \ddots \\ -a_{du} & & &1\end{matrix}\right)\, .
\]
and define for $\widetilde{\beta} =(\beta_0,\beta) \in \mbz^{d+1}$:
\[
\widetilde{\beta}^u := (\beta^u_0,\beta^u) := C_u \cdot \widetilde{\beta}
\]
Notice that $\widetilde{\beta}^0 = \widetilde{\beta}$ since $\underline{a}_0 = 0$.
\begin{lemma}\label{lem:compFuCT}
\begin{enumerate} $ $\\[-1em]
\item
With respect to the coordinates defined by $\phi_u$ the complex $q^*\, {^p}\mbc^{H,\beta^u}_T \otimes F_u^* j_{!} {^p}\mbc^{H,-\beta_0-1}_{\mbc^*}[2n+d+1]$ is isomorphic to
\[
{^p}\mbc^{H,\beta^u}_T \boxtimes pr_u^*j_!{^p}\mbc_{\mbc^*}^{H,-\beta_0-1}[n]
\]
where $pr_u:V \ra \mbc$ is the projection $(\widetilde{\lambda}_0,\ldots, \widetilde{\lambda}_n) \mapsto \widetilde{\lambda}_u$. In particular we have
\[
\mch^{k}\left(q^*\, {^p}\mbc^{H,\beta^u}_T \otimes F_u^* j_{!} {^p}\mbc^{H,-\beta_0-1}_{\mbc^*} \right) = 0 \quad \text{for} \; k \neq 2n+d+1
\]
and the underlying $\mcd$-module of $\mch^{2n+d+1}\left(q^*\, {^p}\mbc^{H,\beta^u}_T \otimes F_u^* j_{!} {^p}\mbc^{H,-\beta_0-1}_{\mbc^*} \right)$ is given by the exterior product
\[
\mcd_T/ \left( (\p_{\widetilde{t}_k} \widetilde{t}_k + \beta^u_k)_{k=1,\ldots,d} \right) \boxtimes \mcd_V/\left( (\p_{\widetilde{\lambda}_i})_{i \neq u}, \widetilde{\lambda}_u \p_{\widetilde{\lambda}_u}  - \beta_0\right) \, .
\]
\item For $\alpha \in \mbz^d$ and $u_1,u_2 \in \{0,\ldots,n\}$ the map
\[
\mch^{2n+d+1}\left(q^*\, {^p}\mbc^{H,\beta^{u_1}}_T \otimes F_{u_1}^* j_{!} {^p}\mbc^{H,-\beta_0-1}_{\mbc^*} \right) \lra \mch^{2n+d+1}\left(q^*\, {^p}\mbc^{H,\beta^{u_2}+\alpha}_T \otimes F_{u_2}^* j_{!} {^p}\mbc^{H,-\beta_0-1}_{\mbc^*} \right)\, ,
\]
given by right multiplication with $\underline{t}^\alpha$ at the level of $\mcd_{V \times  T}$-modules, is an isomorphism.
 \end{enumerate}
\end{lemma}
\begin{proof}
Notice that the map $F_u$ is just the projection $((\widetilde{t}_k)_{k=1,\ldots,s}, (\widetilde{\lambda}_i)_{i \neq u}, \widetilde{\lambda}_u) \mapsto \widetilde{\lambda}_u$ with respect to the new coordinates. This gives
\[
q^*\, {^p}\mbc^{H,\beta^u}_T \otimes F_u^* j_{!} {^p}\mbc^{H,-\beta_0-1}_{\mbc^*}[2n+d+1] \simeq {^p}\mbc^{H,\beta^u}_T \boxtimes pr_u^*j_!{^p}\mbc_{\mbc^*}^{H,-\beta_0-1}[n]
\]
(the shifts can be seen by noticing that $q^*[n+1], F^*[n+d]$ and $pr_u^*[n]$ are exact. The rest is clear.\\
For the second point we define coordinates
Let $((\widetilde{t}_k)_{k=1,\ldots,d}, (\widetilde{\lambda}_i)_{i=0,\ldots,n})$ and $((\bar{t}_k)_{k=1,\ldots,d}, (\bar{\lambda}_i)_{i=0,\ldots,n})$ correspond to the maps $\phi_{u_1}$ and $\phi_{u_2}$, respectively. The coordinate change $\phi_{u_2} \circ \phi_{u_1}^{-1}$ is given by
\[
\bar{t}_k = \widetilde{t}_k, \quad  \qquad \bar{\lambda}_{u_1} =   \widetilde{\lambda}_{u_1} - \sum_{i \neq u_1} \widetilde{\lambda}_i \widetilde{\underline{t}}^{\,\underline{a}_i - \underline{a}_{u_1}}, \qquad \bar{\lambda}_{u_2} = \widetilde{\lambda}_{u_1} \widetilde{\underline{t}}^{\,\underline{a}_{u_1}- \underline{a}_{u_2}}, \quad \bar{\lambda}_i = \widetilde{\lambda}_i
\]
for  $k=1,\ldots, d$ and $i \neq u_1,u_2$. We get the following transformations:
\begin{align*}
\p_{\widetilde{\lambda}_i} \mapsto & \;\p_{\bar{\lambda}_i} - \bar{\underline{t}}^{\,\underline{a}_i - \underline{a}_{u_1}} \p_{\bar{\lambda}_{u_1}} \equiv \p_{\bar{\lambda}_i} \\
\widetilde{\lambda}_{u_1} \p_{\widetilde{\lambda}_{u_1}} -\beta_0  \mapsto & \; \bar{\lambda}_{u_2} \bar{\underline{t}}^{\,\underline{a}_{u_2} - \underline{a}_{u_1}}\p_{\bar{\lambda}_{u_1}} +\bar{\lambda}_{u_2} \p_{\bar{\lambda}_{u_2}} -\beta_0 \equiv \bar{\lambda}_{u_2} \p_{\bar{\lambda}_{u_2}} -\beta_0\\
\p_{\widetilde{\lambda}_{u_2}}\mapsto & \; - \bar{\underline{t}}^{\underline{a}_{u_2} - \underline{a}_{u_1}} \p_{\bar{\lambda}_{u_1}} \\
 \p_{\widetilde{t}_k} \widetilde{t}_k + \beta_k^{u_1} \mapsto  & \; \p_{\bar{t}_k} \bar{t}_k -\sum_{i \neq u_1} (a_{ki} -a_{k u_1}) \widetilde{\lambda}_i \bar{\underline{t}}^{\, \underline{a}_i - \underline{a}_{u_1}}\p_{\bar{\lambda}_{u_1}} + (a_{k u_1}  - a_{ku_2})\bar{\lambda}_{u_2}\p_{\bar{\lambda}_{u_2}} + \beta_k^{u_1} \\
 &  \equiv \p_{\bar{t}_k} \bar{t}_k + (a_{k u_1}  - a_{ku_2})\bar{\lambda}_{u_2}\p_{\bar{\lambda}_{u_2}} + \beta_k^{u_1} \\
  &  \equiv \p_{\bar{t}_k} \bar{t}_k + (a_{k u_1}  - a_{ku_2})\beta_0 + \beta_k^{u_1}\\
& = \p_{\bar{t}_k} \bar{t}_k + \beta_k^{u_2}
\end{align*}
where $\equiv$ means equality modulo the ideal generated by the operators on the left hand side. This shows that
\[
\mcd_V/\left( (\p_{\widetilde{\lambda}_i})_{i \neq u_1}, \widetilde{\lambda}_{u_1} \p_{\widetilde{\lambda}_{u_1}}  - \beta_0\right) \boxtimes \mcd_T/ \left( (\p_{\widetilde{t}_k} \widetilde{t}_k + \beta^{u_1}_k)_{k=1,\ldots,d} \right)\, .
\]
is actually  equal to
\[
\mcd_V/\left( (\p_{\bar{\lambda}_i})_{i \neq u_2}, \bar{\lambda}_{u_2} \p_{\bar{\lambda}_{u_2}}  - \beta_0\right) \boxtimes \mcd_T/ \left( (\p_{\bar{t}_k} \bar{t}_k + \beta^{u_2}_k)_{k=1,\ldots,d} \right)\, .
\]
after the change of coordinates $\phi_{u_2} \circ \phi^{-1}_{u_1}$.
It is then easy to see that the map
\[
\mcd_{V\times T}/\!\left( (\p_{\widetilde{\lambda}_i})_{i \neq u_2}, \widetilde{\lambda}_{u_2} \p_{\widetilde{\lambda}_{u_2}}\!\!\!  - \beta_0, (\p_{\widetilde{t}_k} \widetilde{t}_k + \beta^{u_2}_k)_{k=1,\ldots,d} \right) \ra
\mcd_V/\!\left( (\p_{\bar{\lambda}_i})_{i \neq u_2}, \bar{\lambda}_{u_2} \p_{\bar{\lambda}_{u_2}}\!\!\!  - \beta_0, (\p_{\bar{t}_k} \bar{t}_k + \beta^{u_2}_k+\alpha_k)_{k=1,\ldots,d} \right)\, .
\]
is given by right multiplication with $\underline{\bar{t}}^{\,\alpha}$. This shows the second claim.
\end{proof}

Let $(w_0: \ldots : w_n)$ be the homogeneous coordinates on $\mbp(W)$ and denote by $j_u : W_u \hookrightarrow \mbp(W)$ the chart $w_u \neq 0$ with coordinates $w_{iu} := \frac{w_i}{w_u}$ for $i \neq u$. The map $g$ factors over the chart $W_u$ and gives rise to the map
\begin{align}
g_u : T &\lra W_u \notag \\
(t_1, \ldots , t_n) & \mapsto (\underline{t}^{\underline{a}_0 - \underline{a}_u}, \ldots , \underline{t}^{\underline{a}_n - \underline{a}_u})\, . \notag
\end{align}

We define the maps
\begin{align*}
\widetilde{F}_u: W_u \times V &\lra \mbc \\
(w_{iu})_{i \neq u} &\mapsto \lambda_u + \sum_{i=0 \atop i \neq u}^n \lambda_i w_{iu}
\end{align*}

As mentioned above we would like to compute the restriction of $\mcn$ to the affine chart $W_u \times V$. For $u = 0,\ldots,n$ we set
\begin{align}
{^H\!}\mcn_u := {^H\!}\mcn_{\mid W_u \times V}
& \simeq \mch^{2n+d+1}(g_{u}\times id)_* (q^*\, {^p}\mbc^{H,\beta^u}_T \otimes F_u^* j_{!} {^p}\mbc^{H,-\beta_0-1}_{\mbc^*} ) \notag\\
&\simeq \mch^n (g_u \times id)_* \left({^p}\mbc^{H,\beta^u}_T \boxtimes pr_u^*j_!{^p}\mbc_{\mbc^*}^{H,-\beta_0-1}\right) \notag  \\
&\simeq \mch^n  \left({^p}g_{u*}\mbc^{H,\beta^u}_T \boxtimes pr_u^*j_!{^p}\mbc_{\mbc^*}^{H,-\beta_0-1}\right) \notag \\
&\simeq  \mch^0 ({^p}g_{u*}\mbc^{H,\beta^u}_T) \boxtimes \mch^ n( pr_u^*j_!{^p}\mbc_{\mbc^*}^{H,-\beta_0-1}) \label{eq:presWu}
\end{align}

We now apply the main result of section \ref{sec:HodgeOnTorusEmbed} in order to compute the module $g_{u +} \mco_T^{\beta^u}$ together with its corresponding Hodge filtration.

Notice that the embedding $g_u$ is given by the $d \times n$-matrix  $A_u=(a^u_{ki})$ with columns $(\underline{a}_i-\underline{a}_u)$ for $i \in \{0, \ldots , n\} \setminus \{u\}$. We need to check whether the matrices $A_u$ satisfy the conditions in Theorem \ref{thm:hpot}. Recall that  $\widetilde{\beta}^u := (\beta_0^u,\beta^u):= C_u \cdot  \widetilde{\beta}$.

\begin{lemma}\label{lem:MatrixAu}
Assume that  $\mbz \widetilde{A} = \mbz^{d+1}$ and $\mbn\widetilde{A} = \mbz^{d+1} \cap \mbr_{\geq 0} \widetilde{A}$,
then the matrices $A_u$ satisfy the conditions,
\begin{enumerate}
\item $\mbz A_u = \mbz^d$
\item $\mbn A_u = \mbz^d \cap \mbr_{\geq 0}A_u$
\item if $\widetilde{\beta} \in \mathfrak{A}_{\widetilde{A}}$ then $\beta^u \in \mathfrak{A}_{A_u}$
\end{enumerate}
\end{lemma}
\begin{proof}
Denote by $\widetilde{A}_u$ the $(d+1) \times (n+1)$-matrix with columns $(1,\underline{a}_i - \underline{a}_u)$ for $ i \in \{0, \ldots , n\}$.    We will first show the two properties for the matrix $\widetilde{A}_u$.
Notice that we have $C_u \cdot \widetilde{A} = C_u \cdot \widetilde{A}_0 = \widetilde{A}_u$.
Since $C_u$ is a linear, invertible map
we get $C_u (\mbz \widetilde{A}) = \widetilde{A}_u$, $C_u(\mbn \widetilde{A}) =\widetilde{A}_u$ and $C_u(\mbr_{\geq 0} \widetilde{A}) = \mbr_{\geq 0} \widetilde{A}_u$. Therefore the two properties hold for $\widetilde{A}_u$ if and only if they hold for $\widetilde{A}$. \\

Denote by $p: \mbz^{d+1} \ra \mbz^d$ the projection to the last $d$-coordinates. Since $p$ maps $(1, \underline{a}_i - \underline{a}_u)$ to $\underline{a}_i - \underline{a}_u$ it is easy to see that the first two properties also hold for $A_u$.\\

It follows easily from the definition that $\widetilde{\beta} \in \mathfrak{A}_{\widetilde{A}}$ if and only if $\widetilde{\beta}^u \in \mathfrak{A}_{\widetilde{A}_u}$. Hence it is enough to show that $\widetilde{\beta}=(\beta_0,\beta) \in \mathfrak{A}_{\widetilde{A}}$ implies $\beta \in \mathfrak{A}_A$. We notice first that there is a 1-1 correspondence between facets of $\mbr_{\geq 0} A$ and facets of $\mbr_{\geq 0} \widetilde{A}$ containing $\widetilde{\underline{a}}_0 =(1,0,\ldots,0)$ given by
\[
F  \leftrightarrow \widetilde{F}= F + \mbr_{\geq 0} \cdot (1,0,\ldots,0)
\]
If $n_F$ is a primitive, inward-pointing normal vector of a facet $F$ of $\mbr_{\geq 0}A$, the vector $n_{\widetilde{F}} := (0,n_F)$ is a primitive, inward-pointing normal vector of the corresponding facet $\widetilde{F}$ of $\mbr_{\geq 0} \widetilde{A}$. Since $\widetilde{\underline{c}}:= \sum_{i=0}^n \widetilde{\underline{a}}_i = (n+1, \underline{c})$, we have $e_{\widetilde{F}} = \langle n_{\widetilde{F}}, \widetilde{\underline{c}} \rangle = \langle (0, n_F), (n+1,\underline{c}\rangle = \langle n_F, \underline{c} \rangle = e_F$. We get by definition \ref{eq:defAdm}
\[
\widetilde{\beta} \in \mathfrak{A}_{\widetilde{A}} =  \bigcap_{\widetilde{F}\, \text{facet}} \{\mbr \cdot \widetilde{F} - [0,e_{\widetilde{F}}) \cdot \widetilde{\underline{c}}\} \subset \bigcap_{\widetilde{F}\, \text{facet}\atop \widetilde{\underline{a}}_0 \in \widetilde{F}} \{\mbr \cdot \widetilde{F} - [0,e_{\widetilde{F}}) \cdot \widetilde{\underline{c}}\} \Rightarrow \beta \in \bigcap_{F \,\text{facet}} \{\mbr \cdot F - [0,e_F) \cdot \underline{c} \} = \mathfrak{A}_A\, .
\]

\end{proof}

Denote by $\mbl_{A_u}$ the $\mbz$-module of relations among the columns of $A_u$. In order to calculate the direct image of $\mco_T^\beta$ under the map $g_u$,
we use Theorem \ref{thm:hpot} where $A_u$ takes the role of the matrix $B$ in loc.cit.

\begin{proposition}\label{prop:Mbetau}
Consider the $\mcd_{W_u}$-module $\check{\mcm}^{\beta^u}_{A_u}$ as defined in Definition \ref{def:SWGarbe}, that is,
$\check{\mcm}^{\beta^u}_{A_u}=\mcd_{W_u} / \check{\mci}^{\beta^u}_{A_u}$ where the left ideal $\check{\mci}^{\beta^u}_{A_u}$ is generated by
\[
\check{\Box}_{\underline{m} \in \mbl_{A_u}} = \prod_{i\neq u:m_i >0} w_{iu}^{m_i} - \prod_{i\neq u: m_i <0} w_{iu}^{-m_i}
\]
and the Euler vector fields:
\begin{align}
\check{E}^u_k+\beta_k^u &= \sum_{i\neq u} a_{ki}^u \p_{w_{iu}} w_{iu} + \beta_k^u \notag \\
&= \sum_{i \neq u} (a_{ki} - a_{ku})  \p_{w_{iu}} w_{iu}+ \beta_k^u\, . \notag
\end{align}
Then the direct image $g_{u+} \mco^{\beta_u}_T$ is isomorphic to $\check{\mcm}^{\beta^u}_{A_u}$. Moreover, the Hodge filtration on $\check{\mcm}^{\beta^u}_{A_u}$ is the order filtration shifted by $(n-d)$, i.e.
\[
F^H_{p+(n-d)} \check{\mcm}^{\beta^u}_{A_u} = F^{ord}_p \check{\mcm}^{\beta^u}_{A_u}\, .
\]
\end{proposition}
\begin{proof}
The statement follows from Theorem \ref{thm:hpot} and Lemma \ref{lem:MatrixAu}.

\end{proof}
We now want to compute how the $\mcd$-modules $g_{u+}\mco_T^{\beta^u}$ glue on their common domain of definition. Let $u_1, u_2 \in \{0, \ldots , n\}$ and denote by $W_{u_1 u_2}$ the intersection $W_{u_1} \cap W_{u_2}$. We fix $u_1, u_2 \in \{0, \ldots , n\}$ with $u_1 < u_2$. We have the following change of coordinates between the charts $W_{u_1}$ and $W_{u_2}$
\[
w_{i u_1} = w_{i u_2} w_{u_1 u_2}^{-1} \quad \text{for} \; i \neq u_2 \quad \text{and} \quad w_{u_2 u_1} = w_{u_1 u_2}^{-1}
\]
which gives the following transformation rules for vector fields:
\begin{equation}\label{eq:transfcoord}
w_{i u_1} \p_{w_{i u_1}}=  w_{i u_2} \p_{w_{i u_2}} \quad \text{for}\ i \neq u_2 \quad w_{u_2 u_1} \p_{w_{u_2 u_1}} = - \sum_{i \neq u_2} w_{i u_2} \p_{w_{i u_2}}\, .
\end{equation}

These transformation rules define an algebra isomorphism
$$
\iota_{u_1u_2}: D_{W_{u_1}}[w_{u_2 u_1}^{-1}] \longrightarrow D_{W_{u_2}}[w_{u_1 u_2}^{-1}].
$$

The module of global sections $\Gamma(W_{u_1u_2}, g_{u_1 +} \mco_T^{\beta^{u_1}})$
can be expressed as the quotient $ D_{W_{u_1}}[w_{u_2u_1}^{-1}]/ \dot{I}^{\beta^{u_1}}_{A_{u_1}}$, where $ \dot{I}^{\beta^{u_1}}_{A_{u_1}} \subset D_{W_{u_1}}[w_{u_2u_1}^{-1}] := \mbc[(w_{i u_1})_{i \neq u_1}][w_{u_2 u_1}^{-1}] \otimes_{\mbc[(w_{i u_1})_{i \neq u_1}]} D_{W_{u_1}}$ is the left ideal generated by
\begin{align}
&1.\quad \check{E}^{u_1}_k +\beta_k^{u_1} =\sum_{i=0 \atop i \neq u_1}^{n} a_{ki}^{u_1}  \p_{w_{iu_1}}w_{iu_1} +\beta_k^{u_1} = \sum_{i=0 \atop i \neq u_1}^{n} a_{ki}^{u_1} w_{iu_1} \p_{w_{iu_1}}+ \sum_{i=0 \atop i \neq u_1}^{n} a_{ki}^{u_1} +\beta_k^{u_1}  && k= 1, \ldots ,d \notag \\
&2.\quad \check{\Box}_{\underline{m}} = \prod_{m_i >0 \atop i \neq u_1 } w_{iu_1}^{m_i} - \prod_{m_i <0 \atop i \neq u_1 } w_{iu_1}^{-m_i} &&\underline{m} \in \mbl_{A_{u_1}}\, . \notag
\end{align}

Let $\gamma^{u_1} := \sum_{i \neq u_1} a_{ki}^{u_1}$, then
$\dot{I}_{A_{u_1}}^{\beta^{u_1}-\gamma^{u-1}} \subset D_{W_{u_1}}[w_{u_2u_1}^{-1}]$ is the left ideal generated by
\begin{align}
&1.\sum_{i=0 \atop i \neq u_1}^{n} a_{ki}^{u_1} w_{iu_1} \p_{w_{iu_1}} + \beta^{u_1}_k  && k= 1, \ldots ,d \notag \\
&2.\quad \Box_l = \prod_{l_i >0 \atop i \neq u_1 } w_{iu_1}^{l_i} - \prod_{l_i <0 \atop i \neq u_1 } w_{iu_1}^{-l_i} &&l \in \mbl_{A_{u_1}} \notag
\end{align}
We get the following isomorphism of $D_{W_{u_1}}[w_{u_2u_1}^{-1}]$-modules
\begin{align}
D_{W_{u_1}}[w_{u_2u_1}^{-1}] / \dot{I}^{\beta^{u_1}}_{A_{u_1}} &\lra D_{W_{u_1}}[w_{u_2u_1}^{-1}] /\dot{I}^{\beta^{u_1}-\gamma^{u_1}}_{A_{u_1}} \label{eq:opchange} \\
1 &\mapsto \prod_{i \neq u_1} w_{i u_1}^{-1} \notag
\end{align}
which is the image of the isomorphism
\begin{align*}
\mco_T^{\gamma^{u_1} + \beta_k^{u_1} -1} &\lra \mco_T^{  \beta_k^{u_1} -1} \\
1 &\mapsto \underline{t}^{-\gamma^{u_1}} = \underline{t}^{-\sum_{i \neq u_1}(\underline{a}_i - \underline{a}_{u_1})}
\end{align*}
under the functor $g_{u_1 +}$ (cf. equation \eqref{eq:tTransform}).
One obtains the same results for the chart $W_{u_2}$ by exchanging $u_1$ and $u_2$ above.
Using the transformation rules \eqref{eq:transfcoord}, we can identify $D_{W_{u_1}}[w_{u_2u_1}^{-1}]$ with $D_{W_{u_2}}[w_{u_1u_2}^{-1}]$, which gives a well-defined map
\begin{align}
\iota_{u_1 u_2}: D_{W_{u_1}}[w_{u_2u_1}^{-1}]/\dot{I}_{A_{u_1}}^{\beta^{u_1}-\gamma^{u_1}} &\lra D_{W_{u_2}}[w_{u_1u_2}^{-1}]/\dot{I}_{A_{u_2}}^{\beta^{u_2}-\gamma^{u_2}} \notag.
\end{align}
We can now give an explicit expression for the gluing map between the various charts of the module $g_{+}\mco^\beta_T$.
\begin{lemma}\label{lem:gluemapSW}
The isomorphism between $g_{u_1 + } \mco^{\beta^{u_1}}_T$ and $g_{u_2 +} \mco^{\beta^{u_2}}_T$ on their common domain of definition $W_{u_1 u_2} = W_{u_1} \cap W_{u_2}$ is given by
\begin{align}
\Gamma(W_{u_1 u_2}, g_{u_1+} \mco_T) \simeq
D_{W_{u_1}}[w_{u_2 u_1}^{-1}]/\dot{I}^{\beta^{u_1}}_{A_{u_1}} & \longrightarrow D_{W_{u_2}}[w_{u_1 u_2}^{-1}]/\dot{I}^{\beta^{u_2}}_{A_{u_2}} \notag
\simeq \Gamma(W_{u_1 u_2}, g_{u_2 + } \mco_T) \notag \\
P &\longmapsto \iota_{u_1 u_2}(P) w_{u_1 u_2}^{n+1}\, . \notag
\end{align}
\end{lemma}
\begin{proof}
This follows easily from the discussion above by concatenating the three maps
\[
D_{W_{u_1}}[w_{u_2u_1}^{-1}] / \dot{I}_{A_{u_1}}^{\beta^{u_1}} \lra D_{W_{u_1}}[w_{u_2u_1}^{-1}] /\dot{I}_{A_{u_1}}^{\beta^{u_1}-\gamma^{u_1}} \overset{\iota}\lra D_{W_{u_2}}[w_{u_1u_2}^{-1}]/\dot{I}_{A_{u_2}}^{\beta^{u_2}-\gamma^{u_2}} \lra D_{W_{u_2}}[w_{u_1u_2}^{-1}] / \dot{I}_{A_{u_2}}^{\beta^{u_2}}
\]
and by using the simple computation
\[
\left(\prod_{i \neq u_2} w_{i u_2}\right) \cdot \iota\left(\prod_{i \neq u_1} w_{i u_1}^{-1}\right) =w_{u_1 u_2}^{n+1}
\]
\end{proof}

Consider the following change of coordinates $\theta_u$ on $W_u \times V$.
\begin{equation}\label{eq:coordchange}
\widetilde{\lambda}_u = \lambda_u + \sum_{j=0 \atop j \neq u}^{n} \lambda_j w_{ju} \quad , \quad \widetilde{\lambda}_i = \lambda_i \quad \text{and} \quad \widetilde{w}_{iu} = w_{iu}
\end{equation}
for $i=0, \ldots , n$ and $i \neq u$. Notice that $\theta^{-1}_u \circ (g_u \times id)\circ \phi_u^ = g_u \times id$ and $\widetilde{F}_u$ is just the projection $((\widetilde{w}_{iu})_{i \neq u}, \widetilde{\lambda}_0,\ldots,\widetilde{\lambda}_n) \mapsto \widetilde{\lambda}_u$.

\begin{proposition}\label{prop:OriginalCoord}
Consider the original coordinates $\left((w_{iu})_{i\neq u}, (\lambda_0,\ldots,\lambda_n)\right)$ of $W_u\times V$.
Then there is an isomorphism of $\mcd_{W_u\times V}$-modules
$\mcn_u \simeq \mcd_{W_u\times V}/\mck_{A_u}^{\widetilde{\beta}^u}$, where $\mck_{A_u}^{\widetilde{\beta}^u}$ is the left $\mcd_{W_u\times V}$-ideal generated by the following classes of
operators
\begin{align}
&1.\quad \sum_{i=0 \atop i \neq u}^{n} a_{ki}^u  \p_{w_{iu}}w_{iu} -\sum_{i=1}^{n}a_{ki} \lambda_i \p_{\lambda_i} + \beta_k \notag \\
&2.\quad \check{\Box}_{\underline{m}} = \prod_{m_i >0 \atop i \neq u } w_{iu}^{m_i} - \prod_{m_i <0 \atop i \neq u } w_{iu}^{-m_i} &&\underline{m} \in \mbl_{A_u} \notag \\
&3. \quad \p_{\lambda_i} - w_{iu} \p_{\lambda_u} &&\text{for} \quad i = 0, \ldots ,n\; \text{and}\; i \neq u \notag \\
&4. \quad \sum_{j=0}^n \lambda_j\p_{\lambda_j}- \beta_0\, . \notag
\end{align}
Moreover, for $\widetilde{\beta}=(\beta_0,\beta) \in \mathfrak{A}_{\widetilde{A}}$ and $\beta_0 \in (-1,0]$ we have
$$
F^H_{p+(n-d)} \, \mcn_u \simeq F^{ord}_p \, \mcd_{W_u\times V}/\mck^{\widetilde{\beta}_u}_{A_u}\, .
$$
\end{proposition}
\begin{proof}
Recall that $\mcn_u= \check{\mcm}^{\beta^u}_{A_u} \boxtimes  \mcd_V/\left( (\p_{\widetilde{\lambda}_i})_{i \neq u}, \widetilde{\lambda}_{u} \p_{\widetilde{\lambda}_{u}}  - \beta_0\right) = \mcd_{W_u\times V}/\widetilde{\mck}_{A_u}^{\widetilde{\beta}^u}$,
where
\[
\widetilde{\mck}_{A_u}^{\widetilde{\beta}^u}= \left(
(\check{E}^u_k +\beta^u_k)_{k=1,\ldots ,d}\, ,\;(\check{\Box}_{\underline{m}})_{\underline{m} \in \mbl_{A_u}}\,, \;(\p_{\widetilde{\lambda}_j})_{j \neq u}\, , \;(\widetilde{\lambda}_u \p_{\widetilde{\lambda}_u}-\beta_0) \right).
\]
Using the coordinate transformation \eqref{eq:coordchange} we see that $\widetilde{\mck}_{A_u}^{\widetilde{\beta}_u}$ is transformed into the ideal $\mck_{A_u}^{\widetilde{\beta}}$ generated by
the operators
\begin{align}
&\sum_{i=0 \atop i \neq u}^{n} a_{ki}^u ( \p_{w_{iu}} - \lambda_i \p_{\lambda_u})w_{iu} +\beta^u_k && k= 1, \ldots ,d \notag \\
&\check{\Box}_{\underline{m}} = \prod_{m_i >0 \atop i \neq u } w_{iu}^{m_i} - \prod_{m_i <0 \atop i \neq u } w_{iu}^{-m_i} &&\underline{m} \in \mbl_{A_u} \notag \\
&\p_{\lambda_i} - w_{iu} \p_{\lambda_u} &&\text{for} \quad i = 0, \ldots ,n\; \text{and}\; i \neq u \notag \\
&(\lambda_u + \sum_{j= 0 \atop j \neq u}^{n} \lambda_j w_{ju})\p_{\lambda_u} - \beta_0\, . \notag
\end{align}
The last operator can be rewritten (using the relations $\p_{\lambda_i} - w_{iu} \p_{\lambda_u}$, i.e. the third class of operators) as
\[
\qquad \sum_{j=0}^n \lambda_j \p_{\lambda_j} - \beta_0 \equiv (\lambda_u + \sum_{j= 0 \atop j \neq u}^{n} \lambda_j w_{ju})\p_{\lambda_u}-\beta_0\, .
\]

The Euler-type operators $\sum\limits_{i=0 \atop i \neq u}^{n} a_{ki}^u ( \p_{w_{iu}} - \lambda_i \p_{\lambda_u})w_{iu}$ can be further simplified by writing
\begin{align}
\qquad\qquad\sum_{i=0 \atop i \neq u}^{n} a_{ki}^u ( \p_{w_{iu}} - \lambda_i \p_{\lambda_u})w_{iu} + \beta^u_k &\equiv
\sum_{i=0 \atop i \neq u}^{n} a_{ki}^u ( \p_{w_{iu}}w_{iu} - \lambda_i \p_{\lambda_i})+\beta^u_k \notag \\
\qquad\qquad\qquad\qquad\qquad\qquad&=\sum_{i=0 \atop i \neq u}^{n} a_{ki}^u  \p_{w_{iu}}w_{iu} -\sum_{i=1}^{n}a_{ki} \lambda_i \p_{\lambda_i} + a_{ku} \sum_{i=0}^{n} \lambda_i \p_{\lambda_i} + \beta^u_k \notag \\
 \qquad\qquad\qquad\qquad\qquad\qquad&\equiv \sum_{i=0 \atop i \neq u}^{n} a_{ki}^u  \p_{w_{iu}}w_{iu} -\sum_{i=1}^{n}a_{ki} \lambda_i \p_{\lambda_i} + a_{ku} \beta_0 + \beta^u_k\, , \notag \\
\qquad\qquad\qquad\qquad\qquad\qquad&= \sum_{i=0 \atop i \neq u}^{n} a_{ki}^u  \p_{w_{iu}}w_{iu} -\sum_{i=1}^{n}a_{ki} \lambda_i \p_{\lambda_i} + \beta_k\, , \notag
\end{align}
where the first equivalence follows by using the relation $\sum_{j=0}^n \lambda_j \p_{\lambda_j} \equiv (\lambda_u + \sum_{j= 0 \atop j \neq u}^{n} \lambda_j w_{ju})\p_{\lambda_u}$ from above. Hence we obtain the presentation $\mcn_u \simeq \mcd_{W_u\times V}/\mck_{A_u}^{\widetilde{\beta}}$, and the statement on the Hodge filtration follows directly from  Proposition \ref{prop:Mbetau}.
\end{proof}

\subsection{A Koszul complex}
\label{subsec:Koszul}

In this subsection, we will construct a strict resolution of the filtered module $(\mcn_u, F^H)$. For this purpose,
we first describe an alternative presentation of the ideal $\mck_{A_u}^{\widetilde{\beta}} \subset \mcd_{W_u\times V}$.
Let $A_u^s$ be the $(d+1)\times(2n+1)$-matrix with columns $ (0,\underline{a}_0 - \underline{a}_u), \ldots , \widehat{(0,\underline{a}_u - \underline{a}_u)},\ldots , (0, \underline{a}_n - \underline{a}_u),(1, \underline{a}_0), \ldots,(1, \underline{a}_n)$ (here the symbol $\widehat{\quad}$ means that the zero column $(0,\underline{a}_u - \underline{a}_u)$ is omitted).  In other words, we have

\vspace*{0.2cm}

$$
A^s_u=\left(
\begin{array}{ccc|c|ccc}
0 &          \ldots         & 0  & 1 & 1&\ldots & 1 \\ \hline
  &                         &   & 0 &  &                     \\
  &   \textup{\Large $A_u$} &   & \vdots && \textup{\Large $A$} \\
  &                         &   & 0
\end{array}
\right)\, .
$$

\vspace*{0.2cm}

We prove in  Lemma \ref{lem:descrNu} that the $\mcd$-module underlying $\mcn_u$ is isomorphic to a  a partial Fourier-Laplace transformed GKZ-system with respect to the matrix $A^s_u$ and parameter $\widetilde{\beta}_u$. With the help of the results in section \ref{subsec:GKZ} we construct a $\mcd$-free strictly filtered resolution of the filtered module $(\mcn_u,F^H_\bullet)$ in Proposition \ref{prop:resofNu}.

As a first step we prove some properties of the matrices $A^s_u$.
\begin{lemma}\label{lem:AsuNormal}
If, as before, $\mbz \widetilde{A} = \mbz^{d+1}$ and $\mbn\widetilde{A} = \mbz^{d+1} \cap \mbr_{\geq 0} \widetilde{A}$ hold,
then we have $\dZ A^s_u=\dZ^{d+1}$ and $\mbn A^s_u = \mbz^{d+1} \cap \mbr_{\geq 0} A^s_u$.
\end{lemma}
\begin{proof}
From $\dZ \widetilde{A}=\dZ^{d+1}$ we conclude $\dZ A^s_u=\dZ^{d+1}$  since evidently $\dZ\widetilde{A} \subset \dZ A^s_u$. Hence it remains
to show that the semi-group $\dN A^s_u$ is normal. We have
$$
C_u\cdot A^s_u =
\left(
\begin{array}{ccc|c|ccc}
0 &          \ldots         & 0  & 1 & 1&\ldots & 1 \\ \hline
  &                         &   & 0 &  &                     \\
  &   \textup{\Large $A_u$} &   & \vdots && \textup{\Large $A_u$} \\
  &                         &   & 0
\end{array}
\right)=:\left(\underline{\overline{a}}^u_1, \ldots, \underline{\overline{a}}^u_n, \underline{\widetilde{a}}^u_0,
\underline{\widetilde{a}}^u_1,\ldots,\underline{\widetilde{a}}^u_n\right)\in M((d+1)\times (2n+1),\dZ),
$$
where $C_u\in GL(d+1,\dZ)$
is the matrix already used in Lemma \ref{lem:MatrixAu}. It suffices to show the normality
property for the semi-group $\dN(C_u \cdot A^s_u)$ since $C_u$ is an invertible linear mapping, hence a homeomorphism.
Suppose that we are given a linear combination
$$
\underline{v}=\sum_{i=1}^n \lambda_i \underline{\overline{a}}^u_i + \sum_{j=0}^n \mu_j \underline{\widetilde{a}}^u_j \in \dZ^{d+1},
$$
where $\lambda_i,\mu_j\in\dR_{\geq 0}$. Then $v=\sum_{i=1}^n (\lambda_i+\mu_i) \underline{\overline{a}}^u_i + \underline{\widetilde{a}}^u_0\cdot \left(\sum_{j=0}^n \mu_j\right)$. Clearly, $\sum_{j=0}^n \mu_j\in \dN$, and moreover, the vector $\sum_{i=1}^n (\lambda_i+\mu_i) \underline{\overline{a}}^u_i$ lies
in $\dR_{\geq 0} A_u$, but the latter semi-group is normal according to Lemma \ref{lem:MatrixAu}. Hence
we have $\sum_{i=1}^n (\lambda_i+\mu_i) \underline{\overline{a}}^u_i \in \dN A_u$, and therefore
$v\in \dN \widetilde{A}_u \subset \dN (C_u \cdot A^s_u)$, as required.
\end{proof}

We now show that $\mcn_u$ can be interpreted as a partial Fourier-Laplace transformed GKZ-system. For this we consider the GKZ system $\mcm_{A^s_u}$ on $\hat{W}_u \times V$ with coordinates $(\hat{w}_{iu})_{i \neq u}, \lambda_0, \ldots , \lambda_n$. Let $\FL_{\hat{W}_u}$ be the partial Fourier-Laplace transformation which interchanges $\p_{\hat{w}_{iu}}$ with $(w_{iu})_{i\neq u}$ and $\hat{w}_{iu}$ with $- \p_{w_{iu}}$.

\begin{lemma}\label{lem:descrNu}
Let $\iasu$ be the left $\cD_{W_u\times V}$ ideal generated by the operators
\[
\ck{\Box}_{(\underline{m},\underline{l})} := \prod_{m_i > 0 \atop i \neq u} w_{iu}^{m_i} \prod_{l_i > 0 \atop 0 \leq i \leq n} \p_{\lambda_i}^{l_i} - \prod_{m_i < 0 \atop i \neq u} w_{iu}^{-m_i} \prod_{l_i > 0 \atop 0 \leq i \leq n} \p_{\lambda_i}^{-l_i} ,
\]
where $(\underline{m},\underline{l})=((m_i)_{i \neq u},l_0, \ldots,l_n) \in \mbl_{A^s_u}$,
\[
\ck{E^u_k}-\beta_k := -\sum^n_{i = 0 \atop i \neq u} a^u_{ki} \p_{w_{iu}} w_{iu} + \sum_{i=1}^n a_{ki}  \lambda_i \p_{\lambda_i} -\beta_k \qquad \text{for}\quad k=1, \ldots ,d
\]
(notice that the operators $\ck{E^u_k}$ are the same operators as inProposition \ref{prop:OriginalCoord} 1. above, but multiplied
with $-1$, which is useful for a Fourier-Laplace transformation that will be performed below)
and
\[
\ck{E^u_0} -\beta_0 := \sum_{i=0}^n \lambda_i \p_{\lambda_i}-\beta_0\, .
\]
Then we have $\iasu=\mck^{\widetilde{\beta_u}}_{A_u}$, and hence the $\mcd_{W_u \times V}$-module $\mcn_u$ is isomorphic to $\mcd_{W_u \times V} / \siasu$\;\;.
In other words, we have an isomorphism
$$
\cN_u \simeq \FL_{\hat{W}_u} \mcm^{\widetilde{\beta}}_{A^s_u}.
$$
\end{lemma}

\begin{proof}
For the first statement, notice that $\ck{\Box}_{(\underline{m},0)}$ equals the operator $\check{\Box}_{\underline{m}}$ from the definition
of the ideal $\cK_u$. On the other hand, one can obtain all operators $\ck{\Box}_{(\underline{m},\underline{l})}$ from the operators $\check{\Box}_{(\underline{m}',0)}$
using the relations $\p_{\lambda_i} - w_{iu} \p_{\lambda_u}$.
The last statement follows
by interchanging $\p_{w_{iu}}$ with $-\hat{w}_{iu}$ and $w_{iu}$ with $ \p_{\hat{w}_{iu}}$ in the classes of operators of type 1., 2., 3. and 4. in the definition of the ideal $\mck^{\widetilde{\beta_u}}_{A_u}$.
\end{proof}

In order to construct a strictly filtered resolution of $\cN_u$, we use the theory of Euler-Koszul complexes, which we explained in section \ref{subsec:GKZ}. It will be be applied to the $\mcd_{\hat{W}_u\times V}$-module $\mcm^{\widetilde{\beta_u}}_{A^s_u}$. As before, we work at the level of global sections.\\

Let $F^{\hat{\omega}}_\bullet D_{\hat{W}_u \times V}$ the filtration on $D_{\hat{W}_u \times V}$ corresponding to the weight vector
\begin{align}
\hat{\omega} &= \left((\textup{weight}(\hat{w}_{iu}))_{i \neq u}, (\textup{weight}(\p_{\hat{w}_{iu}}))_{i \neq u}, \textup{weight}(\lambda_0), \ldots , \textup{weight}(\lambda_n),\textup{weight}(\p_{\lambda_0}), \ldots , \textup{weight}(\p_{\lambda_n})\right)\notag \\
&:= (\underbrace{1,\ldots,1}_{n-\textup{times}},\qquad\qquad\quad\underbrace{0,\ldots,0}_{n-\textup{times}},\qquad\qquad\qquad\underbrace{0,\ldots,0}_{(n+1)-\textup{times}},\qquad\qquad\qquad\quad\underbrace{1,\ldots,1}_{(n+1)-\textup{times}})\, . \notag
\end{align}
Notice that this filtration corresponds to the order filtration $F^{ord}_\bullet D_{W_u \times V}$ under the Fourier-Laplace transformation functor $\FL_{\hat{W}_u}$.
We obtain a filtered resolution $((K_u^\bullet,d),F^{\hat{\omega}})$ of $M_{A^s_u}^{\widetilde{\beta}_u}$. Using Remark \ref{rem:strictresprop} we show that resolution is strict.

\begin{lemma}\label{lem:filtResEulerKoszul}
The Euler-Koszul complex $(K_u^\bullet, F^{\hat{\omega}}_\bullet)$ is a resolution of $(M^{\widetilde{\beta}}_{A^s_u}, F^{\hat{\omega}}_\bullet)$ in the category of filtered $D_{\hat{W}_u\times V}$-modules (with respect to the filtration $F_\bullet^{\hat{\omega}} D_{\hat{W}_u\times V}$),
i.e., we have a quasi-isomorphism $K_u^\bullet \twoheadrightarrow M_{A^s_u}$ and the complex $K_u^\bullet$ is strictly filtered.\end{lemma}
\begin{proof}
By Remark \ref{rem:strictresprop} above it is enough to show that $H^{-i} (\gr^{F^{\hat{\omega}}}_\bullet K^\bullet_u) = 0$ for $i \geq 1$ and $H^0( \gr^{F^{\hat{\omega}}}_\bullet K^\bullet_u) \simeq \gr^{F^{\hat{\omega}}}_{\bullet} M_{A^s_u}$. Denote by $\textit{GD}_{\hat{W}_u \times V} = \gr^{\hat{\omega}}_\bullet D_{\hat{W}_u \times V}$ the associated graded object of $D_{\hat{W}_u \times V}$, by $(\hat{v}_{iu})_{i\neq u}$ the symbol of $(\p_{\hat{w}_{iu}})_{i \neq u}$ and by $\mu_j$ the symbol of $\p_{\lambda_j}$ in $\textit{GD}_{\hat{W}_u\times V}$. Since $\Box_{(\underline{m},\underline{l})}$ is homogeneous in $(\p_{\lambda_j})$ and $ord_{\hat{\omega}}(\p_{\hat{w}_{iu}}) = 0$ for all $i \neq u$ we have
\[
\gr^{\hat{\omega}}(D_{\hat{W}_u \times V} / D_{\hat{W}_u \times V} J_{A^s_u}) = \textit{GD}_{\hat{W}_u\times V} / J^g_{A^s_u}
\]
where $J^g_{A^s_u}$ is generated by
\[
^{g}\Box_{(\underline{m},\underline{l})} := \prod_{m_i > 0 \atop i \neq u} \hat{v}_{iu}^{m_i} \prod_{l_i > 0 \atop 0 \leq i \leq n} \mu_i^{l_i} - \prod_{m_i < 0 \atop i \neq u} \hat{v}_{iu}^{-k_i} \prod_{l_i > 0 \atop 0 \leq i \leq n} \mu_i^{-l_i}
\]
Notice that
\[
\textit{GD}_{\hat{W}_u\times V}/ J^g_{A^s_u} \simeq \mbc[(\hat{w}_{iu})_{i \neq u}, \lambda_0, \ldots, \lambda_n]\otimes_\mbc \mbc[\mbn A^s_u]
\]
The associated graded complex $\gr^{{\hat{\omega}}} K^\bullet_u$ is isomorphic to a Koszul complex
\[
\gr^{\hat{\omega}} K^\bullet_u \simeq \textup{Kos}(\textit{GD}_{\hat{W}_u\times V}/ \hat{J}^g_{A^s_u}, (^{g\!}E^u_k)_{k=0,\ldots,d})
\]
where $^{g\!}E^u_k$ is defined by
\[
^{g\!}\hat{E}^u_k := \sum_{i \neq u}a_{ki}^u \hat{w}_{iu} \hat{v}_{iu} + \sum_{i=1}^n a_{ki} \lambda_i \mu_i \qquad \text{for} \quad k= 1, \ldots ,d
\]
and
\[
^{g\!}E^u_0 := \sum_{i=0}^{n} \lambda_i \mu_i
\]
It is shown in \cite[Theorem 1.2]{Berk} that the $^{g\!}E^u_k$ are part of a system of parameters. Since $\mbn A^s_u$ is a normal semi-group (see Lemma \ref{lem:AsuNormal} above), the ring $GD_{\hat{W}_u\times V}/ J^g_{A^s_u}$ is Cohen-Macaulay. Hence $(^{g\!}E^u_k)_{k=0,\ldots,d}$ is a regular sequence in $\textit{GD}_{\hat{W}_u\times V}/ J^g_{A^s_u}$. This shows that $H^{-i}(\gr^{{\hat{\omega}}}_\bullet K^\bullet_u) = 0$ for $i\geq 1$.  On the other hand, it follows from \cite[Theorem 4.3.5]{SST} that $H^0(\gr^{{\hat{\omega}}}_\bullet K^\bullet_u) = \textit{GD}_{\hat{W}_u\times V}/ (J^g_{A^s_u} + (^{g\!}E^u_k)_{k=0,\ldots,d}) \simeq \gr_\bullet^{{\hat{\omega}}} M^{\widetilde{\beta}}_{A^s_u}$, as required.
\end{proof}

As a consequence, we obtain the filtered resolution of $\cN_u$ we are looking for.
Let $\ck{J}_{\!\!A^s_u}$ be the ideal in $D_{W_u \times V}$ generated by the box operators $\ck{\Box}_{(\underline{m},\underline{l})}$ for $(\underline{m},\underline{l}) \in \mbl_{A^s_u}$. Put
\[
\ck{K}{}^{-l}_u = \bigoplus_{0 \leq i_1 < \ldots < i_l \leq d} D_{W_u\times V} / \ck{J}_{\!\!A^s_u}e_{i_1\ldots,i_l}\,
\]
and define
\[
\ck{K}{}^\bullet_u := \textup{Kos}( D_{W_u\times V}/\ck{J}_{\!\!A^s_u}, (\ck{E}{}^u_k-\beta_k)_{k = 0 , \ldots ,d})\, ,
\]
where the $\ck{E}{}^u_k$ denote the (pairwise commuting) endomorphisms of $D_{W_u\times V}/\ck{J}_{\!\!A^s_u}$ induced
from right multiplication by $\ck{E}{}^u_k$ on $D_{W_u\times V}$.
Define a filtration $\{ F_{\bullet} \ck{K}{}_u^\bullet\}$ on $\ck{K}{}^\bullet_u$ by
\[
F_p \ck{K}{}_u^l := \bigoplus_{0 \leq i_1 < \ldots < i_p \leq d} F_{p-l+(n-d)}^{ord}D_{W_u\times V} / \ck{J}_{\!\!A^s_u}.
\]
Then we have
\begin{proposition}\label{prop:resofNu}
We have a filtered quasi-isomorphism $(\ck{K}{}_u^\bullet, F_\bullet) \simeq (N_u, F^{ord}_\bullet) \simeq (N_u, F^H_{\bullet+(n-d)})$, i.e.
the complex $(\ck{K}{}_u^\bullet, F_\bullet)$ is a resolution of $(N_u, F^H_{\bullet+n-d})$ in the category of filtered $D_{W_u\times V}$-modules.
\end{proposition}
\begin{proof}
The filtered quasi-isomorphism $(\ck{K}{}_u^\bullet, F_\bullet) \simeq (N_u, F^{ord}_\bullet)$ is obtained by applying the Fourier-Laplace functor $\FL_{\hat{W}_u}$ to the (filtered) Euler-Koszul complex $(K^\bullet_u, F_\bullet^{\hat{\omega}})$ from above (using Lemma \ref{lem:filtResEulerKoszul}).
The second filtered (quasi-)isomorphism $(N_u, F^{ord}_\bullet) \simeq (N_u, F^H_{\bullet+(n-d)})$ is just the content of Proposition \ref{prop:OriginalCoord}.
\end{proof}

\subsection{$\msr$-modules}

In Subsection \ref{subsec:Koszul} we explicitly computed the filtered $\mcd$-module $(\mcn, F^H_\bullet)$ in the charts $W_u \times V$. Since the direct image of $\mcn$ under $\pi_2 : \mbp(W) \times V \ra V$ is the GKZ-system we are looking for we would just need to compute a filtered version of $R\pi_{2*} \DR_{\mbp(W) \times V /V}(\mcn)$ using a \v{C}ech argument.  It turns out that the theory of $\msr$-modules is most suitable for this task. Hence here we lift the results of the Subsection \ref{subsec:Koszul} to the level of $\msr$-modules. Following \cite{Sa6}, we first recall very briefly the basic notion of $\msr$-modules and the Rees construction which provides a functor from the category of filtered $\mcd$-modules to that of $\msr$-modules. In Proposition \ref{prop:NuRMod} we compute the corresponding Rees object of $\mcn_u$ and of its resolution  \ref{lem:ReesKu}. We glue these  resolutions in order to obtain a global resolution of the Rees object of $\mcn$ in Proposition \ref{Prop:Glueing}.\\

 Let $X$ be a smooth variety of dimension $n$. The order filtration of $\mcd_X$ gives rise to the Rees ring $R_F \mcd_X$. Given a filtered $\mcd_X$-module $(\mcm, F_\bullet \mcm)$ we construct the corresponding graded $R_F \mcd_X$-module $R_F \mcm := \bigoplus_{k \in \mbz} F_k \mcm z^k$. In local coordinates the sheaf of rings $R_F \mcd_X$ is given by
\[
R_F \mcd_X = \mco_X[z]\langle z \p_{x_1}, \ldots , z \p_{x_n} \rangle
\]
Denote by $\msx$ the product $X \times \mbc$.  We will consider the sheaf
\[
\msr_\msx = \mco_\msx \otimes_{\mco_X[z]} R_F \mcd_X
\]
and its ring of global sections
\[
\textup{R}_\msx:= \Gamma(\msx, \msr_\msx) = \mco_X(X)[z]\langle z \p_{x_1}, \ldots , z \p_{x_n} \rangle
\]
Given an $R_F \mcd_X$-module $R_F \mcm$ the corresponding $\msr_\msx$-module is
\[
\msm := \mco_\msx  \otimes_{\mco_X[z]} R_F \mcm
\]
This gives an exact functor $\mst$ from the category of filtered $\mcd_X$-modules $MF(\mcd_X)$ to the category of $\msr_\msx$-modules $Mod(\msr_\msx)$
\begin{align}
\mst: MF(\mcd_X) &\lra Mod(\msr_\msx) \notag \\
 (\mcm, F_\bullet \mcm) &\mapsto \msm \notag
\end{align}

We denote by $Mod_{qc}(\msr_\msx)$ the category of $\msr_\msx$-modules which are quasi-coherent $\mco_\msx$-modules. We denote by $\Omega^1_\msx = z^{-1} \Omega^1_{X \times \mbc / \mbc}$ the sheaf of algebraic $1$-forms on $\msx$ relative to the projection $\msx \ra \mbc$ having at most a pole of order one along $z=0$. If we put $\Omega^k_\msx = \wedge^k \Omega^1_\msx$, we get a deRham complex
\[
0 \lra \mco_\msx \overset{d}{\lra} \Omega^1_\msx \overset{d}{\lra} \ldots \overset{d}{\lra} \Omega^n_\msx \lra 0
\]
where the differential $d$ is induced by the relative differential $d_{X \times \mbc / \mbc}$.
If $X$ is a smooth affine variety we get the following equivalence of categories.

\begin{lemma}\label{lem:equivcatD-R}
Let $X$ be a smooth affine variety. The functor
\[
\Gamma(\msx, \bullet) : Mod_{qc}(\msr_\msx) \lra Mod(R_\msx)
\]
is exact and gives an equivalence of categories.
\end{lemma}
\begin{proof}
The proof is completely parallel to the $\mcd$-module case (see e.g. \cite[Proposition 1.4.4]{Hotta}).
\end{proof}
One can also define a notion of direct image in the category of $\msr$-modules. Since we only need the case of a projection, we will restrict ourselves to this special situation. Let $X,Y$ smooth algebraic varieties and $f:X \times Y \ra Y$ be the projection to the second factor. Similarly as above we have a relative de Rham complex $\Omega^\bullet_{\msx\times \msy /\msy} = z^{-1} \Omega^\bullet_{X \times Y \times \mbc / Y \times \mbc}$. If $\msm$ is an $\msr_{\msx \times \msy}$-module the relative de Rham complex $\DR_{\msx \times \msy / \msy}(\msm)$  is locally given by
\[
d(\omega \otimes m) = d \omega \otimes m + \sum_{i=1}^n (\frac{dx_i}{z} \wedge \omega) \otimes z \p_{x_i} m
\]
where $(x_i)_{1 \leq i \leq n}$ is a local coordinate of $X$. The direct image with respect to $f$ is then defined as
\[
f_+ \msm := Rf_* \DR_{\msx \times \msy / \msy}(\msm)[n]
\]

Recall that for a filtered $\mcd$-module $(\mcm,F_{\bullet} \mcm)$ the direct image under $f$ is given by
\[
f_+ \mcm = Rf_* \left( 0 \ra \mcm \ra \Omega^1_{X \times Y / Y} \otimes \mcm \ra \ldots  \ra \Omega^n_{X \times Y / Y} \otimes \mcm \ra 0 \right)[n]
\]
together with its filtration
\[
F_p f_+ \mcm = Rf_* \left( 0 \ra F_p \mcm \ra \Omega^1_{X \times Y / Y} \otimes F_{p+1} \mcm \ra \ldots  \ra \Omega^n_{X \times Y / Y} \otimes F_{p+n} \mcm \ra 0 \right)[n]
\]
Notice, however, that if $(\mathcal{M},
F_\bullet\mathcal{M})$ underlies a mixed Hodge module on $X\times
Y$, then the Hodge filtration on the cohomology modules of the
direct image complex is not, in general, given by this definition,
unless $X$ is projective. It is a straightforward exercise to check that the functor $\mst$ commutes with the direct image functor $f_+$.\\

We will apply this to the filtered $\mcd$-module $(\mcn, F^H_\bullet)$ as defined in equation \eqref{eq:SheafN} in order to compute $\mch^0\pi_{2 +} \mcn \simeq \mch^{2n+d+1}(p_{+} (  q^\dag \mco_T^{\beta} \otimes F^\dag  j_{\dag} \mco_{\mbc^*}^{-\beta_0-1}))$ together with its corresponding Hodge filtration. We will denote by $\msp \times \msv$ the space $\mbp(W) \times V \times \mbc$. The corresponding $\msr$-module is
\[
\msn := \mst(\mcn) = \mco_{\msp \times \msv} \otimes_{\mco_{\mbp(W) \times V}[z]} R_{F^H}\mcn
\]
The direct image with respect to $\pi_2$ is then given by
\begin{equation}\label{eq:pi2Rmoddirectim}
\pi_{2+} \msn \simeq R\pi_{2 *} \left(0 \ra \msn \ra \Omega^1_{\msp \times \msv / \msv} \otimes \msn \ra \ldots \ra \Omega^n_{\msp \times \msv / \msv} \otimes \msn \ra 0  \right)[n]
\end{equation}
Since this is rather hard to compute, we will replace the complex
\[
0 \ra \msn \ra \Omega^1_{\msp \times \msv / \msv} \otimes \msn \ra \ldots \ra \Omega^n_{\msp \times \msv / \msv} \otimes \msn \ra 0
\]
by a quasi-isomorphic one. For this we will construct a resolution of $\msn$.
Let $\msw_u \times \msv := W_u \times V \times \mbc$ and denote by $\msn_u$ the restriction of $\msn$ to $\msw_u \times \msv$. We write $\textup{R}_{\msw_u \times \msv} = \Gamma(\msw_u \times \msv, \msr_{\msw_u \times \msv})$, then the module of global sections of $\msn_u$ is the $\textup{R}_{\msw_u \times \msv}$-module
\[
\textup{N}_u:= \Gamma(\msw_u \times \msv, \msn_u)\, .
\]

\begin{proposition}\label{prop:NuRMod}
The $\textup{R}_{\msw_u \times \msv}$-module $\textup{N}_u$ is isomorphic to
\[
z^{n-d} \cdot \textup{R}_{\msw_u \times \msv} / \tI_{A_u^s}
\]
where $\tI_{A_u^s}$ is generated by
\[
\overline{\Box}_{(m,l)} := \prod_{m_i > 0 \atop i \neq u} w_{iu}^{m_i} \prod_{l_i > 0 \atop 0 \leq i \leq n} (z\p_{\lambda_i})^{l_i} - \prod_{m_i < 0 \atop i \neq u} w_{iu}^{-m_i} \prod_{l_i < 0 \atop 0 \leq i \leq n} (z\p_{\lambda_i})^{-l_i} ,
\]
where $(m,l)=((m_i)_{i \neq u},l_0, \ldots,l_n) \in \mbl_{A^s_u}$,
\[
\overline{E}^u_k-\beta_k := -\sum^n_{i = 0 \atop i \neq u} a^u_{ki} z\p_{w_{iu}} w_{iu} + \sum_{i=1}^n a_{ki}  \lambda_i z\p_{\lambda_i}  -\beta_k \qquad \text{for}\quad k=1, \ldots ,d
\]
and
\[
\overline{E}^u_0-\beta_0 := \sum_{i=0}^n \lambda_i z\p_{\lambda_i} -\beta_0
\]
\end{proposition}
\begin{proof}
This follows easily from Lemma \ref{lem:descrNu} and Lemma  \ref{lem:equivcatD-R}.
\end{proof}

We will now define a Koszul complex $\tK^\bullet_u$ in the category of $\tR_u$-modules which corresponds to the Koszul complex $\ck{K}{}^\bullet_u$ alluded to above.
Write $\tJ_{A^s_u}$ for the left ideal in $\tR_{\msw_u \times \msv}$ generated by all operators $\overline{\Box}_{(m,l)}$ for $(m,l)\in \mbl_{A^s_u}$, then a computation similar to formula \eqref{eq:EulerFieldsWellDefined} shows that the maps
\begin{align}
\textup{R}_{\msw_u \times \msv} / \tJ_{A^s_u} &\lra \textup{R}_{\msw_u \times \msv} / \tJ_{A^s_u} \notag \\
P &\mapsto P \cdot (\overline{E}^u_k -\beta_k) \qquad \text{for} \quad k= 0, \ldots ,d
\end{align}
are well-defined. Since $[\overline{E}^u_{k_1}-\beta_{k_1}, \overline{E}^u_{k_2}-\beta_{k_2}] = 0$ for $k_1, k_2 \in \{0, \ldots ,d \}$ we can built a Koszul complex
\[
\tK^\bullet_u := Kos\left(z^{n-d} \cdot\tR_{\msw_u \times \msv} / \tJ_{A^s_u} , (\,\cdot\, \overline{E}^u_k-\beta_k)_{k = 0 , \ldots ,d}\right)
\]
whose terms are given by
\[
z^{n-2d-1} \cdot \tR_{\msw_u \times \msv} / \tJ_{A^s_u} \ra \ldots \ra \bigoplus_{i=1}^n z^{n-d-1} \cdot \tR_{\msw_u \times \msv} / \tJ_{A^s_u} \;  e_1 \wedge \widehat{e}_i \wedge \ldots \wedge e_d\ra z^{n-d} \cdot \tR_{\msw_u \times \msv} / \tJ_{A^s_u} \; e_1 \wedge \ldots \wedge e_d
\]
\begin{lemma}\label{lem:ReesKu}
The Koszul complex $\tK_u^\bullet$ is a resolution of $\tN_u$.
\end{lemma}
\begin{proof}
In order to prove the lemma it is enough to apply the exact Rees functor $\mst$ to the Koszul complex $\ck{K}{}^\bullet_u$ which is a strict resolution of $N_u$ in the category of filtered $D_{W_u \times V}$-modules by Proposition \ref{prop:resofNu}.

\end{proof}

We denote by $\msk_u^\bullet$ the corresponding resolution of $\msn_u = \msn_{|W_u \times V}$. We are now able to construct a resolution of $\msn$.

\begin{proposition}\label{Prop:Glueing}
There exists a resolution $\msk^\bullet$ of $\msn$ in the category of $\msr_{\msp \times \msv}$-modules which is locally given by
\[
\Gamma(\msw_u \times \msv, \msk^\bullet) = \tK^\bullet_u
\]
\end{proposition}
\begin{proof}
The resolution $\msk^\bullet$ is constructed by providing glueing maps between the $\tR_{\msw_{u_1u_2} \times \msv}$-modules
\[
\Gamma(\msw_{u_1 u_2} \times \msv, \msk^\bullet_{u_1}) \simeq  \tK_{u_1}^\bullet[w_{u2u1}^{-1}] \lra \Gamma(\msw_{u_1 u_2} \times \msv, \msk^\bullet_{u_2}) \simeq  \tK_{u_2}^\bullet[w_{u1u2}^{-1}]
\]
which are compatible with the glueing maps on
\[
\Gamma(\msw_{u_1 u_2} \times \msv, \msn_{u_1}) \simeq  \tN_{u_1}[w_{u2u1}^{-1}] \lra \Gamma(\msw_{u_1 u_2} \times \msv, \msn_{u_2}) \simeq  \tN_{u_2}[w_{u1u2}^{-1}]
\]
Notice that the latter maps are given by
\begin{align}
\tN_{u_1}[w_{u2u1}^{-1}] &\lra  \tN_{u_2}[w_{u1u2}^{-1}] \notag \\
P &\mapsto \iota_{u_1 u_2}(P) w_{u_1 u_2}^{n+1} \notag
\end{align}
which follows from Lemma \ref{lem:gluemapSW} and by tracing back the functors applied to $g_{u +} \mco_T$. Using the same argument as in Lemma \ref{lem:gluemapSW} shows that the maps
\begin{align}
\tK_{u_1}^\bullet[w_{u2u1}^{-1}] &\lra \tK_{u_2}^\bullet[w_{u1u2}^{-1}] \notag \\
P &\mapsto \iota_{u_1 u_2}(P) w_{u_1 u_2}^{n+1} \notag
\end{align}
are well defined. We have to check that they give rise to a morphism of complexes. But this follows from the commutativity of the diagram
\[
\begin{tikzcd}
P \ar{rr} && \iota_{u_1 u_2}(P) w_{u_1u_2}^{n+1} \\
\tR_{\msw_{u_1} \times \msv} / \tJ_{u_1} \ar{rr} & &
\tR_{\msw_{u_2} \times \msv} / \tJ_{u_2}\\
&  \\ \\
\tR_{\msw_{u_1} \times \msv} / \tJ_{u_1}
 \ar{uuu}{\cdot \overline{E}^{u_1}_k-\beta_k}  \ar{rr} &&
 \tR_{\msw_{u_2} \times \msv} / \tJ_{u_2} \ar{uuu}{\cdot \overline{E}^{u_2}_k-\beta_k} \\
 P \ar{rr} && \iota_{u_1 u_2}(P) w_{u_1u_2}^{n+1}
\end{tikzcd}
\]
\end{proof}
\subsection{A quasi-isomorphism}

In order to compute the direct image of the $\msr$-module $\msn$  under $\pi_2$ we have to deal with with the relative de Rham complex $\DR_{\msp \times \msv / \msv}(\msn)$ (cf. formula \eqref{eq:pi2Rmoddirectim}). In this subsection we show in Proposition \ref{prop:quasiIsos} that this complex is quasi-isomorphic to a complex $\msl^\bullet$ which is the top cohomology (with respect to the deRham differential) of the double complex $\DR_{\msp \times \msv / \msv}(\msk^\bullet)$. In Proposition \ref{prop:LocalDescriptionComplL} we give a local description of this complex $\mcl^\bullet$ on the charts $\msw_u \times \msv$.\\

As announced above we now apply the relative de Rham functor $\DR_{\msp \times \msv / \msv}$ to the resolution $\msk^\bullet$ and get a double complex $\Omega^{\bullet+n}_{\msp \times \msv / \msv} \otimes \msk^\bullet$:
\[
\begin{tikzcd}
\hdots \ar{r} & \Omega^{n-1}_{\msp \times \msv / \msv} \otimes \msk^0 \ar{r}{{_{II}}d^{n,0}} & \Omega^n_{\msp \times \msv / \msv} \otimes \msk^0 \\ \hdots \ar{r} & \Omega^{n-1}_{\msp \times \msv / \msv} \otimes \msk^{-1} \ar{r}{{_{II}}d^{n,-1}} \ar{u}{{_I}d^{n-1,0}} & \Omega^n_{\msp \times \msv / \msv} \otimes \msk^{-1}  \ar{u}{{_I}d^{n,0}} \\ & \ar{u} \vdots& \ar{u} \vdots
\end{tikzcd}
\]
The corresponding total complex is denoted by  $Tot\left(\Omega^{\bullet+n}_{\msp \times \msv / \msv} \otimes \msk^\bullet\right)$.
\begin{proposition}\label{prop:quasiIsos}
The following natural morphisms of complexes
\[
\begin{tikzcd}
\Omega^{\bullet+n}_{\msp \times \msv/\msv} \otimes \msn & Tot\left(\Omega^{\bullet+n}_{\msp \times \msv / \msv} \otimes \msk^\bullet\right) \ar{l} \ar{r} & \Omega^n_{\msp \times \msv / \msv} \otimes \msk^\bullet / {_{II}}d\left(\Omega^{n-1}_{\msp \times \msv / \msv} \otimes \msk^\bullet \right)
 =: \msl^\bullet
\end{tikzcd}
\]
are quasi-isomorphisms.
\end{proposition}
\begin{proof}
Since the double complex $\Omega^{\bullet+n}_{\msp \times \msv / \msv} \otimes \msk^\bullet$ is bounded we can associate with it two spectral sequences which both converge. The first one is given by taking cohomology in the vertical direction which gives the ${_I}E_1$-page of the spectral sequence. Since $\msk^\bullet$ is a resolution of $\msn$ and $\Omega^l_{\msp \times \msv/\msv}$ is a locally free (i.e. flat) $\mco_{\msp \times \msv}$-module for every $l = 1, \ldots, n$, the  only terms which are non-zero are the ${_I}E^{0,q}_1$-terms which are isomorphic to $\Omega^{q+n}_{\msp \times \msv/\msv} \otimes \msn$. Hence the first spectral sequence degenerates at the second page which shows that $\Omega^{\bullet+n}_{\msp \times \msv/\msv} \otimes \msn \leftarrow Tot\left(\Omega^{\bullet+n}_{\msp \times \msv / \msv} \otimes \msk^\bullet\right)$ is a quasi-isomorphism.\\

We now look at the second spectral sequence which is given by taking cohomology in the horizontal direction. We claim that ${_{II}}E^{p,q}_1 = 0$ for $q \neq 0$. It is enough to check this locally on the charts $\msw_u \times \msv$ using, moreover,  Lemma \ref{lem:equivcatD-R} at the level of global sections. Notice that the complex
\[
\Gamma(\msw_u \times \msv, \Omega^{\bullet+n}_{\msp \times \msv/ \msv} \otimes \msk^l)
\]
is isomorphic to a direct sum of Koszul complexes $Kos^{\bullet}( z^{-d-l}\tR_{\msw_u\times \msv} / \tJ_{A^s_u}, \frac{1}{z}(z\p_{w_{iu}} \cdot )_{i \neq u}) )$ where each summand is given by
\[
z^{n-d-l} \tR_{\msw_u \times \msv}/\tJ_{A^s_u} \lra \ldots \lra z^{-d-l}\tR_{\msw_u \times \msv}/\tJ_{A^s_u} e_1 \wedge \ldots \wedge e_n \, .
\]
The quotient $\tR_{\msw_u\times \msv} /\tJ_{A^s_u}$ can be written as
\[
\mbc[z,(z \p_{w_{iu}})_{i \neq u}] \otimes_{\mbc[z]} \left(\mbc[z, \lambda_0,\ldots, \lambda_n, (w_{iu})_{i \neq u}]\langle z \p_{\lambda_0},\ldots,z \p_{\lambda_n}\rangle / \left(\left( \overline{\Box}_{(m,l)}\right)_{(m,l)\in \mbl_{A^s_u}} \right)  \right)
\]
Since the operators $z \p_{w_{iu}} \cdot $ act only on the first term in the tensor product, we immediately see that ${_{II}}E^{p,q}_1 = 0$ for $ q \neq 0$.

The fact that $ Tot\left(\Omega^{\bullet+n}_{\msp \times \msv / \msv} \otimes \msk^\bullet\right) \ra  \msl^\bullet$ is a quasi-isomorphism follows from the fact that ${_{II}}E^{p,q} = 0 $ for $q \neq 0$, i.e. the second spectral sequence degenerates at the second page.
\end{proof}
The next result is an explicit local description of the complex $\msl^\bullet$.
\begin{proposition}\label{prop:LocalDescriptionComplL}
For any $u\in\{0,\ldots,n\}$ define the ring
\[
\tS_{\msw_u\times \msv}:= \mbc[z, \lambda_0,\ldots, \lambda_n, (w_{iu})_{i \neq u}]\langle z \p_{\lambda_0},\ldots,z \p_{\lambda_n}\rangle
\]
and denote by $\mss$ the sheaf of rings on $\msp\times \msv$ which is locally given by
\[
\Gamma(\msw_u \times \msv, \mss) = \tS_{\msw_u\times \msv}
\]
with glueing maps
\begin{align}
\tS_{\msw_{u_1}\times \msv}[w_{u_2 u_1}^{-1}] &\lra \tS_{\msw_{u_2}\times \msv}[w_{u_1 u_2}^{-1}] \notag \\
P &\mapsto \iota_{u_1 u_2} (P) \notag
\end{align}
Denote by $\tJ_{A^s_u}$ the left $\tS_{\msw_u\times \msv}$-ideal generated by the Box operators $\overline{\Box}_{(m,l)}$ for $(m,l) \in \mbl_{A^s_u}$.
Note that this is a slight abuse of notation, as the ideal generated by the same set of operators in the ring $\tR_{\msw_u\times\msv}$
was also denoted by $\tJ_{A^s_u}$, but which is justified by the fact that these generators do not contain the variables $z\partial_{w_{iu}}$.
Then the complex $\msl^\bullet$ is given locally by
\begin{equation}\label{eq:localL}
\Gamma(\msw_u \times \msv, \msl^\bullet) \simeq Kos^\bullet(z^{-d}\tS_{\msw_u \times \msv} / \tJ_{A^s_u}, (\tilde{E}_k-\beta_k)_{k=0,\ldots ,d} )
\end{equation}
whose terms are given by
\[
z^{-2d-1}\tS_{\msw_u \times \msv} / \tJ_{A^s_u} \lra \ldots \lra z^{-d}\tS_{\msw_u \times \msv} / \tJ_{A^s_u} e_1 \wedge \ldots \wedge e_d
\]
where
\begin{align}
&\tilde{E}_k -\beta_k := \sum_{i=1}^n a_{ki}  \lambda_i z\p_{\lambda_i}  -\beta_k \qquad \text{for}\quad k=1, \ldots ,d \notag \\
&\tilde{E}_0 - \beta_0 := \sum_{i=0}^n \lambda_i z\p_{\lambda_i} - \beta_0 \notag
\end{align}
\end{proposition}

\begin{proof}
It follows from Proposition \ref{prop:quasiIsos} that the $0$-th cohomology of the complex
$\left(\Omega^{\bullet+n}_{\msp\times\msv/\msv}\otimes \msk^p, { _{\mathit{II}}}d^{\bullet,p}\right)$ is a direct sum of terms
of the form $H^0(z^{-d}Kos^\bullet( \tR_{\msw_u\times \msv} / \tJ_{A^s_u}, \frac{1}{z}(z \p_{w_{iu}} \cdot)_{i \neq u}))$.
Taking the cokernel of left multiplication on $\tR_{\msw_u\times \msv} / \tJ_{A^s_u}$ by $z \p_{w_{iu}}$ shows that we have
an isomorphism of $\tS_{\msw_u\times\msv}$-modules
\[
H^0(z^{-d}Kos^\bullet( \tR_{\msw_u\times \msv} / \tJ_{A^s_u}, (z \p_{w_{iu}} \cdot)_{i \neq u})) \simeq z^{-d}\tS_{\msw_u\times \msv}/\tJ_{A^s_u}.
\]
Hence equation \eqref{eq:localL} follows.
\end{proof}

The ideals $\tJ_{A^s_u}$ glue to an ideal $\msj \subset \mss$. Notice that the Euler vector fields $(\tilde{E}_k-\beta_k)_{k=0,\ldots ,d}$ are global sections of $\mss$.
We recall from Proposition \ref{Prop:Glueing} that the glueing maps for $\Gamma(\msw_u \times \msv, \Omega^n_{\msp \times \msv} \otimes \msk^p)$ are given by:

\[
\bigwedge_{i=0 \atop i \neq u_1}^{n} dw_{iu_1} \otimes P \longmapsto \bigwedge_{i=0 \atop i \neq u_2}^{n} dw_{iu_2} \cdot (w_{u_1 u_2})^{-n-1} \otimes \iota_{u_1 u_2}(P)w_{u_1 u_2}^{n+1}  \]
Since both powers of $w_{u_1 u_2}$ on the right hand side cancel  when considering the quotient $\msl^p$, we see that
\[
\msl^\bullet \simeq Kos^\bullet(z^{-d}\mss / \msj, (\tilde{E}_k-\beta_k)_{k=0,\ldots,d})
\]
Summarizing, Proposition \ref{prop:quasiIsos} and Proposition \ref{prop:LocalDescriptionComplL} show that instead of computing the direct image \eqref{eq:pi2Rmoddirectim} we can compute
\[
R\pi_{2 *}( \msl^\bullet ) \simeq R\pi_{2*}(z^{-d}Kos^\bullet(\mss / \msj, (\tilde{E}_k-\beta_k)_{k=0,\ldots,d})
\]

\subsection{Computation of the direct image}

In this subsection we continue our computation of $\pi_{2+}\msn \simeq R\pi_{2*}(\msl^\bullet)$. Because $\msv$ is affine  it is enough to work at the level of global sections (cf. Lemma \ref{lem:equivcatD-R} ). We get:
\begin{equation}\label{eq:GLqisGKos}
\Gamma R\pi_{2 *}( \msl^\bullet) \simeq R\Gamma R\pi_{2 *}(\msl^\bullet) \simeq R\Gamma(\msl^\bullet) \simeq R\Gamma(Kos^\bullet(z^{-d}\mss / \msj, (\tilde{E}-\beta_k)_{k=0,\ldots,d}))
\end{equation}
where the first isomorphism follows from the exactness of $\Gamma(\msv, \bullet)$. Hence we have to consider the hypercohomology of a Koszul complex on $\mss/ \msj$.\\

We will show that each term of this Koszul complex  is $\Gamma$-acyclic. For this it is enough to show that $\mss / \msj$ is $\Gamma$-acyclic.  In Proposition \ref{prop:GrLocCoh} we show that $\mss / \msj$ is acyclic and its corresponding associated graded module has an easy description in terms of the "semi" R-module $S/J_{A^s}$ on $\mbc_z \times W \times V$ if all local cohomologies of $S/J_{A^s}$ with respect to the ideal $(w_0,\ldots,w_n)$ vanish. Finally we show in Lemma \ref{lem:tensorprod} that the computation of these local cohomologies can be reduced to the computation of local cohomology groups of semigroup rings, which is a problem in commutative algebra and which we tackle in section \ref{subsec:LocalCohom}.\\

Recall that $\msp \times \msv = \mbc_z \times \mbp(W) \times V$. We denote by $\msw \times \msv$ the space $\mbc_z \times W \times V$. Let
\[
\tS:= \mbc[z,w_0, \ldots, w_n, \lambda_0, \ldots, \lambda_n]\langle z \p_{\lambda_0}, \ldots ,  z\p_{\lambda_n}\rangle
\]
and consider the $\tS$-module
\[
\tS / \tJ_{A^s}
\]
where the left ideal $\tJ_{A^s}$ is generated by
\[
\Box_{(m,l)} = \prod_{m_i > 0} w_i^{m_i} \prod_{l_i > 0} (z \p_{\lambda_i})^{l_i} - \prod_{m_i < 0} w_i^{-m_i} \prod_{l_i < 0} (z \p_{\lambda_i})^{-l_i}\qquad \text{for} \quad (m,l) \in \mbl_{A^s},
\]
the matrix $A^s$ is given by
\[
A^s :=(\underline{a}^s_0,\ldots,\underline{a}^s_n,\underline{b}^s_{0},\ldots,\underline{b}^s_{n}) := \left(\begin{matrix}1 & 1 & \dots & 1 & 0 & 0 &\dots & 0 \\ 0 & 0 & \dots & 0 & 1 & 1 & \dots & 1 \\ 0 & a_{11} & \dots & a_{1n} & 0 & a_{11} & \dots & a_{1n}\\ \vdots & \vdots & &\vdots & \vdots & \vdots & & \vdots \\ 0 & a_{d1} & \dots & a_{dn} & 0 & a_{d1} & \dots & a_{dn}\end{matrix}\right)
\]

and $\mbl_{A^s}$ is the $\mbz$-module of relations among the columns of $A^s$.
Notice that $\tS/\tJ_{A^s}$ is $\mbz$-graded by the degree of the $w_i$. Denote by $\tS_{w_u}$ the localization of $\tS$ with respect to $w_u$, then one easily sees that the degree zero part $[\tS_{w_u} / \tJ_{A^s}]_0$ of $ \tS_{w_u} / \tJ_{A^s}$ is equal to $\Gamma(\msw_u \times \msv, \mss / \msj) \simeq \tS_{\msw_u\times\msv} / \tJ_{A^s_u}$ if we identify $w_i / w_u$ with $w_{iu}$. Let $\widetilde{\tS / \tJ_{A^s}}$ be the associated sheaf on $\msp \times \msv$ having global sections $[\tS / \tJ_{A^s}]_0$, then we obviously
have
\[
\widetilde{\tS / \tJ_{A^s}} \simeq \mss / \msj
\]
Define
\[
\Gamma_\ast(\mss / \msj) := \bigoplus_{a \in \mbz} \Gamma(\msp \times \msv, (\mss / \msj)(a))
\]
We want to use the following result applied to the graded module $\tS/\tJ_{A^s}$
\begin{proposition}\cite[Proposition 2.1.3]{EGA31}\label{prop:GrLocCoh}
There is the following exact sequence of $\mbz$-graded $\tS$-modules
\[
0 \lra H^0_{(\underline{w})}(\tS / \tJ_{A^s}) \lra \tS / \tJ_{A^s} \lra \Gamma_*(\mss / \msj) \lra H^1_{(\underline{w})}(\tS / \tJ_{A^s}) \lra 0
\]
and for each $ i \geq 1$, there is the following isomorphism
\begin{equation}\label{eq:RGammavsLoc}
\bigoplus_{a \in \mbz} H^i(\msp \times \msv, (\mss / \msj)(a)) \simeq H^{i+1}_{(\underline{w})}(\tS/\tJ_{A^s})
\end{equation}
where $(\underline{w})$ is the ideal in $\mbc[z,\lambda_0,\ldots, \lambda_n,w_0,\ldots,w_n]$ generated by $w_0, \ldots ,w_n$.
\end{proposition}
\begin{proof}
In the category of $\mbc[z, \lambda_0, \ldots, \lambda_n, w_0, \ldots w_n]$-modules, the statement follows from \cite[Proposition 2.1.3]{EGA31}.
The statement in the category of $\tS$-modules follows from the proof given there.
\end{proof}

In order to compute the local cohomology of  $S/J_{A^s}$ we introduce a variant of the Ishida complex (see e.g. \cite[Theorem 6.2.5]{HerzBr}). Let $T:= \mbc[w_0,\ldots,w_n, z\p_{\lambda_0}, \ldots , z\p_{\lambda_n}] \subset S$ be a commutative subring and let $\mbc[\mbn A^s]$ be the affine semigroup algebra of $A^s$, i.e.
\[
\mbc[\mbn A^s] = \{ y^{\underline{c}} \in \mbc[y_0^\pm, \ldots,y_{n+1}^\pm] \mid \; \underline{c} \in \mbn A^s \subset \mbz^{d+2}\}
\]
We have a map
\begin{align}
\Phi_{A^s}: T &\lra \mbc[\mbn A^s] \notag \\
w_i &\mapsto y^{\underline{a}_i^s} \notag \\
z\p_{\lambda_i} &\mapsto  y^{\underline{b}_i^s} \notag
\end{align}
Notice that the kernel $K_{A^s}$ of $\Phi_{A^s}$ is equal to the ideal in $T$ generated by the elements $\Box_{(k,l)}$, hence
\[
T/K_ {A^s} \simeq \mbc[\mbn A^s]
\]
\begin{remark}
The $\mbz$-grading of $T$ by the degree of the $w_i$ induces a $\mbz$-grading on $\mbc[\mbn A^s]$ since the operators $\Box_{(k,l)}$ are homogeneous. The semi-group ring $\mbc[\mbn A^s]\subset \mbc[\mbz^{d+2}]$ carries also a natural $\mbz^{d+2}$-grading. Looking at the matrix $A^s$ one sees that the $\mbz$-grading coming from $T$ is the first component of this $\mbz^{d+2}$-grading.
\end{remark}
 We regard $\mbc[\mbn A^s]$ as a $T$-module using the map $\Phi_{A^s}$, which gives the isomorphisms
\[
S/ J_{A^s} \simeq S \otimes_T T / K_{A^s} \simeq  S \otimes_T \mbc[\mbn A^s]
\]
We want to express the local cohomology of $S /J_{A^s}$ by the local cohomology of the commutative ring $\mbc[\mbn A^s]$. For this, let $I$ be the ideal in $\mbc[\mbn A^s]$ generated by $y^{\underline{a}_0^s}, \ldots y^{\underline{a}_n^s}$, then we have the following  change of rings formula:
\begin{lemma}\label{lem:tensorprod}
There is the following isomorphism of $\mbz$-graded $S$-modules:
\[
H^k_{(\underline{w})}(S / J_{A^s}) \simeq S \otimes_T H^k_I(\mbc[\mbn A^s])
\]
\end{lemma}
\begin{proof}
Notice that if $S$ was commutative this would be a standard property of the local cohomology groups. Here we have to adapt the proof slightly. First notice that it is enough  to compute $H^k_{(\underline{w})}(S/J_{A^s})$ with an injective resolution of $T$-modules. To see why, let $I^\bullet$ be an injective resolution (in the category of $S$-modules) of $S/ J_{A^s}$. Since $S$ is a free, hence flat, $T$-module, it follows from $Hom_S( S \otimes_T M, I) \simeq S \otimes_T Hom_T( M,I)$, that an injective $S$-module is also an injective $T$-module. Therefore we have
\[
H^k_{(\underline{w})}(S/J_{A^s}) \simeq H^k \Gamma_{(\underline{w})}(I^\bullet) = H^k \Gamma_{I'}(I^\bullet) \simeq H^k_{I'}(S/J_{A^s})
\]
where $I'$ is the ideal in $T$ generated by $w_0, \ldots, w_n$ and the second isomorphism follows from the equality
\[
\Gamma_{(\underline{w})}(I^k) = \{x \in I^k \mid \forall i\, \exists k_i\, \text{such that} \; w_i^{k_i} x = 0 \} = \Gamma_{I'}(I^ k)
\]
Let $J^\bullet$ be an injective resolution of $T/K_{A^s}$.
In order to show the claim consider the following isomorphisms
\begin{align}
S \otimes_T H^K_I(\mbc[\mbn A^s]) &\simeq S \otimes_T H^k_{I'}(T/K_{A^s})\notag \\
 &\simeq S \otimes_T H^k \Gamma_{I'} (J^\bullet) \notag \\
 &\simeq H^k(S \otimes_T \Gamma_{I'} (J^\bullet))\notag \\
 &\simeq H^k \Gamma_{I'} ( S \otimes_T J^\bullet)\notag \\   &\simeq H^k_{I'} ( S /J_{A^s}) \notag
\end{align}
where the third isomorphism follows from the fact that $S$ is a flat $T$-module and the fifth isomorphism follows from the fact that $S \otimes_T J^\bullet$ is a $T$-injective resolution of $S/J_{A^s} \simeq S \otimes_T T/K_{A^s}$.
\end{proof}

\subsection{Local cohomology of semi-group rings}
\label{subsec:LocalCohom}

In this subsection we compute local cohomology groups of some special semigroup rings associated with the semigroups $\mbn A^s$. These local cohomology groups turned up in Lemma \ref{lem:tensorprod}. We will use their vanishing in Subsection
\ref{subsec:MainThm} (specifically in  Corollary \ref{cor:SJGacyclic}) to prove $\Gamma$-acyclicity of a Koszul complex.

We show that these local cohomology groups can be expressed as the cohomology of the Ishida complex of $\mbc[\mbn A^s]$ (cf. Proposition \ref{prop:LocCohIshida}). Since $A^s$ is a $(d+2)\times (2n+2)$ integer matrix the Ishida complex carries a natural $\mbz^{d+2}$-grading. The vanishing of certain graded pieces of the Ishida complex resp. of its cohomology depends on the position of the degree as seen as an element in $\mbr^{d+2}$ in relation with the cone spanned by the columns of $A^s$(cf. Lemma \ref{lem:vis} and Proposition \ref{prop:Ishidadeg}).\\

Let $\mcf$ be the face lattice of $\mbr_{\geq 0} A^s$ and denote by $\mcf_\sigma$ the sublattice of faces which lie in the face $\sigma$ spanned  by $\underline{a}_0^s, \ldots , \underline{a}^s_n$. For a face $\sigma$ of $\mbr_{\geq 0} A^s$ consider the multiplicatvely closed set
\[
U_\sigma := \{y^{\underline{c}} \mid \underline{c} \in \mbn(A^s \cap \sigma)\} \]
and denote by $\mbc[\mbn A^s]_\sigma = \mbc[\mbn A^s + \mbz(A^s \cap \sigma)]$ the localization with respect to $U_\sigma$.   We put
\[
L^k_\sigma  = \bigoplus_{\tau \in \mcf_\sigma \atop \dim \tau = k} \mbc[\mbn A^s]_\tau
\]
and define maps $f^k : L^k_\sigma \ra L^{k+1}_\sigma$ by specifying their components
\[
f^k_{\tau',\tau}: \mbc[\mbn A^s]_{\tau'} \ra \mbc[\mbn A^s]_\tau \quad \text{to be} \quad \begin{cases} 0 & \text{if}  \; \tau' \not\subset \tau \\ \epsilon(\tau',\tau)nat &\text{if}\; \tau' \subset \tau \end{cases}
\]
where $\epsilon$ is a suitable incidence function on $\mcf_\sigma$ and $nat$ is the natural localization morphism. The Ishida complex with respect to the face $\sigma$ is
\[
L^\bullet_\sigma: 0 \ra L^0_\sigma \ra L^1_\sigma \ra \ldots \ra L^{d+1}_\sigma \ra 0
\]
The Ishida complex with respect to the face $\sigma$ can be used to calculate local cohomology groups of $\mbc[\mbn A^s]$.
\begin{proposition}\label{prop:LocCohIshida}
As above, denote by $I\subset \mbc[\mbn A^s]$ the ideal generated by the elements $\Phi_{A^s}(w_i) = y^{\underline{a}^s_i}$. Then for all $k$ we have the isomorphism
\[
H^k_{I}(\mbc[\mbn A^s]) \simeq H^k(L^\bullet_\sigma)
\]
\end{proposition}
\begin{proof}
The proof can be easily adapted from \cite[Theorem 6.2.5]{HerzBr}. For the convenience of the reader we sketch it here together with the necessary modifications . In order to show the claim we have to prove that the functors $N \mapsto H^k(L^\bullet_\sigma \otimes N)$ form a universal $\delta$-functor (see e.g. \cite{Harts}). If we can additionally show that
\begin{equation}\label{eq:H0loc}
H^0_{I}(\mbc[\mbn A^s]) \simeq H^0(L^\bullet_\sigma)
\end{equation}
the claim follows by \cite[Corollary III.1.4]{Harts}. Let $\mcf_\sigma(1)$ be the set of one-dimensional faces in $\mcf_\sigma$ and notice that
\[
H^0_{I'}(\mbc[\mbn A^s]) \simeq \ker\left(\mbc[\mbn A^s] \lra  \bigoplus_{\tau \in  \mcf_\sigma(1)} \mbc[\mbn A^s]_\tau\right) \simeq H^0(L^\bullet_\sigma \otimes_T M)
\]
where $I' \subset \mbc[\mbn A^s]$ is the ideal generated by $\{y^{\underline{a}^s_i} \mid \mbr_{\geq 0} \underline{a}^s_i \in \mcf_\sigma(1)\}$. In order to show \eqref{eq:H0loc} we have to show that $rad\, I' =  I$ since obviously
$H^0_{I'}(\dC[\dN A^s])=H^0_{rad I'}(\dC[\dN A^s])$.
Since $I' \subset I$ and $I = rad\, I$ ($I$ is a prime ideal corresponding to the face spanned by $\underline{a}_0^s, \ldots , \underline{a}_n^s$), it is enough to check that a multiple of  every $y^{\underline{c}} \in I$ lies in $I'$. But this follows easily from the fact that the elements $\{\underline{a}^s_i \mid \mbr_{\geq 0}\underline{a}_i^s \in \mcf_\sigma(1)\}$ span the same cone over $\mbq$ as the elements $\{\underline{a}^s_0, \ldots , \underline{a}^s_n\}$.

The proof  that $N \mapsto H^k(L^\bullet_\sigma \otimes_T N)$ is a $\delta$-functor is completely parallel to the proof in \cite{HerzBr}.
\end{proof}

Notice that the complex $L^\bullet_\sigma$ is $\mbz^{d+2}$-graded since $\mbc[\mbn A^s]$ is $\mbz^{d+2}$-graded. In order to analyze the cohomology of $L^\bullet_\sigma$ we look at its $\mbz^{d+2}$-graded parts.
For this we have to determine when $(\mbc[\mbn A^s]_\tau)_x \neq 0$ (and therefore $(\mbc[\mbn A^s]_\tau)_x \simeq \mbc$) for $x \in \mbz^{d+2}$.\\

We are following \cite[Chapter 6.3]{HerzBr}. Denote by $C_{A^s}$ the cone $\mbr_{\geq 0} A^s \subset \mbr^{d+2}$. Let $x,y \in \mbr^{d+2}$. We say that $y$ is visible from $x$ if $y \neq x$ and the line segment $[x,y]$ does not contain a point $y' \in C_{A^s}$ with $y' \neq y$. A subset $S$ is visible from $X$ if each $v \in S$ is visible from $x$.\\

Recall that     the cone $C_{A^s}$ is given by the intersection of finitely many half-spaces
\[
H^+_\tau := \{ x \in \mbr^{d+2} \mid \langle n_\tau,x \rangle \geq 0 \} \qquad \tau \in \mcf(d+1)
\]
where $\mcf(d+1)$ is the set of $d+1$-dimensional faces (facets) of $C_{A^s}$. We set
\[
x^0 =\{ \tau \mid \langle n_\tau,x \rangle =0 \}, \qquad x^+ =\{ \tau \mid \langle n_\tau,x \rangle  > 0 \}, \qquad x^- =\{ \tau \mid \langle n_\tau,x \rangle  < 0 \}
\]

\begin{lemma}\cite[Lemma 6.3.2, 6.3.3]{HerzBr}\label{lem:vis}
\begin{enumerate}
\item A point $y \in C_{A^s}$ is visible from $x \in \mbr^{d+2} \setminus C_{A^s}$ if and only if $y^0 \cap x^ - \neq \emptyset$.
\item Let $x \in \mbz^{d+2}$ and let $\tau$ be a face of $C_{A^s}$. The $\mbc$-vector space $(\mbc[\mbn A^s]_\tau)_x$ is non-zero if and only if $\tau$ is not visible from $x$.
\end{enumerate}
\end{lemma}

Recall that the facet $\sigma \in \mcf(d+1)$ is spanned by $\underline{a}_0^s , \ldots, \underline{a}_n^s$.  It is the unique maximal element in the face lattice $\mcf_\sigma \subset \mcf$. Denote by $H_\sigma$ its supporting hyperplane (i.e. $\sigma = C_{A^s} \cap H_\sigma$ ) which is given by
\[
H_\sigma = \{ x \in \mbr^{d+2} \mid \langle n_\sigma, x \rangle = 0\}
\]
where $n_\sigma = (0,1,0,\ldots ,0)$. Let $\tau \in \mcf_\sigma$ be a $k$-dimensional face contained in $\sigma$ and set $I_\tau := \{i \mid \underline{a}_i^s \in \tau\}$. Notice that the vectors $\{\underline{a}_i^s \mid i \in I_\tau\}$ span the face $\tau$. This face $\tau$ gives rise to two other faces, namely its ''shadow'' $\tau^s$ which is spanned by the vectors $\{\underline{b}^s_i \mid \underline{a}^s_i \in \tau\}$ and the unique $k+1$-dimensional face $\tau^c$ which contains both $\tau$ and $\tau^s$. Let $\{\tau_1, \ldots, \tau_m\} = \mcf_\sigma(d)$ be the faces of dimension $d$ contained in $\sigma$, which give rise to the facets $\tau_1^c, \ldots, \tau_m^ c$.\\

\begin{example}\label{ex:locCohP1s}
Consider the matrix
\[
A^s = \left(\begin{matrix}1 & 1 & \phantom{-}1 & 0 & 0 & \phantom{-}0 \\ 0 & 0 & \phantom{-}0 & 1 & 1 & \phantom{-}1\\ 0 & 1 & -1 & 0 & 1 & -1 \end{matrix} \right)
\]
where the face $\sigma$ is generated by $(1,0,0),(1,0,1),(1,0,-1)$ and its shadow $\sigma^s$ is generated by $(0,1,0),(0,1,1),(0,1,-1)$. The facet $\tau^c$ is generated by $\tau$ and its shadow $\tau^s$.
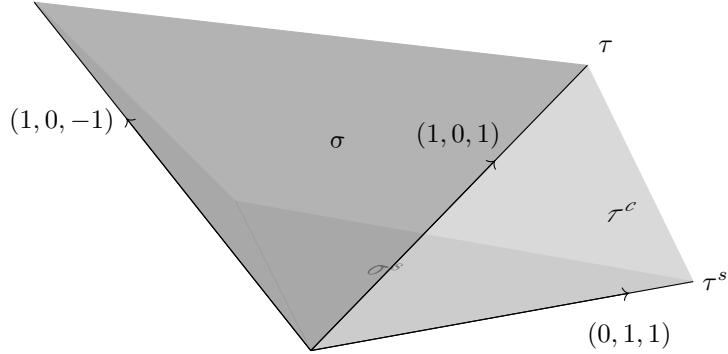
\begin{center}
\tikzset{
	MyPersp/.style={scale=2.8,rotate around z=90, rotate around x=255},
		}
\ifpdf
\begin{tikzpicture}[MyPersp]
\coordinate (A) at (0,0,0);
\coordinate (B) at (1.5,0,0);
\coordinate (C) at (1.5,0,1.5);
\coordinate (D) at (1.5,0,-1.5);
\coordinate (E) at (0,1.5,0);
\coordinate (F) at (0,1.5,1.5);
\coordinate (G) at (0,1.2,-1.2);
\draw (A)--(C);
\draw (A)--(D);
\draw (A)--(G);
\draw[->] (A)--(1,0,1);
\node[] at (.95,0,1.35) {$(1,0,-1)$};
\draw[->] (A)--(1,0,-1);
\node[] at (1.1,0,-.8) {$(1,0,1)$};
\draw[->] (A)--(0,1,-1);
\node[] at (-.2,1,-1) {$(0,1,1)$};
\node[opacity=.4] at (0,1,0.5) {\rotslant{-13}{30}{$\sigma^s$}};
\filldraw[black,opacity=.30] (A) -- (C) -- (D) -- (A);
\filldraw[black,opacity=.06] (A) -- (F) -- (G) -- (A);
\filldraw[black,opacity=.06] (A) -- (C) -- (F) -- (A);
\filldraw[black,opacity=.15] (A) -- (D) -- (G) -- (A);
\node[] at (1,0,0) {\rotslant{-6}{-12}{$\sigma$}};

\node[] at (1.6,0,-1.6){$\tau$};
\node[] at (-.03,1.3,-1.25){$\tau^s$};
\node[] at (0.55,0.55,-1.1){\rotslant{4}{30}{$\tau^c$}};
\end{tikzpicture}
\fi
\end{center}
\end{example}

First notice that by Lemma \ref{lem:vis}.1 the facet $\sigma$ is visible from a point $ x\in \mbr^{d+2}$ if and only if $\langle n_\sigma, x \rangle < 0$. If $\langle n_\sigma ,x \rangle \geq 0$ holds, it follows from Lemma \ref{lem:vis} 1. that a face $\tau_i\subset \sigma$ is visible from $x$ if and only if the facet $\tau_i^c$ is visible from $x$, i.e. $\langle n_{\tau_i^c}, x \rangle < 0$.\\

We define
\[
S := \mbz^{d+2} \cap \left(\mbr(\underline{a}_0^s , \ldots , \underline{a}_n^s) +  \mbr_{\geq 0}(\underline{b}_0^s , \ldots , \underline{b}_n^s) \right),
\]
this is the set of $\mbz^{d+2}$-degrees occurring in $\mbc[\mbn A^s]_{\sigma}$. Notice also that we have
$$
H_\sigma = \mbr(\underline{a}_0^s , \ldots , \underline{a}_n^s)
\quad\textup{ and } \quad
H^+_\sigma = \mbr(\underline{a}_0^s , \ldots , \underline{a}_n^s) +  \mbr_{\geq 0}(\underline{b}_0^s , \ldots , \underline{b}_n^s).
$$

Given a point $x \in S$ with $\langle n_\sigma ,x \rangle \geq 0$ we will construct a point $y_x \in \mbz^{d+2}$ which lies in $H_\sigma$ such that  $\tau_i$ is visible from $x$ if and only if it is visible from $y_x$ for all $i=1,\ldots, m$. Denote by $z_x$ the projection of $x$ to the sub-vector space generated by $\underline{b}_0^s , \ldots , \underline{b}_n^s$. Since the semi-group generated by these vectors is saturated, we can express $z_x$ as a linear combination with positive integers
\[
z_x = \sum_{i=0}^n r^x_i \underline{b}_i^s \qquad \text{with} \quad r^x_i \in \mbn
\]
Since we have $0 = \langle n_{\tau_i^c}, \underline{a}_j^s - \underline{b}_j^s\rangle = \langle n_{\tau_i^c}, (1,-1,0,\ldots,0)\rangle$ for any $\underline{a}_j^s, \underline{b}_j^s \in \tau_i^c$ the first two components of the vector $n_{\tau_i^c}$ are equal. Hence, if we set
\[
y_x := x + \sum_{j=0}^n r_j^x \underline{a}_j^s - \sum_{i=0}^n r_j^x \underline{b}_j^s
\]
we easily see that
\begin{equation}\label{eq:xprojection}
\langle n_{\tau_i^c}, x \rangle = \langle n_{\tau_i^c}, y_x\rangle \qquad \text{for} \quad i=1, \ldots ,m.
\end{equation}
It follows that $\tau_i$ is visible from any point $x\in S$ if and only if it is visible from $y_x$, as required.
Let us remark that the vectors $x$ and $y_x$ differ only in the first two components, because the same is true
for the pair of vectors $(\underline{a}_i^s,\underline{b}_i^s)$ for all $i\in\{0,\ldots, n\}$.
\begin{lemma}\label{lem:CohomIshidaDoesNotChange}
In the above situation, let $x\in S$. Then $y_x \in S\cap H_\sigma$ and we have
$$
(L_\sigma^\bullet)_x = (L_\sigma^\bullet)_{y_x}
$$
\end{lemma}
\begin{proof}
For the first point, notice that the vector $x-z_x$ is precisely the projection of $x$ to $H_\sigma$.
On the other hand, we have $y_x=x-z_x+\sum_{i=0}^n n_i^x \underline{a}_i^s$, and
 $\sum_{i=0}^n n_i^x \underline{a}_i^s$ is an element of $H_\sigma$ anyhow.

The second statement is an easy consequence of Lemma \ref{lem:vis}.2. More precisely,
Equation \eqref{eq:xprojection} shows that the visibility of some facet $\tau^c_i$ is the same
from $x$ and from $y_x$. Moreover, $\sigma$ is not visible from both $x$ and $y_x$ (i.e.,
$\langle n_\sigma,x\rangle \geq 0$, $\langle n_\sigma,y_x\rangle \geq 0$), hence, also the
visibility of $\tau_i$ is the same from $x$ and from $y_x$. We conclude that
any localization $\dC[\dN A^s]_\tau$ (for any face $\tau\subset \sigma$) vanishes in degree $x$ if and only if
it vanishes in degree $y_x$. This yields the desired equality $(L_\sigma^\bullet)_x = (L_\sigma^\bullet)_{y_x}$.
\end{proof}

We are now able to compute the cohomology of the Ishida complex with respect to the face $\sigma$. Set
\[
H^-_{\tau^c_i} := \{ x \in \mbr^{d+2} \mid \langle n_{\tau^c_i},x\rangle < 0\} \qquad \text{for } i =1,\ldots,m
\]
and define
\[
S^- := \mbz^{d+2} \cap H^+_\sigma \cap \bigcap_{i=1}^m H^-_{\tau^c_i}.
\]
Notice that $S= \dZ^{d+2} \cap H^+_\sigma$, hence we have a natural inclusion
$S^- \subset S$.
\begin{example}
We consider again the matrix
\[
A^s = \left(\begin{matrix}1 & 1 & \phantom{-}1 & 0 & 0 & \phantom{-}0 \\ 0 & 0 & \phantom{-}0 & 1 & 1 & \phantom{-}1\\ 0 & 1 & -1 & 0 & 1 & -1 \end{matrix} \right)
\]
and take the point $x=(-1,1,1)$. Its projection to $\mbr(\underline{b}_0^s, \ldots, \underline{b}^s_n)$ is $(0,1,1)$, hence we get
\[
y_x = x + (1,0,1)-(0,1,1) = (0,0,1) \in H^\sigma.
\]

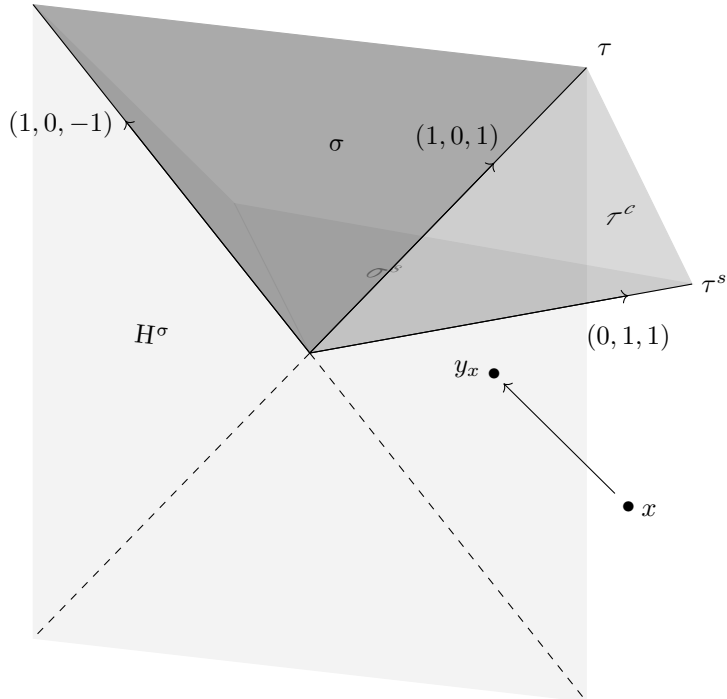
\ifpdf
\begin{center}
\tikzset{
	MyPersp/.style={scale=2.8,rotate around z=90, rotate around x=255},
		}
\begin{tikzpicture}[MyPersp]
\coordinate (A) at (0,0,0);
\coordinate (B) at (1.5,0,0);
\coordinate (C) at (1.5,0,1.5);
\coordinate (D) at (1.5,0,-1.5);
\coordinate (E) at (0,1.5,0);
\coordinate (F) at (0,1.5,1.5);
\coordinate (G) at (0,1.2,-1.2);
\draw (A)--(C);
\draw (A)--(D);
\draw (A)--(G);
\draw[->] (A)--(1,0,1);
\node[] at (.95,0,1.35) {$(1,0,-1)$};
\draw[->] (A)--(1,0,-1);
\node[] at (1.1,0,-.8) {$(1,0,1)$};
\draw[->] (A)--(0,1,-1);
\node[] at (-.2,1,-1) {$(0,1,1)$};
\node[] at (-1,1,-1) {$\bullet$};
\node[] at(-1.05,1.1,-1.04) {$x$};
\node[] at (0,0,-1) {$\bullet$};
\node[] at (0.1,-0.2,-1) {$y_x$};
\draw[->] (-0.9,0.9,-1) --(-0.07,0.07,-1);
\filldraw[black,opacity=.05] (1.5,0,1.5) --(1.5,0,-1.5) -- (-1.5,0,-1.5) -- (-1.5,0,1.5) -- cycle;
\draw[dashed] (0,0,0) -- (-1.5,0,-1.5);
\draw[dashed] (0,0,0) -- (-1.5,0,1.5);
\node[opacity=.4] at (0,1,0.5) {\rotslant{-13}{30}{$\sigma^s$}};
\filldraw[black,opacity=.30] (A) -- (C) -- (D) -- (A);
\filldraw[black,opacity=.06] (A) -- (F) -- (G) -- (A);
\filldraw[black,opacity=.06] (A) -- (C) -- (F) -- (A);
\filldraw[black,opacity=.15] (A) -- (D) -- (G) -- (A);
\node[] at (1,0,0) {\rotslant{-6}{-12}{$\sigma$}};
\node[] at (0,0,1) {\rotslant{-6}{-12}{$ H^\sigma$}};

\node[] at (1.6,0,-1.6){$\tau$};
\node[] at (-.03,1.3,-1.25){$\tau^s$};
\node[] at (0.55,0.55,-1.1){\rotslant{4}{30}{$\tau^c$}};
\end{tikzpicture}
\end{center}
\fi
\end{example}

\begin{proposition}\label{prop:Ishidadeg}
Let $A^s$ as above. Take any $x\in \dZ^{d+2}$ and denote by $L^\bullet_\sigma$ the Ishida complex
with respect to the face $\sigma$ generated by $\underline{a}^s_0, \underline{a}^s_1, \ldots, \underline{a}^s_n$.
\begin{enumerate}
\item If $x \notin S$, then $(L^\bullet_\sigma)_x = 0$.
\item If $x \in S \setminus S^-$, then $H^ i(L^\bullet_\sigma)_x = 0$ for all $i$
\item If $x \in S^-$, then $H^i(L^\bullet_\sigma)_x=0$ for $i \neq d+1$ and $H^{d+1}(L^\bullet_\sigma)_x \simeq \dC$.
\end{enumerate}
\end{proposition}
\begin{proof}
The first point follows from the fact that we have $(\mbc[\mbn A^s]_\sigma)_x = 0$ for $x \notin S$, hence $(L^i_\sigma)_x = 0 $ for all $i$.
For the proof of the second and third point, it is sufficient to consider the case where $x\in H_\sigma$: Namely, in both cases
we have $x\in S$ so that Lemma \ref{lem:CohomIshidaDoesNotChange} apply. We can thus replace $x$ by $y_x$, i.e., $(L^\bullet_\sigma)_x=
(L^\bullet_\sigma)_{y_x}$. Moreover, $x \in S^-$ if and only if $y_x \in S^- \cap H_\sigma$ by formula \eqref{eq:xprojection}. Hence we will suppose
in the remainder of this proof that $x\in S\cap H_\sigma$.

We will reduce statements 2. and 3. for $x\in S\cap H_\sigma$ to the computation of the
local cohomology of a semi-group ring with respect to a maximal ideal via the Ishida complex
as done in \cite[Theorem 6.3.4]{HerzBr}.
For this, we will use the matrix $\widetilde{A}=\left(\widetilde{\underline{a}}_0,\widetilde{\underline{a}}_1,\ldots,\widetilde{\underline{a}}_n\right)$,
which can be seen as the matrix of the first $n+1$ columns of $A^s$, with the second row deleted.
The semigroup $\mbn \widetilde{A}$ (resp. the cone $C_{\widetilde{A}}$) embeds into $\mbn A^s$ (resp. into $C_{A^s}$) via the map
$\widetilde{\underline{a}}_i \mapsto \underline{a}^s_i$, and these embeddings are compatible
with the embeddings $\dR^{r+1} \hookrightarrow \dR^{r+2}$ (resp. $\dZ^{r+1} \hookrightarrow \dZ^{r+2}$)
given by
$$
(x_1,x_3,x_4\ldots,x_{d+2}) \longmapsto (x_1,0,x_3,x_4,\ldots,x_{d+2}).
$$
The following equality of semi-groups holds true:
\begin{equation}\label{eq:Sminus}
S^-\cap H_\sigma = \dZ^{d+1}\cap \textup{Int}\left(-C_{\widetilde{A}}\right),
\end{equation}
where both intersections are taken in $\dZ^{d+2}$. To show this, notice
that $C_{\widetilde{A}}=\dR_{\geq 0}(\underline{a}_0^s, \ldots, \underline{a}_n^s) \cap H_\sigma$, that $\tau_i^c \cap H_\sigma = \tau_i$ and hence
$$
\bigcap_{i=1}^m \left(H_{\tau_i^c}^- \cap H_\sigma \right) = \textup{Int}(-C_{\widetilde{A}})
$$

Consider the projection map
\begin{align}
p: \mbr^{d+2} &\lra \mbr^{d+1} \notag \\
 (x_1,x_2,x_3,\ldots,x_{d+2}) &\mapsto (x_1,x_3,\ldots,x_{d+2}) \notag
\end{align}
which forgets the second component, then for all $\tau\subset \sigma$ and all elements $x\in S\cap H_\sigma$ we have that
\[
\left(\mbc[\mbn A^s]_\tau\right)_x \simeq \left(\mbc[\mbn \widetilde{A}]_{p(\tau)}\right)_{p(x)}
\]
Under this isomorphism the $\mbz^{d+2}$-graded part $(L^\bullet_\sigma)_x$ of the Ishida complex with respect to the face $\sigma$  goes over to the $\mbz^{d+1}$-graded part $(L^\bullet)_{p(x)}$ of the Ishida complex considered in \cite{HerzBr} (i.e., the Ishida complex
of the semi-group $\dC[\dN \widetilde{A}]$ with respect to the maximal ideal generated by
$(w_0,\ldots, w_n)$). Using formula \eqref{eq:Sminus}, the proposition follows now from Theorem 6.3.4 in loc.cit.

\end{proof}
We finish this subsection by the following easy consequence, which will be crucial in the proof of the main result (Theorem \ref{thm:HodgeGKZ} below).
\begin{corollary}\label{cor:IshidaNegative}
In the above situation, we have $H^i(L^\bullet_\sigma)=0$ for
all $i\neq d+1$, $H^{d+1}(L^\bullet_\sigma)_x=0$ for all $x\in \dZ^{d+2}\backslash S^-$ and $\deg_\dZ(H^{d+1}(L^\bullet_\sigma)_x)< 0$
for $x\in S^-$, where $\deg_\dZ(-)$ refers to the $\dZ$-grading of $H^i(L^\bullet_\sigma)$ corresponding to the first row of $A^s$.

In other words, the cohomology groups of the Ishida complex (with respect to the face $\sigma$) are concentrated in negative degrees.
\end{corollary}
\begin{proof}
The first two statements are precisely those from Proposition \ref{prop:Ishidadeg}, points 1. and 2.
In order to show the third one, notice that for any $x\in S$, we have
$\deg_\mbz(x) \leq \deg_\mbz(y_x)$ (this follows from the very definition of the vector $y_x$).
Now let $x\in S^-$, and suppose that $H^{d+1}(L^\bullet_\sigma)_x \neq 0$.
From Lemma \ref{lem:CohomIshidaDoesNotChange} we deduce that
$$
H^{d+1}(L^\bullet_\sigma)_x  = H^{d+1}(L^\bullet_\sigma)_{y_x},
$$
and as already remarked above, $y_x\in S^-\cap H_\sigma$ because $x\in S^-$. However,
we deduce from formula \eqref{eq:Sminus} that $\deg_\dZ(y_x)<0$ if $y_x\in S^-\cap H_\sigma$,
so that we obtain $\deg_\dZ(H^{d+1}(L^\bullet_\sigma)_x)<0$, as required.
\end{proof}

\subsection{Statement and proof of the main theorem}\label{subsec:MainThm}

In  this subsection we finally finish the computation of the direct image $\pi_{2+} \msn$. The $\Gamma$-cyclicity of $\mss/\msj$  (cf. \ref{cor:SJGacyclic}) which follows directly from the results of the previous two subsections enables us to compute the global sections of $\pi_{2+} \msn$ as the cohomology of a Koszul complex (cf. Proposition \ref{prop:directImKos}). Finally we are able to compute the Hodge filtration on the GKZ-system in Theorem \ref{thm:HodgeGKZ}.

\begin{corollary}\label{cor:SJGacyclic}
The  $\mss$-modules $\mss/\msj$ are $\Gamma$-acyclic.
\end{corollary}
\begin{proof}
If we consider the degree zero part of formula \eqref{eq:RGammavsLoc}, then it suffices to show that
the $\mbz$-graded local cohomology $S$-modules $H_{(w)}^*(S/J_{A_s})$ are concentrated in negative degrees.
By Proposition \ref{prop:LocCohIshida} and Lemma \ref{lem:tensorprod}, these local cohomology groups
are calculate by the Ishida complex $L^\bullet_\sigma$, i.e., we have isomorphisms
$$
H_{(w)}^k(S/J_{A_s}) \simeq S \otimes_T H^k_I(\mbc[\mbn A^s]) \simeq S \otimes_T H^k(L^\bullet_\sigma).
$$
The cohomology groups $H^k(L^\bullet_\sigma)$ are concentrated in negative degrees by Corollary \ref{cor:IshidaNegative}
(and tensoring with $S$ does not change the $\dZ$-degree which is counted with respect to the degree of the variables $w_0,\ldots,w_n$).
Hence the result follows.
\end{proof}
\begin{proposition}\label{prop:directImKos}
There is the following isomorphism in $D^b(\msr_V)$:
\[
\Gamma \pi_{2+} \msn \simeq \Gamma (Kos^{\bullet}(z^{-d} \mss / \msj, (\widetilde{E}-\beta_k)_{k=0,\ldots,d})
\]
\end{proposition}
\begin{proof}
By formula \eqref{eq:pi2Rmoddirectim}, Proposition \ref{prop:quasiIsos} and Proposition \ref{prop:LocalDescriptionComplL} we have the isomorphisms
\begin{align}
\Gamma \pi_{2+} \msn &\simeq \Gamma R \pi_{2*} ( \Omega^{\bullet+n}_{\msp \times \msv / \msv} \otimes \msn)\notag \\ &\simeq R\Gamma R\pi_{2*}( \Omega^{\bullet+n}_{\msp \times \msv / \msv} \otimes \msn) \notag \\
&\simeq R\Gamma ( \Omega^{\bullet+n}_{\msp \times \msv / \msv} \otimes \msn)\notag \\
&\simeq R\Gamma(\msl^\bullet)
\end{align}
Using the last isomorphism in \eqref{eq:GLqisGKos} and Corollary \ref{cor:SJGacyclic} we get
\[
R\Gamma(\msl^\bullet) \simeq R\Gamma(Kos^\bullet(z^{-d}\mss / \msj, (\tilde{E})_{k=0,\ldots,d})) \simeq \Gamma(Kos^\bullet(z^{-d}\mss / \msj, (\tilde{E})_{k=0,\ldots,d}))
\]
\end{proof}

Denote by $\textup{R}_\msv$ the ring
\[
\textup{R}_\msv = \mbc[z, \lambda_0, \ldots , \lambda_n]\langle z \p_{\lambda_0}, \ldots , z \p_{\lambda_n} \rangle \, ,
\]
let $J^\lambda_{\widetilde{A}} \subset \textup{R}_\msv$ be the left ideal generated by
\[
\Box^{\lambda}_{\underline{l}} = \prod_{l_i > 0} (z\p_{\lambda_i})^{l_i} - \prod_{l_i < 0} (z\p_{\lambda_i})^{- l_i} \qquad \text{for} \quad \underline{l} \in \mbl_{\widetilde{A}}
\]
and let $I^\lambda_{\widetilde{A}} \subset \textup{R}_\msv
$ be the left ideal generated by $J^\lambda_{\widetilde{A}}$ and the operators
\begin{align}
&\tilde{E}_k -\beta_k := \sum_{i=1}^n a_{ki}  \lambda_i z\p_{\lambda_i}-\beta_k  \qquad \text{for}\quad k=1, \ldots ,d \notag \\
&\tilde{E}_0-\beta_0 := \sum_{i=0}^n \lambda_i z\p_{\lambda_i}-\beta_0 \notag
\end{align}
\begin{lemma}
There is the following isomorphism of $\textup{R}_\msv$-modules
\[
\Gamma\, \mch^0 (\pi_{2+} \msn ) \simeq \mch^0( \Gamma \pi_{2 + } \msn) \simeq z^{-d}\textup{R}_\msv / I^\lambda_{\widetilde{A}}
\]
\end{lemma}
\begin{proof}
The first isomorphism follows from Lemma \ref{lem:equivcatD-R}. The second isomorphism follows from Proposition \ref{prop:directImKos}, the isomorphism
\[
\Gamma(\mss/\msj) \simeq  \textup{R}_\msv / J^\lambda_{\widetilde{A}}
\]
and the isomorphism
\[
z^{-d}\textup{R}_\msv  / I^\lambda_{\widetilde{A}} \simeq H^0\left(Kos^\bullet\left(z^{-d}\textup{R}_\msv / J^\lambda_{\widetilde{A}}, (\tilde{E}_k - \beta_k)_{k=0,\ldots,d}\right)\right)
\]
\end{proof}
We are now able to prove the main theorem of this paper. Let $\widetilde{A}$ be the $(d+1)\times(n+1)$ integer matrix
\[
\widetilde{A} = \left(\begin{matrix}
1 & 1 & \ldots & 1 \\ 0 & a_{11} & \dots & a_{1n} \\ \vdots & \vdots & &\vdots \\ 0 & a_{d1} & \dots &a_{dn}
\end{matrix} \right)
\]
given by a matrix $A = (a_{jk})$ such that $\mbz \widetilde{A} = \mbz^{d+1}$ and such that $\mbn \widetilde{A} = \mbz^{d+1} \cap \mbr_{\geq 0} \widetilde{A}$.
\begin{theorem}\label{thm:HodgeGKZ}
Let $\widetilde{A}$ be an integer matrix as above, $\widetilde{\beta} \in \mathfrak{A}_{\widetilde{A}}$ and $\beta_0 \in (-1,0]$. The GKZ-system $\mcm_{\widetilde{A}}^{\widetilde{\beta}}$ carries the structure of a mixed Hodge module whose Hodge filtration is given by the shifted order filtration, i.e.
\[
(\mcm^{\widetilde{\beta}}_{\widetilde{A}}, F^H_\bullet) \simeq (\mcm^{\widetilde{\beta}}_{\widetilde{A}},F^{ord}_{\bullet+d})\, .
\]
\end{theorem}
\begin{proof}
Recall from Proposition \ref{prop:GKzeqtwRad} that we have the isomorphism
\[
{^H\!\!}\mcm^{\widetilde{\beta}}_{\widetilde{A}} \simeq \mch^{2n+d+1} (p_*( q^* {^p}\mbc_T^\beta \otimes F^* j_{!} {^p}\mbc^{-\beta_0-1}_{\mbc^*}) \in \MHM(V).
\] The underlying $\cD_V$-module of this mixed Hodge module is
$$
\mch^{2n+d+1}(p_{+} (q^\dag \mco_T^\beta \otimes_\mco   F^\dag (j_{\dag} \mco_{\mbc^*}^{-\beta_0-1}))) \simeq \mch^{0} (\pi_{2+} \mcn).
$$
We have already computed the Hodge filtration of $\mcn$
(more precisely, we have computed it on the restrictions $\mcn_u$ of $\cN$ to each chart $W_u \times V$ in Proposition \ref{prop:OriginalCoord}). In order to compute the Hodge filtration under the direct image of $\pi_2$, we will use the results obtained above and read off the Hodge filtration from the corresponding $\msr_V$-module $\mst(\mcm^{\widetilde{\beta}}_{\widetilde{A}},F^{H})$. We have the following isomorphisms
\[
\Gamma \mst(\mcm^{\widetilde{\beta}}_{\widetilde{A}},F^{H})\simeq  \Gamma \mst(\mch^0(\pi_{2+} \mcn, F^H))\simeq \Gamma\, \mch^0( \pi_{2+} \msn) \simeq z^{-d}\textup{R}_\msv / I^\lambda_{\widetilde{A}}
\]
Using these isomorphisms the claim follows easily.
\end{proof}

\subsection{Duality}\label{sec:Duality}

For applications like the one presented in the next section, it will be useful to extend the
computation of the Hodge filtration on  $\cM_{\widetilde{A}}^\beta$
to the dual Hodge module $\bD \cM_{\widetilde{A}}^\beta$. This is possible under the assumption made in the above main theorem (Theorem \ref{thm:HodgeGKZ})
plus the extra requirement that the semi-group ring $\dC[\dN \widetilde{A}]$ is Gorenstein.
More precisely, it follows from \cite{Walther1}, that under these assumptions, the $\cD_V$-module
$\bD \cM_{\widetilde{A}}^\beta$ is still a GKZ-system. Hence it is reasonable to expect that
its Hodge filtration will also be the order filtration up to a suitable shift.

The Gorenstein condition for normal semi-group rings has a well-known combinatorial expression (see \cite[Corollary 6.3.8]{HerzBr}), namely,
$\dC[\dN \widetilde{A}]$ is Gorenstein if and only if there is a vector $\widetilde{c}$ such that the set of interior points
$int(\dN \widetilde{A})$ (i.e., the intersection of $int\left(\dR_{\geq 0} A\right) \cap \dZ^{d+1}$
is given by $\widetilde{c}+\dN \widetilde{A}$.

\begin{theorem}\label{thm:Duality}
Suppose that $\widetilde{A}\in M(d+1 \times n+1,\dZ)$ is such that
$\dZ \widetilde{A} =\dZ^{d+1}$,  $\dN \widetilde{A} = \dZ^{d+1} \cap \dR_{\geq 0} \widetilde{A}$
and such that $int(\dN \widetilde{A}) = \widetilde{c}+\dN \widetilde{A}$ for some $\widetilde{c} =(c_0,c)\in \dZ^{d+1}$ and $\widetilde{\beta} \in \mathfrak{A}_A$.
Then we have
$$
\bD \cM_{\widetilde{A}}^{\widetilde{\beta}} \simeq \cM_{\widetilde{A}}^{-\widetilde{\beta}-\widetilde{c}},
$$
and the Hodge filtration on $\bD \cM_{\widetilde{A}}^{\widetilde{\beta}}$ is the order filtration,
shifted by $n+c_0$, i.e., we have
$$
F^H_p \bD \cM_{\widetilde{A}}^{\widetilde{\beta}} \simeq F^{ord}_{p-n-c_0} \cM_{\widetilde{A}}^{-\widetilde{\beta}-\widetilde{c}}\, .
$$
\end{theorem}
\begin{proof}
The proof is very much parallel to \cite[Proposition 2.19]{ReiSe} resp. \cite[Theorem 5.4]{ReiSe2}, we will
give the main ideas here once again for the convenience of the reader. We again work with the modules of global sections,
and write $D_V:=\dC[\lambda_0,\ldots,\lambda_n]\langle \partial_{\lambda_0},\ldots,\partial_{\lambda_n} \rangle$ and
$S_{\widetilde{A}}$ for the commutative ring $\dC[ \partial_{\lambda_0},\ldots,\partial_{\lambda_n}]/(\Box_{\underline{l}})_{\underline{l}\in\dL_{\widetilde{A}}}$
These rings are $\dZ^{d+1}$-graded by $\deg(\lambda_i)=-\widetilde{a}_i$, $\deg(\partial_{\lambda_i})=\widetilde{a}_i$.

In order to calculate $\bD M_{\widetilde{A}}^{\widetilde{A}}$ together
with its Hodge filtration, we need to find a strictly filtered free resolution
$(L_\bullet, F_\bullet) \stackrel{\simeq}{\twoheadrightarrow} (M_{\widetilde{A}}^{\widetilde{\beta}},F^H_\bullet) = (M_{\widetilde{A}}^{\widetilde{\beta}},F^{ord}_{\bullet+d}) $.
We have already used in the previous sections of this paper resolutions of ``Koszul''-type for various (filtered)
$\cD$-modules. Here we consider the Euler-Koszul complex
$$
K^\bullet:=\textup{Kos}\left(D_V\otimes_{\dC[ \underline{\partial}_{\lambda}]} S_{\widetilde{A}}, (E_k-\beta_k)_{k=0,\ldots,d}\right),
$$
as defined in section \ref{subsec:GKZ} and a generalization to $\mbz^{d+1}$-graded $\mbc[\underline{\p}_{\lambda}]$-modules (for details see \cite{MillerWaltherMat}).

A free resolution of $M_{\widetilde{A}}^{\widetilde{\beta}}$ is constructed as follows:
Take a $\dC[\underline{\partial}]$-free graded resolution of $T^\bullet \twoheadrightarrow S_{\widetilde{A}}$,
and define $L^\bullet$ to be the total
complex $\textup{Tot}\left(K^\bullet(E-\beta, D_V\otimes_{\dC[ \underline{\partial}_{\lambda}]} T^\bullet) \right)$. Notice
that the double complex $K^\bullet(E- \beta, D_V\otimes_{\dC[ \underline{\partial}_{\lambda}]} T^\bullet)$ exists since
$K^\bullet(E-\beta, D_V\otimes_{\dC[\underline{\partial}]} - ) $ is a functor from the category of $\dZ^{d+1}$-graded $\dC[\underline{\partial}]$-modules
to the category of (bounded complexes of) $\dZ^{d+1}$-graded $D_V$-modules. Then we have $L^{-k}=0$ for all
$k>n+1$ (notice that the length of the Euler-Koszul complexes is $d+1$, and the length of the resolution $T_\bullet \twoheadrightarrow P$ is
$n-d+1$, hence the total complex has length $(d+1)+(n-d+1)-1=n+1$). Moreover, the last term $L^{-n-1}$ of this complex is simply equal to $D_V$
(and so is the first one $L^0$).

As we have $int(\dN \widetilde{A}) = \widetilde{c}+\dN \widetilde{A}$, the ring $\dC[\dN \widetilde{A}]\simeq
S_{\widetilde{A}}$ is Gorenstein, more precisely, we have $\omega_{S_{\widetilde{A}}} \simeq S_{\widetilde{A}}(\widetilde{c})$, where $\omega_{S_{\widetilde{A}}}$ is the canonical module of $S_{\widetilde{A}}$.
Then a spectral sequence argument (see also \cite[Proposition 4.1]{Walther1}), using
$$
\textit{Ext}_{\dC[\underline{\partial}]}^i\left(S_{\widetilde{A}},\omega_{\dC[\underline{\partial}]}\right) \simeq
\left\{
\begin{array}{rcl}
0 & \textup{ if } i< n-d \\
S_{\widetilde{A}}\left(\widetilde{c}\right) & \textup{ if } i=n-d
\end{array}
\right.
$$
shows that
$$
\bD M_{\widetilde{A}}^{\widetilde{\beta}} \cong M_{\widetilde{A}}^{-\widetilde{\beta}-\widetilde{c}}.
$$

In order to calculate the Hodge filtration on $M^{-\widetilde{c}}_{\widetilde{A}}$, we remark that the
Euler-Koszul complex is naturally filtered by putting
\[
F_p K^{-l} := \bigoplus_{0 \leq i_1 < \ldots < i_l \leq l} F^{ord}_{p+d-l}\left(D_V\otimes_{\dC[\underline{\partial}]} S_{\widetilde{A}} \right)e_{i_1 \ldots i_l}.
\]

Notice that $D_V\otimes_{\dC[\underline{\partial}]} S_{\widetilde{A}} \simeq D_V/(\Box_{\underline{l}})_{\underline{l}\in\dL_{\widetilde{A}}}$, so that this $D_V$-module
has an order filtration induced from $F^{ord}_\bullet D_V$. In order to show that $(K^\bullet, F_\bullet) \twoheadrightarrow (M_{\widetilde{A}}^{\widetilde{\beta}}, F^H)$ is
a filtered quasi-isomorphism, it suffices (by Lemma \ref{lem:equivChar}) to show that $\gr^F_\bullet K^\bullet \twoheadrightarrow \gr^{F^H}_\bullet M_{\widetilde{A}}^{\widetilde{\beta}}$ is a quasi-isomorphism. This follows from \cite[Formula 4.32, Lemma 4.3.7]{SST}, as $\dC[\dN \widetilde{A}]$ is Cohen-Macaulay
due to the normality assumption on $\widetilde{A}$.
The final step is to endow the free resolution $L^\bullet=\textup{Tot}\left(K^\bullet(E-\beta, D_V\otimes_{\dC[ \underline{\partial}_{\lambda}]} T^\bullet) \right)$ with a strict filtration $F_\bullet$ and to show that $(L_\bullet, F_\bullet) \stackrel{\simeq}{\twoheadrightarrow} (M_{\widetilde{A}}^{\widetilde{\beta}}, F^H_\bullet)$.
As the resolution $T_\bullet \twoheadrightarrow S_{\widetilde{A}}$ is taken in the category of $\dZ^{d+1}$-graded $\dC[\underline{\partial}]$-modules, the morphisms
of this resolution are homogenous for the ($\dZ$-)grading $\deg(\lambda_i)=-1$ and $\deg(\partial_{\lambda_i})=1$ (notice that this
is the grading  given by the first component of the $\dZ^{d+1}$-grading of the ring $D_V\otimes_{\dC[\underline{\partial}]} S_{\widetilde{A}}$)
. Hence these morphisms are naturally filtered for the order filtration $F^{ord}_\bullet (D_V\otimes_{\dC[\underline{\partial}]} S_{\widetilde{A}})$ and they are even strict: for a map given by homogenous operators from $\dC[\partial]$ taking the symbols has simply no effect, so that
$\gr^F_\bullet (D_V\otimes_{\dC[\underline{\partial}]} T_\bullet) \twoheadrightarrow \gr^{F^{ord}}_\bullet(D_V\otimes_{\dC[\underline{\partial}]} S_{\widetilde{A}})$ is a filtered quasi-isomorphism (and similarly for the
sums occuring in the terms $K^{-l}$). However, we have to determine the $\dZ$-degree (for the grading $\deg(\partial_{\lambda_i})=1)$
of the highest (actually, the only nonzero) cohomology module $Ext^{n-d}_{\dC[\underline{\partial}]}(S_{\widetilde{A}},\omega_{\dC[\underline{\partial}]})$:
it is the first component of the difference of the degree of $\omega_{\dC[\underline{\partial}]}$ (i.e, the first component of the sum of the columns
of $\widetilde{A}$), which is $n+1$, and the first component of the degree of $\omega_{S_{\widetilde{A}}}$, which is $c_0$. Now the shift of the filtration between $M_{\widetilde{A}}^{\widetilde{\beta}}$ and the dual module $M_{\widetilde{A}}^{-\widetilde{\beta}-(c_0,c)}$ is the sum of the length of the complex $K^\bullet(E-\beta,D_V\otimes_{\dC[\underline{\partial}]} S_{\widetilde{A}})$, i.e., $d+1$, and the above $\dZ$-degree of $Ext^{n-d}_{\dC[\underline{\partial}]}(S_{\widetilde{A}},\omega_{\dC[\underline{\partial}]})$, i.e. $n+1-c_0$. Hence the filtration $F_\bullet L^{-n-1}$
is again the shifted order filtration, more precisely, we have
$$
F_p L^{-n-1} = F^{ord}_{p+d-(d+1)-(n+1-c_0)} D_V = F^{ord}_{p-n-2+c_0} D_V\, .
$$
Now it follows from \cite[page 55]{SaitoonMHM} that
$$
\bD(M_{\widetilde{A}}^{\widetilde{\beta}}, F^H) \simeq
\mathit{Hom}_{D_V}\left(
\left(L^\bullet, F_\bullet\right),
\left((D_V\otimes \Omega^{n+1}_V)^\vee,F_{\bullet-2(n+1)}D_V\otimes (\Omega_V^{n+1})^\vee \right)
\right)
$$
so that we finally obtain
$$
F_p^H \bD M_{\widetilde{A}}^{\widetilde{\beta}} = F^{ord}_{p-n-c_0} M_{\widetilde{A}}^{-\widetilde{\beta} - (c_0,c)}.
$$
\end{proof}
We now consider the special case $\beta=0$. From Proposition \ref{prop:morphdualGKZ}, we know that up to multiplication by a non-zero constant, we have the morphism
\begin{align}
\phi:F^{ord}_{p+d-c_0}M^{-(c_0,c)}_{\widetilde{A}}=F_{p+n+d}^H \mbd(M^0_{\widetilde{A}}) = F_p^H\mbd(M^0_{\widetilde{A}})(-n-d) &\lra F_p^H M^0_{\widetilde{A}} = F_{p+d}^{ord}M^0_{\widetilde{A}} \notag \\
P &\mapsto P \cdot \p^{(c_0,c)} \notag
\end{align}
where $\p^{(c_0,c)} := \prod_{i=0}^n \p_{\lambda_i}^{k_i}$  for any $\underline{k} = (k_0, \ldots , k_n)$ with $\widetilde{A} \cdot \underline{k} = (c_0,c)$. Since $\widetilde{A}$ is homogeneous we have $\sum k_i = c_0$.
As a consequence, we obtain the following result.
\begin{corollary}\label{cor:StrictlyFiltered}
Under the above assumptions on $\widetilde{A}$, the morphism
$$
\begin{array}{rcl}
\phi:   (M^{-(c_0,c)}_{\widetilde{A}},F^{ord}_{\bullet-c_0}) & \longrightarrow & (M^0_{\widetilde{A}} , F_\bullet^{ord}) \\ \\
P &\longmapsto & P \cdot \p^{(c_0,c)}
\end{array}
$$
(where $\partial^{(c_0,c)}$ is as above) is strictly filtered.
\end{corollary}
\begin{proof}
Since both filtered modules
$(M^{-(c_0,c)}_{\widetilde{A}},F^{ord}_{\bullet-c_0})$ and $(M^0_{\widetilde{A}} , F_\bullet^{ord})$ underly mixed Hodge modules on $V$ by Theorem \ref{thm:HodgeGKZ} and the morphism
is induced from a morphism in $\MHM(V,\dC)$, we obtain the strictness statement we are looking for.
\end{proof}

\begin{remark}
If $\mbc[\mbn \widetilde{A}]$ is not Gorenstein but normal (and therefore Cohen-Macaulay) then the proof of Theorem \ref{thm:Duality} shows that
\[
\mbd M^{\widetilde{\beta}}_{\widetilde{A}} \simeq \mch^0\left(E+\beta,D_V \otimes Ext^{n-d}(S_{\widetilde{A}},\omega_{\mbc[\p_{\lambda}]})\right)\simeq \mch^0\left(E+\beta,D_V \otimes \omega_{S_{\widetilde{A}}}\right)
\]
Recall that the canonical module $\omega_{S_{\widetilde{A}}}$ of $S_{\widetilde{A}}$ is isomorphic to $\mbc[int(\mbn \widetilde{A})]$ in the category of $\mbz^{d+1}$-graded $\mbc[\underline{\p}_\lambda]$-modules. The module $\omega_{S_{\widetilde{A}}}$ carries a $\mbz$-grading given by the first component of the $\mbz^{d+1}$-grading. Hence $D_V \otimes_{\mbc[\underline{\p}_\lambda]} \omega_{S_{\widetilde{A}}}$ carries an order filtration which induces a filtration $F^{ord}_{\bullet}$ on $\mch^0(E+\beta,D_V \otimes \omega_{S_{\widetilde{A}}})$. We therefore get
\[
F_p^H \mbd \mcm^{\widetilde{\beta}} \simeq F^{ord}_{p-n-c_0} \mch^0(E+\beta,D_V \otimes \omega_{S_{\widetilde{A}}}).
\]

where $c_0 := \min\{ deg_\mbz(P) \mid P \in \mbc[int(\mbn A)]\}$. Let $\widetilde{c} \in deg(\mbc[int(\mbn \widetilde{A})]$ with $\widetilde{c} = (c_0,c)$. Similar to \cite[Proposition 4.4]{Walther1} it can be shown that the inclusion $S_{\widetilde{A}}[-\widetilde{c}] \hookrightarrow \mbc[int(\mbn \widetilde{A})]$ induces an isomorphism $M^{-\widetilde{\beta}-\widetilde{c}} \overset{\simeq}\lra \mch^0(E+\beta, D_V \otimes \omega_{S_{\widetilde{A}}})$ however we do not expect  $(M^{-\widetilde{\beta}-\widetilde{c}},F^{ord}_\bullet) \lra (\mch^0(E+\beta, D_V \otimes \omega_{S_{\widetilde{A}}}),F^{ord}_\bullet)$ to be a filtered isomorphism.
\end{remark}
\subsection{Hodge structures on affine hypersurfaces of tori}
\label{subsec:Batyrev}

In this subsection we explain how our main result implies in a rather direct way a classical theorem
of Batyrev concerning the description of the Hodge filtration of the relative cohomology of smooth affine hypersurfaces in algebraic tori.

We first want to recall the sheaf theoretic definition of relative cohomology. Let $X$ be a topological space and $K$ be a closed subset. Denote by $j:X \setminus K \ra X$ the open embedding of the complement. The relative cohomology of the pair $(X,K)$ is defined as the following hypercohomology:
\[
H^i(X,K;\mbc) := \mbh^i(X,j_! j^{-1} \mbq_X)
\]
If $X$ and $K$ are quasi-projective varieties the relative  cohomology of the pair $(X,K)$ carries a mixed Hodge structure, which is given by $\mbh^i(X,j_! j^{-1} \mbq_X^H)$.\\

We want to compute this in the following situation: Consider as in section \ref{sec:Radon} the family of Laurent polynomials $\varphi_A: T \times \Lambda \ra V = \mbc_{\lambda_0} \times \Lambda$, where $\dZ \widetilde{A}=\dZ^{d+1}$ and $\dN\widetilde{A}=\dZ^{d+1}\cap \dR_{\geq 0} \widetilde{A}$. Let $\Delta := \textup{Conv}(\underline{a}_0,\underline{a}_1,\ldots, \underline{a}_n)$ be the convex hull  of the exponents of $\varphi_A$, where $\underline{a}_0:=0$. Let $\tau \subset \Delta$ be a face of $\Delta$,  $x \in V$ and
\begin{align}
F_{A,x}^\tau  = \sum_{i : \underline{a}_i \in \tau} x_i \underline{t}^{\underline{a}_i} \notag
\end{align}

\begin{definition}
The fiber $\varphi^{-1}_A(x)$ is non-degenerate if for every face $\tau$ of $\Delta$ the equations
\[
F^\tau_{A,x} = t_1 \frac{\p F^\tau_{A,x}}{\p t_1}= \ldots = t_d \frac{\p F^\tau_{A,x}}{\p t_d} = 0
\]
have no common solution in $T$.
\end{definition}

Let $x \in V$ such that the fiber $\varphi_A^{-1}(x)$ is non-degenerate. We give a model of $H^i(T, \varphi_A^{-1}(x);\mbc)$ as the quotient of a graded semi-group ring and compute explicitly its Hodge filtration.
This recovers a result of Stienstra \cite[Theorem 7]{Sti} using results of Batyrev \cite{Bat4}.\\

\begin{lemma}
Let $x \in V$ and $i_x: \{x\} \ra V$ be the inclusion. Suppose $\varphi_A^{-1}(x)$ is non-degenerate. Then
\begin{enumerate}
\item the fiber $\varphi^{-1}_A(x)$ is smooth.
\item  the map $i_x$ is noncharacteristic with respect to $\mcm^0_{\widetilde{A}}$.
\end{enumerate}
\end{lemma}
\begin{proof}
The first statement follows directly from the definition for $\tau = \Delta$. The second statement follows from \cite[Lemma 3.3]{Adolphson}.
\end{proof}

Consider the following diagram (cf. diagram \eqref{eq:universaldiag})
\[
\begin{tikzcd}
{}Y \ar{rr} \ar[hook]{d}&& U \ar{rrd}{\pi_2^U}  \ar[hook]{d}{j_U} \\
T \times V \ar{rr}{g \times id} && \mbp(W) \times V  \ar{rr}{\pi_2} && V  \\
\Gamma \ar{rr} \ar[hook]{u} && Z  \ar[hook]{u}{i_Z} \ar{rru}[swap]{\pi_2^Z}
\end{tikzcd}
\]
where $Y$ resp. $\Gamma$ are the pull-backs such that both squares on the left are cartesian. Notice that $\Gamma$ is given as a subspace of $T \times V$ by the equation $\lambda_0 + \sum_{i=1}^n \lambda_i \underline{t}^{\underline{a}_i} = 0$; hence, $\Gamma$ is the graph of $\varphi_A$ and $Y$ is its complement in $T \times V$. Restricting this diagram to some $x \in V$ we therefore get
\begin{equation}\label{eq:AdjuncfiberCompl}
\begin{tikzcd}
{}T \setminus \varphi^{-1}(x) \ar{rr}{\overline{g}} \ar[hook]{d}[swap]{\overline{j}} && U_x \ar{rrd}{\pi_x^{U}}  \ar[hook]{d}{j} \\
T  \ar{rr}{g} && \mbp(W)  \ar{rr}{\pi_x} && \{x\}  \\
\varphi^{-1}(x) \ar{rr}{\widetilde{g}} \ar[hook]{u}{\tilde{i}} && Z_x  \ar[hook]{u}{i} \ar{rru}[swap]{\pi_x^Z}
\end{tikzcd}
\end{equation}
We will need the following statement
\begin{lemma}
Let $x \in V$ be such that $\varphi_A^{-1}(x)$ is smooth then we have an isomorphism in $D^b(\MHM(\mbp(W)))$:
\[
g_* \overline{j}_!\, \overline{j}^{-1}\, \mbq^H_T \simeq j_!\, \overline{g}_*\, \mbq^H_{T \setminus \varphi^{-1}(x)}
\]
\end{lemma}
\begin{proof}
The statement follows using the following chain of isomorphisms
\[
g_* \overline{j}_!\, \overline{j}^{-1}\, \mbq^H_T \;\overset{!}{\simeq}\; j_!\,j^{-1}\,g_*\, \mbq^H_T \;\simeq\; j_!\, \overline{g}_*\, \overline{j}^{-1}\, \mbq^H_T \;\simeq\; j_!\, \overline{g}_*\, \mbq^H_{T \setminus \varphi^{-1}(x)}
\]
where the second isomorphism follows from base change. It remains to show the first isomorphism. Notice that we have the following to triangles
\begin{align}
&j_! j^{-1} g_* \lra g_* \lra i_! i^* g_* \overset{+1}{\lra} \notag \\
&g_* \overline{j}_! \overline{j}^{-1} \lra g_* \lra g_* \overline{i}_! \overline{i}^* \overset{+1}{\lra} \notag
\end{align}
So it is enough to show $i_! i^* g_* \mbq^H_T \simeq g_* \tilde{i}_! \tilde{i}^* \mbq^H_T$. But this can be seen as follows:
\[
i_! i^* g_* \mbq^H_T \simeq i_! i^! g_* \mbq^H_T[2] \simeq i_! \tilde{g}_* \tilde{i}^!\mbq^H_T[2] \simeq i_! \tilde{g}_* \tilde{i}^* \mbq^H_T \simeq g_* \tilde{i}_! \tilde{i}^* \mbq^H_T
\]
where we use the smoothness of $\varphi^{-1}(x)$ in the first and third isomorphism.
\end{proof}

In order to proof the statement that the restriction of the GKZ-system is isomorphic to a relative cohomology group, we have to rewrite the GKZ-system as a Radon transform. For this, consider the following diagram
\[
\begin{tikzcd}
{}T \setminus \varphi^{-1}(x) \ar{rr}{\overline{g}} \ar[hook]{d}[swap]{\overline{i}_x} && U_x \ar{rr}{\pi_x^{U}}  \ar[hook]{d}{i_{U_x}} && \{x\} \ar{d}{i_x}\\
Y  \ar{rr} \ar{d}[swap]{\pi_1^Y} && U  \ar{d}{\pi_1^U}\ar{rr}{\pi_2^U} && V  \\
T \ar{rr}{g}  && \mbp(W)
\end{tikzcd}
\]
where all squares are cartesian.
\begin{proposition}
Let $x\in V$ be such that $\varphi^{-1}_A(x)$ is non-degenerate, then there is an isomorphism of mixed Hodge structures:
\[
i_x^* \,\mcm^0_{\tilde{A}} \simeq H^d(T, \varphi_A^{-1}(x);\mbc)
\]
and $H^i(T, \varphi_A^{-1}(x);\mbc) = 0$ for $i \neq d$.
\end{proposition}
\begin{proof}
Consider the following isomorphisms
\begin{align}
i_x^* \pi_{2!}^U (\pi_1^{U})^* g_* {^p}  \mbq^H_T
&\simeq \pi_{x!}^U i_{U_x}^* (\pi_1^{U})^*g_*{^p}\mbq^H_T \notag  &&\text{base change}\\
&\simeq \pi_{x!}^{U} (\pi_1^U \circ i_{U_x})^* g_* {^p}\mbq^H_T \notag \\
&\simeq \pi_{x!}^{U} (\pi_1^U \circ i_{U_x})^! g_* {^p}\mbq^H_T && (\pi_1^U \circ i_{U_x}) \; \text{open}\notag \\
&\simeq \pi_{x!}^{U} \overline{g}_*(\pi_1^Y \circ \overline{i}_{x})^!  {^p}\mbq^H_T \notag &&\text{base change} \\
&\simeq \pi_{x!}^{U} \overline{g}_*(\pi_1^Y \circ \overline{i}_{x})^*  {^p}\mbq^H_T &&(\pi_1^Y \circ \overline{i}_{x}) \; \text{open}\notag \\
&\simeq \pi_{x!}^{U} \overline{g}_* {^p} \mbq^H_{T \setminus \varphi^{-1}(x)} \notag
\end{align}

We can rewrite this further by looking at a part of diagram \eqref{eq:AdjuncfiberCompl}:
\[
\begin{tikzcd}
{}T \setminus \varphi^{-1}(x) \ar{rr}{\overline{g}} \ar[hook]{d}[swap]{\overline{j}} && U_x \ar{rrd}{\pi_x^{U}}  \ar[hook]{d}{j} \\
T  \ar{rr}{g} && \mbp(W)  \ar{rr}{\pi_x} && \{x\}
\end{tikzcd}
\]
We have
\[
\pi_{x*} g_* \overline{j}_!\, \overline{j}^{-1}\, {^p}\mbq^H_T \simeq \pi_{x*}j_!\, \overline{g}_*\, {^p}\mbq^H_{T \setminus \varphi^{-1}(x)} \simeq \pi_{x!}^{U} \overline{g}_* {^p} \mbq^H_{T \setminus \varphi^{-1}(x)} \simeq i_x^* \pi_{2!}^U (\pi_1^{U})^* g_* {^p}  \mbq^H_T
\]
where the last isomorphism follows from the calculation above.
If we take cohomology and keep in mind that $i_x$ is non-characteristic we get
\begin{align}
H^i(T,\varphi_A^{-1}(x);\mbc) &\simeq H^{i-d}(\pi_{x*} g_* \overline{j}_!\, \overline{j}^{-1}\, {^p}\mbq^H_T) \notag \\
&\simeq \mch^{i-d} i^*_x (\pi_{2!}^U (\pi_1^{U})^* g_* {^p}  \mbq^H_T)\notag \\
&\simeq i^*_x \mch^{i-d+n+1}(\pi_{2!}^U (\pi_1^{U})^* g_* {^p}  \mbq^H_T)\notag \\
&\simeq i^*_x \mch^{i-d+n+1}({^*}\mcr^\circ_c(g_* {^p}\mbq^H_T)) \notag
\end{align}
Since $\mch^k({^*}\mcr^\circ_c(g_* {^p}\mbq^H_T)) = 0$ for $k \neq n+1$ and
$\mch^k({^*}\mcr^\circ_c(g_* {^p}\mbq^H_T)) = \mcm^0_{\tilde{A}}$ for $k = n+1$ the claim follows.
\end{proof}

Denote by $S_{\widetilde{A}} := \mbc[\mbn \widetilde{A}] \subset \mbc[u_0^\pm,\ldots,u_d^\pm]$  the semigroup ring generated by $\underline{u}^{\tilde{\underline{a}}_0}, \ldots , \underline{u}^{\tilde{\underline{a}}_n}$ where $\widetilde{A} =(\tilde{\underline{a}}_0 , \ldots , \tilde{\underline{a}}_n)$ is the matrix from \ref{def:Atilde}.

Define the following differential operators
\[
D_k := \sum_{i=0}^n \left( \tilde{a}_{ki} x_i \underline{u}^{\widetilde{\underline{a}_i}} +  u_i \p_{u_i}\right) \qquad \text{for} \quad k=0,\ldots,d \quad \text{and fixed}\quad  x=(x_0,\ldots,x_n) \in V
\]
which act on $\mbc[u_0^\pm,\ldots,u_d^\pm]$ and which preserve $S_{\widetilde{A}}$. For $\underline{u}^{\underline{l}} =\underline{u}^{l_0 \cdot \widetilde{a}_0}\cdot \ldots \cdot \underline{u}^{l_n\cdot\widetilde{a}_n} \in S_{\widetilde{A}}$ we define the degree $deg(\underline{u}^{\underline{l}}) = \sum_{i=0}^n l_i$. Define a descending filtration $F^\bullet$ of $\mbc$-vector spaces on $S_{\widetilde{A}}$ where $F^{d+1} S_{\widetilde{A}}= 0$ and the filtration step $F^{d-k} S_{\widetilde{A}}$ is spanned by monomials $\underline{u}^{\underline{l}}$
with $deg(\underline{u}^{\underline{l}}) \leq k$.
\begin{thm}
Let $x \in V$ such that $\varphi_A^{-1}(x)$ is non-degenerate, then the following isomorphism of filtered vector spaces holds
\[
(i^*_x \mcm^0_{\widetilde{A}},F_\bullet^H) \simeq (S_{\widetilde{A}} / (D_k S_{\widetilde{A}})_{k=0,\ldots,d},F^{-\bullet})
\]
\end{thm}

\begin{proof}
Since we have assumed that $i_x$ is non-characteristic with respect to $\mcm^0_{\widetilde{A}}$ the only non-zero cohomology group of $i_x^* \mcm^0_{\widetilde{A}}$ is
\[
\mch^{0} i^*_x \mcm_{\widetilde{A}}^0 \simeq (\lambda_i-x_i)_{i=0,\ldots,n}\backslash\mcm_{\widetilde{A}}^0
\]

We define a $\mbc$-linear map
\begin{align}
\Psi':S_{\widetilde{A}} &\lra (\lambda_i-x_i)_{i=0,\ldots,n}\backslash\mcm_{\widetilde{A}}^0 \notag \\
\underline{u}^{l_0 \cdot \widetilde{a}_0}\cdot \ldots \cdot \underline{u}^{l_n\cdot\widetilde{a}_n} &\mapsto \p_{\lambda_0}^{l_0}\cdot \ldots \cdot \p_{\lambda_n}^{l_n} \notag
\end{align}
We want to show that this map factors over $S_{\widetilde{A}} / (D_k S_{\widetilde{A}})_{k=0,\ldots,d}$ so that $\Psi'$ descends to a map
\[
\Psi: S_{\widetilde{A}} / (D_k S_{\widetilde{A}})_{k=0,\ldots,d} \lra (\lambda_i-x_i)_{i=0,\ldots,n}\backslash\mcm_{\widetilde{A}}^0
\]
Let $P = \underline{u}^{l_0 \cdot \underline{\tilde{a}}_0}\cdot \ldots \cdot \underline{u}^{l_n \cdot  \underline{\tilde{a}}_n}$, then $\Psi'(P) = \p_{\lambda_0}^{l_0}\cdot \ldots \p_{\lambda_n}^{l_n}$ and
\begin{align}
\Psi'(D_k P) &= \Psi (\sum_{i=0}^n \tilde{a}_{ki} x_i \underline{u}^{\underline{a}_i}P +  \sum_{i=0}^n \tilde{a}_{ki} l_i P) \notag \\
&= \left(\sum_{i=0}^n \tilde{a}_{ki} x_i \p_{\lambda_i} + \sum_{i=0}^n \tilde{a}_{ki}l_i\right) \Psi(P) \notag \\
&= \left(\sum_{i=0}^n \tilde{a}_{ki} \lambda_i \p_{\lambda_i} + \sum_{i=0}^n \tilde{a}_{ki}l_i\right) \Psi(P) \notag \\
&= \Psi(P) \cdot \left(\sum_{i=0}^n \tilde{a}_{ki} \lambda_i \p_{\lambda_i}\right) \notag \\
&= 0 \notag
\end{align}

We will now construct an inverse $\Theta$ to $\Psi$. If $P \in D_V$ is a normally ordered element, we denote by $\overline{P} \in S_{\widetilde{A}}$ the element which is obtained from $P$ by replacing $\lambda_i$ with $x_i$ and $\p_{\lambda_i}$ with $\underline{u}^{l_i \underline{a}_i}$, i.e. if $P = \lambda_0^{k_0}\ldots \lambda_n^{k_n} \p_{\lambda_0}^{l_0}\ldots \p_{\lambda_n}^{l_n}$ the element $\overline{P}$ is given by $x_0^{k_0}\ldots x_n^{k_n} \underline{u}^{l_0\cdot \underline{a}_0}\ldots\underline{u}^{l_n \cdot \underline{a}_n}$. This gives the map

\begin{align}
\Theta': (\lambda_i -x_i)_{i=0,\ldots,n} \backslash D_V &\lra S_{\widetilde{A}} /(D_k S_{\widetilde{A}})_{k=0,\ldots,d} \notag \\
P &\mapsto \overline{P} \notag
\end{align}
 We want to show that $\Theta'$ factors over $(\lambda_i -x_i)_{i=0,\ldots,n} \backslash M^0_A$ so that $\Theta'$ descends to a map
 \begin{align}
\Theta: (\lambda_i -x_i)_{i=0,\ldots,n} \backslash M^0_A &\lra S_{\widetilde{A}} /(D_k S_{\widetilde{A}})_{k=0,\ldots,d} \notag \\
P &\mapsto \overline{P} \notag
\end{align}
We have to show that $\Theta'( P \cdot E_k) = 0$ and $\Theta'( P \cdot \Box_{\underline{l}})= 0 $ for $k= 0,\ldots ,d$ and $l \in \mbl$. We can assume that $P$ is a monomial $ \lambda_0^{j_0}\ldots \lambda_n^{j_n} \p_{\lambda_0}^{l_0}\ldots \p_{\lambda_n}^{l_n}$:
\begin{align}
\Theta'(\lambda_0^{j_0}\ldots \lambda_n^{j_n} \p_{\lambda_0}^{l_0}\ldots \p_{\lambda_n}^{l_n}E_k) &= \Theta'(\lambda_0^{j_0}\ldots \lambda_n^{j_n} \p_{\lambda_0}^{l_0}\ldots \p_{\lambda_n}^{l_n} \left(\sum_{i=0}^n \tilde{a}_{ki} \lambda_i \p_{\lambda_i} \right)) \notag \\
&= \Theta'(\lambda_0^{j_0}\ldots \lambda_n^{j_n} (\sum_{i=0}^n(\tilde{a}_{ki} \lambda_i \p_{\lambda_i} + \tilde{a}_{ki}l_i))\p_{\lambda_0}^{l_0}\ldots \p_{\lambda_n}^{l_n}) \notag \\
&=x_0^{j_0}\ldots x_n^{j_n} (\sum_{i=0}^n(\tilde{a}_{ki} x_i \underline{u}^{\underline{a}_i}+ \tilde{a}_{ki}l_i))\underline{u}^{l_0\cdot \underline{a}_0}\ldots \underline{u}^{l_n\cdot \underline{a}_0} \notag \\
&= (\sum_{i=0}^n(\tilde{a}_{ki} x_i \underline{u}^{\underline{a}_i}+ \tilde{a}_{ki}l_i))\cdot x_0^{j_0}\ldots x_n^{j_n} \cdot \underline{u}^{l_0\cdot \underline{a}_0}\ldots \underline{u}^{l_n\cdot \underline{a}_0} \notag \\
&= (\sum_{i=0}^n(\tilde{a}_{ki} x_i \underline{u}^{\underline{a}_i}+ u_i \p_{u_i}))\cdot x_0^{j_0}\ldots x_n^{j_n} \cdot \underline{u}^{l_0\cdot \underline{a}_0}\ldots \underline{u}^{l_n\cdot \underline{a}_0} \notag \\
&= 0 \notag
\end{align}
and
\begin{align}
\Theta'(\lambda_0^{j_0}\ldots \lambda_n^{j_n} \p_{\lambda_0}^{l_0}\ldots \p_{\lambda_n}^{l_n} \Box_{\underline{l}}) &= \Theta'(\lambda_0^{j_0}\ldots \lambda_n^{j_n} \p_{\lambda_0}^{l_0}\ldots \p_{\lambda_n}^{l_n}\left(\prod_{l_i > 0} \p_{\lambda_i}^{l_i} - \prod_{l_i <0} \p_{\lambda_i}^{-l_i} \right)) \notag \\
&= x_0^{j_0}\ldots x_n^{j_n} \underline{u}^{l_0\cdot \underline{a}_0}\ldots \underline{u}^{l_n\cdot \underline{a}_n}\left(\prod_{l_i > 0} \underline{u}^{l_i\cdot \underline{a}_i} - \prod_{l_i <0} \underline{u}^{-l_i \cdot \underline{a}_i} \right) \notag \\
&= x_0^{j_0}\ldots x_n^{j_n} \underline{u}^{l_0\cdot \underline{a}_0}\ldots \underline{u}^{l_n\cdot \underline{a}_n}\left( \underline{u}^{\sum_{l_i > 0} l_i\cdot \underline{a}_i} -  \underline{u}^{-\sum_{l_i <0} l_i \cdot \underline{a}_i} \right) \notag \\
&= 0 \notag
\end{align}
where the last equality follows from $\sum_{i=0}^n l_i \underline{a}_i = 0$.
This shows that $\Psi$ is an isomorphism. The statement about the Hodge filtration follows from Theorem \ref{thm:HodgeGKZ}, the fact that $x$ is a smooth point of $\mcm^0_{\widetilde{A}}$ and the definition of the inverse image of filtered $\mcd$-modules (cf. \cite{Saito1} chapter 3.5, notice that no shift in the Hodge filtration occurs since we are dealing with left $\mcd$-modules instead of right $\mcd$-modules as in loc. cit.).
\end{proof}

\section{Landau-Ginzburg models and non-commutative Hodge structures}
\label{sec:LGModel}

In this final section we will give a first application of our main result. It is concerned with
Hodge-theoretic properties of differential systems occurring in toric mirror symmetry. More precisely,
we will prove \cite[Conjecture 6.15]{ReiSe2} showing that the so-called \emph{reduced quantum $\cD$-module}
of a nef (also called weak Fano) complete intersection inside a smooth projective toric variety underlies a (variation of) non-commutative
Hodge structure(s). We will recall the necessary notation and results of loc.cit.
and then deduce this conjecture from our main Theorem \ref{thm:HodgeGKZ}. The basic strategy to obtain the proof of the conjecture is to identify the reduced quantum $\cD$-module with an object which is the Fourier-Laplace transform of a filtered $\cD$-module underlying a pure polarized Hodge module. The latter is nothing but the image of the duality morphism  from Corollary \ref{cor:StrictlyFiltered}. Notice that this corollary depends in an essential way on our main Theorem \ref{thm:HodgeGKZ}, since the strictness of the duality morphism with respect to the order filtrations holds only because the latter are (up to a shift) the Hodge filtrations of mixed Hodge modules. The identification
with the reduced quantum $\cD$-module relies on the explicit description of the latter from \cite{MM17} (already used extensively in \cite{ReiSe2}).

Let $\XSig$ be a smooth, projective and toric variety with $\dim_{\dC}(\XSig)=k$. Put
$m:=k+b_2(\XSig)$. Let $\cL_1,\ldots,\cL_l$ be globally generated line bundles
on $\XSig$ (in particular, they are nef according to \cite[Section 3.4]{Fulton}) and assume that $-K_{\XSig}-\sum_{i=1}^l c_1(\cL_i)$ is nef.
Put $\cE:=\oplus_{i=1}^l\cL_i$, and let $\cE^\vee$ be the dual vector bundle. Its total space $\dV(\cE^\vee):=\mathbf{Spec}_{\cO_{\XSig}}(\textup{Sym}_{\cO_{\XSig}}(\cE))$ is a quasi-projective toric variety with defining fan $\Sigma'$. The matrix
$A\in M((k+l)\times (m+l), \dZ)$ whose columns are the primitive integral generators of the rays of $\Sigma'$
then satisfies the conditions in Theorem \ref{thm:Duality}. More precisely, we have $\dZ\widetilde{A}=\dZ^{d+1}$ and
it follows from \cite[Proposition 5.1]{ReiSe2} that the semi-group $\dN \widetilde{A}$ is normal and that
we have $int(\dN\widetilde{A})=\widetilde{c}+\dN\widetilde{A}$, where $\widetilde{c}=\sum_{i=m+1}^{m+l} e_i=(l+1,\underline{0},\underline{1})$, $e_i$ being the
$i$'th standard vector in $\dZ^{1+m+l}$.

The strictly filtered duality morphism $\phi$ from Corollary \ref{cor:StrictlyFiltered} is more concretely given as
$$
\begin{array}{rcl}
\phi:(\cM_{\widetilde{A}}^{-(l+1,\underline{0},\underline{1})},F^{ord}_{\bullet-l-1}) & \longrightarrow & (\cM_{\widetilde{A}}^0,F^{ord}_{\bullet}) \\ \\
P & \longmapsto & P\cdot\partial_{\lambda_0}\cdot\partial_{\lambda_{m+1}}\cdot\ldots\cdot\partial_{\lambda_{m+l}}\, .
\end{array}
$$
\begin{proposition}\label{prop:PureHodge}
The image of $\phi$ underlies a pure Hodge module of weight $m+k+2l$, where the Hodge filtration is given by
$$
F^H_\bullet im(\phi) = im(\phi) \cap F^{ord}_{\bullet+k+l}\cM_{\widetilde{A}}^0.
$$
\end{proposition}

\begin{proof}
This is a consequence of \cite[Theorem 2.16]{ReiSe2} and of Proposition \ref{prop:morphdualGKZ}.
\end{proof}

A main point in the paper \cite{ReiSe2} is to consider the partial localized Fourier transformations of the
GKZ-systems $\cM_{\widetilde{A}}^\beta$. We recall the main construction and refer to \cite[Section 3.1]{ReiSe2} for details
(in particular concerning the definition and properties of the Fourier-Laplace functor $\FL$ and its ``localized'' version $\FL^{loc}$).
Let (as done already in section in \ref{subsec:Reminder}) $\Lambda$ be the affine space $\dC^{m+l}$ with coordinates $\lambda_1,\ldots,\lambda_{m+l}$ (so that $V=\dC_{\lambda_0}\times \Lambda$) and
put $\widehat{V}:=\dC_z\times \Lambda$. Let
$\widehat{\cM}^{(\beta_0,\beta)}_{A}$ be the $\cD_{\widehat{V}}$-module $\cD_{\widehat{V}}[z^{-1}]/ \cI$, where $\cI$ is the left ideal generated
by the operators $\widehat{\Box}_{\underline{l}}$ (for all $l\in\dL_A$), $\widehat{E}_j - \beta_j z$ (for $j=1,\ldots, k+l$) and $\widehat{E}- \beta_0 z$, which are defined by
$$
\begin{array}{rcl}
\widehat{\Box}_{\underline{l}} & := &  \prod\limits_{i:l_i<0} (z \cdot \partial_{\lambda_i})^{-l_i}  -  \prod\limits_{i:l_i>0} (z \cdot \partial_{\lambda_i})^{l_i}\, ,  \\ \\
\widehat{E}   & := & z^2\partial_z +\sum_{i=1}^{m+l} z\lambda_i\partial_{\lambda_i}\, ,  \\ \\
\widehat{E}_j & := & \sum_{i=1}^{m+l} a_{ji}\, z\lambda_i\partial_{\lambda_i}\, .
\end{array}
$$
We denote the corresponding $\mcd_{\widehat{V}}$-module by  $\widehat{\mcm}^{(\beta_0,\beta)}_{A}$.
Then we have (see \cite[Lemma 3.2]{ReiSe2})
$$
\FL^{loc}_\Lambda\left(\cM_{\widetilde{A}}^{(\beta_0,\beta)}\right)=\widehat{\cM}_A^{(\beta_0+1,\beta)}.
$$
Consider the filtration on $\cD_{\widehat{V}}$
for which $z$ has degree $-1$, $\partial_z$ has degree $2$ and $\deg(\lambda_i)=0$, $\deg(\partial_{\lambda_i})=1$.
Write $\MF^z(\cD_{\widehat{V}})$ for the category of well-filtered $\cD_{\widehat{V}}$-modules (that is, $\cD_{\widehat{V}}$-modules
equipped with a filtration compatible with the filtration on $\cD_{\widehat{V}}$ just described and such that
the corresponding Rees module is coherent over the corresponding Rees ring). Denote by
$G_\bullet$ the induced filtrations on the module $\widehat{\cM}_{A}^{(\beta_0,\beta)}$, which are $\cR_{\dC_z\times \Lambda}$-modules.
We have
$$
G_0 \widehat{\cM}^{(\beta_0,\beta)}_A = \cR_{\dC_z\times \Lambda}/\cR_{\dC_z\times \Lambda}(\widehat{\Box}_{\underline{l}})_{l\in\dL_A}+\cR_{\dC_z\times \Lambda} \widehat{E} + \cR_{\dC_z\times \Lambda}(\widehat{E}_j)_{k=1,\ldots,k+l}
$$
and $G_k \widehat{\cM}^{(\beta_0,\beta)}_A = z^k\cdot G_0 \widehat{\cM}^{(\beta_0,\beta)}_A$
In general, the modules $\widehat{\cM}^{(\beta_0,\beta)}_A$ and their filtration steps may be quite complicated.
However, we have considered in \cite{ReiSe2} their restriction to a specific Zariski open subset $\Lambda^\circ \subset \left(\Lambda \backslash \bigcup_{i=1}^{m+l}\{w_i=0\}\right)\subset \Lambda$ (called $W^\circ$ in
\cite[Remark 3.8]{ReiSe2}),
which contains the critical locus of the familiy of Laurent polynomials associated with the matrix $A$ (but
excludes certain singularities at infinity of this family). Denote by ${^\circ\!}\widehat{\cM}_A^{(\beta_0,\beta)}$ the restriction
$(\widehat{\cM}_A^{(\beta_0,\beta)})_{|\dC_z \times \Lambda^\circ}$ together with the induced filtration
$G_\bullet {^\circ\!}\widehat{\cM}_A^{(\beta_0,\beta)}$. Then $G_k {^\circ\!}\widehat{\cM}_A^{(\beta_0,\beta)}$ is $\cO_{\dC_z\times \Lambda^\circ}$-locally
free for all $k$. Moreover, the multiplication by $z$ is invertible on $\widehat{\cM}_A^{(\beta_0,\beta)}$, filtered
with respect to $G_\bullet$ (shifting the filtration by one) and so is its inverse. Hence, we have a strict morphism
$$
\cdot z: ({^\circ\!}\widehat{\cM}_A^{(\beta_0,\beta)}, G_\bullet) \longrightarrow ({^\circ\!}\widehat{\cM}_A^{(\beta_0-1,\beta)}, G_{\bullet+1}) .
$$

We also need a slightly modified version of the Fourier-Laplace transformed GKZ-systems. More precisely,
define the modules ${^\circ\!}\widehat{\cN}_A^\beta$ as the cyclic quotients
of $\cD_{\dC_z\times \Lambda^\circ}[z^{-1}]$ by the left ideal generated by $\widetilde{\Box}_{\underline{l}}$ for $\underline{l}\in\dL_A$ and $\widehat{E}_j-z\beta_j$ for $j=0,\ldots,k+c$, where
$$
\begin{array}{rcl}
\widetilde{\Box}_{\underline{l}} &:= &
\prod\limits_{i\in\{1,\ldots,m\}:\;l_i>0} \lambda_i^{l_i}(z\cdot \partial_i)^{l_i}\prod\limits_{i\in\{m+1,\ldots,m+l\}:\;l_i>0}
\prod\limits_{\nu=1}^{l_i} (\lambda_i(z\cdot\partial_i) -z\cdot \nu) \\ \\
&& -
\prod\limits_{i=1}^{m+l}\lambda_i^{l_i} \cdot\prod\limits_{i\in\{1,\ldots,m\}:\;l_i<0} \lambda_i^{-l_i}(z \cdot \partial_i)^{-l_i}\prod\limits_{i\in\{m+1,\ldots,m+l\}:\;l_i<0}
\prod\limits_{\nu=1}^{-l_i} (\lambda_i(z\cdot\partial_i) -z\cdot \nu).
\end{array}
$$

Consider the \emph{invertible} morphism
\begin{equation}\label{eq:IsoMN}
\Psi: {^\circ\!}\widehat{\cN}^{(0,\underline{0},\underline{0})}_{A} \longrightarrow {^\circ\!}\widehat{\cM}^{-(2l,\underline{0},\underline{1})}_{A}
\end{equation}
given by right multiplication with $z^l\cdot\prod_{i=m+1}^{m+l} \lambda_i$ (recall that $\lambda_i\neq 0$ on $\Lambda^\circ$).
We define $\widetilde{\phi}$ to be the composition $\widetilde{\phi} := \widehat{\phi}\circ\Psi$,
where $\widehat{\phi}$ is the morphism
$$
\widehat{\phi}:  {^\circ\!}\widehat{\cM}_A^{-(2l,\underline{0},\underline{1})} \longrightarrow {^\circ\!}\widehat{\cM}_A^{(-l,\underline{0},\underline{0})},
$$
given by right multiplication with $\partial_{\lambda_{m+1}}\cdot\ldots\cdot\partial_{\lambda_{m+l}}$. In concrete terms, we have:
$$
\begin{array}{rcl}
\widetilde{\phi}:{^\circ\!}\widehat{\cN}_{A}^{(0,\underline{0},\underline{0})} & \longrightarrow & {^\circ\!}\widehat{\cM}_{A}^{(-l,\underline{0},\underline{0})}\, , \\ \\
x & \longmapsto & \widehat{\phi}(x\cdot z^l\cdot\lambda_{m+1}\cdot\ldots\cdot\lambda_{m+l}) =
x\cdot (z\lambda_{m+1}\partial_{m+1})\cdot\ldots\cdot(z\lambda_{m+l}\partial_{m+l}).
\end{array}
$$
We have an induced filtration $G_\bullet {^\circ\!}\widehat{\cN}_A^{(0,\underline{0},\underline{0})}$ which satisfies
$$
G_0 {^\circ\!}\widehat{\cN}_A^{(0,\underline{0},\underline{0})} = \cR_{\dC_z\times \Lambda^\circ} /
\cR_{\dC_z\times \Lambda^\circ}(\widetilde{\Box}_{\underline{l}})_{\underline{l}\in\dL_A}+
\cR_{\dC_z\times \Lambda^\circ}(\widehat{E}_j-z\beta_j)_{j=0,\ldots,m+l}
$$
and $G_k {^\circ\!}\widehat{\cN}_A^{(0,\underline{0},\underline{0})} = z^k\cdot G_0 {^\circ\!}\widehat{\cN}_A^{(0,\underline{0},\underline{0})}$

In order to obtain the lattices $G_\bullet$ we need to extend the functor $\FL^{loc}_\Lambda$ to the category of filtered $\cD$-modules.
\begin{definition}
Let $(\cM, F_\bullet) \in \MF(\cD_V)=\MF(\cD_{\dC_{\lambda_0}\times \Lambda})$. Define $\cM[\partial_{\lambda_0}^{-1}]:=\cD_V[\partial_{\lambda_0}^{-1}]\otimes_{\cD_V}\cM$
and consider the natural localization morphism $\widehat{\textup{loc}}: \cM \rightarrow
\cM[\partial_{\lambda_0}^{-1}]$.
We define the saturation of $F_\bullet$ to be
\begin{equation}\label{eq:Saturation}
F_k \cM[\partial_{\lambda_0}^{-1}] := \sum_{j\geq 0} \partial_{\lambda_0}^{-j} \widehat{\textup{loc}}\left(F_{k+j} \cM\right)\, .
\end{equation}
and we denote by $G_\bullet \widehat{\cM}$ the filtration induced from
$F_k \cM[\partial_{\lambda_0}^{-1}]$ on $\widehat{\cM}:=\FL^{loc}_\Lambda(M) \in M_h(\cD_{\widehat{V}})=M_h(\cD_{\dC_z\times \Lambda})$.
Notice that for $(\cM, F_\bullet)=(\widehat{\cM}_A^{(\beta_0,\beta)},F^{ord}_\bullet)$, the two definitions of $G_\bullet$ coincide: As we have
$$
F^{ord}_k \cM^{(\beta_0-1,\beta)}_{A}[\partial_{\lambda_0}^{-1}] = im\left(\partial_{\lambda_0}^k\dC[\lambda_0,\lambda_1,\ldots,\lambda_{m+l}]
\langle\partial_{\lambda_0}^{-1},\partial_{\lambda_0}^{-1}\partial_{\lambda_1},\ldots,\partial_{\lambda_0}^{-1}\partial_{\lambda_{m+l}}\rangle\right) \textup{  in  }\cM^{(\beta_0-1,\beta)}_{\widetilde{A}}[\partial_{\lambda_0}^{-1}],
$$
the filtration induced by $F^{ord}_k \cM^{(\beta_0-1,\beta)}_{\widetilde{A}}[\partial_{\lambda_0}^{-1}]$ on $\widehat{\cM}^{(\beta_0,\beta)}_{A}$ is precisely
$G_k\widehat{\cM}^{(\beta_0,\beta)}_{A}$.

We denote by $(\FL^{loc}_\Lambda, \textup{Sat})$ the induced functor from the category $\MF(\cD_V)$ to the category
$\MF^z(\cD_{\widehat{\Lambda}})$ which sends $(\cM,F_\bullet)$ to $(\widehat{\cM},G_\bullet)$.
\end{definition}

From the above duality considerations, we deduce the following result.
\begin{proposition}\label{prop:StrictG}
The morphism
$$
\widetilde{\phi}:{^\circ\!}\widehat{\cN}_{A}^{(0,\underline{0},\underline{0})} \longrightarrow  {^\circ\!}\widehat{\cM}_{A}^{(-l,\underline{0},\underline{0})}
$$
is strict with respect to the filtration $G_\bullet$, in particular, we have
$$
\widetilde{\phi}\left(G_0 {^\circ\!}\widehat{\cN}_A^{(0,\underline{0},\underline{0})}\right)
=G_0 {^\circ\!}\widehat{\cM}_A^{(-l,\underline{0},\underline{0})} \cap \textup{im}(\widetilde{\phi})
$$
Moreover, the object $\left(im(\widetilde{\phi}), G_\bullet\right)$ is obtained via the functor
$(\FL^{loc}_\Lambda, \textup{Sat})$ from $\left(im(\phi), F^H_\bullet = F^{ord}_{\bullet+k+l}\right)$, which underlies
a pure Hodge module of weight $m+k+2l$ by
Proposition \ref{prop:PureHodge}.
\end{proposition}
\begin{proof}
The morphism $\Psi$ is invertible, filtered (shifting the filtration by $-l$) and its inverse is also filtered. Hence
it is strict. Therefore the strictness of $\widetilde{\phi}$ follows from the strictness of $z\widehat{\phi}$.
We will deduce it from the strictness property of the morphism $\phi$ in Corollary \ref{cor:StrictlyFiltered}.

 Notice that the morphism $\widehat{\phi}$ is obtained from $\phi$ by linear extension
in $\partial_{\lambda_0}^{-1}$. Recall that the morphism
$$
\phi:(\cM_{\widetilde{A}}^{-(l+1,\underline{0},\underline{1})},F^{ord}_{\bullet}) \longrightarrow  (\cM_{\widetilde{A}}^0,F^{ord}_{\bullet+l+1})
$$
was strict, hence  equation
\eqref{eq:Saturation} yields the strictness of
$$
\widehat{\phi}:(\widehat{\cM}_A^{-(2l,\underline{0},\underline{1})},G_{\bullet}) \longrightarrow (\widehat{\cM}_A^{(-l,\underline{0},\underline{0})},G_{\bullet+l})
$$
Finally, as already noticed above, this yields the strictness of
$$
\widetilde{\phi}=\widehat{\phi}\circ\Psi:(\widehat{\cN}_{A}^{(0,\underline{0},\underline{0})},G_{\bullet}) \longrightarrow (\widehat{\cM}_A^{(-l,\underline{0},\underline{0})},G_{\bullet}).
$$
\end{proof}
The next corollary is now a direct consequence of \cite[Corollary 3.15]{Sa8}.
\begin{corollary}\label{cor:ncHodge1}
The free $\cO_{\dC_z\times \Lambda^\circ}$-module
$G_0{^\circ\!}\widehat{\cM}^{(-l,\underline{0},\underline{0})}_A \cap im(\widetilde{\phi})$ underlies a variation of pure polarized non-commutative
Hodge structures on $\Lambda^\circ$ (see \cite{SaNcHodgeSurvey} for a detailed discussion of this notion).
\end{corollary}

The main result in \cite{ReiSe2} concerns a mirror statement for several quantum $\cD$-modules
which are associated with the toric variety $\XSig$ and the split vector bundle $\cE.$ In particular, one can consider
the reduced quantum $\cD$-module $\overline{\QDM}(\XSig,\cE)$ which is a vector bundle on
$\dC_z\times H^0(\XSig,\dC)\times B^*_\varepsilon$, where $B^*_\varepsilon:=\{q\in (\dC^*)^{b_2(\XSig)},|\,0<|q|<\varepsilon\}$ together with a flat connection
$$
\nabla: \overline{\QDM}(\XSig,\cE) \rightarrow \overline{\QDM}(\XSig,\cE)\otimes_{\cO_{\dC_z\times H^0(\XSig,\dC)\times B^*_\varepsilon}}
z^{-1}\Omega^1_{\dC_z\times H^0(\XSig,\dC)\times B^*_\varepsilon} \left(\log(\{0\}\times H^0(\XSig,\dC)\times B^*_\varepsilon)\right).
$$
We refer to \cite{MM17} for a detailed discussion of the definition of $\overline{\QDM}(\XSig,\cE)$, a short version
can be found in \cite[Section 4.1]{ReiSe2}.
Notice that in loc.cit., $\overline{\QDM}(\XSig,\cE)$ is defined on some larger set, but
in mirror type statements only its restriction to $H^0(\XSig,\dC)\times\dC_z\times B^*_\varepsilon$ is considered.
In the sequel, we will need to consider a Zariski open subset of $\KM \subset (\dC^*)^{b_2(\XSig)}$ which contains $B^*_\varepsilon$.
We recall the main result from \cite{MM17}, which gives a GKZ-type description
of $\overline{\QDM}(\XSig,\cE)$. We present it in a slightly different form, taking into account
\cite[Proposition 6.9]{ReiSe2}. Let $\cR_{\dC_z\times \cK\mcm^0}$ be the sheaf of Rees rings on $\dC_z\times \KM$, and
$R_{\dC_z\times \cK\mcm^0}$ its module of global sections. If we write $q_1, \ldots, q_r$ for the coordinates on $(\dC^*)^r$ (with
$r:=b_2(\XSig)$, then $R_{\dC_z\times \cK\mcm^0}$) is generated by $z q_i\partial_{q_i}$ and $z^2\partial_z$ over $\cO_{\dC_z\times \KM}$.
\begin{theorem}\label{thm:MirMM17}
For any $\cL\in \textup{Pic}(\XSig)$, write $\widehat{\cL}\in R_{\dC_z\times \KM}$ for the associated ``quantized operator''
as defined in \cite[Notation 4.2.]{MM17} or \cite[Theorem 6.7]{ReiSe2}.
Define the left ideal $J$ of $R_{\mbc_z \times \mck\mcm^\circ}$ by
$$
J:=R_{\dC_z\times \KM}(Q_{\underline{l}})_{\underline{l}\in\dL_{A'}}+R_{\dC_z\times \KM}\cdot \widehat{E}\, ,
$$
where
$$
\begin{array}{rcl}
Q_{\underline{l}} & := &
\prod\limits_{i\in\{1,\ldots,m\}:l_i>0}\prod\limits_{\nu=0}^{l_i-1}\left(\widehat{\cD}_i-\nu z\right)
\prod\limits_{j\in\{1,\ldots,c\}:l_{m+j}>0}\prod\limits_{\nu=1}^{l_{m+l}}\left(\widehat{\cL}_j+\nu z\right)\\ \\
& - & \underline{q}^{\underline{l}}\cdot
\prod\limits_{i\in\{1,\ldots,m\}:l_i<0}\prod\limits_{\nu=0}^{-l_i-1}\left(\widehat{\cD}_i-\nu z\right)
\prod\limits_{j\in\{1,\ldots,c\}:l_{m+j}<0}\prod\limits_{\nu=1}^{-l_{m+l}}\left(\widehat{\cL}_j+\nu z\right)\, ,\\ \\
\widehat{E} & := & z^2\partial_z-\widehat{K}_{\dV(\cE^\vee)}\, .
\end{array}
$$
Here we write $\cD_i\in\textup{Pic}(\XSig)$ for a line bundle associated with the torus invariant divisor
$D_i$, where $i=1,\ldots,m$.
Let $K\subset R_{\dC_z\times \KM}$ be the ideal
$$
K:=\left\{P\in R_{\dC_z\times \KM}\,|\, \exists p\in \dZ, k\in \dN: \prod_{i=0}^k\prod_{j=1}^c (\widehat{\cL}+p+i)P\in J \right\}
$$
and $\cK$ the associated sheaf of ideals in $\cR_{\dC_z\times \KM}$.

Suppose as above that the bundle $-K_{\XSig}-\sum_{j=1}^l \cL_j$ is nef, and moreover that each individual bundle $\cL_j$ is ample.
Then there
is a map $\textup{Mir}: B^*_\varepsilon \rightarrow H^0(\XSig,\dC)\times B^*_\varepsilon$ such that we have an isomorphism
of $\cR_{\dC_z\times B^*_\varepsilon}$-modules
$$
\left(\cR_{\dC_z\times \KM}/\cK\right)_{|\dC_z\times B^*_\varepsilon} \stackrel{\cong}{\longrightarrow} \left(\id_{\dC_z}\times \textup{Mir}\right)^* \overline{\QDM}(\XSig,\cE).
$$
\end{theorem}

In order to relate the quantum $\cD$-module $\overline{\QDM}(\XSig,\cE)$ with our results on GKZ-systems,
we will use the restriction map $\overline{\rho}:\KM\hookrightarrow \Lambda$ as constructed in \cite{ReiSe2} (discussion before Definition
6.3. in loc.cit.). Then it follows from the results of loc.cit., Proposition 6.10, that we have an isomorphism of $\cR_{\dC_z\times \KM}$-modules
$$
\cR_{\dC_z\times \KM}/\cK\cong \left(\id_{\dC_z}\times\overline{\rho}\right)^*\left(\widetilde{\phi}\left(G_0 {^\circ\!}\widehat{\cN}_A^{(0,\underline{0},\underline{0})}\right)\right)
$$

Now we can deduce from Corollary \ref{cor:ncHodge1} the main result of this section.
\begin{theorem}\label{thm:ncHodge2}
Consider the above situation of a $k$-dimensional toric variety $\XSig$, globally generated line bundles
$\cL_1,\ldots,\cL_l$ such that $-K_{\XSig}-\cE$ is nef, where $\cE=\oplus_{j=1}^l\cL_j$, with $\cL_j$ ample for $j=1,\ldots,l$. Then the smooth $\cR_{\dC_z\times\KM}$-module
$(\id_{\dC_z}\times \textup{Mir})^*\overline{\QDM}(\XSig,\cE)$ (i.e., the vector bundle over
$\dC_z\times\KM$ together with its connection operator $\nabla$)
underlies a variation of pure polarized non-commutative Hodge structures.
\end{theorem}
\begin{proof}
The strictness of $\widetilde{\phi}$ as shown in Proposition \ref{prop:StrictG} shows that
$G_0\widehat{\cM}^{(-l,\underline{0},\underline{0})}_A \cap im(\widetilde{\phi}) = \widetilde{\phi}\left(G_0\cN^{(0,\underline{0},\underline{0})}\right)$,
hence, by Corollary \ref{cor:ncHodge1}, the module $\widetilde{\phi}\left(G_0\cN^{(0,\underline{0},\underline{0})}\right)$ underlies a variation of pure polarized
non-commutative Hodge structures on $\Lambda^0$. Hence the assertion follows from the mirror statement of Theorem \ref{thm:MirMM17}.
\end{proof}

\bibliographystyle{amsalpha}
\def\cprime{$'$}
\providecommand{\bysame}{\leavevmode\hbox to3em{\hrulefill}\thinspace}
\providecommand{\MR}{\relax\ifhmode\unskip\space\fi MR }
\providecommand{\MRhref}[2]{  \href{http://www.ams.org/mathscinet-getitem?mr=#1}{#2}
}
\providecommand{\href}[2]{#2}

\vspace*{1cm}

\nd
Thomas Reichelt\\
Mathematisches Institut \\
Universit\"at Heidelberg\\
69120 Heidelberg\\
Germany\\
treichelt@mathi.uni-heidelberg.de

\vspace*{1cm}

\nd
Christian Sevenheck\\
Fakult\"at f\"ur Mathematik\\
Technische Universit\"at Chemnitz\\
09107 Chemnitz\\
Germany\\
christian.sevenheck@mathematik.tu-chemnitz.de

\end{document}